\documentclass[oneside,english]{amsart}
\usepackage[LGR,T1]{fontenc}
\usepackage[latin9]{inputenc}
\usepackage{prettyref}
\usepackage{textcomp}
\usepackage{amstext}
\usepackage{amsthm}
\usepackage{amssymb}

\makeatletter

\DeclareRobustCommand{\greektext}{%
  \fontencoding{LGR}\selectfont\def\encodingdefault{LGR}}
\DeclareRobustCommand{\textgreek}[1]{\leavevmode{\greektext #1}}
\ProvideTextCommand{\~}{LGR}[1]{\char126#1}

\numberwithin{equation}{section}
\numberwithin{figure}{section}
\theoremstyle{plain}
\newtheorem{thm}{\protect\theoremname}[section]
  \theoremstyle{definition}
  \newtheorem{defn}[thm]{\protect\definitionname}
  \theoremstyle{plain}
  \newtheorem{conjecture}[thm]{\protect\conjecturename}
  \theoremstyle{plain}
  \newtheorem{prop}[thm]{\protect\propositionname}
  \theoremstyle{remark}
  \newtheorem{rem}[thm]{\protect\remarkname}
  \theoremstyle{plain}
  \newtheorem{cor}[thm]{\protect\corollaryname}
  \theoremstyle{plain}
  \newtheorem{lem}[thm]{\protect\lemmaname}
  \theoremstyle{definition}
  \newtheorem{example}[thm]{\protect\examplename}

\usepackage{amsthm, amssymb, amsmath, color, marvosym,
rotating, graphicx, mathrsfs}

\usepackage{epsfig}
\usepackage{amscd}

{
      \theoremstyle{plain}
      \newtheorem{assumption}{Assumption}
  }

\makeatletter
\newcommand*\leftdash{\rotatebox[origin=c]{-45}{$\dabar@\dabar@\dabar@$}}
\newcommand*\rightdash{\rotatebox[origin=c]{45}{$\dabar@\dabar@\dabar@$}}
\makeatother

\makeatother

\usepackage{babel}
  \providecommand{\conjecturename}{Conjecture}
  \providecommand{\corollaryname}{Corollary}
  \providecommand{\definitionname}{Definition}
  \providecommand{\examplename}{Example}
  \providecommand{\lemmaname}{Lemma}
  \providecommand{\propositionname}{Proposition}
  \providecommand{\remarkname}{Remark}
\providecommand{\theoremname}{Theorem}

\begin{document}

\title{The $p$-Cycle of Holonomic $\mathcal{D}$-modules and Quantization
of Exact Algebraic Lagrangians}

\author{Christopher Dodd}
\begin{abstract}
Let $X=\mathbb{A}^{n}$ be complex affine space, and let $T^{*}X$
be its cotangent bundle. For any exact Lagrangian $L\subset T^{*}X$,
we define a new invariant, $A$, living in $\text{Div}_{\mathbb{Q}/\mathbb{Z}}(L)$.
We call this invariant the monodromy divisor of $L$. We conjecture
that the existence of a finite order character of $\pi_{1}(L)$ whose
monodromy is exactly $A$ defines an obstruction to attaching a holonomic
$\mathcal{D}_{X}$-module $M$ associated to $L$ - here, the association
goes via positive characteristic and $p$-supports. In the case where
$\mathbb{H}_{dR}^{1}(L)=0$, we prove this conjecture, and then go
on to show that the set of such holonomic $\mathcal{D}_{X}$-modules
forms a torsor over the group of finite order characters of $\pi_{1}$.
This proves a version of a conjecture of Kontsevich. As a consequence,
we deduce that the group of Morita autoequivalences of the $n$-th
Weyl algebra is isomorphic to the group of symplectomorphisms of $T^{*}\mathbb{A}^{n}$.
This generalizes an old theorem of Dixmier (in the case $n=1$) and
settles a conjecture of Belov-Kanel and Kontsevich in general. 
\end{abstract}

\maketitle
\tableofcontents{}

\section{Introduction}

It is an overarching theme of symplectic geometry that one should
study a symplectic manifold via the geometry its Lagrangian submanifolds.
Often, this takes the form of associating to a symplectic manifold
$(X,\omega)$ a category $\mathcal{C}$, linear over some field $\mathbb{K}$,
whose objects are Lagrangian submanifolds (possibly equipped with
extra structure) and whose morphisms are determined by the intersection
theory between Lagrangians (suitably interpreted). 

Of course, this is easier said than done, and the details of the construction
of $\mathcal{C}$ will vary substantially, depending on the precise
set-up in which one works. If $X$ is a $\mathcal{C}^{\infty}$-manifold,
then the Fukaya category $\mathcal{F}uk(X)$ is the traditional answer
for the corresponding category; c.f. e.g., \cite{key-51}, \cite{key-50}.
In addition, one can also consider categories build out of micro-local
sheaves (as in \cite{key-52},\cite{key-53}, \cite{key-54}), or
out of modules over a suitable ring of pseudo-differential operators
(as in \cite{key-55}). Generally speaking, an object in the relevant
category attached to a Lagrangian is called a ``quantization'' of
it. Arbitrary elements of the category $\mathcal{C}$ define not just
smooth Lagrangian submanifolds, but rather Lagrangian cycles (e.g.,
formal sums over $\mathbb{K}$ of Lagrangian submanifolds). 

Despite the differences between these approaches, several common themes
do emerge. In order to construct $\mathcal{C}$ in general, one needs
to assume that $\mathbb{K}$ is a very large ring- often, it is the
Novikov field of formal power series with real exponents. However,
the situation simplifies substantially if $X$ is an exact symplectic
manifold; i.e., there is a one form $\theta$ such that $d\theta=\omega$.
In fact, for the purposes of this paper, one may suppose that $X$
is a cotangent bundle. In this case, one can consider the class of
exact Lagrangians; i.e., Lagrangians $L$ for which $\theta|_{L}$
is an exact one-form. In this setting, one can often construct $\mathcal{C}$
over a much smaller ring of coefficients, such as $\mathbb{R}$ (or
even $\mathbb{Q}$ or $\mathbb{Z}$), in the case of the Fukaya category
this is discussed in the introduction to \cite{key-51}. The situation
simplified further if we place certain cohomological conditions on
$L$. In the $\mathcal{C}^{\infty}$ case, a typical demand is that
$L$ is orientable, admits a spin structure, and that the Mazlov class
$m_{L}$ (which is a map $\pi_{1}(L)\to\mathbb{Z}$ which depends
on the embedding $L\subset X$) vanishes. In case this happens, then
we should expect that each local system on $L$ gives rise to a distinct
object of $\mathcal{C}$. 

Now suppose that $X$ is a complex algebraic or analytic variety,
equipped with a complex algebraic (or holomorphic) symplectic form.
The above constructions are not directly available. However, there
is a robust theory of deformation quantization of algebraic varieties,
which is the analogue of \cite{key-55} in this situation. In particular,
one may find an (often canonical; c.f. \cite{key-72}) sheaf of flat
$\mathbb{C}[[h]]$-algebras $\mathcal{A}_{h}$ which satisfies $\mathcal{A}_{h}/h\tilde{=}\mathcal{O}_{X}$.
It is crucial here that $\mathcal{A}_{h}$ be non-commutative, for
we demand that the bracket on $\mathcal{A}_{h}$ recover the Poisson
structure on $\mathcal{O}_{X}$ via the formula 
\[
\{\overline{f},\overline{g}\}\equiv\frac{1}{h}[f,g]\phantom{i}\text{mod}\phantom{}(h)
\]
for any local sections $f,g\in\mathcal{A}_{h}$ (here $\bar{f}$ denotes
the image of $f$ in $\mathcal{O}_{X}$, and similarly for $g$).
Then one may expect that a suitable derived category of modules over
$\mathcal{A}_{h}$ is an algebraic analogue of the kind of category
we considered above; these matters have been taken up extensively
in the monograph \cite{key-59}, where, essentially, they confirm
this expectation for complex analytic symplectic varieties (c.f. also
\cite{key-60} for the basic result on quantization of Lagrangians
in this context). 

In the algebraic context, the analogous result appears in the paper
\cite{key-58}. Suppose, for simplicity that our symplectic variety
is of the form $T^{*}X_{\mathbb{C}}$ for a smooth affine\footnote{In fact, there is a version of their result when $X_{\mathbb{C}}$
is not affine, but the cohomological obstruction is slightly more
complicated to state, c.f., \cite{key-58}, theorem 1.1.4} complex variety $X_{\mathbb{C}}$. Then a natural quantization (and
the one of interest in this paper) is the sheaf of $h$-differential
operators $\mathcal{D}_{h}$. If $L\subset T^{*}X$ is a smooth Lagrangian,
then we are looking for a $\mathbb{C}[[h]]$-flat $\mathcal{D}_{h}$-module
$\mathcal{L}_{h}$ such that $\mathcal{L}_{h}/h$ is a line bundle
$\mathcal{L}$ on $L$. The main theorem of \cite{key-58} in this
context is that such an $\mathcal{L}_{h}$ exists iff $c_{1}(\mathcal{L})=\frac{1}{2}(c_{1}(\omega_{L/X}))$
inside $\mathbb{H}_{dR}^{2}(L)$. In this case, the set of such modules
is a torsor over the group of isomorphism classes of $(\mathcal{O}_{L}[[h]])^{*}$-torsors
on $L$ equipped with a flat connection (this the analogue of the
set of rank $1$ local systems in this context). 

In the remarkable paper \cite{key-4}, Kontsevich has introduced a
totally different perspective on the algebraic version of this question.
We let $X_{\mathbb{C}}$ be a smooth affine complex algebraic variety,
and $T^{*}X_{\mathbb{C}}$ its cotangent bundle. The first insight
here is that there is already a well defined category which can be
thought of as ``quantizing'' the geometry of $T^{*}X_{\mathbb{C}}$:
namely, the category of holonomic $\mathcal{D}_{X_{\mathbb{C}}}$-modules.
The question is how to construct a correspondence between holonomic
$\mathcal{D}_{X_{\mathbb{C}}}$-modules and suitable Lagrangians.
The main idea of \cite{key-4} is to do this via reduction to positive
characteristic. 

Let $\mathcal{M}_{\mathbb{C}}$ be a holonomic $\mathcal{D}_{X_{\mathbb{C}}}$-module.
We may select a commutative ring $R$, which is a finitely generated
$\mathbb{Z}$-algebra, and a scheme $X$ which is smooth over $S=\mbox{Spec}(R)$,
such that $X\times_{S}\mbox{Spec}(\mathbb{C})\tilde{=}X_{\mathbb{C}}$.
Localizing $R$ if necessary, one may also find (c.f. \prettyref{subsec:Notations-and-Conventions}
below) a sheaf $\mathcal{M}$ of $\mathcal{D}_{X}$-modules, flat
over $R$, such that $\mathcal{M}\otimes_{R}\mathbb{C}=\mathcal{M}_{\mathbb{C}}$.
We call such $\mathcal{M}$ an $R$-model for $\mathcal{M}_{\mathbb{C}}$. 

Then, for any algebraically closed $k$ of positive characteristic,
and any point $\mbox{Spec}(k)\to S$, one may base change to obtain
a $\mathcal{D}_{X_{k}}$ module $\mathcal{M}_{k}$. Now, $\mathcal{D}_{X_{k}}$
is is a sheaf of algebras with a very large center; more precisely,
we have an isomorphism $\mathcal{Z}(\mathcal{D}_{X_{k}})\tilde{=}\mathcal{O}_{T^{*}X_{k}^{(1)}}$,
and $\mathcal{D}_{X_{k}}$ is Azumaya over $\mathcal{Z}(\mathcal{D}_{X_{k}})$
(by {[}BMR{]}, chapter 2)\footnote{Here $T^{*}X_{k}^{(1)}$ is the Frobenius twist of $T^{*}X_{k}$}.
By taking the support of $\mathcal{M}_{k}$, considered as a sheaf
on $T^{*}X_{k}^{(1)}$, counted with appropriate multiplicities, one
associates to $\mathcal{M}_{k}$ a cycle in $T^{*}X_{k}^{(1)}$, called
the $p$-cycle. It was conjectured in \cite{key-4}, conjecture $2$,
and proved by T. Bitoun (\cite{key-7}) that the variety underlying
this cycle, called the $p$-support, is Lagrangian, for all $p>>0$.
In {[}Ko{]}, Kontsevich goes on to outline a program in which these
Lagrangian subvarieties can be fit together to form some kind of cycle
for $\mathcal{M}_{\mathbb{C}}$; which he calls the ``arithmetic
support.'' Furthermore, he conjectures that this cycle provides,
in various senses, a very refined invariant of $\mathcal{M}_{\mathbb{C}}$.
This arithmetic support is no longer a Lagrangian cycle inside $T^{*}X_{\mathbb{C}}$,
but rather lives over a certain very large ring $\mathbb{C}_{\infty}$
defined by putting together all of the reductions mod $p$ of all
of the rings $R$ over which $\mathcal{M}_{\mathbb{C}}$ admits a
model (c.f. \cite{key-4}, section 2.2 for details). 

Natural questions arising out this set-up, then, are: which Lagrangian
cycles are the arithmetic supports of holonomic $\mathcal{D}$-modules? 

When does the arithmetic support live over a smaller field (than $\mathbb{C}_{\infty}$)? 

Can one characterize the set of $\mathcal{D}$-modules with the same
arithmetic support?

To make the second question more rigorous, we give, following \cite{key-4},
the
\begin{defn}
\label{def:Const-arith-supp}1) Let $L_{\mathbb{C}}\subset T^{*}X_{\mathbb{C}}$
be a Lagrangian subvariety. After choosing an appropriate finitely
generated ring $R$ we may suppose $L_{\mathbb{C}}\tilde{=}L\times_{S}\mbox{Spec}(\mathbb{C})$
for some $L\subset T^{*}(X/S)$ which is smooth over $S=\text{Spec}(R)$.
Then, after base changing to $\mbox{Spec}(k)$, for an algebraically
closed field $k$ of positive characteristic for which $R\to k$,
we can look at $L_{k}^{(1)}\subset T^{*}X_{k}^{(1)}$. 

Let $\mathcal{M}$ be an $R$-model for $\mathcal{M}_{\mathbb{C}}$.
We say that $\mathcal{M}_{\mathbb{C}}$ has \emph{constant arithmetic
support equal to $L_{\mathbb{C}}$ }if, for any such $\mathcal{M}$,
$R$ as above, the $p$-support of $\mathcal{M}_{k}$ is equal to
$L_{k}^{(1)}$ for all $k$ of characteristic $p>>0$. 

2) In this situation, let $\{L_{\mathbb{C},i}\}_{i=1}^{s}$ be the
components of $L_{\mathbb{C}}$. We say the multiplicity of $\mathcal{M}_{\mathbb{C}}$
(along $L_{\mathbb{C},i}$) is $r$ if $\mathcal{M}_{k}$ is generically
a vector bundle of rank $p^{r}$ on $L_{k,i}$ for $p>>0$. 
\end{defn}

It is not hard to see that if these conditions hold for one such $R$,
and one particular $R$-model, they hold for all of them. So we can
refine the last two questions by asking: for a given Lagrangian subvariety
$L_{\mathbb{C}}$, is there a holonomic $\mathcal{D}$-module which
has constant arithmetic support equal to $L_{\mathbb{C}}$ (say, with
fixed multiplicities)? Is there a parametrization of the set of such? 

Before discussing this question further, we should note that, in a
special case, there is a precise conjecture of Kontsevich (this is
\cite{key-4}, conjecture $5$) about this question: 
\begin{conjecture}
Suppose that $L_{\mathbb{C}}$ is a smooth Lagrangian and the singular
homology $H_{1}^{sing}(L_{\mathbb{C}},\mathbb{Z})=0$. Then there
is a unique holonomic $\mathcal{D}_{X_{\mathbb{C}}}$-module $\mathcal{M}_{\mathbb{C}}$
with constant arithmetic support equal to $L_{\mathbb{C}}$, with
multiplicity $1$. We further have $\text{Ext}^{1}(\mathcal{M}_{\mathbb{C}},\mathcal{M}_{\mathbb{C}})=0$. 
\end{conjecture}

In fact, in a private communication, Kontsevich expressed the idea
that this conjecture is probably too ambitious. He expects that, by
analogy with the case of algebraic quantization discussed above, $L_{\mathbb{C}}$
should have to satisfy the condition that there is a line bundle $\mathcal{L}$
with $c_{1}(\mathcal{L})=\frac{1}{2}(c_{1}(\omega_{L/X}))$. 

Although we cannot prove this conjecture, in this paper, we will define
and study a different obstruction for $L_{\mathbb{C}}$. I do not
know the exact relationship between the two obstructions, but the
one defined here appears to be stronger. Our obstruction is defined
for smooth, exact algebraic Lagrangians, which, by analogy with the
case of symplectic manifolds, seem to yield the simplest version of
the theory. To explain where this obstruction comes from, we begin
by recalling that the answer to both of the above questions is known
in a special case. Before stating it, let us note that if $L_{\mathbb{C}}=X_{\mathbb{C}}$,
then the equations for $L_{\mathbb{C}}$ are given (locally) by $\{\partial_{i}=0\}_{i=1}^{n}$,
where $\{\partial_{i}\}_{i=1}^{n}$ is a set of coordinate derivations.
Therefore, for $\mathcal{M}_{\mathbb{C}}$ to have constant arithmetic
support equal to $L_{\mathbb{C}}$ means that $\mathcal{M}_{k}$ is
locally annihilated by $\{\partial_{i}^{p}\}_{i=1}^{n}$. This condition
is classically known as having $p$-curvature $0$ (and is well studied).
One has:
\begin{thm}
\label{thm:CC}Let $L_{\mathbb{C}}=X_{\mathbb{C}}\subset T^{*}X_{\mathbb{C}}$.
Then a holonomic $\mathcal{D}_{X_{\mathbb{C}}}$-module $\mathcal{M}_{\mathbb{C}}$
has constant arithmetic support equal to $L_{\mathbb{C}}$, with multiplicity
$1$, iff $\mathcal{M}_{\mathbb{C}}$ is a line bundle with flat connection,
which has regular singularities, and whose associated monodromy representation
is a finite-order character of $\pi_{1}(L_{\mathbb{C}})$.
\end{thm}

This was proved by Chudonovsky-Chudonovsky, \cite{key-65}, theorem
8.1, when $X$ is a curve, following the foundational work of Katz
(\cite{key-24}). Further, the deduction of the result in the general
case from the curve case, was given in \cite{key-74} (c.f. the proof
of theorem 10.5). Another elegant proof of the theorem was given by
Bost, in \cite{key-16}, corollary 2.8 (in fact, he proves an analogous
result for a $G$-bundle with connection, where $G$ is any solvable
group. The above theorem is the case $G=GL_{1}$). If one removes
the condition that the multiplicity be $1$, then the analogue of
this theorem is the famous Katz-Grothendieck conjecture (c.f. \cite{key-74}),
which is open in general. This suggest that the analogue of a coefficient
local system in this context is a local system with finite monodromy
group. This leads us to the 
\begin{conjecture}
\label{conj:Torsor}Suppose $L_{\mathbb{C}}$ is a smooth Lagrangian.
Then the set of all holonomic $\mathcal{D}_{X_{\mathbb{C}}}$-modules
with constant arithmetic support of multiplicity $1$ equal to $L_{\mathbb{C}}$
is a pseudo-torsor over $\pi^{*}$, the group of finite order characters
of $\pi_{1}(L_{\mathbb{C}})$. 
\end{conjecture}

Now let us return to a smooth, exact Lagrangian $L_{\mathbb{C}}$.
Let $f$ on $L_{\mathbb{C}}$ be such that $df=\theta|_{L_{\mathbb{C}}}$.
Below we will prove 
\begin{prop}
(c.f. \prettyref{lem:p-support-of-E-l} below) There is an open subset
$U_{\mathbb{C}}\subset T^{*}X_{\mathbb{C}}$ whose inverse image $L_{U_{\mathbb{C}}}$
in $L_{\mathbb{C}}$ is open and dense; and a $\mathcal{D}_{U_{\mathbb{C}}}$-module
$\mathcal{E}_{\mathbb{C}}$ with constant arithmetic support $L_{U_{\mathbb{C}}}$,
of multiplicity $1$. 
\end{prop}

In fact, the construction of $\mathcal{E}_{\mathbb{C}}$ is not difficult;
for instance, if $\pi:L_{\mathbb{C}}\to X_{\mathbb{C}}$ is dominant,
then $U_{\mathbb{C}}$ is the set over which $\pi$ is finite etale,
and $\mathcal{E}_{\mathbb{C}}=\pi_{*}(e^{f})$ where $e^{f}$ is the
line bundle with flat connection given by $\nabla(1)=df$. In this
case the fact that $\mathcal{E}_{\mathbb{C}}$ has constant arithmetic
support is a direct calculation (c.f. \cite{key-4}, section 2.5). 

Thus if we suppose \prettyref{conj:Torsor} is true, then if $L_{\mathbb{C}}$
admits an $\mathcal{M}_{\mathbb{C}}$ with constant arithmetic support
of multiplicity $1$ equal to $L_{\mathbb{C}}$, the ``difference''
of the two modules $\mathcal{M}_{\mathbb{C}}|_{U_{\mathbb{C}}}$ and
$\mathcal{N}_{\mathbb{C}}$ will be a finite order character of $\pi_{1}(L_{U_{\mathbb{C}}})$.
Such a character will have a finite monodromy around each component
of the divisor $L_{\mathbb{C}}\backslash L_{U_{\mathbb{C}}}$. 

Below, we given an unconditional construction of a $\mathbb{Q}/\mathbb{Z}$-valued
divisor on $L_{\mathbb{C}}$ which (conjecturally)\footnote{This is related to another conjecture of Kontsevich, according to
which all of the $\mathcal{D}$-modules we are discussing should be
of the so-called motivic-exponential type. We'll elaborate in \prettyref{subsec:Invariance}
below} will yield the monodromy around each component of the divisor $L_{\mathbb{C}}\backslash L_{U_{\mathbb{C}}}$.
This divisor is defined in terms of deformations of $\mathcal{O}_{L_{\mathbb{C}}}$
to the world of $\lambda$-connections (this is essentially a rephrasing
of the notion of $\mathcal{D}_{h}$-module discussed above). We call
the resulting output (an element of $\text{Div}_{\mathbb{Q}}(L_{\mathbb{C}})/\text{Div}_{\mathbb{Z}}(L_{\mathbb{C}})$)
the monodromy divisor of $L_{\mathbb{C}}$. Now we can state: 
\begin{conjecture}
\label{conj:Deformation}Suppose $L_{\mathbb{C}}$ is a smooth, exact
Lagrangian. Then there exists a holonomic $\mathcal{D}_{X_{\mathbb{C}}}$-module
$\mathcal{M}_{\mathbb{C}}$ which has constant arithmetic support
equal to $L_{\mathbb{C}}$, with multiplicity $1$ iff there exists
a finite-order character of $\pi_{1}(L_{U_{\mathbb{C}}})$ whose monodromy
around each component of $L_{\mathbb{C}}\backslash L_{U_{\mathbb{C}}}$
is equal to the monodromy divisor of $L_{\mathbb{C}}$. 
\end{conjecture}

Together with \prettyref{conj:Torsor}, this would provide a complete
description of the structure of the set of $\mathcal{D}$-modules
whose arithmetic support is $L_{\mathbb{C}}$. Currently, this conjecture
seems out of reach. 

As mentioned above, the condition that there exists a finite-order
character of $\pi_{1}(L_{U_{\mathbb{C}}})$ whose monodromy around
each component of $L_{\mathbb{C}}\backslash L_{U_{\mathbb{C}}}$ is
equal to the monodromy divisor of $L_{\mathbb{C}}$, is related to
the notion of ``quantization'' in the sense of modules over $\mathcal{D}_{h}$.
In fact, this condition implies the existence of an ``order 2''
deformation of $\mathcal{O}_{L_{\mathbb{C}}}$ (c.f. \prettyref{prop:Characterization-of-ASSUMPTION}
below) with very special properties. If this deformation extends to
a deformation to arbitrary orders, then, by the main theorem of \cite{key-58},
the condition $c_{1}(\mathcal{L})=\frac{1}{2}(c_{1}(\omega_{L/X}))$
would be satisfied. I suspect, though cannot at the moment prove,
that this is the case. 

Now let us state the main theorem of the paper: 
\begin{thm}
\label{thm:1}Let $X_{\mathbb{C}}$ be a smooth affine variety. Let
$L_{\mathbb{C}}\subset T^{*}X_{\mathbb{C}}$ be a smooth, irreducible
Lagrangian subvariety. Suppose $\mathbb{H}_{dR}^{1}(L_{\mathbb{C}})=0$
(in particular, $L_{\mathbb{C}}$ is exact); and suppose that there
exists a finite-order character of $\pi_{1}(L_{U_{\mathbb{C}}})$
whose monodromy around each component of $L_{\mathbb{C}}\backslash L_{U_{\mathbb{C}}}$
is equal to the monodromy divisor of $L_{\mathbb{C}}$. 

Then there exists an irreducible holonomic $\mathcal{D}_{X_{\mathbb{C}}}$-module
$\mathcal{M}_{\mathbb{C}}$ which has constant arithmetic support
equal to $L_{\mathbb{C}}$, with multiplicity $1$. Furthermore, the
set of such $\mathcal{D}_{X_{\mathbb{C}}}$-modules forms a torsor
over $\pi^{*}$, the (finite) group of finite-order characters of
$\pi_{1}(L_{\mathbb{C}})$. Furthermore, each such $\mathcal{D}_{X_{\mathbb{C}}}$-module
$\mathcal{M}'_{\mathbb{C}}$ satisfies $\text{Ext}^{1}(\mathcal{M}'_{\mathbb{C}},\mathcal{M}'_{\mathbb{C}})=0$.
\end{thm}

\begin{rem}
In a previous version (arxiv:1510.05734, v.1 and v.2) of this work,
a different version of this theorem was claimed; namely, I claimed
that $\mathcal{M}_{\mathbb{C}}$ exists under the condition that $H^{0}(\Omega_{\overline{L}}^{2})=0$
for $\overline{L}$ being (any) normal crossings compactification
of $L$. Although this is an interesting condition, the argument presented
there did not work, and I suspect that the result stated there is
not correct. In fact, I discovered the current version of the argument
by analyzing the error in the previous version. Personal and world
events of the past few years have contributed to a delay in getting
the corrected version written up; for this I apologize.
\end{rem}

As explained in \cite{key-4} and \cite{key-3}, this higher dimensional
version has a number of interesting consequences, including a description
of the Picard group of the Weyl algebra. To see why the theorem must
apply here, we note the 
\begin{cor}
Suppose $L_{\mathbb{C}}\subset T^{*}X_{\mathbb{C}}$ is as in the
previous theorem; and suppose $\text{Pic}(L_{\mathbb{C}})=0$. Then
for any $\mathbb{Q}$-divisor ${\displaystyle \sum_{i}\alpha_{i}E_{i}}$
on $L_{\mathbb{C}}$ there is a closed one form $\phi$ on $L_{U_{\mathbb{C}}}$,
with log poles along any compactification of $L_{U_{\mathbb{C}}}$,
whose monodromy about $E_{i}$ is $e^{2\pi i\alpha_{i}}$; in particular,
$\mathcal{M}_{\mathbb{C}}$ as in the previous theorem always exists
in this case. 
\end{cor}

\begin{proof}
The condition is equivalent to asking that $\mathbb{C}[L_{\mathbb{C}}]$
is a UFD; i.e., any prime divisor is principal. If $f_{i}$ is a function
whose zero set is exactly $E_{i}$, with multiplicity $m_{i}$, then
we can set 
\[
\phi=\sum_{i}\frac{\alpha_{i}}{m_{i}}\frac{df_{i}}{f_{i}}
\]
and this is a closed one-form which satisfies the condition.
\end{proof}
Now we can briefly explain the application (some details are provided
in \prettyref{sec:Applications} below). Recall that $\mbox{MAut}(\mathcal{D})$,
where $\mathcal{D}$ is any algebra, is the group of isomorphism classes
of invertible $\mathcal{D}-\mathcal{D}$-bimodules (the group operation
is tensor product). By Morita theory, this isomorphic to the group
of autoequivalences of the abelian category $\text{Mod}(\mathcal{D})$. 

On the other hand, consider the group of algebraic symplectomorphisms
of the symplectic variety $T^{*}\mathbb{A}_{\mathbb{C}}^{m}$. Taking
the graph of such a morphism, $\phi$, yields a Lagrangian subvariety
$L_{\mathbb{C}}^{\phi}$ of $T^{*}\mathbb{A}_{\mathbb{C}}^{m}\times T^{*}\mathbb{A}_{\mathbb{C}}^{m}\tilde{=}T^{*}\mathbb{A}_{\mathbb{C}}^{2m}$;
here, we note that the isomorphism uses the opposite of the usual
symplectic structure on the second factor of $T^{*}\mathbb{A}_{\mathbb{C}}^{m}$.
By construction $L_{\mathbb{C}}^{\phi}\tilde{=}\mathbb{A}_{\mathbb{C}}^{2m}$
and so satisfies the assumption of the previous corollary. Thus we
obtain a unique $\mathcal{D}_{2m}\tilde{=}\mathcal{D}_{m}\otimes\mathcal{D}_{m}^{op}$-module
$\mathcal{M}_{\mathbb{C}}^{\phi}$ corresponding to $L_{\mathbb{C}}^{\phi}$.
One verifies that the bimodule corresponding to the Lagrangian $L_{\mathbb{C}}^{\phi^{-1}}$
is the inverse bimodule to $\mathcal{M}_{\mathbb{C}}^{\phi}$. Combining
this fact with the ideas of \cite{key-25} and \cite{key-3} allows
one to conclude
\begin{thm}
\label{thm:3}There is an isomorphism of groups 
\[
\mbox{Symp}(T^{*}\mathbb{A}_{\mathbb{C}}^{m})\tilde{\to}\mbox{MAut}(\mathcal{D}_{m})
\]
\end{thm}

This proves \cite{key-3}, conjecture $6$. We remark that, in the
case $m=1$, it is known that $\mbox{MAut}(\mathcal{D}_{1})=\mbox{Aut}(\mathcal{D}_{1})$.
In this case the theorem is due to Dixmier in \cite{key-41}. Its
reproof in \cite{key-31} using positive characteristic techniques
is, in a sense, the starting point for everything done here. 

In a related development, in the paper \cite{key-64}, a proof of
the statement that 
\[
\mbox{Aut}(\mathcal{D}_{m})\tilde{\to}\mbox{Symp}(T^{*}\mathbb{A}_{\mathbb{C}}^{m})
\]
was given, also making use of reduction to positive characteristic.
Their technique emphasizes the topological structure of the automorphism
group on the left (considered as an ind-scheme), and, as far as I
can tell, there seems to be almost no overlap between the two methods.
However, combining the two maps does show that every element of $\mbox{MAut}(\mathcal{D}_{m})$
is the bimodule attached to an automorphism of $\mathcal{D}_{m}$,
and in particular that all such bimodules are free of rank $1$ as
left $\mathcal{D}_{m}$-modules. 

\subsection{Key ideas of the paper}

Each chapter of the paper contains an introduction which summarizes
its contents; therefore, instead of repeating that summary here, we'll
just outline some of the key ideas that appear. Let $X_{\mathbb{C}}$
be a smooth irreducible variety and let $L_{\mathbb{C}}\subset T^{*}X_{\mathbb{C}}$
be an irreducible, smooth, exact Lagrangian.

In chapter 2 (\prettyref{sec:The-Monodromy-Divisor}), the key idea
is the construction of a certain $\mathcal{D}_{\lambda}$-module $\mathcal{E}_{\lambda}$
(see the section directly below for our conventions on $\mathcal{D}_{\lambda}$-modules)
on an open subset $U_{\mathbb{C}}\subset X_{\mathbb{C}}$. The associated
module $\mathcal{E}_{\lambda}/\lambda$ (over $\mathcal{O}_{T^{*}U_{\mathbb{C}}}$)
is simply $\mathcal{O}_{L_{U_{\mathbb{C}}}}$. 

Further, $\mathcal{E}_{1}$ solves the quantization problem over $L_{U_{\mathbb{C}}}$-
in particular, after reduction to an algebraically closed field $k$
of large positive characteristic, the $p$-support of $\mathcal{E}_{1}|_{U}$
is equal to $L_{k}^{(1)}|_{U_{k}}$ ). Constructing $\mathcal{E}_{\lambda}$
requires a bit of knowledge about the structure of $L_{\mathbb{C}}$,
and this is contained in \prettyref{thm:Actual-Structure-Theorem}
(see also the introduction to \prettyref{sec:The-Monodromy-Divisor}). 

In general, $\mathcal{E}_{1}$ will not extend to a $\mathcal{D}$-module
on $X_{\mathbb{C}}$ whose arithmetic support is exactly $L_{\mathbb{C}}$.
In order to measure this failure, we first look at the corresponding
failure of the associated infinitesimal\footnote{We use this terminology as it is a module over $R[\lambda]/\lambda^{2}$}
$\lambda$-connection $\mathcal{E}_{\lambda}/\lambda^{2}$ to extend
to a deformation of $\mathcal{O}_{L_{\mathbb{C}}}$. We show in \prettyref{thm:Microlocal-form-of-E}
than this failure is measured by a rational number attached to each
codimension $1$ component of $L_{\mathbb{C}}\backslash L_{U_{\mathbb{C}}}$.
This yields the definition of the monodromy divisor, and sets up the
obstruction class used in the paper. The (a priori, somewhat surprising)
fact that all of the obstruction information is contained in the reduction
mod $\lambda^{2}$ is explained by the $p$-curvature formulas appearing
in \prettyref{prop:Difference-of-two-connections} (c.f. also the
remarks at the end of \prettyref{subsec:Invariance}). Let me also
mention here \prettyref{lem:Extend-mod-p}, which explains why, after
reduction to $k$, the object $\mathcal{E}_{\lambda,k}$ does extend
to a suitable quantization of $L_{k}$ (in other words, there is no
obstruction theory in characteristic $p$).

In chapter 3, we pursue the idea of lifting the module constructed
in \prettyref{lem:Extend-mod-p}. The key idea, as explained in the
introduction, is to compactly the whole situation so as to limit the
space of possible lifts. One must then wrestle with the problem of
putting the correct ``conditions at infinity'' on the resulting
meromorphic connections. The ``toy model'' for this is the case
of Higgs sheaves (i.e., $\lambda$-connections for which $\lambda=0$)
and it is handled in \prettyref{subsec:Higgs-Sheaves}; this leads
to the definition of the key object, a $\theta$-regular Higgs bundle.
The abstract deformation theory\footnote{from Higgs sheaves to $\lambda$-connections over $R[\lambda]/\lambda^{n}$}
for such bundles is developed in \prettyref{subsec:Local-Deformation-theory}
and \prettyref{subsec:-regular-connections-over}; I would like to
emphasize that the proof of the key technical result (\prettyref{lem:Adjust-the-extension})
depends on the precise formulation of the $\theta$-regularity condition
over $R[\lambda]/\lambda^{n}$ (in \prettyref{def:Theta-reg-infinitesimal}).
The existence of deformations, for a suitable collection of $\theta$-regular
Higgs sheaves, is ensured by the vanishing of the obstruction class
constructed in chapter 2 (c.f. \prettyref{def:M-Lambda-2}). Furthermore,
the condition that $H_{dR}^{1}(L_{\mathbb{C}})=0$ is then used to
ensure that deformations, if they exist, are unique (c.f. \prettyref{thm:Infinitesimal-Def}).
Finally, going from formal deformations to actual algebraic $\lambda$-connections
requires an application of Grothendieck's existence theorem in formal
geometry (c.f . the proof of \prettyref{thm:Algebrization}); so this
again uses the fact that we have compactified $X$. The result is
a family of connections over an open subset $U_{\mathbb{C}}\subset X_{\mathbb{C}}$,
indexed by line bundles on $L|_{U_{\mathbb{C}}}$ with flat connection
(whose associated monodromy group is finite) each of which has constant
arithmetic support (equal to $L|_{U_{\mathbb{C}}}$). 

Chapter 4 is of a completely general nature, and can be read independently
of the rest of the paper. It investigates certain natural questions
about connections in mixed characteristic, i.e., over a variety $X_{W_{m}(k)}$
which is smooth over a truncated Witt ring $W_{m}(k)$ of the perfect
field $k$. The key ideas here are twofold- first, we consider the
construction of an object (the so-called $p^{m}$-curvature) attached
to a vector bundle $\mathcal{E}$ with flat connection on $X_{W_{m}(k)}$
whose reduction to $X_{W_{m-1}(k)}$ is locally trivial. Just as the
usual $p$-curvature measures the failure of a vector bundle with
flat connection over $X_{k}$ to be locally trivial, this invariant
measures the failure of $\mathcal{E}$ to be locally trivial (c.f.
\prettyref{cor:p^m-curvature}). The proof makes essential use of
Berthelot's higher differential operators $\mathcal{D}_{X_{k}}^{(i)}$,
and the computation of their center found in \cite{key-57}. The other
key idea of the chapter is that, if one has a smooth Lagrangian $L_{k}\subset T^{*}X_{k}$,
and a splitting bundle $\mathcal{M}_{k}$ of $\mathcal{D}_{X_{k}}$
along $L_{k}^{(1)}$, then the problem of understanding its local
lifts is reducible to the case of understanding lifts of the trivial
connection; in particular, the above-mentioned theory of the $p^{m}$-curvature
actually has an analogue in this case (c.f. \prettyref{prop:p^m-curv-over-L}).
This is obtained via the study of a certain completion of the algebra
of differential operators in mixed characteristic, and automorphisms
thereof, c.f. \prettyref{prop:Iso-All-Levels!}. 

Chapter 5 combines the ideas of the previous three chapters to prove
\prettyref{thm:1}. A few of the results proved along the way may
be surprising- let me mention here the proof of \prettyref{thm:M-sigma-is-theta-reg}
(at the end of \prettyref{subsec:Analysis-of-M-sigma}), in which
it is shown that, under suitable circumstances, an isomorphism of
connections on formal schemes can be extended to an isomorphism of
algebraic connections; the same idea is used again in the proof of
\prettyref{thm:Torsor!}. In addition, let me mention the essential
use of higher-order differential operators and the Riemann-Hilbert
correspondence for unit $F$-crystals (due to Katz, \cite{key-70},
and Emerton-Kisin \cite{key-69}); this ultimately comes from the
use of the $p^{m}$-curvature construction, which produces from two
locally isomorphic connections (over $W(k)$) a line bundle (on $L|_{U_{k}}$)
which is equipped with the action of the Grothendieck differential
operators $\mathcal{D}^{(\infty)}$. Thus the comparison between the
(abelianization of the) etale fundamental group of $L_{U_{\mathbb{C}}}$
and the (abelianization of the) ordinary fundamental group of $L_{U_{\mathbb{C}}}$
lurks behind the proof of \prettyref{thm:Torsor!}; since this group
is finite (under the assumption that $H_{dR}^{1}(L_{\mathbb{C}})=0$),
the theorem is especially easy to state (and use). 

\subsection{\label{subsec:Notations-and-Conventions}Notations and Conventions}

Throughout, $R$ will be an integral domain of finite type over $\mathbb{Z}$,
which is a subalgebra of $\mathbb{C}$. Unless otherwise specified,
undecorated letters such as $X$, $Y$, $Z$ will denote schemes over
$S=\mbox{Spec}(R)$. $k$ will denote an algebraically closed field
of positive characteristic. Decorated letters such as $X_{\mathbb{C}}$
and $X_{k}$ will denote the base change of $X$ to $\mathbb{C}$,
and, relative to a morphism $k\to\mbox{Spec}(R)$, the base change
of $X$ to $k$. As most of our objects get base-changes to different
characteristics, we insist on the subscripts throughout, with the
(relatively rare) exception of various intermediate objects which
have no hope of being base changed. Products of schemes are always
taken over the ground ring. 

We shall, on occasion, make use of the theory of complex analytic
spaces; we abuse notation slightly and use the same letter for a scheme
over $\mathbb{C}$ and its associated analytic space (the context
will make it clear which we mean).

If $\mathcal{M}$ is an object on $X$ (such as a coherent sheaf or
$\mathcal{D}$-module), the letters $\mathcal{M}_{\mathbb{C}}$ and
$\mathcal{M}_{k}$ will denote the base change (to $\mathbb{C}$ and
$k$, respectively). 

If $\varphi:X\to Y$ is a morphism, we will denote by the same letter
the induced morphism $\varphi:X_{\mathbb{C}}\to Y_{\mathbb{C}}$ and
$\varphi:X_{k}\to Y_{k}$; since we always decorate the varieties,
this should (hopefully) cause no confusion. If $Z\to Y$ is a closed
embedding, we shall often denote the base change by $X|_{Z}$.

If $X\to S$ is a smooth map, then we have the relative tangent sheaf
$\mathcal{T}_{X/S}$ and the relative differential forms $\Omega_{X/S}^{1}$;
these will be vector bundles on $X$, and, taking the relative spectrum
over $X$ of the symmetric algebras of these bundles yields the relative
cotangent and tangent bundles, respectively. The enveloping algebra
of the sheaf $\mathcal{T}_{X/S}$ is denoted $\mathcal{D}_{X/S}$,
the sheaf of relative differential operators\footnote{Note in particular that these are differential operators without divided
powers}. 

Over the scheme $X\times\mathbb{A}^{1}$, we have also the sheaf $\mathcal{D}_{\lambda,X/S}$
of $\lambda$-differential operators, satisfying the relation $[\xi,f]=\lambda\cdot\xi(f)$
for any derivation $\xi$ and function $f$; here $\lambda$ is the
coordinate on $\mathbb{A}^{1}$. Modules over $\mathcal{D}_{\lambda,X/S}$
are equivalent to (flat) $\lambda$-connections; i.e., quasicoherent
sheaves equipped with a map 
\[
\nabla:\mathcal{M}\to\mathcal{M}\otimes_{\mathcal{O}_{X}[\lambda]}\Omega_{X/S}^{1}[\lambda]
\]
satisfying $\nabla(fm)=\lambda df\cdot m+f\cdot\nabla(m)$, as well
a $\nabla\circ\nabla=0$. After inverting $\lambda$, these become
equivalent to relative connections over $S\times\mathbb{G}_{m}$.
After setting $\lambda=0$, this becomes the category of Higgs sheaves
on $X$. As $\lambda$ is a central parameter, we can also take the
quotient by $\lambda^{n}$ and obtain $\lambda$-connections over
$R[\lambda]/\lambda^{n}$, and, taking the inverse limit, $\lambda$-connections
over $R[[\lambda]]$. The algebra $\mathcal{D}_{\widehat{\lambda}}:=\lim\mathcal{D}_{\lambda}/\lambda^{n}$
has the additional property that it can be ``micro-localized'' to
a sheaf on $T^{*}X$; more precisely, for each open subset $U$ of
$T^{*}X$ there is a functorially associated algebra $\mathcal{D}_{\widehat{\lambda}}(U)$
so that, if $\mathcal{D}_{\widehat{\lambda}}(T^{*}X)=\mathcal{D}_{\widehat{\lambda}}(X)$
(c.f. \cite{key-94}, chapter 4 for a detailed introduction). In particular,
every coherent sheaf over $\mathcal{D}_{\widehat{\lambda}}$ comes
with a support inside $T^{*}X$; this is simply the support of the
associated coherent sheaf over $T^{*}X$, obtained via the identification
$\mathcal{D}_{\lambda}/\lambda\tilde{=}\pi_{*}(\mathcal{O}_{T^{*}X})$. 

This paper often uses the basic technique of ``spreading out'' objects
over $\mathbb{C}$ to objects defined (and flat) over $R$; after
possibly extending $R$. We do this for finite type quasiprojective
schemes over $\mathbb{C}$ (and morphisms between them), as well as
for coherent $\mathcal{D}$-modules (and $\mathcal{D}_{\lambda}$-modules).
In general, any finite type scheme quasiprojective scheme over $\mathbb{C}$
can be defined over $R$ by taking $R$ to contain all coefficients
of defining equations. Similarly, any locally finitely presented $\mathcal{D}$
or $\mathcal{D}_{\lambda}$-module can be defined over $R$ by taking
$R$ to contain all of the coefficients involved in a given presentation.
The flatness over $R$ is justified by an appeal to Grothendieck's
generic freeness lemma (\cite{key-93}, Theorem 14.4). In particular,
any finite type algebra over $R$ can be made free over $R$ after
localization at a single element. 

Let us explain how this works for $\mathcal{D}$-modules; for convenience\footnote{and because it suffices for the applications needed in this paper},
we will assume $X$ is affine. If $\mathcal{M}_{\mathbb{C}}$ is a
finite-type $\mathcal{D}_{X_{\mathbb{C}}}$-module, then necessarily
$\mathcal{M}_{\mathbb{C}}$ admits a good filtration $F^{\cdot}(\mathcal{M}_{\mathbb{C}})$
(by \cite{key-6}, proposition 2.1.1). In particular, $\text{gr}(\mathcal{M}_{\mathbb{C}})$
is a finitely generated graded module over $T^{*}X_{\mathbb{C}}$.
Such a good filtration can be chosen by choosing $F^{0}(\mathcal{M}_{\mathbb{C}})$
to be any $\mathcal{O}_{X_{\mathbb{C}}}$-submodule of $\mathcal{M}_{\mathbb{C}}$
which contains a set of $\mathcal{D}_{X_{\mathbb{C}}}$-module generators
for $\mathcal{M}_{\mathbb{C}}$; and setting $F^{i}(\mathcal{M}_{\mathbb{C}})=F^{i}(\mathcal{D}_{X_{\mathbb{C}}})\cdot F^{0}(\mathcal{M}_{\mathbb{C}})$.
In particular, a good filtration can be defined over $R$. Applying
again \cite{key-93}, Theorem 14.4 (the last sentence) to $\text{gr}(\mathcal{M})$,
we obtain that, after localizing $R$, $\text{gr}(\mathcal{M})$ has
each component free over $R$. But then $\mathcal{M}$ is free over
$R$ as well. Similar remarks apply to $\mathcal{D}_{\lambda}$-modules. 

Over $\mathbb{C}$, we shall follow the notations and conventions
of \cite{key-6} concerning algebraic $\mathcal{D}$ modules and the
functors between them; the one exception is that, if $j:U_{\mathbb{C}}\subset X_{\mathbb{C}}$
is an open inclusion of smooth complex varieties, and $\mathcal{M}_{U_{\mathbb{C}}}$
is an irreducible holonomic $\mathcal{D}_{U_{\mathbb{C}}}$-module,
we denote by $j_{!*}(\mathcal{M}_{U_{\mathbb{C}}})$ the unique irreducible
holonomic $\mathcal{D}_{X_{\mathbb{C}}}$-module extending $\mathcal{M}_{U_{\mathbb{C}}}$.
We will use the notation ${\displaystyle \int_{\varphi}}$ for the
$\mathcal{D}$-module pushforward over a morphism $\varphi:X\to Y$,
in addition, the same symbol will denote the pushforward of $\mathcal{D}_{\lambda}$-modules
(as defined, e.g., in \cite{key-81}, chapter 2; we note that the
parameter $\lambda$ is there denoted $h$). 

Over a field, the cotangent bundle is equipped with a standard one-form,
we shall denote $\alpha$; the two-form $d\alpha=\omega$ is given
in local coordinates by ${\displaystyle \sum_{i=1}^{n}}dx_{i}\wedge dy_{i}$
(where $n=\mbox{dim}(X)$). 

We shall sometimes encounter the following situation: we have a map
$\varphi:X_{F}\to Y_{F}$ (where $F$ is a field of any characteristic)
between smooth $F$-varieties. This yields a correspondence: 
\[
T^{*}Y_{F}\leftarrow Y_{F}\times_{X_{F}}T^{*}X_{F}\to T^{*}X_{F}
\]
We shall denote these maps $d\varphi:Y_{F}\times_{X_{F}}T^{*}X_{F}\to T^{*}X_{F}$
and $\mbox{pr}_{\varphi}:Y_{F}\times_{X_{F}}T^{*}X_{F}\to T^{*}X_{F}$,
respectively. We remark that if $\varphi$ is finite etale, the map
$d\varphi$ is an isomorphism, and thus we can (and shall) speak of
the induced map $T^{*}Y_{F}\to T^{*}X_{F}$. 

When working over a field of positive characteristic, $k$, we have
that the center of the algebra $\mathcal{D}_{\lambda,X_{k}}$ is isomorphic
to $\mathcal{O}_{T^{*}X_{k}^{(1)}}[\lambda]$. This gives rise the
to the phenomenon of $p$-support (or $p$-curvature), namely, considering
the module $\mathcal{M}$ as a sheaf over $\mathcal{O}_{T^{*}X_{k}^{(1)}}[\lambda]$,
or, the commutative subalgebra $\mathcal{O}_{X_{k}}\otimes_{\mathcal{O}_{X_{k}^{(1)}}}\mathcal{Z}(\mathcal{D}_{\lambda,X_{k}})\tilde{=}\mathcal{O}_{X_{k}\times_{X_{k}^{(1)}}T^{*}X_{k}^{(1)}}[\lambda]\subset\mathcal{D}_{\lambda,X_{k}}$.
In other words, $\mathcal{M}$ also comes equipped with a map $\Psi:\mathcal{M}\to\mathcal{M}\otimes_{\mathcal{O}_{X_{k}[\lambda]}}(F^{*}\Omega_{X_{k}^{(1)}}^{1})[\lambda]$,
here $F:X_{k}\to X_{k}^{(1)}$ is the relative Frobenius.

In the case where $\mathcal{M}$ is a line-bundle with $\lambda$-connection,
the formula is especially nice; namely, if we consider $\nabla(1)=\varphi$
(on $\mathcal{O}_{X_{k}}[\lambda]$), then $\Psi(1)=\varphi^{(p)}-\lambda^{p-1}C(\varphi)$,
where $\varphi^{(p)}$ is the image of $\varphi$ under the semi-linear
isomorphism $\Omega_{X_{k}}^{1}[\lambda]\tilde{\to}\Omega_{X_{k}^{(1)}}^{1}[\lambda]$,
and $C$ is the Cartier map $\Omega_{X_{k}}^{1,cl}[\lambda]\to\Omega_{X_{k}^{(1)}}^{1}[\lambda]$. 

Finally, we specialize to the case where $X=\mathbb{A}^{n}$. In this
case, there is a group $\text{Sp}_{n}$ of linear symplectomorphisms
acting (on $T^{*}\mathbb{A}^{n}$), and they act also on the global
sections of $\mathcal{D}_{\lambda,\mathbb{A}^{n}}$, by identifying
the span of $\{x_{1},\dots,x_{n},\partial_{1},\dots,\partial_{n}\}$
with the vector space $T^{*}\mathbb{A}^{n}$, and letting $\lambda$
go to $\lambda$. Following Belov-Kanel and Kontsevich (\cite{key-3},
proposition 7.2), after passing to $k$ for $k$ of suitably large
characteristic, the resulting automorphism of $\mathcal{Z}(\mathcal{D}_{\lambda,\mathbb{A}_{k}^{n}})=\mathcal{O}_{(T^{*}\mathbb{A}_{k}^{n})^{(1)}}[\lambda]$
is simply given by the action of $\sigma^{(1)}\times\text{id}$ on
$(T^{*}\mathbb{A}_{k}^{n})^{(1)}\times\mathbb{A}_{k}^{1}$.

\subsection{Acknowledgements }

The author would like to thank Maxim Kontsevich and Dima Arinkin for
some helpful email conversations.

\section{\label{sec:The-Monodromy-Divisor}The Monodromy Divisor of an Exact
Lagrangian}

Let $X_{\mathbb{C}}$ be an irreducible, smooth complex variety, and
$L_{\mathbb{C}}\subset T^{*}X_{\mathbb{C}}$ be a smooth irreducible
exact Lagrangian subvariety. In this chapter we define, and prove
the basic properties of, the obstruction class to quantizing $L_{\mathbb{C}}$
which is considered in this paper (the monodromy divisor). In order
to do this, we need a basic structure theorem about exact Lagrangians
inside $T^{*}X_{\mathbb{C}}$. To motivate the formulation, let us
recall the following result about holonomic $\mathcal{D}$-modules:
\begin{thm}
Let $\mathcal{M}_{\mathbb{C}}$ be a holonomic $\mathcal{D}_{X_{\mathbb{C}}}$-module,
supported along the closed subvariety $Z_{\mathbb{C}}\subset X_{\mathbb{C}}$.
Then there is an open subset $U_{\mathbb{C}}\subset X_{\mathbb{C}}$
so that $U_{\mathbb{C}}\cap Z_{\mathbb{C}}$ is dense in the smooth
locus of $Z_{\mathbb{C}}$, and a vector bundle with flat connection
$(\mathcal{V}_{\mathbb{C}},\nabla)$ on $U_{\mathbb{C}}\cap Z_{\mathbb{C}}$
so that 
\[
\mathcal{M}_{\mathbb{C}}|_{U_{\mathbb{C}}}=\int_{\iota}\mathcal{V}_{\mathbb{C}}
\]
where $\iota:U_{\mathbb{C}}\cap Z_{\mathbb{C}}\to U_{\mathbb{C}}$
is the inclusion. 
\end{thm}

For a proof, c.f. \cite{key-6}, lemma 3.1.6 and lemma 3.2.5. 

Now let us consider what this means for the arithmetic support of
$\mathcal{M}_{\mathbb{C}}|_{U_{\mathbb{C}}}$. Take flat $R$-models
for all the varieties and $\mathcal{D}$-modules appearing above,
and consider $R\to k$ where $k$ is a perfect field $k$ of positive
characteristic. As $\mathcal{V}_{\mathbb{C}}$ is a vector bundle,
so too is its reduction $\mathcal{V}_{k}$; this forces the $p$-support
of $\mathcal{V}_{k}$ to be a Lagrangian inside $T^{*}(Z_{k}\cap U_{k})$
which is finite over $Z_{k}\cap U_{k}$. To ease the notation set
$U_{k}'=Z_{k}\cap U_{k}$. Then
\begin{lem}
We may form the smooth morphism 
\[
(d\iota){}^{*}:(U')_{k}^{(1)}\times_{U_{k}^{(1)}}T^{*}U_{k}^{(1)}\to T^{*}(U')_{k}^{(1)}
\]
and if $\mathcal{N}$ is an $\mathcal{D}$-module over $Z'_{k}$,
whose scheme-theoretic $p$-support is $L'_{k}\subset T^{*}(U')_{k}^{(1)}$,
then the module ${\displaystyle \int_{\iota}\mathcal{N}}$ is scheme-theoretically
$p$-supported on the scheme-theoretic inverse image $(d^{*}\iota)^{-1}(L'_{k})$. 
\end{lem}

The proof of this is straightforward (we also state a generalization
in \prettyref{lem:p-support!} below). 

Now suppose that the projection $L_{\mathbb{C}}\to X_{\mathbb{C}}$
has image equal to $Z_{\mathbb{C}}$. From the above results, we see
that, if $\mathcal{M}_{\mathbb{C}}$ is a holonomic $\mathcal{D}$-module
with constant arithmetic support equal to $L_{\mathbb{C}}$, then,
after restricting to $U_{\mathbb{C}}$ for some open $U_{\mathbb{C}}$
we must have $L_{\mathbb{C}}=(d^{*}\iota)^{-1}(L'_{\mathbb{C}})$
for some Lagrangian $L'_{\mathbb{C}}\subset T^{*}(Z'_{\mathbb{C}})$;
here we denote $d^{*}\iota:Z'_{\mathbb{C}}\times_{U_{\mathbb{C}}}T^{*}U_{\mathbb{C}}\to T^{*}Z'_{\mathbb{C}}$
the natural map. In fact, we shall show in \prettyref{thm:Actual-Structure-Theorem}
below that every exact Lagrangian has such a structure, over s suitable
open subset $U_{\mathbb{C}}$. Furthermore, in \prettyref{subsec:The-monodromy-divisor}
we actually show that such a $\mathcal{D}$-module exists on the open
subset $U_{\mathbb{C}}$ (it is an elementary result). Moreover, we
construct a deformation of this module into the world of $\lambda$-connections,
called $\mathcal{E}_{\lambda,\mathbb{C}}$. The $\lambda$-connection
$\mathcal{E}_{\lambda,\mathbb{C}}/\lambda^{2}$ may, therefore, be
considered as an infinitesimal deformation of $\mathcal{O}_{L_{U_{\mathbb{C}}}}$.
This deformation may fail to extend from $L_{U_{\mathbb{C}}}$ to
a all of $L_{\mathbb{C}}$, and we shall see in \prettyref{thm:Microlocal-form-of-E}
that this failure takes the form of a $\mathbb{Q}/\mathbb{Z}$-valued
divisor on $L_{\mathbb{C}}$. The failure of this divisor to be equal
to the monodromy of a suitable flat connection on $L_{U_{\mathbb{C}}}$
is the obstruction to quantization considered in this paper.

Finally, we would like to make a remark on the content of \prettyref{thm:Actual-Structure-Theorem}
from the perspective of symplectic geometry. One has that 
\[
T^{*}Z'_{\mathbb{C}}\leftarrow Z'_{\mathbb{C}}\times_{U_{\mathbb{C}}}T^{*}U_{\mathbb{C}}\to T^{*}U_{\mathbb{C}}
\]
is a Lagrangian correspondence; i.e., the image of $Z'_{\mathbb{C}}\times_{U_{\mathbb{C}}}T^{*}U_{\mathbb{C}}$
in $T^{*}Z'_{\mathbb{C}}\times T^{*}U_{\mathbb{C}}\tilde{=}T^{*}(Z'_{\mathbb{C}}\times U_{\mathbb{C}})$
is Lagrangian. Therefore, one expects to be able to construct Lagrangians
in $T^{*}U_{\mathbb{C}}$ from Lagrangians in $T^{*}Z'_{\mathbb{C}}$
by taking fibre product with this correspondence. That this exactly
what the theorem does in the case of the Lagrangian $L_{\mathbb{C}}$. 

\subsection{\label{subsec:The-Generic-Structure}The Generic Structure of an
exact Lagrangian}

In this subsection, we discuss the structure of an irreducible exact
Lagrangian in a the cotangent bundle of a smooth affine\footnote{We could start with a more general $X_{\mathbb{C}}$, but the main
result of this section concerns only an open subset of $X_{\mathbb{C}}$,
which we may as well take to be affine} variety over $\mathbb{C}$. Let $L_{\mathbb{C}}\subset T^{*}X_{\mathbb{C}}$
be this Lagrangian. If $\pi:L_{\mathbb{C}}\to X_{\mathbb{C}}$ is
the projection, denote by $Z_{\mathbb{C}}$ the closure of the image
of $\pi$. Since we are in characteristic $0$, we have, by \cite{key-20},
Chapter 3, Corollary 10.7, that $\pi$ is generically smooth; i.e.,
we may select an open affine subset $U_{\mathbb{C}}\subset X_{\mathbb{C}}$,
intersecting $Z_{\mathbb{C}}$ in a nonempty way, such that on $U'_{\mathbb{C}}:=U_{\mathbb{C}}\cap Z_{\mathbb{C}}$
the morphism $\pi^{-1}(U'_{\mathbb{C}})=L_{U'_{\mathbb{C}}}\to U'_{\mathbb{C}}$
is smooth. The main goal of this subsection is to prove 
\begin{thm}
\label{thm:Actual-Structure-Theorem}From the closed embedding $\iota:U'_{\mathbb{C}}\to U_{\mathbb{C}}$
one obtains the affine space bundle $d\iota^{*}:U'_{\mathbb{C}}\times_{U_{\mathbb{C}}}T^{*}U_{\mathbb{C}}\to T^{*}U'_{\mathbb{C}}$.
Then, after possibly shrinking the open subset $U_{\mathbb{C}}$,
there is a smooth exact Lagrangian $L'_{\mathbb{C}}\subset T^{*}U'_{\mathbb{C}}$,
which is finite etale over $U'_{\mathbb{C}}$, and such that the following
holds: regarding $U'_{\mathbb{C}}\times_{U_{\mathbb{C}}}T^{*}U_{\mathbb{C}}$
as a closed subscheme of $T^{*}U_{\mathbb{C}}$, we have that $L_{U'_{\mathbb{C}}}=(d\iota^{*})^{-1}(L'_{\mathbb{C}})$
(here $(d\iota^{*})^{-1}$ is the scheme-theoretic inverse image). 
\end{thm}

To prove this, we will to analyze the behavior of $f$ on $L_{\mathbb{C}}$,
with respect to the map $\pi$. We begin with the
\begin{lem}
\label{lem:Kernel-Vanish}Let $l$ be a closed point of $L_{U'_{\mathbb{C}}}$,
and consider the map $d\pi_{l}:T_{l}(L_{\mathbb{C}})\to T_{\pi(l)}(U'_{\mathbb{C}})$.
Then $df$ vanishes on $\mbox{ker}(d\pi_{l})$. Therefore, if we set
$F_{l}=\pi^{-1}(\pi(l))\subset L_{U'_{\mathbb{C}}}$, we have that
$f$ is constant on connected components of $F_{l}$. 
\end{lem}

\begin{proof}
Let $i_{l}:F_{l}\to L_{\mathbb{C}}$ denote the inclusion. By the
smoothness of $\pi$ on $\pi^{-1}(U')$, we have that $F_{l}$ is
a smooth subvariety of $L_{\mathbb{C}}$. Furthermore, 
\[
i_{l}^{*}(df)=i_{l}^{*}(\sum_{i=1}^{n}y_{i}dx_{i})=\sum_{i=1}^{n}y_{i}(l)i_{l}^{*}(dx_{i})=0
\]
where the last equality follows from the fact that $i_{l}^{*}(dx_{i})=0$
for all $i$; which in turn follows from the fact that $F_{l}\subset T_{l}^{*}(X)$,
the cotangent fibre over the point $l$. Since, by smoothness, the
kernel of $d\pi_{l}$ is the image of the natural inclusion $di_{l}:T_{l}(F_{l})\to T_{l}(L)$,
the result of the first sentence follows; the second follows from
$df|_{F_{l}}=0$ since we are in characteristic $0$. 
\end{proof}
Recall that if $f$ is any regular function on a smooth variety $Y_{\mathbb{C}}$
we may associate the subvariety $\Gamma(df)\subset T^{*}Y$, known
as the graph of $df$, which, as a subvariety of $T^{*}Y$ is (locally)
defined by the equations
\[
\{\partial-<df,\partial>|\partial\in\mathcal{T}_{Y}\}
\]
where $<,>$ denotes the pairing between tangent vectors and one-forms.
By definition this is an exact Lagrangian, since on $\Gamma(df)$
we have $df=\theta$. 
\begin{cor}
\label{cor:df-contained-in-product}Consider the natural map $L_{U'_{\mathbb{C}}}\times_{U_{\mathbb{C}}'}T^{*}(U'_{\mathbb{C}})\to T^{*}(L_{U'_{\mathbb{C}}})$.
Since\linebreak{}
$L_{U'_{\mathbb{C}}}\to U'_{\mathbb{C}}$ is smooth, this map is a
closed embedding. Then the subvariety $\Gamma(df)\subset T^{*}(L_{U'_{\mathbb{C}}})$
is contained in the image of $L_{U'_{\mathbb{C}}}\times_{U'_{\mathbb{C}}}T^{*}(U'_{\mathbb{C}})$. 
\end{cor}

\begin{proof}
Over a given point $l\in L_{U'_{\mathbb{C}}}$, the subspace of $T_{l}^{*}(L_{U'_{\mathbb{C}}})$
in the image of $L_{U'_{\mathbb{C}}}\times_{U'_{\mathbb{C}}}T^{*}(U'_{\mathbb{C}})\to T^{*}(L_{U_{\mathbb{C}}'})$
consists of those functionals $\{v:T_{l}(L)\to\mathbb{C}\}$ which
are trivial on $\mbox{ker}(d\pi_{l})$. By the previous proposition,
$df$ is such a functional, at each $l$. The result follows. 
\end{proof}
From this, we may deduce 
\begin{cor}
\label{cor:Construction-of-L'}Consider the natural projection $\mbox{pr}:L_{U'_{\mathbb{C}}}\times_{U_{\mathbb{C}}'}T^{*}(U_{\mathbb{C}}')\to T^{*}(U_{\mathbb{C}}')$.
After possibly replacing $U_{\mathbb{C}}'$ by a smaller open subset
of $Z_{\mathbb{C}}$, we have that the image of $\Gamma(df)$ in $T^{*}(U_{\mathbb{C}}')$
is a smooth Lagrangian subvariety, $L'_{\mathbb{C}}$. The regular
function $f$ induces a regular function $f'$ on $L'_{\mathbb{C}}$,
such that $f'\circ\mbox{pr}=f$, and we have $df'=\theta|_{L'}$,
i.e., $L'$ is exact in $T^{*}U_{\mathbb{C}}'$. The natural projection
$\pi':L_{\mathbb{C}}'\to U_{\mathbb{C}}'$ is finite etale.
\end{cor}

\begin{proof}
Since $\pi:L_{U_{\mathbb{C}}'}\to U_{\mathbb{C}}'$ is surjective,
we may represent any (closed) point in $U'_{\mathbb{C}}$ as $\pi(l)$
for some $l$. As above, denote the fibre over this point by $F_{l}$.
By \prettyref{lem:Kernel-Vanish}, for each $p\in F_{l}$, the linear
functional $df_{p}:T_{p}L\to\mathbb{C}$ induces a functional on $T_{\pi(p)}U'_{\mathbb{C}}=T_{\pi(l)}U_{\mathbb{C}}'$,
i.e., a point in $T_{\pi(l)}^{*}U'$. Since $f$ is constant on connected
components of $F_{l}$, the collection 
\[
\{df_{p}|p\in F_{l}\}
\]
considered as a set of points in $T_{\pi(l)}^{*}U'_{\mathbb{C}}$,
is a finite set. However, by definition, the set of all such points
is the set-theoretic image of (the closed points of) $\Gamma(df)$
in $T^{*}(U'_{\mathbb{C}})$. This shows the quasi-finiteness of the
projection $L'_{\mathbb{C}}\to U'_{\mathbb{C}}$. 

To obtain finer information, we recall that, by generic base change
for etale cohomology\textbf{ }(\cite{key-13}, section 9.3) we have
that, after possibly shrinking $U'_{\mathbb{C}}$ , we may select
an an etale neighborhood $S_{l}$ of $\pi(l)\in U'_{\mathbb{C}}$
so that the variety $L_{S_{l}}:=S_{l}\times_{U'_{\mathbb{C}}}L_{U'_{\mathbb{C}}}$
is a disjoint union ${\displaystyle \bigsqcup_{n=1}^{r}V_{n}}$, where
each $V_{n}$ is a smooth connected variety, which is smooth over
$S_{l}$; and such that the fibers of each map $V_{n}\to S_{l}$ are
connected. The pull-back of $f$ to each $V_{n}$ yields a regular
function $f_{n}$, which is constant on the (connected) fibre of $\mbox{pr}:V_{n}\to S_{l}$.
By \prettyref{lem:regular-function} below, we see that there exists
a regular function $f'_{n}$ on $S_{l}$ such that $f_{n}=f'_{n}\circ\mbox{pr}$. 

Now we consider $V_{n}\times_{S_{l}}T^{*}S_{l}$. Since the map $V_{n}\to S_{l}$
is smooth, we may, as in \prettyref{cor:df-contained-in-product},
regard $V_{n}\times_{S_{l}}T^{*}S_{l}\subset T^{*}V_{n}$; as in that
corollary we have $\Gamma(df'_{n})\subset V_{n}\times_{S_{l}}T^{*}S_{l}$.
By definition the image of the projection map $\Gamma(df_{n})\to T^{*}S_{l}$
is precisely $\Gamma(df'_{n})$, a smooth exact Lagrangian subvariety
of $T^{*}(S_{l})$. Furthermore, since the map $S_{l}\to U'_{\mathbb{C}}$
is etale, it yields a map $T^{*}S_{l}\to T^{*}U'_{\mathbb{C}}$. Then
the variety $L'$ consists of the union, over all $n$, of the images
of the $\Gamma(df'_{n})$ inside $T^{*}U'_{\mathbb{C}}$. 

After possibly discarding a closed subset, we see that $L'_{\mathbb{C}}$
is a smooth, exact Lagrangian, with finite etale projection to $U'_{\mathbb{C}}$,
as claimed. The existence of the global regular function $f'$ now
follows from \prettyref{lem:regular-function} below, and the exactness
statement follows as it is true after passing to an etale cover. 
\end{proof}
In the above, we used the following, well-known
\begin{lem}
\label{lem:regular-function}Suppose that $\phi:Z\to W$ is a smooth
surjective morphism of smooth connected complex varieties. Let $f$
be a regular function on $Z$ which is constant along the fibers of
$\phi$. Then there exists a regular function $f'$ on $W$ such that
$f=f'\circ\phi$. 
\end{lem}

Finally, we give the
\begin{proof}
(of \prettyref{thm:Actual-Structure-Theorem}) We have defined $L'_{\mathbb{C}}$
and shown it is an exact Lagrangian above. Since $d\iota^{*}$ is
a smooth map with connected fibres, it follows that both subschemes
$L'_{\mathbb{C}}$ and $L_{U'_{\mathbb{C}}}$ are reduced and irreducible
of the same dimension. It we show $d^{*}\iota(L_{U'_{\mathbb{C}}})=L'_{\mathbb{C}}$,
then this yields $L_{U'_{\mathbb{C}}}\subset(d^{*}\iota)^{-1}(L'_{\mathbb{C}})$
which immediately implies the result. 

So, let $l'\in L'_{\mathbb{C}}\subset T^{*}U'_{\mathbb{C}}$. By the
construction of $L'_{\mathbb{C}}$ in \prettyref{cor:Construction-of-L'},
there is some $l_{0}\in L_{U_{\mathbb{C}}'}$ so that $l'=(\pi(l_{0}),df_{l_{0}})$
where we regard $df_{l_{0}}:T_{\pi(l)}U'_{\mathbb{C}}\to\mathbb{C}$.
It follows directly that 
\[
(d\iota^{*})^{-1}(l')=\{\phi:T_{\pi(l)}(U)\to\mathbb{C}|\phi|_{T_{\pi(l)}(U')}=df_{l_{0}}\}
\]
Now consider any $l\in L_{U_{\mathbb{C}}'}$. We consider $l$ as
an element of $T_{\pi(l)}^{*}(U_{\mathbb{C}})$. We have $d\iota^{*}(l)=(\pi(l),\phi_{l})$,
where $\phi_{l}\in T_{\pi(l)}^{*}(U_{\mathbb{C}}')$ is obtained by
restricting $l$ to $T_{\pi(l)}(U_{\mathbb{C}}')$. To prove the corollary
we must show that $\phi_{l}=df_{l}$. 

We have the surjective map 
\[
d\pi:T_{l}(T^{*}U_{\mathbb{C}})\to T_{\pi(l)}(U_{\mathbb{C}})
\]
and by definition we have $l\circ d\pi=\theta$, where $\theta$ is
the canonical one-form on $T^{*}U_{\mathbb{C}}$. Furthermore, the
restriction of $\theta$ to $T_{l}(L_{U_{\mathbb{C}}'})\subset T_{l}(T^{*}U_{\mathbb{C}})$
is exactly $df_{l}$, by our choice of $f$. On the other hand, the
subspace $d\pi(T_{l}(L_{U_{\mathbb{C}}'}))\subset T_{\pi(l)}(U_{\mathbb{C}})$
is exactly $T_{\pi(l)}(U'_{\mathbb{C}})$, since $L_{U_{\mathbb{C}}'}\to U_{\mathbb{C}}'$
is a smooth morphism. Therefore, the restriction of $l:T_{\pi(l)}(U_{\mathbb{C}})\to\mathbb{C}$
to $T_{\pi(l)}(U_{\mathbb{C}}')$ agrees with $df_{l}$, and we deduce
that $\phi_{l}=df_{l}$, as required. 
\end{proof}
Let us note that, in the case where $L_{\mathbb{C}}\to X_{\mathbb{C}}$
was already a dominant morphism, we have done nothing. The image of
$\Gamma(df)$ in $T^{*}(U_{\mathbb{C}})$ is precisely $L_{\mathbb{C}}$.
What we have shown, essentially, is that if we allow ourselves to
change the underlying variety, the general case reduces to this one. 

\subsection{\label{subsec:The-monodromy-divisor}The monodromy divisor of an
exact Lagrangian}

We continue with the notation of the previous section. Let us explain
how the above results allow us to construct a natural quantization
of $L_{\mathbb{C}}$ over the open subset $U_{\mathbb{C}}$. 

Recall that we constructed a regular function $f'$ on $L'_{\mathbb{C}}$
such that $df'=\theta'|_{L'}$. Thus we may define the $\lambda$-connection
$(\mathcal{O}_{L_{\mathbb{C}}'}[\lambda],df')$ for which $\nabla(1)=df'$
(c.f. \prettyref{subsec:Notations-and-Conventions} for our conventions
on $\lambda$-connections). Further, since the map $\pi:L_{\mathbb{C}}'\to U_{\mathbb{C}}'$
is finite etale we may define 
\[
\mathcal{E}'_{\lambda,\mathbb{C}}:=\pi_{*}(\mathcal{O}_{L_{\mathbb{C}}'}[\lambda],df')
\]
and 
\[
\mathcal{E}_{\lambda,\mathbb{C}}:=\int_{\iota}\mathcal{E}'_{\lambda,\mathbb{C}}
\]
Note that this $\mathcal{D}_{\lambda}$-module depends only on the
one form $\theta'=df'$, and not on $f'$ itself; the fact that $\theta=df'$
will become crucially useful when we compute in positive characteristic
below, but it is unnecessary to actually define $\mathcal{E}_{\lambda}$. 

By \prettyref{thm:Actual-Structure-Theorem}, we have $\mathcal{E}_{\lambda,\mathbb{C}}/\lambda\tilde{=}\mathcal{O}_{L_{U_{\mathbb{C}}}}$;
here, we are considering $\mathcal{E}_{\lambda,\mathbb{C}}/\lambda$
as a sheaf over $\mathcal{D}_{U,\lambda}/\lambda\tilde{=}\mathcal{O}_{T^{*}U_{\mathbb{C}}}$.
Finally, we let $j_{*}(\mathcal{E}_{\lambda,\mathbb{C}})$ denote
the push-forward of this module to $X_{\mathbb{C}}$- it is a coherent
$\mathcal{D}_{\lambda}$-module. Note that, if we choose a finitely
generated $\mathbb{Z}$-algebra $R$ such that all the objects in
question are defined over $R$, we obtain an $R$-model which we denote
$\mathcal{E}_{\lambda}$.

Our goal in this section is to describe the ``micro-local singularities''
of this module, at least to first order in $\lambda$. Over $U$,
$\mathcal{E}_{\lambda}$ has no singularities. We let $E:=L\backslash L_{U}$.
Shrinking $U$ if necessary, we can arrange that this is a principle
divisor in $L$, and we let $\{E_{i}\}_{i=1}^{s}$ denote the components.
Let $\{x_{i}\}$ denotes the generic point of the component $E_{i,\mathbb{C}}$,
and, with $i$ fixed, we let $z$ denote a local uniformizor in $\mathcal{O}_{\{x_{i}\}}$. 
\begin{lem}
\label{lem:Def-of-=00005Cpsi}The sheaf $\mathcal{O}_{L_{\mathbb{C}}}$
admits a deformation to a $\mathbb{C}[\lambda]/\lambda^{2}$ flat
$\mathcal{D}_{\lambda}/\lambda^{2}$ module, $\tilde{\mathcal{E}}_{\mathbb{C}}$.
Thus we obtain a one-form 
\[
\psi=[\mathcal{\tilde{E}}_{\mathbb{C}}]-[\mathcal{E}_{\lambda,\mathbb{C}}]\in\Gamma(\Omega_{L_{U_{\mathbb{C}}}}^{1})
\]
For each component $E_{i,\mathbb{C}}$, we obtain an element $[\psi_{E_{i,\mathbb{C}}}]\in\Omega_{L_{\mathbb{C}},\{x_{i}\}}^{1}[z^{-1}]/\Omega_{L_{\mathbb{C}},\{x_{i}\}}^{1}$,
which is independent of the choice of $\tilde{\mathcal{E}}_{\mathbb{C}}$. 
\end{lem}

\begin{proof}
As $\mathcal{D}_{\lambda}$ is flat over $\mathbb{C}[\lambda]$, we
have (c.f. \cite{key-92}, tag 08VQ) the set of isomorphism classes
of deformations of $\mathcal{O}_{L_{\mathbb{C}}}$ to a $\lambda$-connection
over $\mathbb{C}[\lambda]/\lambda^{2}$ (i.e., a module over $\mathcal{D}_{\lambda}/\lambda^{2}$)
is a pseudo-torsor over $\text{Ext}_{T^{*}X_{\mathbb{C}}}^{1}(\mathcal{O}_{L_{\mathbb{C}}},\mathcal{O}_{L_{\mathbb{C}}})\tilde{=}\Gamma(\Omega_{L_{\mathbb{C}}}^{1})$,
with the obstruction to lifting given by a class in $\Gamma(\Omega_{L_{\mathbb{C}}}^{2})$;
and the analogous statements are true for $\mathcal{O}_{L_{U_{\mathbb{C}}}}$.
We have just seen that $\mathcal{O}_{L_{U_{\mathbb{C}}}}$ admits
a deformation, namely $\mathcal{E}_{\lambda,\mathbb{C}}/\lambda^{2}$.
Thus the obstruction class in $\Gamma(\Omega_{L_{U_{\mathbb{C}}}}^{2})$
vanishes. On the other hand, as he obstruction class is constructed
via an injective resolution in the category of sheaves of $\mathcal{D}_{\lambda}/\lambda^{2}$-modules
(c.f. \cite{key-92}, tag 08L8) and so the class is compatible with
restriction to $L_{U_{\mathbb{C}}}$. Since $L_{U_{\mathbb{C}}}$
is an open dense affine inside $L_{\mathbb{C}}$, we have that the
map $\Gamma(\Omega_{L_{\mathbb{C}}}^{2})\to\Gamma(\Omega_{L_{U_{\mathbb{C}}}}^{2})$
is injective. Therefore the obstruction class in $\Gamma(\Omega_{L_{\mathbb{C}}}^{2})$
vanishes as well. 

So, let $\tilde{\mathcal{E}}_{\mathbb{C}}$ denote some deformation
of $\mathcal{O}_{L_{\mathbb{C}}}$ to a $\lambda$-connection (over
$\mathbb{C}[\lambda]/\lambda^{2}$). Then the difference $[\mathcal{E}_{\lambda,\mathbb{C}}/\lambda^{2}]-[\tilde{\mathcal{E}}_{\mathbb{C}}|_{U_{\mathbb{C}}}]$
of the isomorphism classes of $\mathcal{E}_{\lambda,\mathbb{C}}/\lambda^{2}$
and $\tilde{\mathcal{E}}_{\mathbb{C}}|_{U_{\mathbb{C}}}$ is an element
of $\Gamma(\Omega_{L_{U_{\mathbb{C}}}}^{1})$, which we denote by
$\psi$. While $\psi$ is not unique, choosing a different deformation
$\tilde{\mathcal{E}}'_{\mathbb{C}}$ amounts to adding a one-form,
which is the restriction of a form in $\Gamma(\Omega_{L_{\mathbb{C}}}^{1})$,
to $\psi$. Thus, the polar term of $\psi$ along each component of
the divisor $E_{\mathbb{C}}$, is independent of the choice of $\tilde{\mathcal{E}}_{\mathbb{C}}$,
and therefore uniquely determined by $\mathcal{E}_{\lambda,\mathbb{C}}$.
\end{proof}
\begin{rem}
As written, this element $[\psi_{E_{i,\mathbb{C}}}]$ depends on the
structure of $\mathcal{E}_{\lambda,\mathbb{C}}$ as a lift of $\mathcal{O}_{L_{U_{\mathbb{C}}}}$.
Altering this structure amounts to multiplying the map $\mathcal{E}_{\lambda,\mathbb{C}}/\lambda\to\mathcal{O}_{L_{U_{\mathbb{C}}}}$
by a unit $\alpha\in\mathcal{O}_{L_{U_{\mathbb{C}}}}^{*}$. This changes
the class $\psi$ by adding ${\displaystyle \frac{d\alpha}{\alpha}}$;
this alters the element $[\psi_{E_{i,\mathbb{C}}}]$ by a term of
the form ${\displaystyle \alpha\frac{dz}{z}}$ for $\alpha\in\mathbb{Z}$.
When we consider the monodromy divisor below, we shall see that such
terms disappear. 
\end{rem}

\begin{example}
\label{exa:Airy}The following example of the phenomenon in the previous
lemma is worth keeping in mind. Suppose $L\subset T^{*}\mathbb{A}^{1}$
is given by $\xi^{2}-x$ (here $x$ is the coordinate on $\mathbb{A}^{1}$
and $\xi$ is the coordinate in the cotangent direction) Then the
map $\pi:L\to\mathbb{A}^{1}$ is a double cover of the line, branched
at the origin. A $\mathcal{D}_{\lambda}/\lambda^{2}$-module $\tilde{\mathcal{E}}_{\mathbb{C}}$
which lifts $\mathcal{O}_{L}$ is given as follows: as a module over
$R[x]=\Gamma(\mathcal{O}_{\mathbb{A}^{1}})$, $\tilde{\mathcal{E}}_{\mathbb{C}}$
is free of rank $2$, with basis $\{e_{1},e_{2}$\}. The action of
$d={\displaystyle \frac{d}{dx}}$ is given by $de_{1}=e_{2}$, $de_{2}=xe_{1}$.
When one sets $\lambda=0$, $d$ becomes $\xi$ and this is simply
the structure of $\mathcal{O}_{L}$ as a module over $R[x,\xi]$.
The module $\tilde{\mathcal{E}}$ is the $\lambda$-version of the
Airy connection. 

On the other hand, the natural meromorphic connection structure on
the bundle $\pi_{*}(\mathcal{O}_{L},df)$ is computed as follows:
we have $df=2z^{2}dz$ where $z$ is the coordinate on $L$ (so that
$z^{2}=x$). Pushing forward to $\mathbb{A}^{1}$ gives a rank two
bundle whose $\lambda$-connection is
\[
\begin{pmatrix}0 & x\\
1 & \frac{\lambda}{2}x^{-1}
\end{pmatrix}dx
\]
in the basis $\{1,z\}$ of $\pi_{*}(\mathcal{O}_{L_{\mathbb{C}}})$.
So, if we compute the difference between $\tilde{\mathcal{E}}_{\mathbb{C}}$
and $\pi_{*}(\mathcal{O}_{L_{\mathbb{C}}},df)$ in the torsor of lifts
of $\pi_{*}(\mathcal{O}_{L_{\mathbb{C}}})$, we obtain the one-form
${\displaystyle \frac{1}{2}\frac{dz}{z}}$ on $L_{U_{\mathbb{C}}}$. 
\end{example}

Our goal is to characterize these polar terms as follows: 
\begin{thm}
\label{thm:Microlocal-form-of-E}Let $E_{i,\mathbb{C}}$ be a component
of the divisor $E_{\mathbb{C}}=L_{\mathbb{C}}\backslash U_{\mathbb{C}}$.
Then we have
\[
[\psi_{E_{i,\mathbb{C}}}]=\alpha\frac{dz}{z}
\]
for some $\alpha\in\mathbb{Q}$; i.e., $[\psi_{E_{i,,\mathbb{C}}}]$
is a one-form with log poles and rational residue along each component
of $E_{\mathbb{C}}$. 
\end{thm}

The proof of this will require a few steps, and the key point is to
reduce mod $p$ and look at the $p$-curvature of a one-form closely
related to $\psi$. Before carrying this out, it is useful to have
a purely local characterization of $\psi_{E_{i,\mathbb{C}}}$, at
any point $x\in E_{i,\mathbb{C}}$. Recall that the sheaf $\mathcal{D}_{\lambda,X_{\mathbb{C}}}/\lambda^{2}$
can be naturally extended to a sheaf on $T^{*}X_{\mathbb{C}}$ (via
the micro-localization); for any point $x\in T^{*}X_{\mathbb{C}}$
the projection $\pi$ induces an isomorphism on stalks 
\[
(\mathcal{D}_{\lambda}/\lambda^{2})_{x}\tilde{\to}(\mathcal{D}_{\lambda}/\lambda^{2})_{\pi(x)}
\]
After completing at $\pi(x)$, a choice of local coordinates induces
an isomorphism 
\[
(\mathcal{\widehat{D}}_{\lambda}/\lambda^{2})_{\pi(x)}\tilde{\to}(\widehat{D}_{n,\lambda}/\lambda^{2})_{\pi(x)}
\]
where $D_{n,\lambda}$ is the $\lambda$-deformed $n$th Weyl algebra;
i.e. the $\mathbb{C}[\lambda]$-algebra generated by $\{z_{1},\dots,z_{n},\partial_{1},\dots,\partial_{n}\}$
with relations $[\partial_{i},z_{j}]=\lambda\delta_{ij}$, $[z_{i},z_{j}]=0=[\partial_{i},\partial_{j}]$;
and the completion is at the ideal $(z_{1},\dots z_{n},\partial_{1},\dots,\partial_{n})$. 

Let $\widehat{T^{*}X}_{\mathbb{C},x}$ denote the formal completion
of $T^{*}X_{\mathbb{C}}$ at the point $x$. It follows that we may
choose an automorphism, $\sigma:(\mathcal{\widehat{D}}_{\lambda}/\lambda^{2})_{x}\to(\mathcal{\widehat{D}}_{\lambda}/\lambda^{2})_{x}$,
whose reduction mod $\lambda$ is a symplectomorphism, $\sigma':\widehat{T^{*}X}_{\mathbb{C},x}\to\widehat{T^{*}X}_{\mathbb{C},x}$,
so that $\sigma'(\widehat{L})_{x}$ projects isomorphically onto $\widehat{X}_{\pi(x)}$. 

Now suppose that $x\in L_{\mathbb{C}}$ is a closed point in the smooth
part of the divisor $E_{\mathbb{C}}=L_{\mathbb{C}}\backslash U_{\mathbb{C}}$.
Let $\{z_{1},\dots,z_{n}\}$ be local coordinates at $x$ so that
$E_{\mathbb{C}}=\{z_{1}=0\}$ near $x$. 

Then we have 
\begin{prop}
\label{prop:alternative-characterization} Let $x\in L_{\mathbb{C}}$
be a point in the smooth part of the divisor $E_{\mathbb{C}}=L_{\mathbb{C}}\backslash U_{\mathbb{C}}$.
Let $\sigma:(\mathcal{\widehat{D}}_{\lambda}/\lambda^{2})_{x}\to(\mathcal{\widehat{D}}_{\lambda}/\lambda^{2})_{x}$
be as above; and denote 
\[
\widehat{j_{*}(\mathcal{E}_{\lambda})}/\lambda^{2}:=(\mathcal{\widehat{D}}_{\lambda}/\lambda^{2})_{x}\otimes_{(\mathcal{D}_{\lambda}/\lambda^{2})_{x}}j_{*}(\mathcal{E}_{\lambda}/\lambda^{2})
\]

Then we have an isomorphism 
\[
\sigma^{*}\widehat{j_{*}(\mathcal{E}_{\lambda})}/\lambda^{2}\tilde{=}\mathbb{C}[[z_{1},\dots,z_{n}]][z_{1}^{-1}][\lambda]/(\lambda^{2})
\]
under which the $(\widehat{\mathcal{D}_{\lambda}}/\lambda^{2})_{(x)}$-module
structure is given by a flat $\lambda$-connection of the form 
\[
\nabla(1)=\lambda\cdot\psi_{P}+\psi_{O}
\]
where $\psi_{P}={\displaystyle \sum_{s=1}^{n}\sum_{j=1}^{m}z_{1}^{-j}(\sum_{I}a_{I}z_{2}^{i_{2}}\cdots z_{n}^{i_{n}}dx_{s})}$,
where $I$ ranges over multi-indices $(i_{2},\dots,i_{n})$, and the
sum is finite; and $\psi_{O}\in\Omega_{\mathbb{C}[[z_{1},\dots,z_{n}]]}^{1}[\lambda]/(\lambda^{2})$.
The residue class $[\psi_{P}]$ in $\Omega_{\mathbb{C}[[z_{1},\dots,z_{n}]][z_{1}^{-1}]}^{1}/\Omega_{\mathbb{C}[[z_{1},\dots,z_{n}]]}^{1}$
is unique, up to the addition of terms of the form $mz_{1}^{-1}dz_{1}$
where $m\in\mathbb{Z}$; and after adding such a term it agrees with
$\psi_{E_{i,\mathbb{C}},x}$; the power series expansion of $\psi_{E_{i},\mathbb{C}}$
at $x$, where $\psi_{E_{i},\mathbb{C}}$ is defined above. In particular,
this term does not depend on the choice of $\sigma$. 
\end{prop}

\begin{proof}
The choice of $\sigma'$ ensures us that 
\[
\sigma'^{*}\widehat{j_{*}(\mathcal{E}_{\lambda})}/\lambda\tilde{=}\mathbb{C}[[z_{1},\dots,z_{n}]][z_{1}^{-1}]
\]
where we have chosen the numbering so that $x_{1}$ defines the completion
at $x$ of the divisor $E$. The isomorphism 
\[
a:\sigma^{*}\widehat{j_{*}(\mathcal{E}_{\lambda})}/\lambda^{2}\tilde{=}\mathbb{C}[[z_{1},\dots,z_{n}]][z_{1}^{-1}][\lambda]/(\lambda^{2})
\]
follows immediately, as $j_{*}(\mathcal{E}_{\lambda})$ is flat over
$\lambda$. The action of the coordinate derivations endows this module
with the structure of a flat $\lambda$-connection, and writing out
$\nabla(1)$ in coordinates yields the expression $\nabla(1)=\lambda\cdot\psi_{P}+\psi_{O}$.
For this choice of $\sigma$, this expression is unique up to changing
the isomorphism $a$, which amounts to multiplication by a unit in
$\mathbb{C}[[z_{1},\dots,z_{n}]][z_{1}^{-1}][\lambda]/(\lambda^{2})$.
Any such unit is of the form $u=z_{1}^{m}(1+q)+\lambda r$ where $m\in\mathbb{Z}$,
$q\in\mathbb{C}[[z_{1},\dots,z_{n}]]$, and $r\in\mathbb{C}[[z_{1},\dots,z_{n}]][z_{1}^{-1}]$.
Multiplying by $u$ changes the connection by adding 
\[
\lambda\frac{du}{u}=\lambda mz_{1}^{-1}dz_{1}+\lambda(1+q)^{-1}dq
\]
which alters the class $[\psi_{P}]\in\Omega_{\mathbb{C}[[z_{1},\dots,z_{n}]][z_{1}^{-1}]}^{1}/\Omega_{\mathbb{C}[[z_{1},\dots,z_{n}]]}^{1}$
by adding $\lambda mz_{1}^{-1}dz_{1}$. 

To relate this to our construction of $\psi_{E_{i,\mathbb{C}}}$;
we let $\tilde{\mathcal{E}}$ denote any deformation of $\pi_{*}(\mathcal{O}_{L})$
to a $\mathcal{D}_{\lambda}/\lambda^{2}$-module. Let $\widehat{\tilde{\mathcal{E}}}$
be the micro-localization of $\tilde{\mathcal{E}}$ at $x$. After
applying the automorphism $\sigma$, we have that $\sigma^{*}(\widehat{\tilde{\mathcal{E}}})$
is a deformation of $\widehat{L}_{\mathbb{C}}$; i.e., $\sigma^{*}(\widehat{\tilde{\mathcal{E}}})\tilde{=}\mathbb{C}[[x_{1},\dots,x_{n}]][\lambda]/(\lambda^{2})$,
equipped with a flat connection $\tilde{\nabla}$. The restriction
of the one-form $\psi=[\mathcal{E}_{\lambda}/\lambda^{2}]-[\tilde{\mathcal{E}}|_{U}]$
to the formal neighborhood of $x$ can therefore be computed as 
\[
\frac{1}{\lambda}(\nabla-\tilde{\nabla})(1)
\]
where $\nabla$ is the flat connection of $\sigma^{*}\widehat{j_{*}(\mathcal{E}_{\lambda})}/\lambda^{2}$.
Thus the polar term $\psi_{P}$ of $\nabla$ is equal to the polar
term of $\psi$, in the formal neighborhood of $x$, which implies
the proposition. 
\end{proof}
In order to prove \prettyref{thm:Microlocal-form-of-E}, we shall
in fact show that the one-form $\psi_{P}$ takes the form ${\displaystyle \alpha\frac{dz_{1}}{z_{1}}}$.
The key to doing this, in turn, is to use the theorem of Katz \cite{key-24},
theorem 13.0, to limit the shape of the singularities via reduction
mod $p$. Recall that the theorem reads
\begin{thm}
(Katz) Let $\psi$ be a closed one-form on a smooth algebraic variety
$Y$ (over $R$). Let $R\to k$, where $k$ is an algebraically closed
field of characteristic $p$, and consider induced form $\psi$ on
$Y_{k}$. Suppose that the $p$-curvature of $\psi$ is $0$ for all
such $k$ of sufficiently large characteristic. Then the flat connection
defined by $\psi$ has regular singularities and quasi-unipotent monodromy. 

Concretely, this means that if $\overline{Y}$ is a smooth compactification
of $Y$, and $D$ is an irreducible divisor in $\overline{Y}\backslash Y$,
then at the generic point of $D$ we have 
\[
\psi=\alpha\frac{dz}{z}+\psi_{0}
\]
where $z$ is a local equation for $D$ and $\psi_{0}$ has no singularizes. 
\end{thm}

In fact, in loc. cit. the theorem is only stated when $Y$ is a curve;
however, the properties of having regular singularities and being
quasi-unipotent are testable via curve restriction (c.f., e.g., \cite{key-6},
section 6.1). 

We are going to show that, at least locally, we can choose a one-form
which satisfied the conditions of Katz's theorem, and whose whose
poles agree with the polar term $\psi_{P}$ of \prettyref{prop:alternative-characterization}. 

The first step is the following straightforward description of the
behavior of the $p$-support under basic maps:
\begin{lem}
\label{lem:p-support!}1) Suppose $\pi:X_{k}\to Y_{k}$ is a finite
etale morphism of smooth schemes over $k$; then we may form the finite
etale morphism 
\[
d^{*}\pi:T^{*}X^{(1)}\times_{k}\mathbb{A}^{1}\to T^{*}Y^{(1)}\times_{k}\mathbb{A}_{k}^{1}
\]
and if $\mathcal{M}$ is an $\mathcal{D}_{\lambda}$-module over $X_{k}$,
whose scheme-theoretic $p$-support is $Z_{k}\subset T^{*}X_{k}^{(1)}\times_{k}\mathbb{A}^{1}$,
then the module ${\displaystyle \int_{\pi}\mathcal{M}=\pi_{*}\mathcal{M}}$
is scheme-theoretically $p$-supported on the scheme-theoretic image
$(d^{*}\pi)_{*}(Z_{k})$. 

2) Suppose $\iota:X_{k}\to Y_{k}$ is a closed embedding. Then we
may form the smooth morphism 
\[
d^{*}\iota:X_{k}^{(1)}\times_{Y_{k}^{(1)}}T^{*}Y_{k}^{(1)}\times_{k}\mathbb{A}_{k}^{1}\to T^{*}X_{k}^{(1)}\times_{k}\mathbb{A}_{k}^{1}
\]
and if $\mathcal{M}$ is an $\mathcal{D}_{\lambda}$-module over $X_{k}$,
whose scheme-theoretic $p$-support is $Z_{k}\subset T^{*}X_{k}^{(1)}\times_{k}\mathbb{A}^{1}$,
then the module ${\displaystyle \int_{\iota}\mathcal{M}}$ is scheme-theoretically
$p$-supported on the scheme-theoretic inverse image $(d^{*}\iota)^{-1}(Z_{k})$. 
\end{lem}

From this, we conclude: 
\begin{lem}
\label{lem:p-support-of-E-l}(c.f. \cite{key-4}, section 2.5) The
scheme-theoretical $p$-support of $\mathcal{E}_{\lambda}$ is equal
to $L_{U_{k}}^{(1)}\times\mathbb{A}_{k}^{1}\subset T^{*}U_{k}^{(1)}\times\mathbb{A}_{k}^{1}$;
in fact, $\mathcal{E}_{\lambda}$ is a vector bundle of rank $p^{n}$
over this subscheme. 
\end{lem}

\begin{proof}
The $p$-support of the $\lambda$-connection $(\mathcal{O}_{L'}[\lambda],df')$
is precisely $\Gamma(df')^{(1)}\times\mathbb{A}_{k}^{1}$. Indeed,
for any local derivation we have 
\[
\partial(1)=\partial(f')\cdot1
\]
and so; if $\partial$ satisfies $\partial^{[p]}=0$, then 
\[
\partial^{p}(1)=(\partial f')^{p}+\lambda^{p-1}\partial^{p-1}(\partial f')=(\partial f')^{p}
\]
so that $\partial^{p}-(\partial f')^{p}$ annihilates $\mathcal{N}_{\lambda}$.
Thus we see directly that $(\mathcal{O}_{L'}[\lambda],df')$ is a
vector bundle on $\Gamma(df)^{(1)}\times\mathbb{A}_{k}^{1}$. Now
apply \prettyref{lem:p-support!}. 
\end{proof}
In fact, there is, essentially, a converse to this result, characterizing
$\lambda$-connections with $p$-curvature equal to $L_{U_{k}}^{(1)}\times\mathbb{A}_{k}^{1}$: 
\begin{lem}
\label{lem:All-bundles-are-pi-push}Let $\mathcal{P}_{\lambda}$ be
a vector bundle with $\lambda$-connection on $U_{k}$, such that
$\mathcal{P}_{\lambda}/\lambda$ is the Higgs bundle corresponding
to $\mathcal{O}_{L_{U_{k}}}$. Suppose also that the scheme-theoretical
$p$-support of $\mathcal{P}_{\lambda}$ is equal to $L_{U_{k}}^{(1)}\times\mathbb{A}_{k}^{1}\subset T^{*}U_{k}^{(1)}\times\mathbb{A}_{k}^{1}$.
Then there exists a $\lambda$-connection $\nabla'$ on $\mathcal{O}_{L_{U_{k}}}[\lambda]$
such that $\mathcal{P}_{\lambda}\tilde{=}\pi_{*}(\mathcal{O}_{L_{U_{k}}}[\lambda],\nabla')$.
The $p$- support of $(\mathcal{O}_{L_{U_{k}}}[\lambda],\nabla')$
is equal to $\Gamma(df)^{(1)}$. 
\end{lem}

\begin{proof}
Since we have a closed immersion $L_{U_{k}}\to T^{*}U_{k}$, we also
have a closed immersion $L_{U_{k}}\times_{U_{k}}L_{U_{k}}\to L_{U_{k}}\times_{U_{k}}T^{*}U_{k}\tilde{\to}T^{*}L_{U_{k}}$
(the last isomorphism follows because $L_{U_{k}}\to U_{k}$ is etale).
It follows that $\Gamma(df)\subset T^{*}L_{U_{k}}$ is a component
of $L_{U_{k}}\times_{U_{k}}L_{U_{k}}$. 

Since $\pi^{*}\mathcal{P}_{\lambda}$ is $p$-supported on the image
of $L_{U_{k}^{(1)}}\times_{U_{k}^{(1)}}L_{U_{k}^{(1)}}\times\mathbb{A}_{k}^{1}$,
the component $\Gamma(df)^{(1)}$ corresponds to a summand, $\mathcal{L}_{\lambda}$,
of $\pi^{*}\mathcal{P}_{\lambda}$. Looking at the stalk of $\mathcal{L}_{\lambda}$
as a sheaf on $T^{*}L_{U_{k}}^{(1)}\times\mathbb{A}_{k}^{1}$, one
sees that it is a vector bundle of rank $p^{n}$. Therefore $\mathcal{L}_{\lambda}$
is a line bundle on $L_{U_{k}}\times\mathbb{A}_{k}^{1}$. Further,
$\mathcal{L}_{\lambda}/\lambda$ is the Higgs bundle on $T^{*}L_{U_{k}}$
corresponding to $\pi^{*}\mathcal{O}_{L_{U_{k}}}|_{\Gamma(df)}=\mathcal{O}_{\Gamma(df)}$.
Since the projection $\Gamma(df)\to L_{U_{k}}$ is an isomorphism,
we see that $\mathcal{L}_{\lambda}/\lambda\tilde{=}\mathcal{O}_{L_{U_{k}}}$,
which implies $\mathcal{L}_{\lambda}\tilde{=}\mathcal{O}_{L_{U_{k}}}[\lambda]$.
It has a connection $\nabla'$. 

Now, by adjunction, the projection morphism $\pi^{*}\mathcal{P}_{\lambda}\to\mathcal{L}_{\lambda}$
corresponds to a morphism $\mathcal{P}_{\lambda}\to\pi_{*}\mathcal{L}_{\lambda}$;
by regarding both sheaves as sheaves on $L_{U_{k}}^{(1)}\times\mathbb{A}_{k}^{1}$
this is easily seen to be an isomorphism; and the result follows. 
\end{proof}
We'll need another general result about line bundles with flat $\lambda$-connection;
which have a prescribed $p$-curvature: 
\begin{prop}
\label{prop:Difference-of-two-connections}Let $Y_{k}$ be a smooth
$k$-scheme, and suppose $(\mathcal{O}_{Y_{k}}[\lambda],\nabla_{1})$
and $(\mathcal{O}_{Y_{k}}[\lambda],\nabla_{2})$ are two flat $\lambda$-connections
on the trivial bundle; suppose $(\nabla_{1}-\nabla_{2})(1)\in\lambda\cdot\Omega_{Y_{k}}^{1}$.
Let $\nabla_{1}^{(p)},\nabla_{2}^{(p)}:\mathcal{O}_{Y_{k}}[\lambda]\to\Omega_{Y_{k}^{(1)}}^{1}[\lambda]$
denote the resulting $p$-curvatures, and suppose $\nabla_{1}^{(p)}=\nabla_{2}^{(p)}$.
Then we have $(\nabla_{1}-\nabla_{2})(1)=\lambda\psi$ where $\psi\in\Omega_{Y_{k}}^{1}$
is a closed one-form with $p$-curvature $0$. In particular, $\nabla_{1}$
and $\nabla_{2}$ are locally isomorphic as $\lambda$-connections.
\end{prop}

\begin{proof}
Write 
\[
(\nabla_{1}-\nabla_{2})(1)=\sum_{i=1}^{m}\lambda^{i}\psi_{i}
\]
Then since both $\nabla_{1}$ and $\nabla_{2}$ are flat we have $0={\displaystyle \sum_{i=1}^{m}\lambda^{i}d\psi_{i}}$
so that each $d\psi_{i}=0$. We shall show $\psi_{j}=0$ for all $j>1$
and $\psi_{1}$ has $p$-curvature $0$. 

By the basic formula for the $p$-curvature of a line bundle, we have
that the $p$-curvature is additive; i.e. the $p$-curvature of the
flat connection 
\[
(\mathcal{O}_{Y_{k}}[\lambda],\nabla_{1})\otimes(\mathcal{O}_{Y_{k}}[\lambda],\nabla_{2})^{*}
\]
 which takes $1$ to $(\nabla_{1}-\nabla_{2})(1)$ is $\nabla_{1}^{(p)}-\nabla_{2}^{(p)}=0$.
Working locally, pick coordinate derivations $\{\partial_{j}\}_{j=1}^{n}$
and write 
\[
\partial_{j}\cdot1=\sum_{i=1}^{m}\lambda^{i}\psi_{i,j}
\]
where $\psi_{i,j}\in\mathcal{O}$; this is the action of $\mathcal{D}_{\lambda}$
on the module associated to the connection $(\mathcal{O}_{Y_{k}}[\lambda],\nabla_{1})\otimes(\mathcal{O}_{Y_{k}}[\lambda],\nabla_{2})^{*}$.
Then 
\[
\partial_{j}^{p}\cdot1=\sum_{i=1}^{m}\lambda^{pi}\psi_{i,j}^{p}+\sum_{i=1}^{m}\lambda^{p-1+i}\partial_{j}^{p-1}(\psi_{i,j})
\]
If this is zero, then, looking at the highest term in $\lambda$,
we see $\psi_{i,j}^{p}=0$ for all $i>1$; which gives $\psi_{i,j}=0$
for all $i>1$. Looking at the $i=1$ term, we also obtain $\psi_{i,j}^{p}+\partial_{j}^{p-1}(\psi_{i,j})=0$.
As this holds for all $j$, we obtain the result about $\psi$. Finally,
to obtain the last statement, note that by Cartier descent there exists,
locally on $Y_{k}$, an invertible function $u$ so that ${\displaystyle \psi=\frac{du}{u}}$.
But then the $\lambda$-connections $\nabla_{1}$ and ${\displaystyle \nabla_{2}=\nabla_{1}+\lambda\frac{du}{u}}$
are isomorphic on the open subset where $u$ is defined.
\end{proof}
Let's proceed to the 
\begin{proof}
(of \prettyref{thm:Microlocal-form-of-E}) To start off, we note that
it suffices to prove the result after taking an embedding $\iota:X_{\mathbb{C}}\to Y_{\mathbb{C}}$,
where $Y_{\mathbb{C}}$ is another smooth affine variety. In that
case, we have the correspondence
\[
T^{*}X_{\mathbb{C}}\xleftarrow{\rho}X_{\mathbb{C}}\times_{Y_{\mathbb{C}}}T^{*}Y_{\mathbb{C}}\xrightarrow{p}T^{*}Y_{\mathbb{C}}
\]
where $\rho$ is smooth and $p$ is a closed immersion. The exact
Lagrangian $L_{\mathbb{C}}$ is replaced by $L_{\mathbb{C}}'=\rho^{-1}(L_{\mathbb{C}})\subset T^{*}Y_{\mathbb{C}}$;
from the construction of $\mathcal{N}_{\lambda}$, one sees directly
that $L'_{\mathbb{C}}$ is an exact Lagrangian and we have that ${\displaystyle \int_{\iota}\mathcal{E}_{\lambda}}$
is exactly the $\mathcal{D}_{\lambda}$-module associated (by the
construction at the beginning of this section) to $L_{\mathbb{C}}'\subset T^{*}Y_{\mathbb{C}}$.
Furthermore, if $\tilde{\mathcal{E}}$ is a flat $\mathbb{C}[\lambda]/\lambda^{2}$-deformation
of $L_{\mathbb{C}}$ on $X_{\mathbb{C}}$, then ${\displaystyle \int_{\iota}\tilde{\mathcal{E}}}$
is a flat $\mathbb{C}[\lambda]/\lambda^{2}$ deformation of $L'_{\mathbb{C}}$.
Therefore, if $\psi\in\Gamma(\Omega_{L_{U_{\mathbb{C}}}}^{1})$ denotes
the difference class $[\tilde{\mathcal{E}}]-[{\displaystyle \mathcal{E}_{\lambda}}/\lambda^{2}]$,
then $\rho^{*}\psi$ will be the difference class $[{\displaystyle \int_{\iota}\tilde{\mathcal{E}}}]-[{\displaystyle \int_{\iota}\mathcal{E}_{\lambda}/\lambda^{2}}]$.
Thus, by the construction given in \prettyref{lem:Def-of-=00005Cpsi},
we see that it suffices to prove the result for $L'_{\mathbb{C}}$. 

Thus we may, and will, assume that $X_{\mathbb{C}}=\mathbb{A}_{\mathbb{C}}^{n}$;
from now on denote $L'_{\mathbb{C}}$ by $L_{\mathbb{C}}$ and replace
${\displaystyle \int_{\iota}\mathcal{E}_{\lambda}}$ by $\mathcal{E}_{\lambda}$.
Fix a component $E_{1,\mathbb{C}}$ of $L_{\mathbb{C}}$, and choose
a point $x\in E_{1,\mathbb{C}}^{sm}$. Let $\sigma$ be a linear symplectomorphisms
of $T^{*}\mathbb{A}_{\mathbb{C}}^{n}$ such that the differential
of the projection map $d\pi:T(\sigma(L_{\mathbb{C}}))_{x}\to\mathbb{A}_{\mathbb{C},\pi(x)}^{n}$
is an isomorphism at $x$. This implies that $\pi:\sigma(L_{\mathbb{C}})\to\mathbb{A}_{\mathbb{C}}^{n}$
is dominant. Thus, applying the constructions of the beginning of
this section to $\sigma(L_{\mathbb{C}})$, we obtain an open subset
$U'_{\mathbb{C}}\subset\mathbb{A}_{\mathbb{C}}^{n}$ and $\mathcal{D}_{\lambda}$-module
$\tilde{\mathcal{E}}_{\lambda}=e^{g}$ on $U'_{\mathbb{C}}$; here
$dg=\theta|_{\sigma(L_{\mathbb{C}})}$, such that $\tilde{\mathcal{E}}_{\lambda}/\lambda=\mathcal{O}_{L_{U'_{\mathbb{C}}}}$.
We shall assume (shrinking $U'_{\mathbb{C}}$ if needed) that $\sigma(E_{1,\mathbb{C}})$
is in the complement of $\sigma(L_{\mathbb{C}})_{U'_{\mathbb{C}}}$. 

Now, the complement of $U'_{\mathbb{C}}$ in $\mathbb{A}_{\mathbb{C}}^{n}$
is a divisor $D'_{\mathbb{C}}$, whose inverse image in $\sigma(L_{\mathbb{C}})$
is a divisor $E'_{\mathbb{C}}$. The singular locus of $E'_{\mathbb{C}}$
is a codimension $2$ subset of $\sigma(L_{\mathbb{C}})$. Therefore,
by replacing $x$ by a nearby point (in the classical topology)\footnote{this is allowable because, during the proof of \prettyref{lem:Def-of-=00005Cpsi},
we have seen that $[\psi_{P}]$ is completion of the the polar part
of an algebraic one-form on $U_{\mathbb{C}}$. So the result only
depends on the behavior of $\psi$ at the generic point of $E_{1,\mathbb{C}}^{sm}$.} on $E_{1,\mathbb{C}}^{sm}$, we can again suppose that $x$ is in
the smooth part of $E'_{\mathbb{C}}$. 

Now we employ the result of \prettyref{prop:alternative-characterization}:
since the projection to $X_{\mathbb{C}}$ is an isomorphism at $x$,
the micro-localization $\widehat{\sigma^{*}\mathcal{E}}_{\lambda}/\lambda^{2}$
can be regarded as a rank one free $\mathbb{C}[[z_{1},\dots z_{n}]][z_{1}^{-1}][\lambda]/(\lambda^{2})$
module with flat $\mathbb{C}[\lambda]/\lambda^{2}$-connection; and
$[\psi_{P}]$ is the polar part of this connection evaluated at $1$
(divided by $\lambda$). In addition, we have that the micro-localization
of $\tilde{\mathcal{E}}_{\lambda}/\lambda^{2}=e^{g}$ at $x$ can
be regarded as a rank one free $\mathbb{C}[[z_{1},\dots z_{n}]][z_{1}^{-1}][\lambda]/(\lambda^{2})$
module with flat $\mathbb{C}[\lambda]/\lambda^{2}$-connection; this
connection has no pole at $x$; since $dg=\theta$ and the projection
to $\mathbb{A}_{\mathbb{C}}^{m}$ is an isomorphism at $x$. Thus
to compute $[\psi_{P}]$ we can look at the difference between the
one-form associated to $\widehat{\sigma^{*}\mathcal{E}}_{\lambda}/\lambda^{2}$
and the one-form associated to $\widehat{\tilde{\mathcal{E}}_{\lambda}}/\lambda^{2}$. 

To compute this difference, in turn, we can employ \prettyref{prop:Difference-of-two-connections}.
Choose a finite type $\mathbb{Z}$-algebra $R$ over which everything
in sight is defined and flat. We have that the one-form $\psi$ on
$L_{U'}$, which is the difference $[\sigma^{*}\mathcal{E}_{\lambda}/\lambda^{2}]-[\tilde{\mathcal{E}}_{\lambda}/\lambda^{2}]$,
is therefore defined over $R$. Now base change to a perfect field
$k$. According to \cite{key-3}, section 7, proposition $2$, the
$p$-support of $\sigma^{*}(j_{*}\mathcal{E}_{\lambda})$ is equal
to $\sigma(L_{U_{k}})$. Restricting to $U'_{k}$, and applying \prettyref{lem:All-bundles-are-pi-push},
we have that $\sigma^{*}\mathcal{E}_{\lambda}|_{U'_{k}}$ is isomorphic
to $\pi_{*}(\mathcal{O}_{L_{U'_{k}}},\nabla')$, where $\nabla'$
is a connection whose $p$-support is $\Gamma(dg)^{(1)}$. On the
other hand, we have, by definition, $\tilde{\mathcal{E}}_{\lambda}=\pi_{*}(\mathcal{O}_{L_{U_{k}'}},\nabla)$
where $\nabla(1)=dg$. So, we may apply \prettyref{prop:Difference-of-two-connections}
to see that $(\nabla'-\nabla)(1)$ is a one-form which flat and has
$p$-curvature $0$. On the other hand, up to adding a term of the
form ${\displaystyle \frac{du}{u}}$, this one form is simply the
reduction of $\psi$ to $k$. As this is true for all fields $k$,
we conclude that $\psi$ is closed (as its reduction to $k$ is for
all $k$) and, by Katz' theorem quoted above, the associated connection
is regular singular with quasi-unipotent monodromy; but this implies
that the polar term of $\psi$ is as desired. 
\end{proof}
So, we may now make the 
\begin{defn}
\label{def:The-monodromy-divior}The monodromy divisor of $L_{\mathbb{C}}\subset T^{*}X_{\mathbb{C}}$
is the image in $\text{Div}_{\mathbb{Q}}(L_{\mathbb{C}})/\text{Div}_{\mathbb{Z}}(L_{\mathbb{C}})$
of ${\displaystyle \sum_{i}\alpha_{i}E_{i}}$ on $L_{\mathbb{C}}$,
where $E_{i}$ are the components of $E_{\mathbb{C}}$ and $\alpha_{i}$
is the rational number appearing in \prettyref{thm:Microlocal-form-of-E};
associated to any (and hence every) point in $E_{\mathbb{C}}^{sm}\cap E_{i}$.
By \prettyref{prop:alternative-characterization} this is a well-defined
element of $\text{Div}_{\mathbb{Q}}(L_{\mathbb{C}})/\text{Div}_{\mathbb{Z}}(L_{\mathbb{C}})$. 
\end{defn}

\begin{rem}
Let us note that further shrinking $U_{\mathbb{C}}$ does not alter
the monodromy divisor. For, if we an additional component to $E$
over which $\mathcal{E}_{\lambda}/\lambda^{2}$ extends to a deformation
of $\mathcal{O}_{L_{\mathbb{C}}}$. Then, by \prettyref{prop:alternative-characterization},
the resulting one-form will have no poles along any such divisor.
So the resulting monodromy divisor is the same as the original one. 
\end{rem}

As explained in the next subsection, this divisor will be the obstruction
to quantization that we consider in this paper. It will be useful
to note the following compatibility: 
\begin{lem}
\label{lem:Pullback-of-monodromy}Suppose $g:X_{\mathbb{C}}\to Y_{\mathbb{C}}$
is a closed immersion of affine varieties. There is a smooth morphism
$d^{*}g:X_{\mathbb{C}}\times_{Y_{\mathbb{C}}}T^{*}Y_{\mathbb{C}}\to T^{*}X_{\mathbb{C}}$,
as well as a closed immersion $\iota:X_{\mathbb{C}}\times_{Y_{\mathbb{C}}}T^{*}Y_{\mathbb{C}}\to T^{*}Y_{\mathbb{C}}$.
If $L_{\mathbb{C}}\subset T^{*}X_{\mathbb{C}}$ is an exact Lagrangian,
then so is $\iota((d^{*}g)^{-1}(L_{\mathbb{C}}))$. The monodromy
divisor of $\iota((d^{*}g)^{-1}(L_{\mathbb{C}}))$ is the pullback
of the monodromy divisor of $L_{\mathbb{C}}$, under the smooth morphism
$d^{*}g$. 
\end{lem}

This follows directly from the construction. Now, let us record the
following fact which will be used in the next chapter: 
\begin{lem}
\label{lem:Extend-mod-p} Let $\mathcal{E}_{\lambda}$ be as above.
Then for all $k$ of characteristic $p>>0$, $\mathcal{E}_{\lambda,k}$
can be extended to a $\mathcal{D}_{\lambda}$-module $\tilde{\mathcal{E}}_{\lambda,k}$
which is scheme-theoretically supported on $L_{k}^{(1)}\times\mathbb{A}_{k}^{1}$,
and such that $\tilde{\mathcal{E}}_{\lambda,k}/\lambda$ is a line
bundle on $L_{k}$. The set of (isomorphism classes of) such $\mathcal{D}_{\lambda}$-modules
is a torsor over $\text{Pic}(L_{k}^{(1)})$. 
\end{lem}

\begin{proof}
As $\mathcal{E}_{\lambda,k}$ is a vector bundle on $L_{U_{k}}^{(1)}\times\mathbb{A}_{k}^{1}$,
there exists an extension $\tilde{\mathcal{E}}_{\lambda,k}$ which
is scheme-theoretically supported on $L_{k}^{(1)}\times\mathbb{A}_{k}^{1}$
and torsion-free; restricting to a codimension $2$ open and pushing
forward as needed, we can suppose that $\tilde{\mathcal{E}}_{\lambda,k}$
is a reflexive sheaf on $L_{k}^{(1)}\times\mathbb{A}_{k}^{1}$. Thus
we have an injection $\tilde{\mathcal{E}}_{\lambda,k}/\lambda\to\mathcal{L}$
for some line bundle $\mathcal{L}$ on $L_{k}$, and this injection
is an isomorphism is codimension $2$. We want to show it is an isomorphism
everywhere. By pushing forward along an inclusion $X_{k}\subset\mathbb{A}_{k}^{m}$
we can suppose $X_{k}=\mathbb{A}_{k}^{m}$ from the start. Choose
some $x\in L_{k}$. After applying a suitable linear symplectomorphism,
we can suppose that the projection $\pi:L_{k}\to\mathbb{A}_{k}^{m}$
is an isomorphism at $x$. By assumption the completion $\widehat{\tilde{\mathcal{E}}}_{\lambda,k}$
of $\tilde{\mathcal{E}}_{\lambda,k}$ at the ideal of $\{x\}\times\mathbb{A}_{k}^{1}\subset L_{k}^{(1)}\times\mathbb{A}_{k}^{1}$
is reflexive. On the other hand, by \cite{key-49}, proposition 1.6,
a sheaf is reflexive iff it is torsion-free and equal to its pushforward
over an open subset whose complement has codimension $2$. Thus we
see that $\widehat{\tilde{\mathcal{E}}}_{\lambda,k}$ can be regarded
as rank $1$ reflexive coherent sheaf with $\lambda$-connection on
$\widehat{\mathcal{O}}_{\mathbb{A}^{m},\pi(x)}[\lambda]$. But any
such sheaf is a line bundle, which shows that $\widehat{\tilde{\mathcal{E}}}_{\lambda,k}/\lambda$
is a line bundle, which implies the claim. 

For the second statement, we first prove that any two such modules,
say $\tilde{\mathcal{E}}_{\lambda,k}$ and $\tilde{\mathcal{E}}'_{\lambda,k}$,
are locally isomorphic. After inverting $\lambda$, the result of
the previous paragraph implies that $\mathcal{D}_{\lambda}[\lambda^{-1}]$
is a split Azumaya algebra on $L_{k}^{(1)}\times(\mathbb{A}_{k}^{1}\backslash\{0\})$
with $\tilde{\mathcal{E}}_{\lambda,k}[\lambda^{-1}]$ and $\tilde{\mathcal{E}}'_{\lambda,k}[\lambda^{-1}]$
are splitting bundles. Therefore they differ by the action of an element
of $\text{Pic}(L_{k}^{(1)}\times(\mathbb{A}_{k}^{1}\backslash\{0\}))\tilde{=}\text{Pic}(L_{k}^{(1)})$;
in other words, there is an open affine covering of $L_{k}^{(1)}$,
$\{U_{i}\}$, such that $\tilde{\mathcal{E}}_{\lambda,k}[\lambda^{-1}]$
and $\tilde{\mathcal{E}}'_{\lambda,k}[\lambda^{-1}]$ are isomorphic
on $\{U_{i}\times(\mathbb{A}^{1}\backslash\{0\}\}$. Restricting attention
to such a $U_{i}$, we may regard both $\tilde{\mathcal{E}}_{\lambda,k}$
and $\tilde{\mathcal{E}}'_{\lambda,k}$ as $\mathcal{D}_{\lambda}$-
lattices inside $\tilde{\mathcal{E}}_{\lambda,k}[\lambda^{-1}]$.
Multiplying these lattices by powers of $\lambda$ as needed, we may
suppose that $\tilde{\mathcal{E}}_{\lambda,k}\subset\mathbf{\tilde{\mathcal{E}}'}_{\lambda,k}$,
with quotient annihilated by a power of $\lambda$, and also that
the induced map $\tilde{\mathcal{E}}_{\lambda,k}/\lambda\to\tilde{\mathcal{E}'}_{\lambda,k}/\lambda$
is nonzero. But this is a nonzero map of line bundles on the integral
scheme $L_{k}$; hence injective. Thus the quotient $\mathbf{\tilde{\mathcal{E}}'}_{\lambda,k}/\mathbf{\tilde{\mathcal{E}}}_{\lambda,k}$
is $\lambda$-torsion-free; since it is also annihilated by a power
of $\lambda$ it must be $0$ and so $\tilde{\mathcal{E}}_{\lambda,k}=\mathbf{\tilde{\mathcal{E}}'}_{\lambda,k}$
as desired. 

Next we claim 
\[
\mathcal{E}nd_{\mathcal{D}_{\lambda}}(\tilde{\mathcal{E}}_{\lambda,k})\tilde{=}\mathcal{O}_{L_{k}^{(1)}}[\lambda]
\]
There is a map $\mathcal{O}_{L_{k}^{(1)}}[\lambda]\to\mathcal{E}nd_{\mathcal{D}_{\lambda}}(\tilde{\mathcal{E}}_{\lambda,k})$
given by the action of the center $\mathcal{Z}(\mathcal{D}_{\lambda})$,
which by assumption factors through $\mathcal{O}_{L_{k}^{(1)}}[\lambda]$.
On the other hand, the induced map $\mathcal{O}_{L_{k}^{(1)}}[\lambda,\lambda^{-1}]\to\mathcal{E}nd_{\mathcal{D}_{\lambda}[\lambda^{-1}]}(\tilde{\mathcal{E}}_{\lambda,k}[\lambda^{-1}])$
is an isomorphism since $\mathcal{D}_{\lambda}[\lambda^{-1}]$ is
Azumaya and $\tilde{\mathcal{E}}_{\lambda,k}[\lambda^{-1}]$ is a
splitting bundle. Thus we have inclusions 
\[
\mathcal{O}_{L_{k}^{(1)}}[\lambda]\subset\mathcal{E}nd_{\mathcal{D}_{\lambda}}(\tilde{\mathcal{E}}_{\lambda,k})\subset\mathcal{O}_{L_{k}^{(1)}}[\lambda,\lambda^{-1}]
\]
Further, since $\tilde{\mathcal{E}}_{\lambda,k}$ is a coherent sheaf
on $L_{k}^{(1)}\times\mathbb{A}_{k}^{1}$, any endomorphism of it
satisfies a monic polynomial with coefficients in $\mathcal{O}_{L_{k}^{(1)}}[\lambda]$.
Since $\mathcal{O}_{L_{k}^{(1)}}[\lambda]$ is integrally closed we
deduce $\mathcal{O}_{L_{k}^{(1)}}[\lambda]=\mathcal{E}nd_{\mathcal{D}_{\lambda}}(\tilde{\mathcal{E}}_{\lambda,k})$. 

In sum, we see that the set of isomorphism classes of $\mathcal{D}_{\lambda}$-modules
which are scheme-theoretically supported on $L_{k}^{(1)}\times\mathbb{A}_{k}^{1}$,
and whose reduction mod $(\lambda)$ is a line bundle on $L_{k}$,
form a torsor over 
\[
H^{1}(\text{Aut}(\tilde{\mathcal{E}}_{\lambda,k}))\tilde{=}\text{Pic}(L_{k}^{(1)}\times\mathbb{A}^{1})\tilde{=}\text{Pic}(L_{k}^{(1)})
\]
as desired.
\end{proof}

\subsection{\label{subsec:Invariance}Invariance}

Now we can introduce our obstruction to the quantizing $L_{\mathbb{C}}$
\begin{defn}
\label{def:Unobstructed!}Let ${\displaystyle \sum_{i}\alpha_{i}E_{i}}$
be the monodromy divisor of $L_{\mathbb{C}}$. We say that $L_{\mathbb{C}}$
is unobstructed if there exists a line bundle with flat connection
$(\mathcal{K},\nabla)$ on $L_{U_{\mathbb{C}}}$, with finite monodromy
group, whose monodromy around the component $E_{i}$ of $E$ is equal
to $-\alpha_{i}$ (inside $\text{Div}_{\mathbb{Q}}(L_{\mathbb{C}})/\text{Div}_{\mathbb{Z}}(L_{\mathbb{C}})$). 
\end{defn}

As mentioned in the introduction, this condition should be the obstruction
to finding a $\mathcal{D}_{X_{\mathbb{C}}}$-module with constant
arithmetic support $L_{\mathbb{C}}$, of multiplicity $1$. The rest
of the paper is devoted to understanding this question. For now, we'll
just prove the following useful characterization:
\begin{prop}
\label{prop:Characterization-of-ASSUMPTION}Let $L_{\mathbb{C}}\subset T^{*}X_{\mathbb{C}}$
be an exact Lagrangian, with monodromy divisor $\sum\alpha_{i}E_{i}$.
Then $L_{\mathbb{C}}$ is unobstructed iff there is an open subset
$V_{\mathbb{C}}\subset L_{\mathbb{C}}$ whose complement has codimension
$2$, a $\mathcal{D}_{\lambda}/\lambda^{2}(V_{\mathbb{C}})$-module
$\mathcal{P}_{\lambda,\mathbb{C}}$, which is flat over $\mathbb{C}[\lambda]/\lambda^{2}$,
such that $\mathcal{P}_{\lambda,\mathbb{C}}/\lambda^{2}$ is a line
bundle on $V_{\mathbb{C}}$, and which satisfies the following: for
any $R$-model $\mathcal{P}_{\lambda}$, and for all $k$ of characteristic
$p>>0$, $\mathcal{P}_{\lambda,k}|_{U_{k}}$ can be lifted to a $k[\lambda]$-flat
$\mathcal{D}_{\lambda,k}$-module $\tilde{\mathcal{P}}_{\lambda,k}$
which is scheme-theoretically supported on $L_{U_{k}}^{(1)}\times\mathbb{A}_{k}^{1}$. 
\end{prop}

\begin{proof}
Suppose $L_{\mathbb{C}}$ is unobstructed. Momentarily shrinking $U_{\mathbb{C}}$
if necessary, there is a closed one-form $\phi$ on $L_{U}$ and so
that the flat connection on $\mathcal{O}_{L_{U_{\mathbb{C}}}}$, defined
by $\nabla(1)=\phi$, which has a finite monodromy group and whose
monodromy, in $\text{Div}_{\mathbb{Q}}(L_{\mathbb{C}})/\text{Div}_{\mathbb{Z}}(L_{\mathbb{C}})$,
is equal to ${\displaystyle -\sum_{i}\alpha_{i}E_{i}}$. So we can
deduce the existence of $\mathcal{P}_{\lambda,\mathbb{C}}$ as follows:
over $U_{\mathbb{C}}$ define $[\mathcal{P}_{\lambda,\mathbb{C}}]=[\mathcal{E}_{\lambda}/\lambda^{2}]-\phi$
(as a deformation of the Higgs sheaf $\pi_{*}(\mathcal{O}_{L_{\mathbb{C}}})$).
Let $x$ be a closed point contained in the smooth part of $E_{\mathbb{C}}$.
By \prettyref{prop:alternative-characterization}, after applying
an automorphism of $\mathcal{\widehat{D}}_{\lambda,x}/\lambda^{2}$,
we have the description 
\[
\sigma^{*}\widehat{j_{*}(\mathcal{P}_{\lambda,\mathbb{C}})}\tilde{=}\mathbb{C}[[z_{1},\dots,z_{n}]][z_{1}^{-1}][\lambda]/(\lambda^{2})
\]
equipped with a flat $\lambda$-connection. By its very definition,
$\mathcal{P}_{\lambda,\mathbb{C}}$ has been chosen so that this connection
has a pole of the form ${\displaystyle \frac{mdz_{1}}{z_{1}}}$ for
some $m\in\mathbb{Z}$. Therefore, multiplying the basis element by
$z_{1}^{-m}$ we obtain a basis over which the $\lambda$-connection
has no singularity; and therefore an extension of the $\lambda$-connection
$\mathcal{P}_{\lambda,\mathbb{C}}$ to a bundle with $\lambda$-connection
over a formal neighborhood of $x$. We can then consider the subsheaf
of $j_{*}(\mathcal{P}_{\lambda,\mathbb{C}})$ consisting of sections
whose image in the completion at $x$ is contained in this extension.
Doing this over each component, we obtain an open subset $V_{\mathbb{C}}\subset L_{\mathbb{C}}$
of codimension $2$, a line bundle $\mathcal{L}$ on $V_{\mathbb{C}}$,
and a $\mathcal{D}_{\lambda}/\lambda^{2}$-module $\mathcal{P}_{\lambda,\mathbb{C}}$
which deforms the Higgs sheaf corresponding to $\mathcal{L}$ on $V_{\mathbb{C}}$.
Since $\phi$ has $p$-curvature $0$ for all $p>>0$, the condition
on the reduction mod $p$ of $\mathcal{P}_{\lambda}|_{U}$ follows
from \prettyref{prop:Difference-of-two-connections}. 

For the converse, choose an open affine $U'\subset U$ on which $\mathcal{P}_{\lambda}/\lambda$
is the trivial bundle. Then $[\mathcal{P}_{\lambda}|_{U'}]-[\mathcal{E}_{\lambda}/\lambda^{2}|_{U'}]$
is a one form $\phi$ on $U'$, and, arguing as in the proof of \prettyref{thm:Microlocal-form-of-E}
we see that $d\phi=0$ and that the reduction mod $p$ of $\phi$
has $p$-curvature $0$ for all $p>>0$. Thus, by \prettyref{thm:CC},
$\phi$ defines a finite order connection on $U'$, and by construction
$\overline{\text{res}}(\phi)=\sum\alpha_{i}E_{i}$, and the result
follows. 
\end{proof}
Here is a very useful consequence: 
\begin{cor}
\label{cor:Invariance} If $L_{\mathbb{C}}\subset T^{*}\mathbb{A}_{\mathbb{C}}^{m}$
is an exact Lagrangian, then the condition that $L_{\mathbb{C}}$
is unobstructed is invariant under linear symplectomorphism; i.e.,
it is satisfied for $L_{\mathbb{C}}$ iff it is satisfies for $\sigma(L_{\mathbb{C}})$,
when $\sigma$ is some linear symplectomorphism of $T^{*}\mathbb{A}_{\mathbb{C}}^{m}$.
Therefore, \prettyref{thm:1} holds iff it holds for $X=\mathbb{A}^{m}$
and in the case where $L_{\mathbb{C}}\to X_{\mathbb{C}}$ is dominant. 
\end{cor}

\begin{proof}
The first part follows immediately from the previous proposition and
the invariance of the $p$-curvature under linear symplectomorphisms
of $T^{*}\mathbb{A}^{m}$. For the second sentence, note that, if
$X_{\mathbb{C}}\subset\mathbb{A}_{\mathbb{C}}^{m}$, then Kashiwara's
equivalence of $\mathcal{D}_{X_{\mathbb{C}}}$-modules and $\mathcal{D}_{\mathbb{A}_{\mathbb{C}}^{m}}$
modules supported on $X_{\mathbb{C}}$ reduces the theorem to $\mathbb{A}_{\mathbb{C}}^{m}$. 
\end{proof}
Finally, to finish out this section, we give a few remarks about the
meaning of the obstruction. According to \cite{key-4}, conjecture
$6$, any $\mathcal{D}$-module with $\text{Ext}^{1}(\mathcal{M},\mathcal{M})=0$
should be of the extended motivic exponential type. Roughly speaking,
this means that $\mathcal{M}$ is obtained from modules of the type
$"e^{f}"$ by the basic $\mathcal{D}$-module operations: pushing
forward, pulling back, tensor product, etc. As $e^{f}$ admits an
extension to a $\lambda$-connection, which has constant $p$-support
in $\lambda$ for all $p>>0$, one can conjecture that the same is
true of $\mathcal{M}$. If this is so, then the ``difference'' between
$\mathcal{M}$ and $\mathcal{E}$ should be visible mod $\lambda^{2}$
(via \prettyref{prop:Difference-of-two-connections}). 

\section{\label{sec:Higher-D}Meromorphic Connections}

Consider an irreducible smooth exact Lagrangian $L\subset T^{*}X$.
In the previous chapter, specifically \prettyref{lem:Extend-mod-p},
we have seen that, for any $k$, we can find some $\mathcal{D}_{X_{k}}$-module
$\mathcal{E}$ which is a splitting bundle for $L_{k}^{(1)}$. In
fact, we found a $\mathcal{D}_{\lambda}$-module $\mathcal{E}_{\lambda}$,
specializing to $\mathcal{E}$ at $\lambda=1$, whose $p$-curvature
is equal to $L_{k}^{(1)}\times\mathbb{A}_{k}^{1}$. We want to lift
$\mathcal{E}$ (in fact, $\mathcal{E}_{\lambda}$ ) to something in
characteristic $0$. There are several problems with doing so, both
of which we shall solve by employing a compactification. 

The first problem is the very large space of possible deformations.
Choose a morphism $R\to W(k)$ lifting $R\to k$. As $\mathcal{E}$
is a splitting bundle for $L_{k}^{(1)}$, deformations of $\mathcal{E}$
over $W_{2}(k)$ are indexed by $\text{Ext}^{1}(\mathcal{E},\mathcal{E})\tilde{=}\Omega_{L_{k}^{(1)}}^{1}$
(for this isomorphism see \prettyref{lem:Ext for a splitting bundle}
below); the same holds for infinitesimal deformations and every level.
Further, even if one is able to single out a particular sequence of
deformations, taking the inverse limit yields an object over the formal
scheme $\mathfrak{X}_{W(k)}$; there is no guarantee that it comes
from an algebraic $\mathcal{D}$-module on $X$. 

Passing to a suitable compactification of $X$ will allow us to both
limit the space of deformations, and then, once one has a suitable
object over $W(k)$, algebrize it via Grothendieck's existence theorem
in formal geometry. We wish to compactify in such a way that the set
of deformations becomes a torsor over $H^{0}(\Omega_{\overline{L}_{k}^{(1)}}^{1})\oplus H^{1}(\mathcal{O}_{\overline{L}_{k}^{(1)}})$,
where $\overline{L}$ is a suitable compactification of $L$. The
assumption that $H_{dR}^{1}(L_{\mathbb{C}})=0$ then forces $H^{0}(\Omega_{\overline{L}_{k}^{(1)}}^{1})\oplus H^{1}(\mathcal{O}_{\overline{L}_{k}^{(1)}})=0$
for $k$ of sufficiently large characteristic; thus allowing us to
construct unique infinitesimal deformations to all orders. There is
also the issue of obstructions to deforming. This is why we need $\lambda$-connections
instead of just ordinary connections, as we will deal with the obstruction
in the world of $\lambda$-connections using precisely the assumption
that the Lagrangian is unobstructed, c.f. \prettyref{def:M-Lambda-2}
below. 

Let us give here the basic set-up for how this is done. Then we will
discuss a bit more the contents of the different sections of the chapter. 

We suppose that the projection $L_{\mathbb{C}}\to X_{\mathbb{C}}$
is dominant. Therefore, there is an open set $U_{\mathbb{C}}\subset X_{\mathbb{C}}$
such that $L|_{U_{\mathbb{C}}}\to U_{\mathbb{C}}$ is finite etale.
Let $D_{\mathbb{C}}$ be the hypersurface $X_{\mathbb{C}}\backslash U_{\mathbb{C}}$.
As we are in characteristic $0$, we may invoke resolution of singularities
to deduce that there is a birational projective map $\varphi:\tilde{X}_{\mathbb{C}}\to X_{\mathbb{C}}$
so that $\varphi^{-1}(D_{\mathbb{C}})$ is a global normal crossings
divisor in $\tilde{X}_{\mathbb{C}}$. The map $\varphi$ is an isomorphism
over $U_{\mathbb{C}}$, and we shall write $U_{\mathbb{C}}\subset\tilde{X}_{\mathbb{C}}$
for $\varphi^{-1}(U_{\mathbb{C}})$. This implies that $T^{*}U_{\mathbb{C}}$
is an open dense subset inside $T^{*}\tilde{X}_{\mathbb{C}}$. In
addition, we can further compactify $\tilde{X}_{\mathbb{C}}$ to $\overline{\tilde{X}}_{\mathbb{C}}$,
such that $\overline{\tilde{X}}_{\mathbb{C}}\backslash\tilde{X}_{\mathbb{C}}$
is also a normal crossings divisor, whose union with $\tilde{D}_{\mathbb{C}}$,
also denoted $\tilde{D}_{\mathbb{C}}$, is normal crossings as well. 

Since $L|_{U_{\mathbb{C}}}\to U_{\mathbb{C}}$ is a finite map, we
may define $\overline{L}_{\mathbb{C}}$ to be the normalization of
$\overline{\tilde{X}}_{\mathbb{C}}$ inside $K(L_{\mathbb{C}})$.
Then $\overline{L}_{\mathbb{C}}\to\tilde{\bar{X}}_{\mathbb{C}}$ is
a finite map, which is simply equal to $L|_{U_{\mathbb{C}}}\to U_{\mathbb{C}}$
upon restriction to $U_{\mathbb{C}}$. 

For later use, we shall need a few technical assumptions, which can
always be arranged after further blowing up. 

\begin{assumption}For each componant $E_{i,\mathbb{C}}$ of $E_{i,\mathbb{C}}$, the valuation associated to $E_{i,\mathbb{C}}$ defined a divisor in $\overline{\tilde{X}}_{\mathbb{C}}$
\end{assumption}

Let $\tilde{E}_{\mathbb{C}}$ denote the complement of $L_{U_{\mathbb{C}}}$
in $\overline{L}_{\mathbb{C}}$. Then $\tilde{E}_{\mathbb{C}}$ is
a divisor in $\overline{L}_{\mathbb{C}}$, and $\pi:\tilde{E}_{\mathbb{C}}\to\tilde{D}_{\mathbb{C}}$
is onto; in particular the birational map $L_{\mathbb{C}}\dasharrow\overline{L}_{\mathbb{C}}$
extends over the generic point of each $E_{i,\mathbb{C}}$, and is
therefore an open immersion on an open subset, whose complement has
codimension $\geq2$ in $L_{\mathbb{C}}$. 
\begin{defn}
\label{def:Belongs-to-L}Let $\tilde{E}_{i,\mathbb{C}}$ be a component
of $\tilde{E}_{\mathbb{C}}$. We say that $\tilde{E}_{i,\mathbb{C}}$
belongs to $L_{\mathbb{C}}$ if the generic point of $\tilde{E}_{i,\mathbb{C}}$
is in the image of the map $L_{\mathbb{C}}\dasharrow\overline{L}_{\mathbb{C}}$. 
\end{defn}

With this in hand, we can now state our second technical assumption: 

\begin{assumption} \label{assumption:Main-Assump} Let $\tilde{E}_i$ be a componant of the divisor $E \subset \overline{L}$, which belongs to $L$. Let $\tilde{D}_i$ be the image of $\tilde{E}_i$ under $\pi$. Then, if $\tilde{D}_j$ is another componant of $\tilde{D}$ which has nonempty intersection with $\tilde{D}_i$, we suppose that there is no componant of $\tilde{E}$ belonging to $L$ whose image under $\pi$ is equal to $\tilde{D}_i$.
\end{assumption}

This can be arranged by blowing up intersections of divisors inside
$\overline{\tilde{X}}_{\mathbb{C}}$. 

Now we may choose a ring $R\subset\mathbb{C}$, finitely generated
over $\mathbb{Z}$, so that $X_{\mathbb{C}}$, $\tilde{X}_{\mathbb{C}}$,
$L_{\mathbb{C}}$,$\overline{L}_{\mathbb{C}}$ and $\tilde{D}_{\mathbb{C}}$
all admit $R$-models. We choose such models now and denote them $X$,
$\tilde{X}$, $L$, $\overline{L}$,$\tilde{D}$,$\tilde{E}$; extending
$R$ as necessary, we may also suppose that each of these models has
an $R$-point. We may also suppose that the function $f$ belongs
to $\Gamma(\mathcal{O}_{L})$. 

Further, since $R$ is a finite type $\mathbb{Z}$-algebra, after
localizing at an integer it becomes smooth over $\mathbb{Z}$. Thus
we may suppose that all schemes appearing here (except $\overline{L}$)
are smooth over $\mathbb{Z}$, and, in particular, are regular schemes.
We may suppose $\tilde{\bar{L}}$ is a normal scheme.

In \prettyref{subsec:Higgs-Sheaves} below, we will construct a family
(indexed by rank one reflexive coherent sheaves on $\overline{L}$)
of meromorphic Higgs bundles on $\overline{\tilde{X}}$, which capture
well the ``behavior at infinity'' of $\overline{L}$. we call them
$\theta$-regular at infinity; or just $\theta$-regular for short.
The definition is chosen so that infinitesimal deformations of such
Higgs bundles (which have the same proscribed behavior at infinity)
are in bijection with $H^{0}(\Omega_{\overline{L}}^{1})\oplus H^{1}(\mathcal{O}_{\overline{L}})$
(which is in particular $0$ in the case of interest to us, c.f. \prettyref{thm:Infinitesimal-Def}
below. 

In \prettyref{subsec:-regular--connections-in-positive-char}, we
deform the $\theta$-regular Higgs sheaves to $\lambda$-connections
in positive characteristic, essentially by the method of \prettyref{lem:Extend-mod-p}.
However, we have to keep track of the behavior at infinity in a way
that parallels the discussion of the previous section; we obtain in
particular \prettyref{def:Theta-Reg-Pos-Char} for a $\theta$-regular
connection in positive characteristic. Such connections have a residue,
and, once one fixes it, the resulting collection is a torsor over
rank one reflexive coherent sheaves on $\overline{L}_{k}^{(1)}$ (c.f.
\prettyref{cor:Pic-with-residue-over-k}). 

In \prettyref{subsec:Infinitesimal--regular-Connectio}, we fully
develop the theory of infinitesimal $\theta$-regular connections,
whose definition over $R[\lambda]/\lambda^{n}$ is modeled after the
constructions of the previous section. Here we show the aforementioned
results about uniqueness of infinitesimal deformations (when $H_{dR}^{1}(L_{\mathbb{C}})=0$).
We also use the unobstructedness condition to construct a family of
$\theta$-regular sheaves over $R[\lambda]/\lambda^{2}$, denoted
$\mathcal{L}\star\overline{\mathcal{M}}_{\lambda,2}$ (indexed by
line bundles with finite order connection on $L_{U}$). After reduction
mod $p$ these connections match up with the ones constructed in the
previous section; this allows us to kill the obstructions to lifting
$\mathcal{L}\star\overline{\mathcal{M}}_{\lambda,2}$ to obtain a
family of $\lambda$-connections over $R[\lambda]/\lambda^{n}$ (this
is carried out in \prettyref{subsec:-regular-connections-over}),
and then a family of $\lambda$-connections over over $R[[\lambda]]$;
denoted by $\mathcal{L}\star\overline{\mathcal{M}}_{\widehat{\lambda}}$
In particular, we can conclude that, after reduction to $k$, each
$\mathcal{L}\star\overline{\mathcal{M}}_{\widehat{\lambda},k}$ is
the $\lambda$-adic completion of a unique $\overline{\mathcal{N}}_{\lambda,k}$
which was constructed in \prettyref{subsec:-regular--connections-in-positive-char}. 

In \prettyref{subsec:Lifting-to}, we show that if $\overline{\mathcal{N}}_{\lambda,k}$
is a meromorphic $\lambda$-connection over $\overline{\tilde{X}}_{k}$,
whose completion along $(\lambda)$ agrees with $\mathcal{L}\star\overline{\mathcal{M}}_{\widehat{\lambda},k}$,
then $\overline{\mathcal{N}}_{\lambda,k}$ admits a unique lifting
to $\overline{\mathcal{N}}_{\lambda,W_{m}(k)}$ for each $m$ (the
fact that the completion along $(\lambda)$ is already known to admit
such a lifting is used crucially). Taking the inverse limit, we obtain
an algebraic vector bundle with $\lambda$-connecton over $\overline{\tilde{X}}_{W(k)}$
(c.f. \prettyref{thm:Algebrization}). 

Finally, in \prettyref{subsec:Arithmetic-Support}, we show that these
constructions ``agree'' for different fields $k$. More precisely,
we have the following: if $\overline{\mathcal{N}}$ is a $\theta$-regular
Higgs sheaf over $\overline{\tilde{X}}$, which is the reduction mod
$\lambda$ of a $\lambda$-connection of the form $\mathcal{L}\star\overline{\mathcal{M}}_{\widehat{\lambda}}$,
then, if we have $R\to W(k)$, we have constructed a $\lambda$-connection
$\overline{\mathcal{N}}_{\lambda,W(k)}$. Then, after possibly extending
$R$, there is a vector bundle with connection on $U$, $\mathcal{N}_{1}$,
so that 
\[
\mathcal{N}_{1}\otimes_{R}W(k)\tilde{=}\overline{\mathcal{N}}_{\lambda,W(k)}
\]
Furthermore, if $R\to W(k')$ is any other map to an algebraically
closed field of positive characteristic, then we also have 
\[
\mathcal{N}_{1}\otimes_{R}W(k')\tilde{=}\overline{\mathcal{N}}_{\lambda,W(k')}
\]
In particular, $\mathcal{N}_{1}$ is constant arithmetic support equal
to $L_{U}$ (c.f. \prettyref{thm:Uniqueness-of-lambda-conns} for
details). 

\subsection{Higgs Sheaves \label{subsec:Higgs-Sheaves}}

We suppose that the projection $L_{\mathbb{C}}\to X_{\mathbb{C}}$
is dominant. Therefore, there is an open set $U_{\mathbb{C}}\subset X_{\mathbb{C}}$
such that $L|_{U_{\mathbb{C}}}\to U_{\mathbb{C}}$ is finite etale. 
\begin{defn}
\label{def:The-Higgs-Bundle-H}Let $\mathcal{L}$ be any line bundle
on $L_{U_{\mathbb{C}}}$; equip it with the Higgs field given by the
one form $df$. Via the BNR correspondence of \cite{key-45}, the
fact that $L_{U_{\mathbb{C}}}\subset T^{*}U_{\mathbb{C}}$ implies
that the bundle $\pi_{*}\mathcal{L}:=\mathcal{N}_{\mathcal{L}}$ is
equipped with a canonical Higgs field, denoted $\Theta$. Said more
concretely, the fact that $L_{U_{\mathbb{C}}}\subset T^{*}U_{\mathbb{C}}$
means that $\pi_{*}(\mathcal{O}_{L_{U_{\mathbb{C}}}})$ inherits the
structure of a module over $\pi_{*}(\mathcal{O}_{T^{*}U_{\mathbb{C}}})$
which is the symmetric algebra of the tangent sheaf. The associated
action of tangent vectors yields the aforementioned Higgs field $\Theta$. 
\end{defn}

Most of the time, the line bundle $\mathcal{L}$ will be implicit,
and we will simply write $\mathcal{N}$. 

Let $\mathcal{L}$ be a line bundle on $L_{U}$. 

We would like to study how the Higgs bundle $\mathcal{N}=\pi_{*}(\mathcal{L})$
degenerates as we approach the divisor $\tilde{D}_{\mathbb{C}}$.
The key to doing this is Abhyankar's Theorem; to state it, let $x\in\tilde{D}$
be any closed point, let \textbf{$N_{x}$ }be an open neighborhood
of $x$ in which $\tilde{D}$ is defined by local coordinates\textbf{
$\{z_{i}\}_{i=1}^{d}$.} Recall that $L_{U}\to U$ is a finite etale
cover of degree $r>0$, and, after localizing $R$ at $r$, it is
a tamely ramified cover at each point of $U$. 

For any integer $l>0$, define the affine scheme $N_{x}^{(l)}$ as
the $\mbox{Spec}$ of\linebreak{}
 $O(N_{x})[y_{1},\dots,y_{d}]/(y_{i}-z_{i}^{l})$. Then we have finite
flat morphism $p_{l}:N_{x}^{(l)}\to N_{x}$. 
\begin{thm}
\label{thm:Abhyankar}(Abhyankar's Theorem) After possibly shrinking
$N_{x}$, there exists $l>0$ so that the base change of $L_{U}\to N_{x}\cap U_{\mathbb{C}}$
over the map $p_{l}:N_{x}^{(l)}\to N_{x}$ extends to a finite etale
cover $\pi_{l}:\tilde{L}^{(l)}\to N_{x}^{(l)}$. 
\end{thm}

The proof of this theorem may be found in {[}SGA1{]}, chapter 13,
appendix 1. Let us consider the consequences for $\mathcal{N}$. If
$j:U\to\tilde{\overline{X}}$ denotes the inclusion, we can consider
$j_{*}\mathcal{N}$ as a Higgs sheaf on $\tilde{\overline{X}}$.

Given the map $p_{l}:N_{x}^{(l)}\to N_{x}$ as in \prettyref{thm:Abhyankar},
we may further consider the pullback $p_{l}^{*}(j_{*}\mathcal{N})$
as a Higgs sheaf on $N_{x}^{(l)}$; we note that this sheaf is naturally
a sheaf of modules over the sheaf of rings $\mathcal{O}_{\tilde{\overline{X}_{\mathbb{C}}}}[y_{1}^{-1},\dots,y_{m}^{-1}]$,
where, as above $\{y_{1},\dots,y_{m}\}$ are defining equations of
the components of $\tilde{D}$ near $x_{i}$. 

By construction, the sheaf $p_{l}^{*}(j_{*}\mathcal{N})$ admits an
action of the group $G_{l}$. Then we have the 
\begin{prop}
\label{prop:Theta-i's-higher-dim} After passing to an etale neighborhood
$V_{x}$ of $x$, for any point $x'\in V_{x}$, the Higgs sheaf $p_{l}^{*}(j_{*}\mathcal{N})$
decomposes, as a Higgs sheaf, as a direct sum of rank one free modules
$\mathcal{O}_{x',V_{x}}[y_{1}^{-1},\dots,y_{m}^{-1}]$, equipped with
pairwise distinct Higgs fields. The group $G_{l}$ acts by permuting
the summands. 
\end{prop}

\begin{proof}
We have the commutative diagram $$ \begin{CD} L^{(l)}_{U} @> p_{l}^{'} >> L_{U}  \\ \pi_{l} @VVV \pi @VVV  \\\  N^{(l)}  @> p_l >> N \end{CD} $$

By \prettyref{thm:Abhyankar}, the etale map $\pi_{l}:L_{U}^{(l)}\to N$
extends to a $G_{l}$-equivariant etale map $\pi_{l}:L^{(l)}\to N^{(l)}$.
After passing to a further $G_{l}$-equivariant etale cover $V_{x}^{(l)}$
of $N^{(l)}$, the cover $\pi_{l}$ splits ($G_{l}$-equivariantly).
On the other hand, the $G_{l}$-invariant subalgebra of the Henselization
of $N^{(l)}$ at $x$ is the Henselization of $N$ at $x$. Therefore
we may suppose that $V_{x}^{(l)}=V_{x}\times_{N}N^{(l)}$ for some
etale neighborhood $V_{x}$ of $x$. By construction we have

\begin{equation}
L^{(l)}\times_{N^{(l)}}V_{x}^{(l)}\tilde{=}\bigsqcup_{i=1}^{r}V_{x}^{(l)}\label{eq:split}
\end{equation}
and the induced action of $G_{l}$ permutes the components. 

Now we examine the Higgs sheaves. Pulling everything back to $V_{x}$,
we have the isomorphisms
\[
p_{l}^{*}(j_{*}\mathcal{N}_{V_{x}})=p_{l}^{*}((\pi)_{*}\mathcal{L}|_{U_{x}^{(l)}})\tilde{=}(\pi_{l})_{*}(p_{l}')^{*}\mathcal{L}|_{U_{x}^{(l)}}
\]
where $U_{x}^{(l)}$ is the inverse image of $U$ in $V_{x}^{(l)}$.
By definition, $(p_{l}')^{*}\mathcal{L}|_{U_{x}^{(l)}}$ is a rank
one Higgs bundle on $L_{U_{x}^{(l)}}$. After pushing forward to $V_{x}^{(l)}$,
we obtain a rank one Higgs bundle $j_{*}(p_{l}')^{*}\mathcal{L}|_{U_{x}^{(l)}}$
over $\mathcal{O}_{V_{x}^{(l)}}[y_{1}^{-1},\dots,y_{d}^{-1}]$. On
the other hand, as noted above, the sheaf $p_{l}^{*}(j_{*}\mathcal{N}_{V_{x}})$
is obtained from $j_{*}(p_{l}')^{*}\mathcal{L}|_{U_{x}^{(l)}}$ by
applying the pushforward $\pi_{l}$. Thus \prettyref{eq:split} now
implies that the bundle is a direct sum of rank one Higgs bundles. 

To see that the corresponding Higgs fields are pairwise distinct,
note that the map $V_{x}^{(l)}\to U$ is etale. Therefore the closed
immersion
\[
L_{U}\subset T^{*}U
\]
pulls back to a closed immersion
\[
L^{(l)}\times_{N^{(l)}}V_{x}^{(l)}\subset T^{*}V_{x}^{(l)}
\]
so the result follows from \prettyref{eq:split}. 
\end{proof}
Given this, we make the
\begin{defn}
\label{def:omega-reg-prime-higher-dim}Let $\overline{\mathcal{N}}'$
be an extension of the Higgs bundle $p_{l}^{*}\mathcal{N}$ to a $G_{l}$-equivariant
meromorphic Higgs bundle on $N^{(l)}$. Let $\{e_{i}\}_{i=1}^{r}$
denote a $G_{l}$-equivariant basis of $\mathcal{N}'_{V_{x}^{(l)}}[y_{1}^{-1},\dots,y_{d}^{-1}]$
on which the Higgs fields acts diagonally (which exists by \prettyref{prop:Theta-i's-higher-dim}).
Then $\overline{\mathcal{N}}'$ is said to be $\theta$-regular at
infinity (at $x$) if $\overline{\mathcal{N}}'_{V_{x}^{(l)}}$ has
a basis $\{\overline{e}_{i}\}$ whose localization at $\{y_{1},\dots,y_{d}\}$
is $\{e_{i}\}$. 

Fixing such a basis for each closed point, we let $[\Theta]_{x}$
(or simply $[\Theta]$ if the point $x$ is understood) denote the
matrix of the Higgs field of $\overline{\mathcal{N}}'_{x}$ in this
basis. It follows immediately (from the fact that the $\overline{e}_{i}$
are eigenvectors for the Higgs field) that this matrix is diagonal;
let $\theta_{i}\in\Omega_{\widehat{N}_{x}^{(l)}}^{1}$ denote the
$(i,i)$ entry of $[\Theta]$. 
\end{defn}

To motivate this definition, we give the 
\begin{example}
\label{exa:Push} As in the proof of \prettyref{prop:Theta-i's-higher-dim},
we let $\pi_{l}:L^{(l)}\to N_{x}^{(l)}$ be the etale cover from \prettyref{thm:Abhyankar}.
Then, for any line bundle $\mathcal{L}$ on $L^{(l)}$, we have the
meromorphic Higgs sheaf $(\mathcal{L},df)$. By pushing forward under
the etale map $\pi_{l}$ we see that $(\pi_{l})_{*}\mathcal{L}$ inherits
the structure of a meromorphic Higgs bundle. If we set $\overline{\mathcal{N}}'=(\pi_{l})_{*}\mathcal{L}$,
we obtain (by the proof of \prettyref{prop:Theta-i's-higher-dim})
that $\overline{\mathcal{N}}'$ is $\theta$-regular at infinity. 

From the construction, we see that the eigenvalues of the Higgs field
on this $(\pi_{l})_{*}\mathcal{L}$ correspond to the distinct components
of the completion of $L^{(l)}$ along the fibre $\pi_{l}^{-1}(x)$. 
\end{example}

Next, we consider the descent of such bundles to $\overline{\tilde{X}}$.
Since $V_{x}^{(l)}\to V_{x}$ is finite flat, each $G_{l}$-isotypic
componant of $\overline{\mathcal{N}}'$ is a vector bundle over $V_{x}$.
So we can make the
\begin{defn}
\label{def:omega-reg-higher-dim}Let $\overline{\mathcal{N}}$ be
an extension of $\mathcal{N}$ to a meromorphic Higgs bundle on $\overline{\tilde{X}}$.
Then $\overline{\mathcal{N}}$ is said to be $\theta$-regular at
infinity if, for each point $x\in\overline{\tilde{X}}\backslash U$,
and any etale neighborhood $V_{x}$ as in \prettyref{thm:Abhyankar},
there is a $G_{l}$-equivariant meromorphic Higgs bundle $\overline{\mathcal{N}}'$
on $V_{x}^{(l)}$ which is $\theta$-regular at infinity, 
\[
\overline{\mathcal{N}}=(\overline{\mathcal{N}}')^{G_{l}}
\]
\end{defn}

Let $\text{Pic}^{r}(\overline{L})$ denote the group of isomorphism
classes of rank $1$ reflexive coherent sheaves on $\overline{L}$-
these are equivalent to line bundles on the codimension $2$ smooth
open subset of $\overline{L}$. Then we have
\begin{prop}
\label{prop:push-L-is-theta-reg}Let $\mathcal{L}\in\text{Pic}^{r}(\overline{L})$.
Then $(\mathcal{L},df)$ is a meromorphic Higgs sheaf, and $\pi_{*}(\mathcal{L},df)$
is a Higgs bundle on $\overline{\tilde{X}}$, which is $\theta$-regular
at infinity. 
\end{prop}

\begin{proof}
The result is clear over $U$. Let $x\in\overline{\tilde{X}}\backslash U$
and let $N_{x}^{(l)}\to N_{x}$ be as above. Then since $\mathcal{L}$
is a line bundle in codimension $2$ on $\overline{\tilde{L}}$, we
can pull it back (in codimension $2$) under $p_{l}:L^{(l)}\to\overline{\tilde{L}}|_{N^{(l)}}$;
the resulting object extends uniquely to a $G_{l}$-equivariant line
bundle $\mathcal{L}'$ on $L^{(l)}$ (by \cite{key-49}, proposition
1.8; since $L^{(l)}$ is a regular scheme). We have $(\mathcal{L}')^{G_{l}}=\mathcal{L}$.
By \prettyref{exa:Push} $(\pi_{l})_{*}\mathcal{L}'$ is $\theta$-regular
at infinity. Therefore so is 
\[
((\pi_{l})_{*}\mathcal{L}')^{G_{l}}=\pi_{*}(\mathcal{L})
\]
as claimed. 
\end{proof}
Our next goal is the converse to this theorem; which can be considered
a BNR correspondence with singularities:
\begin{thm}
\label{thm:Pic} There is a bijection between $\text{Pic}^{r}(\overline{L})$
and the set of Higgs bundles which are $\theta$-regular at infinity
on $\overline{\tilde{X}}$; it is given by $\mathcal{L}\to\pi_{*}\mathcal{L}$. 
\end{thm}

In order to prove this, we first show:
\begin{lem}
\label{lem:Local-Endomorphisms}Let $\overline{\mathcal{N}}$ be a
Higgs sheaf on $\overline{\tilde{X}}$ which is $\theta$-regular
at infinity; and let $\mathcal{E}nd(\overline{\mathcal{N}})$ denote
the sheaf of endomorphisms of $\overline{\mathcal{N}}$ which respect
the Higgs field. Then we have 
\[
\mathcal{E}nd(\overline{\mathcal{N}})\tilde{=}\pi_{*}\mathcal{O}_{\overline{L}}
\]
as sheaves of algebras on $\overline{\tilde{X}}$. 
\end{lem}

\begin{proof}
Over $U$ we have that $\mathcal{N}=\pi_{*}(\mathcal{L})$ where $\mathcal{L}$
is a line bundle on $L_{U}$. Since $L_{U}\to U$ is etale, the natural
action of $\mathcal{O}_{L_{U}}$ on $\mathcal{L}$ induces a map $\pi_{*}(\mathcal{O}_{L_{U}})\to\mathcal{E}nd(\mathcal{N})$.
As this map is an isomorphic etale locally (in particular after pullback
to $L_{U}$), it is already an isomorphism. 

As $\overline{\mathcal{N}}$ is a vector bundle, restriction of endomorphisms
yields an inclusion $\mathcal{E}nd(\overline{\mathcal{N}})\subset\pi_{*}j_{*}\mathcal{O}_{\overline{L}_{U}}$;
therefore $\mathcal{E}nd(\overline{\mathcal{N}})$ is commutative
sheaf of domains, whose function field is everywhere equal to $K(\overline{L})$.
Since $\mathcal{E}nd(\overline{\mathcal{N}})$ is a coherent sheaf
over $\overline{\tilde{X}}$, it must be contained in the normalization
of $\overline{\tilde{X}}$ in $K(\overline{L})$, which is $\pi_{*}\mathcal{O}_{\overline{L}}$.
We must therefore show that the resulting map $\mathcal{E}nd(\overline{\mathcal{N}})\to\pi_{*}\mathcal{O}_{\overline{L}}$
is surjective.

Choose a point $x\in\overline{\tilde{X}}\backslash U$, a neighborhood
$N_{x}$ of $x$ (as in \prettyref{thm:Abhyankar}) and a finite flat
cover $N_{x}^{(l)}\to N_{x}$. The $\theta$-regularity condition
implies that there is $\overline{\mathcal{N}}'$ on $N^{(l)}$ so
that $(\overline{\mathcal{N}}')^{G_{l}}=\overline{\mathcal{N}}$.
After pulling back to $V_{x}^{(l)}$ we have 
\[
\overline{\mathcal{N}}_{V_{x}^{(l)}}'=\bigoplus_{i=1}^{r}\mathcal{O}_{V_{x}^{(l)}}\cdot e_{i}
\]
Since the Higgs field acts on the $\{e_{i}\}$ with distinct eigenvalues,
the endomorphisms of this Higgs sheaf are ${\displaystyle \bigoplus_{i=1}^{r}\mathcal{O}_{V_{x}^{(l)}}\tilde{=}\mathcal{O}_{L^{(l)}\times_{N^{(l)}}V_{x}^{(l)}}}$
(here we used \prettyref{eq:split}). 

Now, if $s$ is any local section of $\pi_{*}\mathcal{O}_{\overline{L}}$,
it can be regarded as a $G_{l}$-invariant local section of $\mathcal{O}_{L^{(l)}}$,
and hence, after pulling back to $V_{x}^{(l)}$, it yields an endomorphism
of $\overline{\mathcal{N}}'$, which is equivariant with respect to
the $G_{l}$-action. But the map 
\[
\mathcal{E}nd(\overline{\mathcal{N}}')^{G_{l}}\to\mathcal{E}nd((\overline{\mathcal{N}}')^{G_{l}})=\mathcal{E}nd(\overline{\mathcal{N}})
\]
is clearly surjective (consider the decomposition into character spaces).
Therefore the section $s$ is contained in the image of the map $\mathcal{E}nd(\overline{\mathcal{N}})\to\pi_{*}\mathcal{O}_{\overline{L}}$
as required. 
\end{proof}
Now we proceed to the 
\begin{proof}
(of \prettyref{thm:Pic}) We have already seen (in \prettyref{prop:push-L-is-theta-reg})
that each such $\pi_{*}\mathcal{L}$ is $\theta$-regular at infinity.
So we must prove the other direction. Let $\overline{\mathcal{N}}$
be $\theta$-regular at infinity.

By the usual BNR correspondence, $\mathcal{N}$ is isomorphic to $\pi_{*}(\mathcal{L}_{U})$
for a unique line bundle $\mathcal{L}_{U}$ on $L_{U}$. An application
of flat base change yields
\[
\overline{\pi}_{*}(j_{*}\mathcal{L}_{U})\tilde{=}j_{*}\pi_{*}(\mathcal{L}_{U})=j_{*}\mathcal{N}
\]
Now, using this isomorphism, for any open subset $V\subset\overline{L}$
we define 
\[
\mathcal{L}(V):=\{s\in(j_{*}\mathcal{L}_{U})(V)|\overline{\pi}_{*}(s)\in\overline{\mathcal{N}}\}
\]
this is a sheaf on $\overline{L}$ which extends $U$. By \prettyref{lem:Local-Endomorphisms},
this is a sheaf of $\mathcal{O}_{\overline{L}}$-modules. We now check
that it is a reflexive coherent sheaf. 

Let $x\in\overline{\tilde{X}}\backslash U$. By definition, there
is an etale neighborhood $V_{x}$ of $x$, a root cover $V_{x}^{(l)}\to V_{x}$
and a Higgs bundle $\overline{\mathcal{N}}'$ on $V_{x}^{(l)}$ so
that $\overline{\mathcal{N}}=(\overline{\mathcal{N}}')^{G_{l}}$.
There is a $G_{l}$-equivariant line bundle $\mathcal{L}^{(l)}$ on
$L^{(l)}\times_{N^{(l)}}V_{x}^{(l)}|_{U}$ so that $\overline{\mathcal{N}}'|_{V_{x}^{(l)}|_{U}}=(\overline{\pi}_{l})_{*}(\mathcal{L}_{V_{x}^{(l)}|_{U}}^{(l)})$.
Further, the Higgs bundle $\overline{\mathcal{N}}'$ splits as a direct
sum of line bundles with Higgs field. Therefore, if we consider the
analogue of the above construction for $\overline{\mathcal{N}}'$
and define
\[
\mathcal{L}^{(l)}(V):=\{s\in(j_{*}\mathcal{L}_{V_{x}^{(l)}|_{U}}^{(l)})(V)|\overline{\pi}_{*}(s)\in\overline{\mathcal{N}}'\}
\]
then we see that $\mathcal{L}^{(l)}$ is a $G_{l}$-equivariant line
bundle on $L^{(l)}\times_{N^{(l)}}V_{x}^{(l)}$. But then since $\overline{\mathcal{N}}=(\overline{\mathcal{N}}')^{G_{l}}$
we see that $\mathcal{L}_{V_{x}}=(\mathcal{L}^{(l)})^{G_{l}}$ which
shows that the pullback of $\mathcal{L}$ to an etale neighborhood
of each point is a reflexive coherent sheaf; this implies the result. 
\end{proof}
\begin{rem}
The definitions and constructions in this section extend, in the obvious
way, to reflexive coherent Higgs bundles on $\overline{\tilde{X}}_{F}$,
where $F$ is either a field of characteristic zero with $R\subset F$,
or $F=k$, a perfect field such that there is a map $R\to k$. 
\end{rem}

To finish this section, we give a convenient result (to be used later)
about the etale local structure of the of the Higgs bundle $\overline{\mathcal{N}}$,
and the finite flat map $\pi:\overline{L}\to\overline{\tilde{X}}$,
at a point $x\in D$. 
\begin{cor}
\label{cor:Local-Struc-of-Closure} 1) Let $x\in D$ be any point.
There is an etale neighborhood $V_{x}$ of $x$ so that the scheme
$\overline{L}_{V_{x}}$ is a disjoint union of its components
\[
\overline{L}_{V_{x}}=\bigsqcup_{i}\overline{L}_{V_{x},i}
\]
each map $\overline{L}_{V_{x},i}\to V_{x}$ is finite flat, and the
restriction to a neighborhood of the generic point of any component
of $D$, is itself a disjoint union of branched covers.

2) Let $\overline{\mathcal{N}}$ be $\theta$-regular at infinity,
and let $x\in\tilde{D}$ be the generic point of the intersection
of some components of $\tilde{D}$. There is an etale neighborhood
$\varphi:V_{x}\to\overline{\tilde{X}}$ of $x$, and a direct sum
decomposition 
\[
\varphi^{*}\overline{\mathcal{N}}=\bigoplus_{i}\overline{\mathcal{N}}_{V_{x},i}
\]
into sub meromorphic Higgs sheaves, indexed by the components $\overline{L}_{V_{x},i}$,
such that
\[
\overline{\mathcal{N}}_{V_{x},i}\tilde{=}\pi_{*}(\mathcal{O}_{\overline{L}_{V_{x},i}},df)
\]
\end{cor}

\begin{proof}
The second part clearly follows from the first and \prettyref{thm:Pic};
so we prove the first statement only. Let $p:V_{x}^{(l)}\to V_{x}$
be as in \prettyref{thm:Abhyankar}; so that $L^{(l)}\to V_{x}^{(l)}$
is a disjoint union of isomorphisms. Fix local coordinates $\{y_{1},\dots,y_{d},z_{d+1}\dots,z_{n}\}$
on $V_{x}^{(l)}$, so that that action of $G_{l}=<g_{1},\dots,g_{d}>$
is given by $g_{i}\cdot y_{i}=\zeta y_{i}$, for a primitive $l$th
root of unity $\zeta$ (and $z_{i}=y_{i}^{l}$) for $1\leq i\leq d$.
Fix an orbit $\mathfrak{O}$ of $G_{l}$ on the components of $L_{V_{x}}^{(l)}\to V_{x}^{(l)}$.
Call the members of the orbit $\{Z_{i}^{(l)}\}$. So, for each $i$,
if we write ${\displaystyle \theta_{i}=\theta_{1}dy_{1}+\sum_{j\geq2}\theta_{ij}dz_{j}}$,
then the equations for $Z_{i}^{(l)}$ inside $T^{*}V_{x}^{(l)}$ are
given by $\{\partial_{s}-\theta_{is}\}_{s=1}^{n}$. 

Consider the diagram of schemes 
\[
T^{*}V_{x}^{(l)}\leftarrow T^{*}V_{x}\times_{V_{x}}V_{x}^{(l)}\rightarrow T^{*}V_{x}
\]
where the first map is $dp^{*}$ and the second map is projection
along the second factor. Then we choose coordinates on these spaces
so that $\mathcal{O}_{T^{*}V_{x}^{(l)}}=\mathcal{O}_{V_{x}^{(l)}}[\xi_{1},\dots,\xi_{n}]$,
and $\mathcal{O}_{T^{*}V_{x}\times_{V_{x}}V_{x}^{(l)}}=\mathcal{O}_{T^{*}V_{x}^{(l)}}[\xi'_{1},\dots\xi_{n}']$;
and the map $dp^{*}$ is given by the identity on $\mathcal{O}_{T^{*}V_{x}^{(l)}}$
and $\xi_{j}\to\xi_{j}$ for $j>d$ while ${\displaystyle \xi{}_{i}\to lz_{i}^{l-1}\xi'_{i}}$
for $i\leq d$. Thus the inverse image of the subscheme ${\displaystyle (\xi_{1}-\theta{}_{i1},\xi_{2}-\theta{}_{i2},\dots\xi_{n}-\theta{}_{in})}$
inside $T^{*}V_{x}\times_{V_{x}}V_{x}^{(l)}$ is given by ${\displaystyle (lz_{1}^{l-1}\xi'_{1}-\theta{}_{i1},\dots,lz_{d}^{l-1}\xi'_{d}-\theta{}_{d1},\xi{}_{d+1}-\theta{}_{i1},\dots.\xi_{n}-\theta'_{in})}$. 

Now define $Z'_{i}\subset T^{*}V_{x}\times_{V_{x}}V_{x}^{(l)}$ to
be the reduced closure of the subscheme defined by ${\displaystyle (\xi'_{1}-\frac{1}{lz_{1}^{l-1}}\theta'_{i1},\dots,\xi'_{d}-\frac{1}{lz_{d}^{l-1}}\theta'_{d1},\xi_{d+1}-\theta'_{i(d+1)},\dots\xi_{n}-\theta'_{in})}$
inside $T^{*}(V_{x}\backslash D)\times_{V_{x}}(V_{x}^{(l)}\backslash D)$.
The projection to the base $V_{x}^{(l)}$ is an open immersion away
from the intersection of divisors; it surjects onto the complement
of the intersection iff $z_{j}^{l-1}$ divides $\theta_{ij}$ for
all $1\leq j\leq d$. Either way, if we now consider the union (over
$i\in\mathfrak{O}$) of all $Z'_{i}$, we obtain a reduced subscheme
of $T^{*}V_{x}\times_{V_{x}}V_{x}^{(l)}$ with smooth components,
whose components are permuted by the group $G_{l}$. 

Now, consider the scheme $\overline{Z}_{i}'$, which we define to
be the integral closure of $\mathcal{O}_{V_{x}^{(l)}}$ inside $\mathcal{O}_{Z'_{i}}$;
this is a smooth scheme which is isomorphic to $V_{x}^{(l)}$. Since
the projection map $T^{*}V_{x}\times_{V_{X}}V_{x}^{(l)}\rightarrow T^{*}V_{x}$
is simply the quotient by $G_{l}$, we see that we have attached to
the orbit $\mathfrak{O}$, the scheme $\overline{L}_{V_{x},\mathfrak{O}}$,
which is the quotient 
\[
\overline{L}_{V_{x},\mathfrak{O}}=(\bigsqcup_{i\in\mathfrak{O}}\overline{Z}'_{i})/G_{l}
\]
As scheme $\overline{L}_{V_{x}}$ is the union of these schemes, this
implies the result (where $\overline{L}_{V_{x},\mathfrak{O}}$ is
labelled $\overline{L}_{V_{x}i}$ in the statement). 
\end{proof}
In fact, this proof gives slightly more, as the map ${\displaystyle \bigsqcup_{i\in\mathfrak{O}}Z'_{i}\to T^{*}V_{x}\times_{V_{x}}N_{V_{x}}^{(l)}}$
extends to a rational map ${\displaystyle \bigsqcup_{i\in\mathfrak{O}}\overline{Z}'_{i}\dasharrow T^{*}V_{x}\times_{V_{x}}N_{V_{x}}^{(l)}}$;
taking the quotient by $G_{l}$ then yield the rational map $\overline{L}_{V_{x}}\dasharrow T^{*}V_{x}$. 
\begin{defn}
\label{def:embedded-in-T^*}Let $\{x\}$ be the generic point of a
component of $\tilde{D}$, and let $V_{x}$ be an etale neighborhood
of $\{x\}$ as above. Let $\overline{L}_{V_{x}}={\displaystyle \bigsqcup\overline{L}_{V_{x,i}}}$
be the decomposition into components, and let $\overline{L}_{V_{x,i}}\dasharrow T^{*}V_{x}$
be the rational map constructed directly above. Then we say that $\overline{L}_{V_{x,i}}$
is embedded in $T^{*}V_{x}$ if the rational map above extends to
a closed immersion on $\overline{L}_{V_{x,i}}$. 
\end{defn}

\subsection{\label{subsec:-regular--connections-in-positive-char}$\theta$-regular
$\lambda$-connections in positive characteristic}

In this section, working in positive characteristic, we construct
natural deformations of the $\theta$-regular Higgs sheaves encountered
in the previous section, to objects which we call $\theta$-regular
$\lambda$-connections. As our ultimate goal is to construct $\mathcal{D}$-modules
with a certain $p$-curvature, the $\lambda$-connections we consider
will have a strong $p$-curvature constraint; and we will end up with
a very small family of objects- in fact, under the assumption that
$H_{dr}^{1}(L)=0$, we will see that each $\theta$-regular Higgs
bundle admits at most one deformation to a $\theta$-regular $\lambda$-connection
(c.f. \prettyref{cor:At-most-one-N-lambda} below). 

Before giving the details, let us consider the problem to be solved.
Looking at \prettyref{thm:Pic}, one sees that, for the simplest $\theta$-regular
Higgs sheaf $\pi_{*}(\mathcal{O}_{\overline{L}})$, there is a very
natural meromorphic $\lambda$-connection one can associate to it-
namely $\pi_{*}(\mathcal{O}_{\overline{L}}[\lambda])$, which is equipped
with the $\lambda$-connection $\nabla(1)=df$. However, this connection
may not be the one we seek to solve the quantization problem. This
is because, wherever the map $\overline{L}\to\overline{\tilde{X}}$
is branched, the $\lambda$-connection $\pi_{*}(\mathcal{O}_{\overline{L}}[\lambda])$
has singularities when $\lambda\neq0$. So, if the map $L\to X$ is
a branched cover (as in \prettyref{exa:Airy}), this bundle will have
singularities along the branchings, and is therefore not a good candidate
to solve the quantization problem. 

On the other hand, over the locus $U$ where $L_{U}\to U$ is etale,
the bundle $\pi_{*}(\mathcal{O}_{\overline{L}}[\lambda])$ behaves
as expected (as detailed in \prettyref{subsec:The-monodromy-divisor}
above). So we seek to extend this bundle to an object on all of $\overline{\tilde{X}}$
which has no ``unexpected'' singularities. In this section we accomplish
this over $k$, where $k$ is an algebraically closed field of positive
characteristic. 

To set things up, let $(\mathcal{L},\nabla)$ be any line bundle on
$L_{U_{k}}$, equipped with a $\lambda$-connection such that the
Higgs field on $\mathcal{L}/\lambda$ is given by $df$; and such
that $(\mathcal{L},\nabla)$ is locally isomorphic to the connection
given by
\[
\nabla(1)=df
\]
(there is at least one such, namely $(\mathcal{O},\nabla)$ where
$\nabla(1)=df$). 
\begin{defn}
\label{def:N-lambda}Define a $\lambda$-connection on $U$ via $\mathcal{N}_{\lambda,\mathcal{L}}=\pi_{*}(\mathcal{L},\nabla)$.
We set ${\displaystyle \mathcal{N}_{\widehat{\lambda}}:=\lim_{n}(\mathcal{N}_{\lambda}/\lambda^{n})}$.
Let $\mathcal{N}:=\mathcal{N}_{\lambda}/\lambda$, this is the Higgs
bundle (attached to $\mathcal{L}$) that we constructed in the previous
section. 
\end{defn}

As above, we will almost always suppress the subscript $\mathcal{L}$. 

The sheaves $\mathcal{N}_{\lambda}$ are locally isomorphic to the
sheaf $\mathcal{E}_{\lambda}$ which we considered in \prettyref{subsec:The-monodromy-divisor}
above; recall that it was defined to be $\pi_{*}(\mathcal{O}_{L_{U_{k}}},df)$.
In particular, by \prettyref{lem:p-support-of-E-l} we have:
\begin{lem}
\label{lem:p-curvature-of-N-lambda}The sheaf $\mathcal{N}_{\lambda,k}$,
considered as a sheaf on $T^{*}U_{k}^{(1)}\times\mathbb{A}^{1}$ via
the identification $\mathcal{Z}(\mathcal{D}_{\lambda})\tilde{=}\mathcal{O}_{T^{*}U_{k}^{(1)}}[\lambda]$,
is scheme theoretically supported on $L_{U_{k}}^{(1)}\times\mathbb{A}_{k}^{1}$.
Similarly, $\mathcal{N}_{\widehat{\lambda},k}$ is scheme-theoretically
supported on $L_{U_{k}}^{(1)}\times\widehat{\mathbb{A}_{k}^{1}}$. 
\end{lem}

Let us consider how to obtain the correct extension of these objects
to $\overline{\tilde{X}}_{k}$. We will first work in codimension
$2$, obtaining a reflexive coherent sheaf on $\overline{\tilde{X}}_{k}$,
and then show that the resulting object is a bundle in \prettyref{thm:N-lambda-is-theta-reg}
below. We have 
\begin{lem}
\label{lem:Decomop-before-extn}Let $x\in\tilde{D}_{k}^{\text{sm}}$.
Then there is an etale neighborhood $\varphi:V_{x}\to\overline{\tilde{X}}_{k}$
of $x$ in which we have a direct sum of $\lambda$-connections
\[
\varphi^{*}(j_{*}\mathcal{N}_{\lambda})=\bigoplus_{i}j_{*}(\mathcal{N}_{\lambda})_{V_{x},i}
\]
where $j:U_{k}\to\overline{\tilde{X}}_{k}$ is the inclusion, and
each summand $j_{*}(\mathcal{N}_{\lambda})_{V_{x},i}$ is $p$-supported
along a component of the inverse image of $L_{k}^{(1)}\times\mathbb{A}_{k}^{1}$
inside $(T^{*}(V_{x,k})^{(1)}\times\mathbb{A}^{1})\backslash\tilde{D}$.
In particular, the summands are indexed by components of $\overline{L}_{V_{x}}\times\mathbb{A}_{k}^{1}$. 
\end{lem}

\begin{proof}
If we consider $\varphi^{*}(j_{*}\mathcal{N}_{\lambda})$ as a sheaf
on $T^{*}V_{x,k}^{(1)}\times\mathbb{A}_{k}^{1}$. It is $p$-supported
along the inverse image of $L_{k}^{(1)}\times\mathbb{A}_{k}^{1}$.
Over $(T^{*}V_{x,k}{}^{(1)}\times\mathbb{A}^{1})\backslash\tilde{D}$,
this scheme is a disjoint union of components, indexed by orbits of
the group $G_{l}$ on the components of the scheme $L^{(l)}\times_{N^{(l)}}V_{x}^{(l)}$
(by \prettyref{cor:Local-Struc-of-Closure}). Thus the action of the
center of $\mathcal{D}_{\lambda}$ on $\varphi^{*}(j_{*}\mathcal{N}_{\lambda})$
induces the claimed decomposition. 
\end{proof}
This reduces the problem of extending the sheaf $\varphi^{*}(j_{*}\mathcal{N}_{\lambda})$
to a bundle on $V_{x}$ to the problem of extending each individual
summand $j_{*}(\mathcal{N}_{\lambda})_{V_{x},i}$. There are two distinct
cases that we have to deal with, according to weather or not the associated
component $\overline{L}_{V_{x},i}$ is embedded in $T^{*}V_{x}$ (c.f.
\prettyref{def:embedded-in-T^*}). In fact, as our only need in this
paper is to control the behavior our connection on $L$ itself, we
can cut things even finer: 
\begin{defn}
Let $x\in\tilde{D}_{k}^{\text{sm}}$ and consider the decomposition
$\overline{L}_{V_{x}}={\displaystyle \bigsqcup\overline{L}_{V_{x,i}}}$
into components. As each map $\overline{L}_{V_{x,i}}\to V_{x}$ is
a branched cover, each component $L_{V_{x},i}$ contains a single
component of the divisor $\tilde{E}_{k}$. We shall say that the component
is of type I if $\overline{L}_{V_{x},i}$ is embedded in $T^{*}V_{x}$
and the component of $\tilde{E}_{k}$ in $\overline{L}_{V_{x},i}$
belongs to $L_{k}$ (as in \prettyref{def:Belongs-to-L}). We say
that the component is of type II in every other case. 
\end{defn}

Now we can construct the (first incarnation of) the extensions we
are looking for: 
\begin{prop}
\label{prop:Extension-in-codim-2}For each $\mathcal{N}_{\lambda}$
as above, there is a reflexive coherent extension, $\overline{\mathcal{N}}_{\lambda}$,
with the following property: let $x$ be the generic point of a component
of $\tilde{D}_{k}$. There is an etale neighborhood $\varphi:V_{x}\to\overline{\tilde{X}}_{k}$
of $x$, on which $\overline{\mathcal{N}}_{\lambda}$ is a bundle,
and for which we have a direct sum of bundles with meromorphic $\lambda$-connection
\[
\varphi^{*}(\overline{\mathcal{N}}_{\lambda})=\bigoplus_{i}(\overline{\mathcal{N}}_{\lambda})_{V_{x},i}
\]
whose restriction to $V_{x}\backslash\tilde{D}_{k}$ agrees with the
decomposition of \prettyref{lem:Decomop-before-extn}. If the summand
$(\overline{\mathcal{N}}_{\lambda})_{V_{x},i}$ corresponds to a component
of type I, then we have that this summand is a bundle with $\lambda$-connection;
i.e., it has no singularities. If the summand $(\overline{\mathcal{N}}_{\lambda})_{V_{x},i}$
corresponds to a component of type II, then we have $(\overline{\mathcal{N}}_{\lambda})_{V_{x},i}=(\overline{\mathcal{N}}'_{i,\lambda})^{G_{l}}$
where $\overline{\mathcal{N}}'_{\lambda}$ is a $G_{l}$-equivsriant
bundle on $V_{x}^{(l)}$ with meromorphic connection, which has a
basis $\{e_{j}\}$ on which $G_{l}$ acts transitively, so that 
\[
\nabla(e_{j})=\theta_{j}e_{j}
\]
where $\theta_{j}$ are the one-forms defined in \prettyref{def:omega-reg-prime-higher-dim}. 

The meromorphic Higgs sheaf $\overline{\mathcal{N}}_{\lambda}/\lambda$
is $\theta$-regular at infinity in codimension $2$. 
\end{prop}

In other words, we take the naive kind of $\lambda$-connection (where
we imitate the definition of $\theta$-regular Higgs bundle) when
we are in type II, but we need to do something different in type I
in order to get a bundle without singularities. 
\begin{proof}
Consider the generic point $\{x\}$ of a component of the divisor
$\tilde{D}_{k}$. By flat descent, to obtain a bundle we can work
in an etale neighborhood of each such $\{x\}$. We apply \prettyref{lem:Decomop-before-extn}
to get the decomposition 
\[
\varphi^{*}(j_{*}\mathcal{N}_{\lambda})=\bigoplus_{i}(j_{*}\mathcal{N}_{\lambda})_{V_{x},i}
\]

Now we shall explain how to define an appropriate sub-bundle of each
$j_{*}(\mathcal{N}_{\lambda})_{V_{x},i}$. First, suppose that the
corresponding component $L_{V_{x},i,k}$ is of type I. This implies
that the center of $\mathcal{D}_{\lambda,V_{x}}$ acts on $j_{*}(\mathcal{N}_{\lambda})_{V_{x},i}$
through a quotient which is a finite flat module over $\mathcal{O}_{V_{x}}[\lambda]$.
Thus any coherent $\mathcal{D}_{\lambda,V_{x}}$ -submodule of $j_{*}(\mathcal{N}_{\lambda})_{V_{x},i}$
is also coherent over $\mathcal{O}_{V_{x}}[\lambda]$. 

Choose a finite locally free module $\mathcal{F}$ over $\mathcal{O}_{V_{x}}[\lambda]$
such that $\mathcal{F}[z^{-1}]=\varphi^{*}(j_{*}\mathcal{N}_{\lambda})_{i}$
(here $z$ is a local coordinate for $\varphi^{-1}(\tilde{D)}$).
Let $\mathcal{G}$ be he coherent $\mathcal{D}_{\lambda,V_{x}}$ -submodule
generated by $\mathcal{F}$. By the above it is contained in $z^{-M}\mathcal{F}$
for some $M>0$. Let $\mathcal{G}'$ denote $\{m\in\varphi^{*}(j_{*}\mathcal{N}_{\lambda})_{i}|\lambda^{i}m\in\mathcal{G}\phantom{a}\text{for some}\phantom{b}i>0\}$.
Then $\mathcal{G}'$ is a $\mathcal{D}_{\lambda,V_{x}}$ -submodule,
clearly contained in $z^{-M}\mathcal{F}$, and hence also coherent
over $V_{x,k}\times\mathbb{A}_{k}^{1}$. By the choice of $\mathcal{G}'$
we have that the map
\[
\mathcal{G}'/\lambda\to j_{*}(\mathcal{N}_{\lambda})_{V_{x},i}/\lambda
\]
is injective. Thus $\mathcal{G}'/\lambda$ is a torsion-free coherent
sheaf over $V_{x}$ which is supported along $\overline{L}_{V_{x,i}}\subset T^{*}(V_{x,k})$.
After replacing $\mathcal{G}'$ by its double dual over $\mathcal{O}_{V_{x}}[\lambda]$
if necessary, we see, applying (the proof of) \prettyref{lem:Extend-mod-p},
that $\mathcal{G}'$ is a bundle over $V_{x,k}\times\mathbb{A}_{k}^{1}$,
which is coherent over $\mathcal{D}_{\lambda,V_{x}}$ ; i.e., it is
a bundle with $\lambda$-connection, without singularities. The reduction
mod $\lambda$ of this bundle is a Higgs bundle, necessarily supported
along the image of $\overline{L}_{V_{x},i}$ in $T^{*}V_{x}$ (as
this is so generically); therefore, by the BNR correspondence, it
is the pushforward of a line bundle on $\overline{L}_{V_{x},i}$. 

Now suppose the component $\overline{L}_{V_{x},i}$ is of type II.
Then, after taking the $l$th root pullback $V_{x,k}^{(l)}\to V_{x,k}$
(along $\tilde{D}_{k}$), we have that $j_{*}(\mathcal{N}_{\lambda})_{V_{x},i}$
becomes isomorphic to a direct sum of line bundles with connection.
Thus we may choose an extension of this sheaf to a direct sum of line
bundles with meromorphic connection on $V_{x,k}^{(l)}$. Then taking
$G_{l}$-invariants yields an appropriate sub-bundle of $j_{*}(\mathcal{N}_{\lambda})_{V_{x},i}$. 

Now, we have defined a bundle on $U\times\mathbb{A}^{1}$ (namely
$\mathcal{N}_{\lambda}$ ) and sub-bundles of $j_{*}(\mathcal{N}_{\lambda})_{V_{x},i}$
along the an etale cover of the generic point of each component of
the divisor $\tilde{D}_{k}$. Thus, after restricting to an open subscheme
whose complement has codimension $2$, faithfully flat descent yields
a bundle, and we may push forward to obtain a reflexive coherent sheaf
on $\overline{\tilde{X}}_{k}\times\mathbb{A}_{k}^{1}$. To obtain
the last sentence, note that we proved it above (at the generic point
of a divisor) for $(\overline{\mathcal{N}}_{\lambda})_{V_{x},i}$
in type I, and for type II it follows immediately from the construction. 
\end{proof}
Now, in order to proceed (and, eventually, to match these constructions
with those of the previous chapter), we shall need to have an action
of the category of locally trivial log connections on the set of $\overline{\mathcal{N}}_{\lambda}$.
The above construction is not quite flexible enough to allow us to
do this; essentially, we need to be able to change the ``monodromy''
along components of the divisor $\tilde{E}_{k}\subset\overline{L}_{k}$.
Therefore, we make the 
\begin{defn}
\label{def:Theta-Reg-Pos-Char}Let $\overline{\mathcal{N}}_{\lambda}$
be a reflexive coherent sheaf with $\lambda$-connection on $\overline{\tilde{X}}_{k}\times\mathbb{A}_{k}^{1}$,
whose restriction to $U_{k}\times\mathbb{A}_{k}^{1}$ is equal to
a bundle of the type defined in \prettyref{def:N-lambda}. We say
that $\overline{\mathcal{N}}_{\lambda}$ is $\theta$-regular at infinity
if, for each $x$ which is the generic point of a component of $\tilde{D}_{k}$,
there is an etale neighborhood $\varphi:V_{x}\to\overline{\tilde{X}}_{k}$
so that
\[
\varphi^{*}(\overline{\mathcal{N}}_{\lambda})=\bigoplus_{i}(\overline{\mathcal{N}}_{\lambda})_{V_{x},i}
\]
whose restriction to $V_{x}\backslash\tilde{D}_{k}$ agrees with the
decomposition of \prettyref{lem:Decomop-before-extn}. If the summand
$(\overline{\mathcal{N}}_{\lambda})_{V_{x},i}$ corresponds to a component
of type I, then we demand that this summand is of the form 
\[
z^{\alpha}\cdot\mathcal{V}_{\lambda}
\]
where $\mathcal{V}_{\lambda}$ is a bundle with $\lambda$-connection,
$\alpha\in\mathbb{Z}$, and the terminology is as follows: we have
$j_{*}(\mathcal{N}_{\lambda})_{V_{x},i}\tilde{=}\pi_{*}(\mathcal{L})$
where $\mathcal{L}$ is a line bundle on $\overline{L}_{V_{x},i}$.
Thus $j_{*}(\mathcal{N}_{\lambda})_{V_{x},i}$ has an action of $\pi_{*}\mathcal{O}_{\overline{L}_{V_{x},i}}$,
and so we can rescale any sub-bundle by a power of $z$, which is
a local coordinate on $\overline{L}_{V_{x},i}$. 

If the summand $(\overline{\mathcal{N}}_{\lambda})_{V_{x},i}$ corresponds
to a component of type II, then we have $(\overline{\mathcal{N}}_{\lambda})_{V_{x},i}=(\overline{\mathcal{N}}'_{i,\lambda})^{G_{l}}$
where $\overline{\mathcal{N}}'_{\lambda}$ is a $G_{l}$-equivariant
bundle on $V_{x}^{(l)}$ with meromorphic connection, which has a
basis $\{e_{j}\}$ on which $G_{l}$ acts transitively, so that 
\[
\nabla(e_{j})=(\theta_{j}+\lambda\beta\frac{dz}{z}))e_{j}
\]
where $\theta_{j}$ are the one-forms defined in \prettyref{def:omega-reg-prime-higher-dim}
and and $\beta\in\mathbb{F}_{p}$. 
\end{defn}

We note that the integer $\alpha$ appearing in the first part of
the definition in unique, up to adding a multiple of $p$. Indeed,
as we have seen during the proof of \prettyref{lem:Extend-mod-p},
the bundle $\mathcal{V}_{\lambda}$ is unique up to multiplication
by a power of $\lambda$, and up to rescaling by an element of $\mathcal{O}_{\overline{L}_{V_{x,i}}^{(1)}}$.
As this bundle is supposed to equal $\mathcal{N}_{\lambda}$ after
restriction to $U_{k}$, we see that the power of $\lambda$ must
be $0$, and that the rescale by an element of $\mathcal{O}_{\overline{L}_{V_{x,i}}^{(1)}}$
must be a rescale by $z^{pn}\cdot u$ where $u$ is a unit. Therefore
$\mathcal{V}_{\lambda}$ is unique up to multiplication by by an element
of the form $z^{pn}$. 

From this definition we see that each $\lambda$-connection which
is $\theta$-regular at infinity comes with a ``residue,'' which
is an $\mathbb{F}_{p}$-divisor in $\overline{L}_{k}$. More precisely,
we have 
\begin{defn}
Let $\overline{\mathcal{N}}_{\lambda}$ be a $\lambda$-connection
which is $\theta$-regular at infinity. To $\overline{\mathcal{N}}_{\lambda}$
we attach an $\mathbb{F}_{p}$-divisor supported on $\tilde{E}_{k}$
as follows: let $E_{i}$ be a component of $\tilde{E}_{k}$, living
over some component $D'_{k}$of $\tilde{D}_{k}$. After passing to
an etale neighborhood $V_{x}$ of the generic point of $D'_{k}$,
we have that $E_{i}$ is contained in a unique component $\overline{L}_{V_{x},i}$
of $\overline{L}_{V_{x}}$. If $\overline{L}_{V_{x},i}$ is of type
I, then we have $(\overline{\mathcal{N}}_{\lambda})_{V_{x},i}=z^{\alpha}\mathcal{V}_{\lambda}$
for some $\alpha\in\mathbb{Z}$, and we attach $\overline{\alpha}E_{i}$,
where $\overline{\alpha}$ is the image of $\alpha$ in $\mathbb{F}_{p}$.
If $\overline{L}_{V_{x},i}$ is of type II, then we have $\overline{\mathcal{N}}_{\lambda}=(\overline{\mathcal{N}}'_{i,\lambda})^{G_{l}}$
where $\overline{\mathcal{N}}'_{\lambda}=\pi_{*}(\mathcal{L},\nabla+\lambda\beta\frac{dz}{z})$,
and we attach $\beta E_{i}$. The sum over all these divisors is the
residue of $\overline{\mathcal{N}}_{\lambda}$. 
\end{defn}

Before proceeding, we record for later use the 
\begin{rem}
\label{rem:Theta-reg-over-W}Suppose $W\subset\overline{\tilde{X}}_{k}\times\mathbb{A}_{k}^{1}$
is an open subset. Then if $\overline{\mathcal{P}}_{\lambda}$ is
some vector bundle with meromorphic $\lambda$-connection on $W$,
we can extend the notion of $\theta$-regularity to $\overline{\mathcal{P}}_{\lambda}$,
by noting that the construction of \prettyref{prop:Extension-in-codim-2}
is completely local. 
\end{rem}

Now we can classify $\theta$-regular $\lambda$-connections. We start
with the analogues of \prettyref{lem:Local-Endomorphisms} and \prettyref{thm:Pic}
for $\theta$-regular $\lambda$-connections:
\begin{lem}
\label{lem:Endomorphisms-of-N-lambda}Let $\overline{\mathcal{N}}_{\lambda}$
be strongly $\theta$-regular at infinity. Let $\mathcal{E}nd(\overline{\mathcal{N}}_{\lambda})$
denote the sheaf of endomorphisms of $\overline{\mathcal{N}}_{\lambda}$
which respect the $\lambda$-connection. Then 
\[
\mathcal{E}nd(\overline{\mathcal{N}}_{\lambda})\tilde{=}\pi_{*}\mathcal{O}_{\overline{L}_{k}^{(1)}}[\lambda]
\]
\end{lem}

\begin{proof}
As in the proof of \prettyref{lem:Local-Endomorphisms}, one easily
obtains an injective morphism 
\[
\mathcal{E}nd(\overline{\mathcal{N}}_{\lambda})\to\pi_{*}\mathcal{O}_{\overline{L}_{k}^{(1)}}[\lambda]
\]
which is an isomorphism over $U_{k}$. To show that it is surjective,
let $W$ denote an open subset of $\overline{\tilde{X}}_{k}$ of codimension
$2$ over which $\overline{\mathcal{N}}_{\lambda}$ is a bundle; and
$j:W\to\overline{\tilde{X}}_{k}$ the inclusion. Note that
\[
j_{*}\mathcal{E}nd_{\mathcal{O}_{W}}(\overline{\mathcal{N}}_{\lambda}|_{W})\tilde{=}\mathcal{E}nd_{\mathcal{O}_{\overline{\tilde{X}}_{k}}}(\overline{\mathcal{N}}_{\lambda})
\]
which implies that the same holds for $\mathcal{E}nd(\overline{\mathcal{N}}_{\lambda})$.
Thus it suffices to prove the result over $W$. To do that, we can
work locally and pull back to the etale neighborhood $V_{x}$ of $\{x\}$,
the generic point of some component of $\tilde{D}_{k}$. We can then
use the decomposition 
\[
\varphi^{*}\mathcal{N}_{\lambda}=\bigoplus_{i}(\overline{\mathcal{N}_{\lambda}})_{V_{x},i}
\]
For a summand of type II, then by definition $(\overline{\mathcal{N}_{\lambda}})_{V_{x},i}$
is the $G_{r_{i}}$-invariants inside a direct sum of line bundles
with connection; and we can proceed in an identical manner to the
proof of \prettyref{lem:Local-Endomorphisms}. In the type I case,
we can apply the argument of \prettyref{lem:Extend-mod-p} to see
that the endomorphisms are equal to $\pi_{*}\mathcal{O}_{\overline{L}_{V_{x,i}}^{(1)}}[\lambda]$,
as needed.
\end{proof}
Next, we have the following consequence, the analogue of \prettyref{thm:Pic}: 
\begin{thm}
\label{thm:Pic-over-k}The set of $\theta$-regular $\lambda$-connections
with a given residue is a torsor over $\text{Pic}^{r}(\overline{L}_{k}^{(1)})$. 
\end{thm}

\begin{proof}
First, we note that there is a $\theta$-regular connection which
has any given residue. To see it, take the $\theta$-regular connection
$\overline{\mathcal{N}}_{\lambda}$ of \prettyref{prop:Extension-in-codim-2},
and then modify it at the generic point of each component of $\tilde{D}_{k}$
as follows: let $V_{x}$ be an etale neighborhood of the generic point
of a component of $\tilde{D}_{k}$, so that we have the decomposition
\begin{equation}
\varphi^{*}\mathcal{\overline{N}}_{\lambda}=\bigoplus_{i}(\overline{\mathcal{N}_{\lambda}})_{V_{x},i}\label{eq:Direct-sum-decomp}
\end{equation}
Then, by rescaling each summand by an appropriate power of the uniformizor
in $\overline{L}_{V_{x},i}$, we can arrange the residue to be any
divisor (supported in $\tilde{E}_{k}$) with coefficients in $\mathbb{F}_{p}$. 

Now, let $\overline{\mathcal{N}_{\lambda}}$ and $\overline{\mathcal{P}_{\lambda}}$
be two meromorphic $\lambda$-connections which are $\theta$-regular
at infinity, with the same residue. We start by showing that they
are locally isomorphic over an open subset of codimension $2$. Over
the open subset $U$ this is true by definition. Let's examine the
situation over the generic point $\{x\}$ of a component of $\tilde{D}_{k}$.
There, we have the etale neighborhood $\varphi:V_{x}\to\overline{\tilde{X}}_{k}$,
and we have the decomposition \prettyref{eq:Direct-sum-decomp} and
its analogue for $\overline{\mathcal{P}}_{\lambda}$. 

If we are in type II, we have $(\overline{\mathcal{N}_{\lambda}})_{V_{x},i}=(\overline{\mathcal{N}}'_{i,\lambda})^{G_{l}}$
for a bundle $\overline{\mathcal{N}}'_{i,\lambda}$ on $N_{k}^{(l)}$
which possesses a basis of eigenvectors for the connection; and the
same for $(\overline{\mathcal{P}_{\lambda}})_{V_{x},i}$. So the result
follows immediately (from the fact that both the $p$-curvature and
the residue are fixed). If we are in type I, we have that $(\overline{\mathcal{N}_{\lambda}})_{V_{x},i}$
and $(\overline{\mathcal{P}_{\lambda}})_{V_{x},i}$ are $\lambda$-connections
without singularities; or a multiple of such a connection by a power
of the uniformizor. Then the proof of \prettyref{lem:Extend-mod-p}
shows that $(\overline{\mathcal{N}_{\lambda}})_{V_{x},i}\tilde{=}(\overline{\mathcal{P}_{\lambda}})_{V_{x},i}$
as claimed. 

Therefore, $\overline{\mathcal{N}}_{\lambda}$ and $\overline{\mathcal{P}}{}_{\lambda}$
are locally isomorphic in the etale topology over an open subset of
codimension $2$. Applying the previous lemma, we see that $\text{\ensuremath{\mathcal{H}}om}(\overline{\mathcal{N}}_{\lambda},\overline{\mathcal{P}}_{\lambda})$
is, locally in the etale topology, isomorphic to $\pi_{*}(\mathcal{O}_{\overline{L}_{k}^{(1)}}[\lambda])$;
therefore this is true in the Zariski topology as well and we see
that $\overline{\mathcal{N}}_{\lambda}$ and $\overline{\mathcal{P}_{\lambda}}$
are locally isomorphic in codimension $2$. 

Thus, over an open subset $V$ of codimension $2$, the set of all
$\theta$-regular $\lambda$-connections is a torsor over $H^{1}(\text{Aut}(\overline{\mathcal{N}}_{\lambda}))=H^{1}(\pi_{*}(\mathcal{O}_{\pi^{-1}(V^{(1)})}[\lambda])^{*})\tilde{=}H^{1}(\mathcal{O}_{\pi^{-1}(V^{(1)})}^{*})$.
But this is just the group of line bundles on $\pi^{-1}(V^{(1)})\subset\overline{L}_{k}^{(1)}$.
Since any $\theta$-regular $\lambda$-connection is a pushforward
of its restriction to codimension $2$, the result follows. 
\end{proof}
Now we would like to consider the case where the residue is not fixed.
Let $\text{Pic}^{r}(\overline{L},\tilde{E},\nabla)$ denote the group
of reflexive coherent sheaves of rank $1$ with log connection with
respect to $\tilde{E}_{k}$ on $\overline{L}_{k}$, of trivial $p$-curvature.
In \prettyref{cor:Marked-Descent} below, we show that to give an
element $(\mathcal{L},\nabla)\in\text{Pic}^{r}(\overline{L}_{k},\tilde{E}_{k},\nabla)$
is to give an element $\mathcal{L}'\in\text{Pic}^{r}(\overline{L}_{k}^{(1)})$,
and embedding of sheaves $F^{*}\mathcal{L}'\to\mathcal{L}$ and for
each component $\tilde{E}_{i}$ of $\tilde{E}_{k}$, an element of
$\{0,\dots,p-1\}$ so that $z^{-\alpha}F^{*}\mathcal{L}'=\mathcal{L}$
at the generic point of $\tilde{E}_{i}$. With this in hand, we can
show 
\begin{cor}
\label{cor:Pic-with-residue-over-k}The set of $\theta$-regular $\lambda$-connections
is a torsor over the group $\text{Pic}^{r}(\overline{L},\tilde{E},\nabla)$.
If $\overline{\mathcal{N}}_{\lambda}$ is strongly $\theta$-regular
with residue $A$, and $(\mathcal{L},\nabla)$ is such a log connection,
whose residue along $\tilde{E}_{k}$ is given by ${\displaystyle B=\sum_{i}\beta_{i}\tilde{E}_{i}}$,
then the residue of the strongly $\theta$-regular sheaf $(\mathcal{L},\nabla)\star\overline{\mathcal{N}}_{\lambda}$
is $A+B$. Further, if $(\overline{\mathcal{N}}_{\lambda})/\lambda=\pi_{*}(\mathcal{L}')$
(in codimension $2$) via \prettyref{thm:Pic}, then we have $((\mathcal{L},\nabla)\star\overline{\mathcal{N}}_{\lambda})/\lambda\tilde{=}\pi_{*}(\mathcal{L}\otimes\mathcal{L}')$
in codimension $2$. 
\end{cor}

\begin{proof}
Let $\mathcal{L}'\in\text{Pic}^{r}(\overline{L}_{k}^{(1)})$ be associated
to $(\mathcal{L},\nabla)$, as discussed above. We can define the
action of such a bundle as follows: first, act by the line bundle
$\mathcal{L}'\in\text{Pic}^{r}(L_{k}^{(1)})$ as in the previous theorem
to obtain a $\theta$-regular bundle $\mathcal{L}'\star\overline{\mathcal{N}_{\lambda}}$,
whose residue is equal to that of $\overline{\mathcal{N}}_{\lambda}$. 

Now we argue as in the proof of \prettyref{lem:2-d-goodness}; let
the notation be as in the that proof. So we have 
\[
\varphi^{*}(\mathcal{L}'\star\overline{\mathcal{N}_{\lambda}})=\bigoplus_{i}(\mathcal{L}'\star\overline{\mathcal{N}_{\lambda}})_{V_{x},i}
\]
For each $i$, we let $\alpha$ be the element of $\{0,\dots,p-1\}$
associated to the divisor $\tilde{E}_{i}$ by the construction of
\prettyref{cor:Marked-Descent}. If we are in type I, then we modify
$(\mathcal{L}'\star\overline{\mathcal{N}_{\lambda}})_{V_{x},i}$ to
$z^{-\alpha}(\mathcal{L}'\star\overline{\mathcal{N}_{\lambda}})_{V_{x},i}$.
If we are in type II, then we have $(\mathcal{L}'\star\overline{\mathcal{N}_{\lambda}})_{V_{x},i}=(\mathcal{L}'\star\overline{\mathcal{N}}'_{i,\lambda})^{G_{l}}$
where $\mathcal{L}'\star\overline{\mathcal{N}}'_{\lambda}$ possesses
a basis $\{e_{i}\}$ as in \prettyref{def:Theta-Reg-Pos-Char}.. We
modify this bundle by multiplying each basis element by $z^{-\alpha}$.
The required properties of the action follow directly from the construction. 
\end{proof}
Now, the final major result we need about $\theta$-regular connections
is:
\begin{thm}
\label{thm:N-lambda-is-theta-reg}Let $\overline{\mathcal{N}}_{\lambda}$
be $\theta$-regular at infinity. Then $\overline{\mathcal{N}}_{\lambda}$
is a vector bundle. In particular, $\overline{\mathcal{N}}_{\lambda}/\lambda$
is a $\theta$-regular Higgs sheaf on all of $\overline{\tilde{X}}_{k}$. 
\end{thm}

This follows immediately from the following lemma, in which we show
that $\overline{\mathcal{N}}_{\lambda}$ is, etale locally, the $G_{l}$-invariants
inside some vector bundle- which we explicitly construct, on a root
cover.
\begin{lem}
\label{lem:2-d-goodness}Let $x$ be a point of $\overline{\tilde{X}}_{k}$
which is contained in the intersection of exactly $m$ components
of $\tilde{D}_{k}$. Then there is an etale neighborhood of $x$,
$\varphi:V_{x,k}\to\overline{\tilde{X}}_{k}$ , and a vector bundle
$\overline{\mathcal{N}}''_{\lambda}$ on $V_{x,k}^{(l)}$, so that
$\varphi^{*}\overline{\mathcal{N}}_{\lambda}=(\overline{\mathcal{N}}''_{\lambda})^{G_{l}}$. 
\end{lem}

\begin{proof}
Let $\varphi:V_{x,k}\to\overline{\tilde{X}}_{k}$ be the etale morphism
of \prettyref{cor:Local-Struc-of-Closure}. Let $\{z_{1},\dots,z_{n}\}$
be local coordinates on $V_{x}$, for which $\tilde{D}$ is given
by $\{z_{1}\cdots z_{m}=0\}$. We have the decomposition 
\[
\varphi^{*}\overline{\mathcal{N}_{\lambda}}=\bigoplus_{i}(\overline{\mathcal{N}_{\lambda}})_{V_{x},i}
\]
according to $p$-curvature. We can work with one summand at a time,
and the only nontrivial case is the one where the component $\overline{L}_{V_{x},i}$
embeds into $T^{*}V_{x,k}$ over $\{z_{1}=0\}$ (i.e., $\overline{L}_{V_{x},i}$
is type I). Here we are using assumption $2$ concerning the structure
of $\overline{L}$, which says in particular that any type I divisor
does does not intersect another. 

Consider the partial root cover $V_{x,k}^{(0,l\dots,l)}\to V_{x,k}$
(where we take $l$th roots of $\{z_{2},\dots z_{m}\}$). We denote
by $L_{V_{x,i,k}}^{(0,l,\dots,l)}$ the scheme whose structure sheaf
is the integral closure of $\mathcal{O}_{V_{x,k}^{(0,l\dots,l)}}$
inside $\mathcal{O}_{V_{x,k}^{(0,l\dots,l)}\times_{V_{x,k}}L_{V_{x,i,k}}}[w_{1}^{-1}\cdots w_{m}^{-1}]$.
Then $L_{V_{x,i,k}}^{(0,l,\dots,l)}\to V_{x,k}^{(0,l\dots,l)}$ is
a $G_{(0,l,\dots,l)}$-equivariant finite map, where $G_{(0,l,\dots,l)}$
is the $(m-1)$-fold product of cyclic groups of order $l$, acting
in the obvious way on $\{z_{2},\dots,z_{m}\}$. 

As all of the divisors living above $\{z_{2}\cdots z_{m}=0\}$ are
type II, we have that $(\overline{\mathcal{N}_{\lambda}})_{V_{x},i}[z_{1}^{-1}]=(\overline{\mathcal{N}}'_{i,\lambda}[z_{1}^{-1}])^{G_{(0,l,\dots,l)}}$
where $\overline{\mathcal{N}}'_{i,\lambda}[z_{1}^{-1}]$ is a $G_{(0,l,\dots l)}$-equivariant
meromorphic connection, with a basis which is an eigenbasis for the
connection on which $G_{(0,l,\dots,l)}$ acts transitively. We may
now extend this sheaf to a sheaf $\overline{\mathcal{N}}'_{i,\lambda}$,
which is a reflexive meromorphic $\lambda$-connection on $V_{x,k}^{(0,l,\dots,l)}$
constructed by the method of \prettyref{prop:Extension-in-codim-2}.
In other words, it is a bundle whose connection form has no singularities,
along $\{z_{1}=0\}$; we can readjust the residue by multiplying by
a power of $z_{1}$ as needed. Then we have $(\overline{\mathcal{N}_{\lambda}})_{V_{x},i}=(\overline{\mathcal{N}}'_{i,\lambda})^{G_{(0,l,\dots,l)}}$. 

Let $p:V_{x,k}^{(l)}\to V_{x,k}^{(0,l\dots,l)}$ be the root cover
along $w_{1}$. Then we claim that $p^{*}\overline{\mathcal{N}}'_{i,\lambda}$
is a $G_{l}$-equivariant vector bundle with $(p^{*}\overline{\mathcal{N}}'_{i,\lambda})^{G_{l}}=(\overline{\mathcal{N}_{\lambda}})_{V_{x},i}$.
Note that it is reflexive as the pullback of a reflexive sheaf under
a flat morphism is reflexive.

To prove that it is actually a vector bundle, we first note that $p^{*}\overline{\mathcal{N}}'_{i,\lambda}[z_{2}^{-1},\dots,z_{m}^{-1}]$
is a vector bundle on $V_{x,k}^{(l)}\backslash\{z_{2}\cdots z_{m}=0\}\times\mathbb{A}_{k}^{1}$
by the argument of \prettyref{lem:Extend-mod-p}- we showed there
that a reflexive coherent $\lambda$-connection, whose reduction mod
$\lambda$ is a line bundle on a smooth variety $L$ (in codimension
$2$), and whose $p$-support is contained in $L^{(1)}\times\mathbb{A}_{k}^{1}$,
is actually a bundle. Its reduction mod $(\lambda)$ is equal to the
free module $p^{*}\overline{\mathcal{N}}'_{i,\lambda}=p^{*}\pi_{*}(\mathcal{L}_{L_{V_{x},i}^{(0,l\dots,l)}})$,
where $\mathcal{L}$ is the trivial line bundle on $\mathcal{L}_{L_{V_{x},i}^{(0,l\dots,l)}}$,
with $G_{(0,l,\dots l)}$-action given by some character. The projection
$L_{V_{x},i}^{(0,l\dots,l)}\to V_{x,k}^{(0,l\dots,l)}$ is a finite
map, branched over $\{z_{1}=0\}$ and etale after removing $\{z_{1}=0\}$.
It follows that we can choose a basis of $\pi_{*}(\mathcal{L}_{L_{V_{x},i}^{(0,l\dots,l)}})$
which consists of eigenvectors for the action of $G_{(0,l,\dots l)}$.
The same is therefore true after pulling back to $V_{x,k}^{(l)}$,
and so we obtain a basis $\{e_{1},\dots,e_{r'}\}$ of $p^{*}\pi_{*}(\mathcal{L}_{L_{V_{x},i,k}^{(0,l\dots,l)}})$
which consists of eigenvectors for $G_{l}$. In particular, each $e_{i}$
has no zeros on $\overline{L}_{V_{x,i,k}^{(l)}}$. 

Now, we claim that there is a set of elements $\{\tilde{e}_{1},\dots,\tilde{e}_{r'}\}\subset p^{*}\overline{\mathcal{N}}'_{i,\lambda}[z_{2}^{-1},\dots,z_{m}^{-1}]$
on which $G_{l}$ acts by a character, which lift $\{e_{1},\dots,e_{r'}\}$,
and which have no zeros on $V_{x}^{(l)}|_{U_{k}}\times\mathbb{A}_{k}^{1}$.
To see this, we first rescale the $\tilde{e}_{i}$ by powers of $\{z_{2},\dots,z_{m}\}$
until the subgroup $G_{(0,l,\dots,l)}$ acts trivially. Then, note
that there is a decomposition 
\[
p^{*}\overline{\mathcal{N}}'_{i,\lambda}[z_{2}^{-1},\dots,z_{m}^{-1}]=\bigoplus_{\chi}(p^{*}\overline{\mathcal{N}}'_{i,\lambda}[z_{2}^{-1},\dots,z_{m}^{-1}])_{\chi}
\]
of $p^{*}\overline{\mathcal{N}}'_{i,\lambda}[z_{2}^{-1},\dots,z_{m}^{-1}]$
into submodules on which $G_{(l,0,\dots,0})$ acts with character
$\chi$. Each of these submodules is a vector bundle over $\mathcal{O}_{V_{x}^{(0,l\dots,l)}}[\lambda]$.
So applying Lindel's theorem (\cite{key-71}, Theorem on the bottom
of page 1), we see that each $(p^{*}\overline{\mathcal{N}}'_{\lambda,i}[z_{2}^{-1},\dots,z_{m}^{-1}])_{\chi}$
is induced from $\mathcal{O}_{V_{x,k}^{(0,l\dots,l)}}$. In particular,
we can choose lifts $\{\tilde{e}_{i}\}$ to sections, on which $G_{(l,0,\dots,0)}$
acts by a character, and which have no zeros on $V_{x}^{(0,l,\dots,)}|_{U_{k}}\times\mathbb{A}_{k}^{1}$;
i.e., when regarding them as sections of the pushforward of $p^{*}\overline{\mathcal{N}}'_{\lambda,i}[z_{2}^{-1},\dots,z_{m}^{-1}]$
to $V_{x}^{(0,l,\dots,l)}$. As $V_{x}^{(l)}|_{U_{k}}\times\mathbb{A}_{k}^{1}$
is an etale $G_{(l,0,\dots,0)}$-cover of $V_{x}^{(0,l,\dots,)}|_{U_{k}}\times\mathbb{A}_{k}^{1}$
and $G_{(l,0,\dots,0)}$ acts by a character on each section $\tilde{e}_{i}$,
this implies that these sections, when regarded as elements of $p^{*}\overline{\mathcal{N}}'_{i,\lambda}[z_{2}^{-1},\dots,z_{m}^{-1}]$,
have no zeros on $V_{x}^{(l)}|_{U_{k}}\times\mathbb{A}_{k}^{1}$.
Now we may again rescale by powers of $\{z_{2},\dots,z_{m}\}$ to
see that we can lift the original $\{e_{i}\}_{i=1}^{r'}$. 

Now choose a collection $\{f_{i}\}_{i=1}^{r'}$ of eigenvectors for
$\nabla$ inside $p^{*}\overline{\mathcal{N}}'_{i,\lambda}[z_{1}^{-1}]$
which are a basis for it it as a bundle, on which $G_{l}$ acts transitively.
We may write any of the sections $\tilde{e}_{i}\in p^{*}\overline{\mathcal{N}}'_{i,\lambda}[z_{2}^{-1},\dots,z_{m}^{-1}]$
as a sum of eigenvectors of the form ${\displaystyle \sum_{j=1}^{r'}\alpha_{j}f_{j}}$
(now working inside $p^{*}\overline{\mathcal{N}}'_{i,\lambda}[z_{1}^{-1},z_{2}^{-1},\dots,z_{m}^{-1}]$).
Since the action by $G_{l}$ is transitive, we see that each ratio
${\displaystyle \frac{\alpha_{j}}{\alpha_{j'}}}$ is a root of unity.
From this and the non-vanishing condition for the section $\tilde{e}_{i}$,
it follows that, $\alpha_{j}\in(\mathcal{O}_{V_{x}^{(l)}}[z_{1}^{-1},z_{2}^{-1},\dots,z_{m}^{-1},\lambda])^{*}=(\mathcal{O}_{V_{x}^{(l)}}[z_{1}^{-1},z_{2}^{-1},\dots,z_{m}^{-1}])^{*}$
for all $\alpha_{j}$; and since the reduction mod $\lambda$ of each
such section is contained in $p^{*}\overline{\mathcal{N}}_{i}'$,
we see that in fact $\alpha_{j}$ is contained in $\mathcal{O}_{V_{x,k}^{(l)}}[z_{1}^{-1}]$;
therefore, the sections $\{\tilde{e}_{1},\dots,\tilde{e}_{r'}\}$
are contained in $p^{*}\overline{\mathcal{N}}'_{i,\lambda}$ (since
this sheaf is reflexive, to check that these are sections we only
have to check it in codimension $2$; i.e., after inverting each $z_{i}$),
and in fact these sections must generate $p^{*}\overline{\mathcal{N}}'_{i,\lambda}[z_{1}^{-1}]$
as the change of basis matrix between the $\{e_{i}\}_{i=1}^{r}$ and
the $\{f_{i}\}_{i=1}^{r}$ has nowhere vanishing determinant. 

To see that $\{\tilde{e}_{i}\}_{i=1}^{r'}$ actually generate $p^{*}\overline{\mathcal{N}}'_{i,\lambda}$,
it is enough to check it in codimension $2$; as a sub-bundle of full
rank of a reflexive coherent sheaf, which is equal to the whole sheaf
in codimension $2$, is equal to it everywhere. Since $\{e_{1},\dots,e_{r'}\}$
are a basis for $p^{*}\overline{\mathcal{N}}'_{i}$, it follows that
$p^{*}\overline{\mathcal{N}}'_{i,\lambda}/\lambda\to p^{*}\overline{\mathcal{N}}_{i}'$
is surjective, and therefore\footnote{Since $\overline{\mathcal{N}}_{\lambda}$ is reflexive, by, \cite{key-76}
corollary 1.1.14 , $\overline{\mathcal{N}}_{\lambda}/\lambda$ is
torsion-free} an isomorphism and that $\{\tilde{e}_{i}\}_{i=1}^{r'}$ generate
in an open neighborhood of $V_{x}^{(l)}\subset V_{x}^{(l)}\times\mathbb{A}_{k}^{1}$;
however, by the previous paragraph they also generate $p^{*}\overline{\mathcal{N}}'_{i,\lambda}[z_{1}^{-1}]$;
so they generate in codimension $2$ as required. 
\end{proof}
Finally, to finish off this section, we note:
\begin{cor}
\label{cor:At-most-one-N-lambda}Suppose $H^{0}(\Omega_{\overline{L}_{k}}^{1})=0$.
For each $\mathcal{L}\in\text{Pic}^{r}(\overline{L}_{k})$ there is
at most one $\overline{\mathcal{N}}_{\lambda}$, with a given residue,
for which $\overline{\mathcal{N}}_{\lambda}/\lambda=\pi_{*}(\mathcal{L})$. 
\end{cor}

\begin{proof}
There is a natural map $\text{Pic}^{r}(\overline{L}_{k}^{(1)})\to\text{Pic}^{r}(\overline{L}_{k})$
which is given by the Frobenius pullback. By \prettyref{thm:Pic-over-k}
and \prettyref{thm:Pic}, we have to show that this map is injective.
So, suppose $\mathcal{L}'\in\text{Pic}^{r}(\overline{L}_{k}^{(1)})$
satisfies $F^{*}\mathcal{L}'\tilde{=}\mathcal{O}_{\overline{L}_{k}}$.
By Cartier descent, applied over the smooth locus $\overline{L}_{k}^{sm}$,
we obtain a flat connection on $\mathcal{O}_{\overline{L}_{k}^{sm}}$,
for which $\nabla(1)\in H^{0}(\Omega_{\overline{L}_{k}^{sm}}^{1})$.
But since $\overline{L}_{k}$ is a $V$-manifold, we have $H^{0}(\Omega_{\overline{L}_{k}^{sm}}^{1})=H^{0}(\Omega_{\overline{L}_{k}}^{1})=0$.
Therefore the connection is trivial, and we see that $\mathcal{L}'_{\overline{L}_{k}^{sm,(1)}}\tilde{=}\mathcal{O}_{\overline{L}_{k}^{sm,(1)}}$.
Since $\mathcal{L}'$ is reflexive, we must have $\mathcal{L}'_{\overline{L}_{k}^{(1)}}\tilde{=}\mathcal{O}_{\overline{L}_{k}^{(1)}}$. 
\end{proof}

\subsection{\label{subsec:Infinitesimal--regular-Connectio}Infinitesimal $\theta$-regular
Connections}

In this subsection, we give the definition of a $\theta$-regular
$\lambda$-connection, in the case of $\lambda$-connections over
$R[\lambda]/\lambda^{m}.$ Unlike in the case of positive or mixed
characteristic, we cannot start with a description of our object over
$U$ and then extend it; instead, we demand that our $\lambda$-connection
is $\theta$-regular as a Higgs sheaf (when $m=1$), and then specify
conditions ``at infinity.'' The strong assumptions on vanishing
of cohomology will then ensure that we get a unique object for each
$m$. 

This definition is fairly technically involved, but (we hope) the
constructions of the previous section can act as motivation. After
giving the definition, we'll go on the discuss the local deformation
theory, and then turn to proving the uniqueness and existence of deformations
in this context. This proof, in turn, relies on reduction mod $p$,
and therefore on the results of the previous section. 

Before giving the definition, we recall the following: suppose $\pi:L\to X$
is a finite morphism of varieties, which is a branched cover over
a smooth divisor $D\subset X$, let $E=\pi^{-1}(D)$. Then, if $\Omega_{X}^{1}(D)$
denotes log one-forms along $D$, we have a canonical isomorphism
$\pi^{*}(\Omega_{X}^{1}(D))\tilde{=}\Omega_{L}^{1}(E)$. Therefore,
if $\mathcal{L}$ is a coherent sheaf on $L$, we have 
\[
\pi_{*}(\mathcal{L}\otimes\Omega_{L}^{1}(E))\tilde{\to}\pi_{*}(\mathcal{L})\otimes\Omega_{X}^{1}(D)
\]
Now suppose $\mathcal{L}$ is a line bundle. Then, if $\varphi$ is
a section of $\Omega_{L}^{1}(E)$, $\varphi$ yields a map, which
we shall also call $\varphi:\mathcal{L}\to\mathcal{L}\otimes\Omega^{1}(E)$.
Applying $\pi_{*}$ we obtain a map
\[
\pi_{*}(\varphi):\pi_{*}(\mathcal{L})\to\pi_{*}(\mathcal{L})\otimes\Omega_{X}^{1}(D)
\]
If we pick a basis of $\pi_{*}(\mathcal{L})$, we get a matrix of
log one-forms which we shall denote $A_{\varphi}$. 

Now let $\overline{\mathcal{N}}_{\lambda,n}$ denote a meromorphic
$\lambda$-connection over $R[\lambda]/\lambda^{n}$, which is $R[\lambda]/\lambda^{n}$-flat.
Suppose that $\overline{\mathcal{N}}_{\lambda,n}/\lambda$ is a Higgs
bundle which is $\theta$-regular at infinity. We're going to adapt
\prettyref{def:Theta-Reg-Pos-Char}. To begin with, we should give
a version of \prettyref{lem:Decomop-before-extn}. Before stating
it, note that $\mathcal{N}_{\lambda,n}$ can be regarded as a sheaf
on $U$ via micro-localization. 
\begin{lem}
\label{lem:Decomp-before-extn-infinitesimal}Let $x\in\tilde{D}^{\text{sm}}$.
Then there is an etale neighborhood $\varphi:V_{x}\to\overline{\tilde{X}}$
of $x$ in which we have a direct sum of $\lambda$-connections
\[
\varphi^{*}(j_{*}\mathcal{N}_{\lambda,n})=\bigoplus_{i}j_{*}(\mathcal{N}_{\lambda,n})_{V_{x},i}
\]
where $j:U\to\overline{\tilde{X}}_{k}$ is the inclusion, and each
summand $j_{*}(\mathcal{N}_{\lambda})_{V_{x},i}$ is supported along
a component of the inverse image of $L$ inside $T^{*}(V_{x})\backslash\tilde{D}$.
In particular, the summands are indexed by components of $\overline{L}_{V_{x}}$. 
\end{lem}

This follows immediately from the decomposition of \prettyref{cor:Local-Struc-of-Closure},
as we demand $\overline{\mathcal{N}}_{\lambda,n}/\lambda$ is $\theta$-regular
at infinity (and $\lambda$ is nilpotent). So, we can proceed to the 
\begin{defn}
\label{def:Weakly-Theta-Reg-infinitesimal}We say that $\overline{\mathcal{N}}_{\lambda,n}$
is weakly $\theta$-regular at infinity if, for each $x$ which is
the generic point of a component of $\tilde{D}$, there is an etale
neighborhood $\varphi:V_{x}\to\overline{\tilde{X}}$ so that
\[
\varphi^{*}(\overline{\mathcal{N}}_{\lambda,n})=\bigoplus_{i}(\overline{\mathcal{N}}_{\lambda,n})_{V_{x},i}
\]
whose restriction to $V_{x}\backslash\tilde{D}_{k}$ agrees with the
decomposition of \prettyref{lem:Decomp-before-extn-infinitesimal}. 

1) If the summand $(\overline{\mathcal{N}}_{\lambda,n})_{V_{x},i}$
corresponds to a component of type I, then we demand that there is
a rational number $\alpha$ so that, for any basis, the matrix of
the connection satisfies
\[
[\nabla]=\Theta+\lambda A_{\frac{\alpha dz}{z}}+\lambda^{2}\Phi
\]
where $\Theta$ and $\Phi$ is a matrices of one-forms with no poles,
and $A_{\frac{\alpha dz}{z}}$ is the matrix of one forms described
directly above (with respect to the reduction of our given basis to
$(\overline{\mathcal{N}}_{\lambda,n})_{V_{x},i}/\lambda=\pi_{*}(\mathcal{L})$). 

2) If the summand $(\overline{\mathcal{N}}_{\lambda})_{V_{x},i}$
corresponds to a component of type II, then we have $(\overline{\mathcal{N}}_{\lambda,n})_{V_{x},i}=(\overline{\mathcal{N}}'_{i,\lambda,n})^{G_{l}}$
where $\overline{\mathcal{N}}'_{\lambda,n}$ is a $G_{l}$-equivariant
bundle on $V_{x}^{(l)}$ with meromorphic connection, which has a
basis $\{e_{j}\}$ on which $G_{l}$ acts transitively, so that 
\[
\nabla(e_{j})=(\theta_{j}+\lambda\beta\frac{dz}{z}))e_{j}
\]
where $\theta_{j}$ are the one-forms defined in \prettyref{def:omega-reg-prime-higher-dim}
and and $\beta\in\mathbb{Q}$. 
\end{defn}

Let us note that this definition (in type I) does not depend on the
choice of basis; in other words, if the description given there holds
in one basis it holds in all of them. To see this, we rewrite the
condition as follows: consider the Higgs field $\Theta$ on $(\overline{\mathcal{N}}_{\lambda,n})_{V_{x},i}/\lambda$,
let the matrix of $\Theta$ in the given basis be $[\Theta]$. Then
\[
[\nabla]-[\Theta]=\lambda A_{\frac{\alpha dz}{z}}+\lambda^{2}\Phi
\]
If we change basis, both $[\Theta]$ and $\lambda A_{\frac{\alpha dz}{z}}$
are altered by conjugation, and the condition that $\Phi$ has no
poles is still satisfied. 

We can therefore define the residue of $\overline{\mathcal{N}}_{\lambda,n}$,
denoted $\text{res}(\overline{\mathcal{N}}_{\lambda,n})$, as the
$\mathbb{Q}$-divisor determined by $\alpha E$ if $E$ is a divisor
in a component $L_{V_{x},i}$ of type I, and $\beta E$ in type II.
In type I, if $\alpha=0$ then we have that $(\overline{\mathcal{N}}_{\lambda,n})_{V_{x},i}$
is a bundle with $\lambda$-connection (with no singularities). 

We note that this definition makes sense over any field $k$ to which
$R$ maps, if we modify it by demanding that the numbers $\alpha$
and $\beta$ live in $k$. In particular, if we consider the $\theta$-regular
connections $\overline{\mathcal{N}}_{\lambda,k}$ of the previous
section, then they necessarily satisfy the definitions; this follows
from looking at the condition for divisors of type I satisfied by
$\overline{\mathcal{N}}_{\lambda,k}$ (that the bundle is $z^{\alpha}\cdot\mathcal{V}_{\lambda}$,
where $\mathcal{V}_{\lambda}$ has no singularities); writing this
out in a basis gives the form $A_{\frac{\alpha dz}{z}}$ considered
in the definition above. 

However, the connections $\overline{\mathcal{N}}_{\lambda,k}$ also
satisfy extra rigidity conditions, which are not automatically satisfied
in the infinitesimal case. That is because, when we don't demand that
our bundles extend all the way to $k[\lambda]$, there could be ``extra''
infinitesimal deformations. We need to guarantee that these don't
occur. 

To do this, consider a point $x\in\tilde{X}$ which is contained in
exactly $m$ components of $\tilde{D}$. Recall that, by assumption
$2$ concerning the structure of $\overline{L}$, we can suppose that
at most $1$ component of $\tilde{D}$ containing $x$ has a component
of type I (in $\overline{L}$) living above it. If no such components
are of type I, then, in fact, the above definition is already enough. 

So, let $V_{x}$ be an etale neighborhood of $x$ and let $\{z_{1},\dots,z_{m}\}$
be local equations for the components of $\tilde{D}$, and suppose
that, over the generic point of $\{z_{1}=0\}$, there is a component
of $L_{V_{x}}$, called $L_{V_{x},i}$ of type I . Consider the partial
root cover 
\[
V_{x}^{(0,l,\dots,l)}\to V_{x}
\]
which consists of taking $l$th roots of the variables $\{z_{2},\dots,z_{m}\}$.
This scheme is acted upon by $G_{(0,l,\dots,l)}$, a product of cyclic
groups of order $l$. Arguing exactly as in \prettyref{cor:Local-Struc-of-Closure},
we see that there is disjoint union of schemes $\{L_{V_{x,i}}^{(0,l,\dots,l)}\}$,
equipped with an action of $G_{(0,l,\dots,l)}$, and $G_{(0,l,\dots,l)}$-equivariant
projections $\pi_{(0,l,\dots,l)}:L_{V_{x,i}}^{(0,l,\dots,l)}\to V_{x}^{(0,l,\dots,l)}$.
Note that these schemes are obtain from the smooth scheme $L_{V_{v}}^{(l)}$
by the action of a single cyclic group of order $l$, and are therefore
smooth themselves. The orbits under the action of $G_{(0,l,\dots,l)}$
yield the components $L_{V_{x},i}$. 

Then, there is a $G_{(0,l,\dots,l)}$-equivariant line bundle $\mathcal{L}$
on ${\displaystyle \bigsqcup_{i}L_{V_{x,i}}^{(0,l,\dots,l)}}$ so
that 
\[
\overline{\mathcal{N}}_{V_{x}}=(\pi_{(0,l,\dots,l)})_{*}(\mathcal{L},df)^{G_{(0,l,\dots,l)}}
\]
as Higgs bundles. 

Furthermore, we can decompose
\[
(\pi_{(0,l,\dots,l)})_{*}(\mathcal{L},df)=\bigoplus_{i}(\pi_{(0,l,\dots,l)})_{*}(\mathcal{L}|_{L_{V_{x,i}}^{(0,l,\dots,l)}},df):=\bigoplus_{i}\overline{\mathcal{N}}'_{i}
\]

Now, let 
\[
\overline{\mathcal{N}}'_{\lambda,n}=\bigoplus_{i}(\overline{\mathcal{N}}_{\lambda,n}')_{i}
\]
be an $R[\lambda]/\lambda^{n}$-flat deformation of ${\displaystyle \bigoplus_{i}\overline{\mathcal{N}}'_{i}}$. 
\begin{defn}
\label{def:Theta-reg-infinitesimal}We say that $\overline{\mathcal{N}}'_{\lambda,n}$
is $\theta$-regular at infinity each summand $(\overline{\mathcal{N}}_{\lambda,n}')_{i}$
satisfies conditions analogous to those of \prettyref{def:Weakly-Theta-Reg-infinitesimal};
i.e., if $(\overline{\mathcal{N}}_{\lambda,n}')_{i}$ corresponds
to a component which is embedded over $\{z_{1}=0\}$, then we demand
that the first condition of \prettyref{def:Weakly-Theta-Reg-infinitesimal}
is satisfied etale locally at the generic point of $\{z_{1}=0\}$,
and in all other cases we demand that the second condition of \prettyref{def:Weakly-Theta-Reg-infinitesimal}
is satisfied etale locally at the generic point of $\{z_{1}=0\}$;
in all cases we demand that this second condition is satisfied etale
locally at the generic point of $\{z_{i}=0\}$ for $i\neq1$. 

We say that our bundle $\overline{\mathcal{N}}_{\lambda,n}$ is $\theta$-regular
at infinity if, for each $x$in exactly $m$ components of $\tilde{D}$,
there is a bundle $\overline{\mathcal{N}}'_{\lambda,n}$ on $V_{x}^{(0,l,\dots,l)}$
so that $(\overline{\mathcal{N}}_{\lambda,n})_{V_{x}}=(\overline{\mathcal{N}}'_{\lambda,n})^{G_{(0,l,\dots,l)}}$.
Note that that (the proof of) \prettyref{lem:2-d-goodness} ensures
that, after passing to $k$, the bundles $\overline{\mathcal{N}}_{\lambda}/\lambda^{n}$
satisfy this condition. Furthermore, in the sections directly below,
we'll see that this definition gives us good control over the local
deformation theory of $\overline{\mathcal{N}}_{\lambda,n}$. 
\end{defn}

\subsubsection{\label{subsec:Local-Deformation-theory}Local Deformation theory}

We turn now to understanding the local structure of deformations of
a $\theta$-regular bundle. To set the stage for these results, let
us begin by working over $U$. We have a useful general result which
is similar in spirit to \prettyref{lem:All-bundles-are-pi-push}. 
\begin{lem}
\label{lem:Deformations-over-U}Let $U'\subset U$ be an affine open
subset. Then the set of (isomorphism classes of) deformations of $\mathcal{N}_{\lambda,n}(U')$
is a pseudo-torsor over $\pi_{*}\Omega_{L}^{1}(U')\tilde{=}\text{Ext}_{U'}^{1}(\mathcal{N},\mathcal{N})$.
Further, we have that $\mathcal{N}_{\lambda,n}|_{U'}\tilde{=}\pi_{*}\mathcal{L}_{\lambda,n}$
where $\mathcal{L}_{\lambda,n}$ is a line bundle with $\lambda$-connection
over $R[\lambda]/\lambda^{n}$ on $L_{U'}$. The obstruction to deforming
$\mathcal{N}_{\lambda,n}|_{U'}$ is a class in $\pi_{*}\Omega_{L}^{2}(U')$. 
\end{lem}

\begin{proof}
As $L_{U'}\to U'$ is etale, the scheme $L_{U'}\times_{U'}L_{U'}$
has the diagonal $L_{U'}$ as a component. Since the Higgs bundle
$\mathcal{N}|_{U'}$ is scheme-theoretically supported along $L_{U'}\subset T^{*}U'$,
the Higgs bundle $\pi^{*}\mathcal{N}|_{L_{U'}}$ is supported along
$L_{U'}\times_{U'}L_{U'}\subset T^{*}L_{U'}\tilde{=}L_{U'}\times_{U'}T^{*}U'$.
Thus there is a summand of $\pi^{*}\mathcal{N}|_{L_{U'}}$ which is
scheme-theoretically supported along $L_{U'}\subset L_{U'}\times_{U'}L_{U'}$;
i.e. this summand is a line bundle $(\mathcal{L},\eta)$ where the
one-form $\eta$ has graph equal to $L_{U'}\subset T^{*}L_{U'}$. 

If we write $\pi^{*}\mathcal{N}|_{L_{U'}}=\mathcal{L}\oplus\mathcal{E}$,
then $\mathcal{E}$ is a Higgs bundle supported on the complement
of $L_{U'}$ in $L_{U'}\times_{U'}L_{U'}$; therefore $\text{Ext}^{1}(\mathcal{L},\mathcal{E})=0$.
Thus any formal deformation of $\pi^{*}\mathcal{N}|_{L_{U'}}$ is
a direct sum of deformations of $\mathcal{L}$ and $\mathcal{E}$;
so we may write 
\[
\pi^{*}\mathcal{N}_{n,\lambda}|_{L_{U'}}=\mathcal{L}_{n,\lambda}\oplus\mathcal{E}_{n,\lambda}
\]
where $\mathcal{L}_{n,\lambda}$ is a line bundle with $\lambda$-connection
over $R[\lambda]/\lambda^{n}$ on $L_{U'}$. Thus the projection $\pi^{*}\mathcal{N}_{n,\lambda}\to\mathcal{L}_{n,\lambda}$
yields by adjunction a map $\mathcal{N}_{n,\lambda}\to\pi_{*}\mathcal{L}_{n,\lambda}$;
this map is an isomorphism mod $\lambda$ by \prettyref{thm:Pic},
and hence an isomorphism by Nakayama's lemma. 

Thus, the set of deformations is a torsor over $\text{Ext}_{U'}^{1}(\mathcal{N},\mathcal{N})\tilde{=}\text{Ext}_{L_{U'}}^{1}(\mathcal{L},\mathcal{L})\tilde{=}\Omega_{L}^{1}(L_{U'})$;
while the obstruction is given by choosing any lift of $\mathcal{L}_{\lambda,n}$
and taking the differential of the $\lambda$-connection. 
\end{proof}
Now we consider the situation near a suitable point in $\tilde{D}$.
In the following we fix a $\mathbb{Q}$-divisor and consider $\theta$-regular
$\lambda$-connections whose residue is equal to this fixed divisor. 
\begin{lem}
\label{lem:Local-Deformation-Statement}$1)$ Let $\{x\}$ be the
generic point of a component of $\tilde{D}$. Let $n\geq1$. Suppose
that $\overline{\mathcal{N}}_{\lambda,n}$ is a $\theta$-regular
$\lambda$-connection over $R[\lambda]/\lambda^{n}$. Then there is
an etale neighborhood $\varphi:V_{x}\to\overline{\tilde{X}}$ so that
set of $\theta$-regular deformations of $\varphi^{*}\overline{\mathcal{N}}_{\lambda,n}$
is a pseudo-torsor over $\pi_{*}\Omega_{\overline{L}_{V_{x}}}^{1}$.
The obstruction to deforming $\varphi^{*}\overline{\mathcal{N}}_{\lambda,n}$
is a class in $\pi_{*}\Omega_{\overline{L}_{V_{x}}}^{2}$. 

2) Suppose $\{x\}$ is the generic point of the intersection of $m$
components of $\tilde{D}$, for some $m\geq2$. Then here is an etale
neighborhood $\varphi:V_{x}\to\overline{\tilde{X}}$ so that set of
$\theta$-regular deformations of $\varphi^{*}\overline{\mathcal{N}}_{\lambda,n}$
is a pseudo-torsor over $\pi_{*}\Omega_{\overline{L}_{V_{x}}}^{1}$.
\end{lem}

\begin{proof}
$1)$ First suppose $\{x\}$ is the generic point of a component of
$\tilde{D}$. Write 
\[
\varphi^{*}\overline{\mathcal{N}}_{\lambda,n}=\bigoplus_{i}\overline{\mathcal{N}}_{\lambda,n,i}
\]
a decomposition as in \prettyref{def:Weakly-Theta-Reg-infinitesimal}.
Then either $\overline{\mathcal{N}}_{\lambda,n,i}$ is $G_{l}$-invariants
inside some direct sum of line bundles with meromorphic connection
(in which case the proof of the result is exactly as in the previous
lemma); or, $\overline{\mathcal{N}}_{\lambda,n,i}$ is of type I;
in other words, up to adding the monodromy term $\lambda A_{\frac{\alpha dz}{z}}$,
the bundle has a connection which is a deformation of $\pi_{*}(\mathcal{O}_{L,V_{x,i}},df)$.
This, in turn, is a Higgs bundle corresponding to the line bundle
on a smooth Lagrangian. In this case, deformations form a pseudotorsor
over 
\[
\text{Ext}_{V_{x,i}}^{1}(\pi_{*}(\mathcal{O}_{L,V_{x,i}},df),\pi_{*}(\mathcal{O}_{L,V_{x,i}},df))\tilde{=}\pi_{*}\Omega_{\overline{L}_{V_{x}}}^{1}
\]
with an obstruction class in $\pi_{*}\Omega_{\overline{L}_{V_{x}}}^{2}$.
So we obtain the lemma over the generic point of a component of $\tilde{D}$. 

$2)$ Now suppose $\{x\}$ is the generic point of the intersection
of $m$ components of $\tilde{D}$. Let $\{z_{1},\dots,z_{m}\}$ denote
local equations for the components of $\tilde{D}$. Then either $\overline{\mathcal{N}}_{\lambda,n}$
is simply $G_{l}$-invariants inside a sum of line bundles on the
root cover $V_{x}^{(l)}$, in which case the result is obvious; or
this is only true near the divisors $\{z_{i}=0\}$ for $i>1$. In
that case, we can argue by pulling back under the root cover $V_{x}^{(0,l,\dots,l)}\to V_{x}$
(where the we take the $l$th root of each of $\{z_{2},\dots,z_{m}\}$). 

To carry this out, consider the decomposition
\[
\overline{\mathcal{N}}'_{\lambda,n}=\bigoplus_{i}(\overline{\mathcal{N}}_{\lambda,n}')_{i}
\]
where $\overline{\mathcal{N}}'_{\lambda,n}$ is the bundle appearing
in \prettyref{def:Theta-reg-infinitesimal}, and each $\overline{\mathcal{N}}_{\lambda,n,i}$
is a deformation of the Higgs sheaf $\pi_{*}(\mathcal{O}_{\overline{L}_{V_{x},i}},df)$
appearing in \prettyref{cor:Local-Struc-of-Closure}. Then $\overline{\mathcal{N}}'_{\lambda,n}$
is $G_{(0,l,\dots,l)}$-equivariant vector bundle on $V_{x}^{(0,l,\dots,l)}$;
whose restriction to the generic point of each component of $\tilde{D}$
(except $\{z_{1}=0\}$) is a direct sum of eigenvectors for the connection.

Suppose we have a $\theta$-regular deformation $\overline{\mathcal{N}}'_{\lambda,n+1,i}$
of $\overline{\mathcal{N}}'_{\lambda,n,i}=p^{*}\overline{\mathcal{N}}_{\lambda,n,i}$.
Then $\overline{\mathcal{N}}'_{\lambda,n+1,i}[(z_{2}\cdots z_{m})^{-1}]$
is a flat deformation of the free module $\overline{\mathcal{N}}'_{\lambda,n,i}[(z_{2}\cdots z_{m})^{-1}]$,
and the set of isomorphism classes of such deformations are a torsor
over $\pi_{*}(\Omega_{L_{V_{x},i}}^{1}[(z_{2}\cdots z_{m})^{-1}])$.
Consider a section $\phi$ of $\pi_{*}(\Omega_{L_{V_{x},i}}^{1})\subset\pi_{*}(\Omega_{L_{V_{x},i}}^{1}[(z_{2}\cdots z_{m})^{-1}])$.
We have that $\overline{\mathcal{N}}'_{\lambda,n+1,i}\subset\overline{\mathcal{N}}'_{\lambda,n+1,i}[(z_{2}\cdots z_{m})^{-1}]$
is a sub-bundle (with $\theta$-regular meromorphic $\lambda$-connection).
The deformation $\overline{\mathcal{N}}'_{\lambda,n+1,i}[(z_{2}\cdots z_{m})^{-1}]+\phi$
is the same bundle, with the connection structure modified by adding
a term of the form $\lambda^{n}\Psi$ for a suitable operator $\Psi$.
So we can can define the meromorphic $\lambda$-connection $\overline{\mathcal{N}}'_{\lambda,n+1,i}+\phi$
by adding the term $\lambda^{n}\Psi$ to the connection form. It follows
(by looking at an eigenbasis for the connection in a neighborhood
of each component of $\tilde{D}$) that the deformation $\overline{\mathcal{N}}'_{\lambda,n+1,i}+\phi$
is $\theta$-regular at infinity as required. 
\end{proof}

\subsubsection{Construction of $\theta$-regular sheaves over $R[\lambda]/\lambda^{2}$}

Now we will, for the first time in this chapter, employ the assumption
that our Lagrangian $L$ is unobstructed (in the sense of \prettyref{def:Unobstructed!}),
in order to construct a family of $\theta$-regular $\lambda$-connections
over $R[\lambda/\lambda^{2}]$. We start with a purely local construction,
which allows us to match the definition of $\theta$-regularity in
this chapter with the constructions of chapter 2. 
\begin{lem}
\label{lem:Correct-extension}In the notations of \prettyref{def:embedded-in-T^*},
suppose $L_{V_{x},i}$ is embedded in $T^{*}V_{x}$ (here, $V_{x}$
is an etale neighborhood of the generic point $\{x\}$ of a component
of $\tilde{D}$). Let $\mathcal{E}_{\lambda,2}=\pi_{*}(\mathcal{O}_{L_{V,x}i}[\lambda]/\lambda^{2})$.
Then there is an extension $\tilde{\mathcal{E}}_{\lambda,2}$ of $\mathcal{E}_{\lambda,2}|_{L_{V_{x},i}\backslash\tilde{D}}$
to a bundle with meromorphic connection, in which the matrix for the
connection has the form 
\[
[\nabla]=\Theta+\lambda A_{\frac{\alpha dz}{z}}
\]
as in \prettyref{def:Weakly-Theta-Reg-infinitesimal}; in particular,
$\alpha\in\mathbb{Q}$. 
\end{lem}

\begin{proof}
Consider the log one form $\beta{\displaystyle \frac{dz}{z}}$ on
$L_{V,x}i$ (here, $z$ is coordinate of $\tilde{E}$), where $\beta\in\mathbb{Q}$
is arbitrary. Then, we can define a $\lambda$-connection 
\[
\mathcal{E}_{\lambda,2}|_{L_{V_{x},i}\backslash\tilde{D}}+[\beta{\displaystyle \frac{dz}{z}}]
\]
where $+$ denotes the action of $\Omega_{L_{V_{x},i}\backslash\tilde{D}}^{1}$
on the torsor of infinitesimal deformations of $\pi_{*}(\mathcal{O}_{L_{V,x}i\backslash\tilde{D}},df)$.
We know from \prettyref{lem:Def-of-=00005Cpsi} and \prettyref{thm:Microlocal-form-of-E},
that, for appropriate choice of $\beta$, this object extends to a
$\lambda$-connection on all of $V_{x}$, which deforms $\pi_{*}(\mathcal{O}_{L_{V,x}i})$.
Call this $\lambda$-connection $\mathcal{E}'_{\lambda,2}$. In any
basis, the matrix for the connection, call it $M$, has no singularities.
But then the matrix $M+\lambda A_{\frac{\alpha dz}{z}}$ defines a
new flat $\lambda$-connection on $\mathcal{E}'_{\lambda,2}$, and
if we set $\beta=-\alpha$ we obtain the result. 
\end{proof}
With this is hand, we can construct a $\theta$-regular $\lambda$-connection
in codimension $2$: 
\begin{cor}
There is an open subset $V\subset\overline{\tilde{X}}$, whose complement
has codimension $2$, on which there is a sheaf $\tilde{\mathcal{E}}_{\lambda,2}$
which satisfies the conditions of \prettyref{def:Weakly-Theta-Reg-infinitesimal}
at the generic points of divisors. 
\end{cor}

\begin{proof}
Start with the meromorphic $\lambda$-connction$\mathcal{E}_{\lambda,2}:=\pi_{*}(\mathcal{O}_{\bar{L}}[\lambda]/\lambda^{2},df)$.
For any type II divisor, this sheaf satisfies the $\theta$-regularity
condition automatically; and we can modify it at the generic points
of divisors of type I by the previous lemma to obtain a $\theta$-regular
sheaf on an open subset $V$ as required. 
\end{proof}
Now, we want to extend $\tilde{\mathcal{E}}_{\lambda,2}$ to a $\theta$-regular
bundle on all of $\overline{\tilde{X}}$. For this, we're going to
use the deformation theory of the previous subsection to prove:
\begin{lem}
\label{lem:Adjust-the-extension}Suppose $H^{1}(\mathcal{O}_{\overline{L}})=0=H^{0}(\Omega_{\overline{L}}^{1})$.
Let $\tilde{\mathcal{E}}_{\lambda,2}/\lambda=\pi_{*}(\mathcal{L},df)$
as a Higgs sheaf on $V$ (here, $\mathcal{L}$ is a line bundle on
$L_{V}$). Then, there is a unique deformation of $\pi_{*}(\mathcal{L},df)$
to an $R[\lambda]/\lambda^{2}$-flat $\lambda$-connection on $V$,
which is locally isomorphic to $\tilde{\mathcal{E}}_{\lambda,2}$,
and which extends to a $\theta$-regular bundle on all of $\overline{\tilde{X}}$. 
\end{lem}

\begin{proof}
Consider the set of isomorphism classes of deformations of $\pi_{*}(\mathcal{L},df)$,
on $V$. If $\mathcal{F}_{\lambda,2}$ is any such, then by \prettyref{lem:Local-Deformation-Statement},
the difference $[\tilde{\mathcal{E}}_{\lambda,2}]-[\mathcal{F}_{\lambda,2}]$
yields a section of $\Omega_{\overline{L}_{V}}^{1}$ (i.e., we can
define such a section locally, and they must agree on overlaps). As
$H^{0}(\Omega_{\overline{L}}^{1})=H^{0}(\Omega_{\overline{L}_{V}}^{1})=0$,
we see that $\mathcal{F}_{\lambda,2}$ and $\tilde{\mathcal{E}}_{\lambda,2}$
are locally isomorphic, as deformations. Therefore, there is an open
cover $\{V_{i}\}$ of $V$ on which we have a collection of isomorphisms
$\varphi_{i}:\tilde{\mathcal{E}}_{\lambda,2}\to\mathcal{F}_{\lambda,2}$
. This yields sections $\alpha_{ij}\in\mathcal{O}(\overline{L}_{V_{i}\cap V_{j}})$,
so that the maps $\varphi_{i}\circ\varphi_{j}^{-1}$ are given by
multiplication by $1+\lambda\alpha_{ij}$. Thus we obtain a class
in 
\[
H^{1}(\mathcal{O}_{L_{V}})\tilde{=}H^{0}R^{1}j_{*}(\mathcal{O}_{L_{V}})
\]
where the last isomorphism uses $H^{1}(\mathcal{O}_{\overline{L}})=0$.
Conversely, this module clearly acts on the set of deformations by
modifying the transition maps. 

Now, the sheaf $R^{1}j_{*}(\mathcal{O}_{L_{V}})$ is concentrated
at points of codimension $2$ in $\overline{L}\backslash\overline{L}_{V}$,
and there are only finitely many of these. So, if we can show that
we can locally modify $\tilde{\mathcal{E}}_{\lambda,2}$ to a sheaf
which is $\theta$-regular at infinity about each such point, then
we obtain the desired extension. 

Let $\{x\}$ be such a point. We shall assume that $\{x\}$ is the
intersection of two components of $\tilde{D}$, call then $D_{1}$
and $D_{2}$, the other cases being similar (but simpler). As usual,
the only interesting case is where there is a component $L_{V_{x},i}$
of $L_{V_{x}}$ which is of type I over one of the divisors, say $D_{1}$.
Pulling back to the root cover $V_{x}^{(l,0)}$, we can perform the
analogous construction for $\tilde{\mathcal{E}}_{\lambda,2}$ on $V_{x}^{(l,0)}$to
obtain a sheaf $\tilde{\mathcal{E}}'_{\lambda,2}$ on $V_{x}^{(l,0}\backslash\{x\}$
whose $G_{(0,l)}$-invariants are equal to $\tilde{\mathcal{E}}_{\lambda,2}$.
We have the decomposition 
\[
\tilde{\mathcal{E}}'_{\lambda,2}=\bigoplus_{i}(\tilde{\mathcal{E}}'_{\lambda,2})_{i}
\]
according to the support of the sheaf $\tilde{\mathcal{E}}'_{\lambda,2}$
in $T^{*}V_{x}^{(l,0)}$ (the sum is over components of $L_{V_{x}}^{(l,0)}$).
Then, as above, the space of deformations of $\tilde{\mathcal{E}}'_{\lambda,2}$
which are locally isomorphic to it, are a torsor over $R^{1}j_{*}(\mathcal{O}_{L_{V_{x}}^{(l,0)}})$.
As $L_{V_{x}}^{(l,0)}\to V_{x}^{(l,0)}$ is a branched covering of
smooth varieties (along $D_{1}$), we have an injection 
\begin{equation}
R^{1}j_{*}(\mathcal{O}_{L_{V_{x}}^{(l,0)}})\to R^{1}j_{*}(\mathcal{E}nd_{\mathcal{O}_{V_{x}^{(l,0)}}}(\pi_{*}\mathcal{O}_{L_{V_{x},i}^{(l,0)}}))\label{eq:injection}
\end{equation}
This is because the map $\mathcal{O}_{L_{V_{x}}^{(l,0)}}\to\pi^{*}\pi_{*}\mathcal{O}_{L_{V_{x},i}^{(l,0)}}$
is a split injection of bundles; and these cohomology groups can be
described as 
\begin{equation}
\mathcal{O}_{L_{V_{x}}^{(l,0)}}[(z_{1}z_{2})^{-1}]/\mathcal{O}_{L_{V_{x}}^{(l,0)}}z_{1}^{-1}+\mathcal{O}_{L_{V_{x}}^{(l,0)}}z_{2}^{-1}\label{eq:explicit-groups}
\end{equation}
(and the analogous formula for $\mathcal{E}nd_{\mathcal{O}_{V_{x}^{(l,0)}}}(\pi_{*}\mathcal{O}_{L_{V_{x},i}^{(l,0)}})$).
On the other hand, we can attach to $\tilde{\mathcal{E}}'_{\lambda,2}$
a canonical class in $R^{1}j_{*}(\mathcal{E}nd_{\mathcal{O}_{V_{x}^{(l,0)}}}(\pi_{*}\mathcal{O}_{L_{V_{x}^{(l,0)},i}}))$,
which measures its failure to extend to a bundle on all of $V_{x}^{(l,0)}$.
We want to know that this class is in the image of \prettyref{lem:Correct-extension}.
However, this is the case after reduction mod $p$ for all $p>>0$,
because $\tilde{\mathcal{E}}'_{\lambda,2,k}$ is locally isomorphic
to $\overline{\mathcal{N}}'_{\lambda,k}/\lambda^{2}$, which extends
to a bundle on $V_{x}^{(l,0)}$. Therefore, it must be true over $R$,
as follows from the explicit description of the groups $R^{1}j_{*}(\mathcal{O}_{L_{V_{x}}^{(l,0)}})$,
$R^{1}j_{*}(\mathcal{E}nd_{\mathcal{O}_{V_{x}^{(l,0)}}}(\pi_{*}\mathcal{O}_{L_{V_{x},i}^{(l,0)}}))$.
Thus we can make the required modification of $\tilde{\mathcal{E}}_{\lambda,2}$,
and it is unique by the injectivity of the map \prettyref{eq:injection}.
\end{proof}
\begin{rem}
The previous proof shows the necessity of the $\theta$-regularity
condition (as opposed to only considering weak $\theta$-regularity).
If we did not pull back to $V_{x}^{(l,0)}$ in the previous proof,
we could repeat all the steps until we got to the map 
\[
R^{1}j_{*}(\mathcal{O}_{L_{V_{x}}})\to R^{1}j_{*}(\mathcal{E}nd_{\mathcal{O}_{V_{x}^{(l,0)}}}(\pi_{*}\mathcal{O}_{L_{V_{x},i}}))
\]
As far as I can tell, this map is not injective in general. 
\end{rem}

From now on, we abuse notation and replace $\tilde{\mathcal{E}}_{\lambda,2}$
with the correct extension to $\overline{\tilde{X}}$ constructed
above. Note that, for $k$ of characteristic $p>>0$, the reduction
mod $p$ of this sheaf agrees with $\overline{\mathcal{N}}_{\lambda}/\lambda_{2}$
for a suitable $\overline{\mathcal{N}}_{\lambda}$; this is because
of the construction of $\overline{\mathcal{N}}_{\lambda}$ as a modification
of $\mathcal{E}_{\lambda}=\pi_{*}(\mathcal{O}_{\overline{L}}[\lambda],df)$. 

Next, we need to define an action of a suitable group of line bundles
with log connection (with respect to $\tilde{E}$) on $L_{V}$, on
a subset of the set of $\theta$-regular connections. 

In positive characteristic, we used locally trivial connections. Here,
we'll consider all connections whose monodromy group is finite; this
implies in particular that we suppose that the residue of the log
connection is rational. When $H_{dR}^{1}(L_{V})=0$, the latter condition
actually implies the finiteness of the associated monodromy group.
At any rate, let $\text{Pic}^{\text{fin}}(L_{V},\tilde{E})$ denote
this group. We have 
\begin{prop}
Let $(\mathcal{L},\nabla)\in\text{Pic}^{\text{fin}}(L_{V},\tilde{E})$,
and suppose it has associated residue $B={\displaystyle \sum\beta_{i}\tilde{E}_{i}}$.
Then there is a $\theta$-regular connection over $V$, $(\mathcal{L},\nabla)\star\tilde{\mathcal{E}}_{\lambda,2}$,
whose residue is given by $B+\text{res}(\tilde{\mathcal{E}}_{\lambda,2})$.
This operation defines an action of $\text{Pic}^{\text{fin}}(L_{V},\tilde{E})$
on the subset of (isomorphism classes of) $\theta$-regular connections
of type $(\mathcal{L},\nabla)\star\tilde{\mathcal{E}}_{\lambda,2}$.
We have 
\[
((\mathcal{L},\nabla)\star\tilde{\mathcal{E}}_{\lambda,2})/\lambda=\mathcal{L}\star\tilde{\mathcal{E}}_{\lambda,2}/\lambda
\]
where the object on the right is the action of line bundles on $\overline{L}$
on $\theta$-regular Higgs sheaves given by \prettyref{thm:Pic}.
\end{prop}

\begin{proof}
Over $U$, we have that $\tilde{\mathcal{E}}_{\lambda,2}=\pi_{*}(\mathcal{O}[\lambda]/\lambda^{2},df)$.
So we can define the required sheaf as $\pi_{*}(\mathcal{L}[\lambda]/\lambda^{2},df+\lambda\nabla)$.
Here, we are using the fact that $(\mathcal{L}[\lambda],\lambda\nabla)$
is a line bundle with $\lambda$-connection). Let us consider extensions
of this to $V$. Let $x$ be the generic point of some component of
$\tilde{D}$; employ the notation of \prettyref{cor:Local-Struc-of-Closure}.
If $L_{V_{x,i}}\to V_{x}$ is a component of type II, then we have
\[
\tilde{\mathcal{E}}_{\lambda,2}=\pi_{*}(\mathcal{O}_{L_{V_{x},i}^{(l)}}[\lambda]/\lambda^{2},df+\lambda\beta dz)^{G_{l}}
\]
where $L_{V_{x},i}^{(l)}$ is a suitable union of the components of
the etale cover $L_{V_{x}}^{(l)}\to V_{x}^{(l)}$. The connection
$(\mathcal{L},\lambda\nabla)$ extends to a line bundle with log connection
on $L_{V_{x}}^{(l)}$ (as $L_{V_{x}}^{(l)}\to L_{V_{x,i}}$ is a $G_{l}$-cover),
and so we can define 
\[
\mathcal{L}\star\tilde{\mathcal{E}}_{\lambda,2}=\pi_{*}(\mathcal{L}_{L_{V_{x},i}^{(l)}}[\lambda]/\lambda^{2},df+\lambda\beta dz+\lambda\nabla)^{G_{l}}
\]
which is a $\theta$-regular extension of $\pi_{*}(\mathcal{L}[\lambda]/\lambda^{2},df+\lambda\nabla)$. 

Now consider the case where $L_{V_{x,i}}\to V_{x}$ is a component
of type I. Since the line bundle $(\mathcal{L}[\lambda],\lambda\nabla)$
has trivial $p$-cuvature after reduction mod $p$, we see (by employing
the exact proof of \prettyref{thm:Microlocal-form-of-E} and then\prettyref{lem:Correct-extension}),
that the connection $j_{*}\pi_{*}(\mathcal{L}_{L_{V_{x},i}}[\lambda]/\lambda^{2},df+\lambda\nabla)$
extends to a bundle $\mathcal{L}_{L_{V_{x},i}}\star\pi_{*}(\mathcal{L}_{L_{V_{x},i}})$
which is $\theta$-regular along the divisor $\tilde{E}$. 

As the line bundle $(\mathcal{L},\nabla)$ defines a connection with
finite order, we have that $(\mathcal{L},\nabla)$ is locally trivial
after reduction mod $p$ (and its residue is in $\mathbb{F}_{p}$
if we take $p>>0$. Thus the sheaf $(\mathcal{L},\nabla)\star\tilde{\mathcal{E}}_{\lambda,2}$
agrees with a sheaf of the type $(\mathcal{L},\nabla)\star\overline{\mathcal{N}}_{\lambda,2}$,
where the action is the one defined in \prettyref{thm:Pic-over-k}.
Therefore this sheaf extends to one which is $\theta$-regular at
infinity on all of $\overline{\tilde{X}}$ (as this is true mod $p$
for $p>>0$). 

This defines the action as claimed, the formal properties follow immediately. 
\end{proof}
Now let use the unobstructedness condition to define the bundles we
want to study. Choose a line bundle with log connection $(\mathcal{K},\nabla)$
on $\overline{L}$ with finite monodromy group, whose residue (when
restricted to $L$) is equal to the negative of the monodromy divisor
of $L$. Then we make the 
\begin{defn}
\label{def:M-Lambda-2}We define the $\theta$-regular connection
$\overline{\mathcal{M}}_{\lambda,2}$ on $V$ to be $(\mathcal{K},\nabla)\star\tilde{\mathcal{E}}_{\lambda,2}$. 
\end{defn}

The choice of $(\mathcal{K},\nabla)$ is arbitrary, up to twist by
a character of the fundamental group of $\overline{L}_{\mathbb{C}}$.
Thus we could equally well consider any of the bundles $(\mathcal{L},\nabla)\star\overline{\mathcal{M}}_{\lambda,2}$
where $(\mathcal{L},\nabla)$ has trivial monodromy along each divisor
in $L$. There is, as far as I know, no canonical choice of $\overline{\mathcal{M}}_{\lambda,2}$.

\subsubsection{\label{subsec:-regular-connections-over}$\theta$-regular connections
over $R[\lambda]/\lambda^{n}$}

In this subsection we impose the conditions $H^{0}(\Omega_{\overline{L}}^{1})=0=H^{1}(\mathcal{O}_{\overline{L}})$.
Under this condition we'll show the following
\begin{thm}
\label{thm:Infinitesimal-Def}Let $\overline{\mathcal{M}}_{\lambda,2}$
be as in \prettyref{def:M-Lambda-2}. Then for any $(\mathcal{L},\nabla)\in\text{Pic}^{\text{fin}}(L_{V},\tilde{E})$,
the $\theta$-regular $\lambda$-connection $(\mathcal{L},\nabla)\star\overline{\mathcal{M}}_{\lambda,2}$
admits, for each $n>2$, a unique lift to a $\theta$-regular $\lambda$-connection
over $R[\lambda]/\lambda^{n}$. 
\end{thm}

To prove this, we need to show that deformation theory works as expected
in this setting. Suppose we have found our unique $\overline{\mathcal{M}}_{\lambda,n}$
for some $n\geq2$ which deforms $\overline{\mathcal{M}}_{\lambda,2}$
(identical arguments will work for deformations of $(\mathcal{L},\nabla)\star\overline{\mathcal{M}}_{\lambda,2}$).
We say that $\overline{\mathcal{M}}_{\lambda,n}$ is \emph{locally
liftable }if, for each $x\in\overline{\tilde{X}}$, there is an etale
neighborhood $V_{x}$ of $x$ on which $\overline{\mathcal{M}}_{\lambda,n}$
lifts to a $\theta$-regular connection $\overline{\mathcal{M}}_{\lambda,n+1}$. 
\begin{lem}
The $\lambda$-connection $\overline{\mathcal{M}}_{\lambda,n}$ is
locally liftable. 
\end{lem}

\begin{proof}
First suppose the point $x\in U$. Choose any lift of the bundle $\overline{\mathcal{M}}_{\lambda,n}|_{U}$
to a bundle with (not necessarily flat) $\lambda$-connection over
$R[\lambda]/\lambda^{n+1}$. The curvature of this connection can
be regarded as an $\mathcal{O}_{U}$-linear map from $\overline{\mathcal{M}}_{\lambda,n}/\lambda^{n}$
to $\overline{\mathcal{M}}_{\lambda,n}/\lambda^{n}\otimes_{\mathcal{O}_{U}}\Omega_{U}^{2}$,
which does not depend on the choice of lift. To show that it vanishes,
we can reduce mod $p$. But there, we have that $\overline{\mathcal{M}}_{\lambda,n,k}$
must be isomorphic to $\overline{\mathcal{M}}_{\lambda}/\lambda^{n}$
for some $\theta$-regular bundle $\overline{\mathcal{M}}_{\lambda}$
on $\overline{\tilde{X}}_{k}$ (because this is true when $n=2$,
and by the induction assumption, deformations are unique). But this
$\lambda$-connection lifts to $\overline{\mathcal{M}}_{\lambda}/\lambda^{n+1}$,
which shows that the curvature vanishes. 

Now suppose $x\in\tilde{D}^{\text{sm}}$. Then we can employ a very
similar argument to the above, working in $V_{x}$, and locally lifting
the components in the decomposition 
\[
(\overline{\mathcal{M}}_{\lambda,n})_{V_{x}}=\bigoplus_{i}(\overline{\mathcal{M}}_{\lambda,n,i})_{V_{x}}
\]

Finally, suppose $x\in D$ is in the intersection of more than one
component. Pull back to $V_{x}^{(0,l,l,\dots,l)}$; and let $\overline{\mathcal{M}}'_{\lambda,n}$
denote a $\theta$-regular bundle there whose $G_{(0,l,l,\dots,l)}$-invariants
are $\overline{\mathcal{M}}_{\lambda,n}$. As above we can locally
lift $\overline{\mathcal{M}}'_{\lambda,n}$ in codimension $2$. Therefore,
the obstruction to lifting $\overline{\mathcal{M}}'_{\lambda,n}|_{V_{x}^{(0,l,l,\dots,l)}\backslash\{x\}}$
is a class in $H^{1}(\Omega_{V_{x}^{(0,l,l,\dots,l)}\backslash\{x\}}^{1})$.
(c.f. the proof of the lemma directly below). But, arguing as in \prettyref{lem:Adjust-the-extension},
we see that this class vanishes if it does so mod $p$ for $p>>0$,
which it must as a deformation exists there. Then, to see that this
lift (on $V_{x}^{(0,l,l,\dots,l)}\backslash\{x\}$) can be extended
to all of $V_{x}^{(0,l,l,\dots,l)}$, we act by a class in $H^{1}(\mathcal{O}_{L_{V_{x}}^{(0,l,\dots,l)}})$,
exactly as in the proof of \prettyref{lem:Adjust-the-extension}. 
\end{proof}
Now let's show that local liftability implies liftability: 
\begin{lem}
There is a lift of $\overline{\mathcal{M}}_{\lambda,n}$ to a $\theta$-regular
$\lambda$-connection over $R[\lambda]/\lambda^{n+1}$. 
\end{lem}

\begin{proof}
The previous lemma gives us local liftability. Choosing a collection
of local lifts on an etale cover $\{U_{i}\}$ of $\overline{\tilde{X}}$,
we know from \prettyref{lem:Local-Deformation-Statement} that their
differences (on $U_{ij}$) live in $\Omega_{\overline{L}_{U_{ij}}}^{1}$.
Thus the obstruction to making these local lifts locally isomorphic
is a class in $H^{1}(\Omega_{\overline{L}}^{1})$. But this group
is finite over $R$, so to test the vanishing we can test mod $p$,
where it clearly holds (because $\overline{\mathcal{M}}_{\lambda,n}$
does lift). Then, the obstruction to gluing these local isomorphisms
is a class in $H^{2}(\mathcal{O}_{\overline{L}})$. But this class
again vanishes by reduction mod $p$. 
\end{proof}
Now we can give the 
\begin{proof}
(of \prettyref{thm:Infinitesimal-Def}) We've seen the existence already.
For uniqueness, note that any two such lifts are locally isomorphic
because $H^{0}(\Omega_{\overline{L}}^{1})=0$, and therefore the set
of such lifts forms a torsor over $H^{1}(\mathcal{O}_{\overline{L}})=0$
as needed. 
\end{proof}

\subsection{\label{subsec:Lifting-to}Lifting to $W(k)$}

We continue with the assumption that $H^{0}(\Omega_{\overline{L}}^{1})=0=H^{1}(\mathcal{O}_{\overline{L}})$.
In the previous section we have constructed a family of bundles with
$\lambda$-connection over $R[[\lambda]]$
\[
(\mathcal{L},\nabla)\star\overline{\mathcal{M}}_{\widehat{\lambda}}:=\lim_{n}(\mathcal{L},\nabla)\star\overline{\mathcal{M}}_{\lambda,n}
\]
indexed by log connections with finite monodromy group. This bundle
has the property that it is the completion at $(\lambda)$ of a (necessarily
unique, by \prettyref{cor:At-most-one-N-lambda}) $\theta$-regular
connection $\overline{\mathcal{N}}_{\lambda,k}$. 

In this section we will ``algebrize'' this construction to make
a family of bundles with $\lambda$-connection over $W(k)[\lambda]$,
where $k$ is an algebraically closed field of suitably large characteristic.
The first main result reads 
\begin{thm}
\label{thm:Theta-reg-in-mixed}Let $R\to k$ where $\text{char}(k)=p$,
and choose a lift of this map to $R\to W(k)$. Then for each $m\geq1$
there is a bundle with $\lambda$-connection $(\mathcal{L},\nabla)\star\overline{\mathcal{M}}_{\lambda,W_{m}(k)}$
over $\overline{\tilde{X}}_{W_{m}(k)}\times\mathbb{A}_{W_{m}(k)}^{1}$
which is a flat deformation of $(\mathcal{L},\nabla)\star\overline{\mathcal{M}}_{\lambda,k}$,
and whose completion at $(\lambda)$ is isomorphic to $(\mathcal{L},\nabla)\star\overline{\mathcal{M}}_{\widehat{\lambda},W_{m}(k)}$. 
\end{thm}

The basic structure of the proof is as follows: first we define an
object in codimension $2$ in a a way similar to \prettyref{prop:Extension-in-codim-2}.
Then we use the deformation theory to modify it at points of codimension
two as needed (as in the proof of, e.g., \prettyref{lem:Adjust-the-extension}).
The starting point is the following 
\begin{lem}
\label{lem:Extend-over-W_m}Let $\{x\}$ be a closed point of $\tilde{D}_{W_{m}(k)}^{\text{sm}}$,
and let $V_{x}$ be an etale neighborhood of $x$; suppose that the
component $\overline{L}_{V_{x,i}}$ embeds into $T^{*}V_{x}$. Then,
after possibly shrinking $V_{x}$, the $\lambda$-connection $\pi_{*}(\mathcal{O}_{\overline{L}_{V_{x},i}}[\lambda],df)|_{V_{x,W_{m}(k)}\backslash\tilde{D}}$
admits an extension to a bundle with $\lambda$-connection on all
of $V_{x}$ (with no singularities). 
\end{lem}

The proof, which is built around the ideas of chapter 2, is somewhat
technical, and uses a ``forward reference'' to \prettyref{lem:(Technical)-We-must}.
To begin, consider the coherent sheaf with $\lambda$-connection ${\displaystyle \int_{\pi}(\mathcal{O}_{L_{V_{x},i}}[\lambda],df):}=\mathcal{F}_{\lambda}$.

Let $V_{x,i}\to\mathbb{A}^{m}$ be a closed embedding, and denote
by $L'$ the Lagrangian defined by $(d\iota^{*})^{-1}(L_{V_{x},i})$;
replace also $\mathcal{F}_{\lambda}$ with ${\displaystyle \int_{\iota}\mathcal{F}_{\lambda}}$.
Choose a point $y$ in the smooth part of $E'$, and let $\sigma:T^{*}\mathbb{A}^{m}\to T^{*}\mathbb{A}^{m}$
be a linear symplectomorphism so that $\sigma(L')$ projects isomorphically
onto $\mathbb{A}^{m}$ in a formal neighborhood of $\sigma(y)$. We
can (and will) assume $\sigma$ is chosen as a composition of simple
coordinate transpositions, as specified at the beginning of \prettyref{subsec:Analysis-of-M-sigma}.
We can apply the automorphism $\sigma$ to obtain $\mathcal{F}_{\lambda}^{\sigma}$
and then complete along $(\lambda)$ to obtain a micro-local sheaf
on $T^{*}\mathbb{A}^{m}$; we can then formally complete at $\sigma(y)$
to obtain a module $\widehat{\mathcal{F}}_{\widehat{\lambda},W}^{\sigma}$
which is a coherent $\widehat{\mathcal{D}}_{\widehat{\lambda}}$-module.
Here $\widehat{\mathcal{D}}_{\widehat{\lambda}}$ denotes the quantization
of a formal disc, i.e., the formal completion along $(x_{1},\dots x_{m},\partial_{1},\dots,\partial_{m},\lambda)$
of the $m$th Weyl algebra $D_{m,\lambda}$. 

Then we have 
\begin{lem}
Let $z_{1}$ denote a local coordinate of $\sigma(E')$. The module
$\widehat{\mathcal{F}}_{\widehat{\lambda},W}^{\sigma}[z_{1}^{-1}]$
is a line bundle over $R[[z_{1},\dots,z_{n},\lambda]][z_{1}^{-1}]$,
whose natural connection is meromorphic in $z_{1}$ (i.e., there is
a power $M$ so that $z_{1}^{M}\nabla$ preserves a generator $e$
of this line bundle). Further, there is an generator $e$ of the line
bundle so that 
\[
\nabla(e)=\alpha\lambda\frac{dz_{1}}{z_{1}}e+ue
\]
where $u$ is a one form with no poles, and $\alpha\in\mathbb{Q}$. 
\end{lem}

\begin{proof}
The fact that the connection is meromorphic is a consequence of \prettyref{lem:(Technical)-We-must}.
Once this is known, we can argue as we did in chapter 2. Namely, choose
a line bundle with connection, denoted $e^{g}$, over $\widehat{\mathcal{D}}_{\widehat{\lambda}}$,
whose reduction mod $p$ has $p$-curvature equal to $\sigma(L'_{k})[[\lambda]]$
for all $R\to k$; this is possible as $L'$ is an exact algebraic
Lagrangian, so we can apply the construction at the beginning of \prettyref{sec:The-Monodromy-Divisor}
and then micro-localize. Then the module $\widehat{\mathcal{F}}_{\widehat{\lambda},W}^{\sigma}[z_{1}^{-1}]\otimes e^{-g}$
is a line bundle over $R[[z_{1},\dots,z_{n},\lambda]][z_{1}^{-1}]$,
possessing a meromorphic connection, whose reduction mod $p$ has
$p$-curvature $0$ for all $p>>0$. 

Now look at the polar term $\Psi$ of the connection, computed in
the basis $e$. This is a finite sum of terms, whose $p$-curvature
is computed by the formula
\[
\Psi^{p}+\lambda^{p-1}\sum_{i=1}^{n}\partial_{i}^{p-1}(\Psi_{i})dz_{i}
\]
where $\Psi_{i}$ is defined via 
\[
\Psi=\sum_{i=1}^{n}\Psi_{i}dz_{i}
\]
and by $\Psi^{p}$ we mean ${\displaystyle \sum_{i=1}^{n}\Psi_{i}^{p}dz_{i}}$. 

Then, the condition that the $p$-curvature vanishes immediately implies
that the connection is is log with respect to $z_{1}$, and the coefficient
of $z_{1}^{-1}$ satisfies $a_{-1}^{p}=\lambda^{p-1}a_{-1}$ after
reduction mod $p$ for $p>>0$; thus the result follows (just as in
\cite{key-24}, section 13). 
\end{proof}
Now this yields the 
\begin{proof}
(of \prettyref{lem:Extend-over-W_m}) It is enough the prove the existence
of such an extension after embedding into $\mathbb{A}^{m}$ and applying
an automorphism $\sigma$. But then it follows from the previous result,
as the reduction of $\alpha\in\mathbb{Q}$ mod $p^{m}$ is an element
in $\mathbb{Z}/p^{m}$ for all $p>>0$, so we may rescale the generating
element $\{e\}$ by a power of $z$ to obtain a connection which extends
across $\{z_{1}=0\}$. 
\end{proof}
Now we turn to the 
\begin{proof}
(of \prettyref{thm:Theta-reg-in-mixed}) Since we have \prettyref{lem:Extend-over-W_m}
at our disposal, we can mimic the argument of \prettyref{prop:Extension-in-codim-2}
to obtain an extension $\overline{\mathcal{N}}_{\lambda,W_{m}(k)}$
of $\pi_{*}(\mathcal{O}_{\overline{L}}[\lambda],df)$ over an open
subset $V\subset\overline{\tilde{X}}_{W_{m}(k)}$ whose complement
had codimension $2$, and which has the required properties there. 

To extend to all of $\overline{\tilde{X}}_{W_{m}(k)}$, we use the
argument of \prettyref{lem:Adjust-the-extension}; namely, the failure
to extend $\overline{\mathcal{N}}_{\lambda,W_{m}(k)}$ to a $W_{m}(k)$-flat
sheaf which deforms $(\mathcal{L},\nabla)\star\overline{\mathcal{M}}_{\lambda,W_{m}(k)}$
and whose completion is $(\mathcal{L},\nabla)\star\overline{\mathcal{M}}_{\widehat{\lambda},W_{m}(k)}$,
is given by a collection of sections $[\mathfrak{o}]_{x}$ of $R^{1}j_{*}(\mathcal{E}nd_{\mathcal{O}_{V_{x}^{(l,0)}}}(\pi_{*}\mathcal{O}_{L_{V_{x},i}^{(l,0)},k})[\lambda])$
over points $\{x\}$ of codimension $2$ (the notation is as in \prettyref{lem:Adjust-the-extension}).
We wish to show that these sections are in the image of 
\[
\eta:R^{1}j_{*}(\mathcal{O}_{L_{V_{x,k}}^{(l,0)}}[\lambda])\to R^{1}j_{*}(\mathcal{E}nd_{\mathcal{O}_{V_{x}^{(l,0)}}}(\pi_{*}\mathcal{O}_{L_{V_{x},i}^{(l,0)},k})[\lambda])
\]
But, after passing to the completion at $(\lambda)$, the corresponding
sections of $R^{1}j_{*}(\mathcal{E}nd_{\mathcal{O}_{V_{x}^{(l,0)}}}(\pi_{*}\mathcal{O}_{L_{V_{x},i}^{(l,0)},k})[[\lambda]])$
are in the image of the map 
\[
\widehat{\eta}:R^{1}j_{*}(\mathcal{O}_{L_{V_{x,k}}^{(l,0)}}[[\lambda]])\to R^{1}j_{*}(\mathcal{E}nd_{\mathcal{O}_{V_{x}^{(l,0)}}}(\pi_{*}\mathcal{O}_{L_{V_{x},i}^{(l,0)},k})[[\lambda]])
\]
precisely because the completion of $\overline{\mathcal{N}}_{\widehat{\lambda},W_{m}(k)}$
can be readjusted at points of codimension $2$ to make $(\mathcal{L},\nabla)\star\overline{\mathcal{M}}_{\widehat{\lambda},W_{m}(k)}$;
therefore $[\mathfrak{o}]_{x}$ is in the image of $\eta$ as required. 
\end{proof}
Denote by $(\mathcal{L},\nabla)\star\widehat{\overline{\mathcal{M}}}_{\lambda,W(k)}$
the inverse limit of the sheaves $(\mathcal{L},\nabla)\star\overline{\mathcal{M}}_{\lambda,W_{n}(k)}$
constructed in the previous proof. A priori, this is a family of $\lambda$-connections
on the formal completion (along the ideal $(p)$) of the scheme $\overline{\tilde{X}}_{W(k)}\times\mathbb{A}_{W(k)}^{1}$.
However, we in fact have: 
\begin{thm}
\label{thm:Algebrization}There is a vector bundle with $\lambda$-connection
$(\mathcal{L},\nabla)\star\overline{\mathcal{M}}_{\lambda,W(k)}$
on $\overline{\tilde{X}}_{W(k)}\times\mathbb{A}_{W(k)}^{1}$, whose
$p$-adic completion is isomorphic to $\widehat{\overline{\mathcal{M}}}_{\lambda,W(k)}$. 

After inverting the prime $p$, we obtain a vector bundle with $\lambda$-connection
$(\mathcal{L},\nabla)\star\overline{\mathcal{M}}_{\lambda,K}$ on
$\overline{\tilde{X}}_{K}\times\mathbb{A}_{K}^{1}$. 
\end{thm}

\begin{proof}
For notational simplicity, we'll give the proof for $\overline{\mathcal{M}}_{\lambda,W(k)}$;
the identical argument works after acting by a line bundle. From Grothendieck's
existence theorem in formal geometry (c.f. \cite{key-37}), we must
show that each $\overline{\mathcal{M}}_{\lambda,W_{m}(k)}$ extends
to a coherent sheaf $(\overline{\mathcal{M}}_{\lambda,W_{m}(k)})'$
on $\overline{\tilde{X}}_{W_{m}(k)}\times\mathbb{P}_{W_{m}(k)}^{1}$,
so that 
\[
(\overline{\mathcal{M}}_{\lambda,W_{m}(k)})'/p^{n-1}\tilde{=}(\overline{\mathcal{M}}_{\lambda,W_{m-1}(k)})'
\]
and so that the $\lambda$-connection extends to an operator on $(\overline{\mathcal{M}}_{\lambda,W_{m}(k)})'$.
To see this, consider the $\lambda$-connection $\overline{\mathcal{E}}_{\lambda,W_{m}(k)}$.
After inverting $\lambda$, we can regard this as a $\lambda^{-1}$-connection,
which degenerates to a Higgs sheaf $\overline{\mathcal{M}}_{\infty,W_{m}(k)}$
as $\lambda^{-1}\to0$. When changing from $\lambda$-connections
to modules over $\mathcal{D}[\lambda,\lambda^{-1}]$, the $\lambda$-connection
$\nabla(e)=df\cdot e$ becomes $\nabla(e)=(\lambda^{-1}df)\cdot e$
(c.f. \prettyref{subsec:Notations-and-Conventions}); and it follows
that $\overline{\mathcal{M}}_{\infty,W_{m}(k)}$ is a torsion free
sheaf on $\overline{\tilde{X}}_{W_{m}(k)}$ which, in codimension
$2$, is equal to the Higgs sheaf $\pi_{*}(\mathcal{L})$, equipped
with the trivial Higgs field. We want to show that $\overline{\mathcal{M}}_{\infty,W_{m}(k)}$
is in fact a bundle; then the fact that its reduction to $W_{m-1}(k)$
is $\overline{\mathcal{M}}_{\infty,W_{m-1}(k)}$ follows directly. 

We start with the case of $m=1$. Let $\{x\}\in\overline{\tilde{X}}_{k}$
be a point which is in the intersection of exactly $m$ components
of the divisor $\tilde{D}_{k}$. Let $V_{x}$ be an etale neighborhood
of $\{x\}$, and call the coordinates of these components $\{z_{i}\}_{i=1}^{m}$.
We're going to prove that there is a bundle with $k[[\lambda^{-1}]]$-connection
on $V_{x}$, call it $\mathcal{F}_{\lambda^{-1}}$, which is locally
isomorphic to $\overline{\mathcal{M}}_{\widehat{\lambda^{-1}},k}$
in codimension $2$ on $V_{x}$; here, $\overline{\mathcal{M}}_{\widehat{\lambda^{-1}},k}$
denotes the $\lambda^{-1}$-adic completion of the extension of $\overline{\mathcal{M}}_{\lambda,k}$
to $\overline{\tilde{X}_{k}}\times\mathbb{P}_{k}^{1}$. In fact, this
is enough to prove the result; because we can then look at the sheaf
\[
\mathcal{H}om_{\nabla}(\mathcal{F}_{\lambda^{-1}},\overline{\mathcal{M}}_{\widehat{\lambda^{-1}},k})
\]
which is a reflexive coherent sheaf over the formal scheme $L_{V_{x}}^{(1)}[[\lambda^{-1}]]$,
which in codimension $2$ is locally isomorphic to $\mathcal{O}_{L_{V_{x}}^{(1)}}[[\lambda^{-1}]]$
(by \prettyref{lem:Endomorphisms-of-N-lambda}, applied to $\overline{\mathcal{M}}_{\widehat{\lambda^{-1}},k}$).
But any reflexive coherent sheaf on a regular (formal) scheme which
is generically a line bundle is a line bundle; so there is an everywhere
non-vanishing section of this sheaf which is necessarily an isomorphism
between $\mathcal{F}_{\lambda^{-1}}$ and $\overline{\mathcal{M}}_{\widehat{\lambda^{-1}},k}$. 

Let us now construct $\mathcal{F}_{\lambda^{-1}}$. We construct a
bundle with meromorphic connection $\mathcal{F}_{\lambda^{-1},n}=\mathcal{F}_{\lambda^{-1}}/(\lambda^{-1})^{n}$
for each $n\geq1$ by induction on $n$; the induction assumption
being that $\mathcal{F}_{\lambda^{-1},n}$ and $\overline{\mathcal{M}}_{\widehat{\lambda^{-1}},k}/\lambda^{n}$
are isomorphic on $\{V_{i}\}$, where the union of the $\{V_{i}\}$
is a set whose complement has codimension $2$. When $n=1$ we set
$\mathcal{F}_{\lambda^{-1},1}=\pi_{*}(\mathcal{L})$, a bundle with
trivial Higgs field. Supposing $\mathcal{F}_{\lambda^{-1},n}$ has
been constructed, let $\mathcal{H}_{n+1}$ be any lift of $\mathcal{F}_{\lambda^{-1},n}$
to a meromorphic connection, whose poles along each $\{z_{i}\}$ have
order bounded by the order of the poles of $\overline{\mathcal{M}}_{\widehat{\lambda^{-1}},k}$.
The set of isomorphism classes of such deformations is a torsor over
\[
\mathcal{E}nd(\pi_{*}(\mathcal{L}))\otimes\Omega_{V_{X}}^{1}(n_{1}D_{1}+\dots+n_{m}D_{m})
\]
where the bundle $\Omega_{V_{X}}^{1}(n_{1}D_{1}+\dots+n_{m}D_{m})$
denotes one-forms with poles in $\tilde{D}$ of order along $D_{i}$
bounded by $n_{i}$ (we do not have to take a cohomology group because
the Higgs field is trivial). 

We can therefore consider $[\mathcal{H}_{n+1}|_{V_{i}}]-[\overline{\mathcal{M}}_{\widehat{\lambda^{-1}},k}/\lambda^{n+1}|_{V_{i}}]$
as a section of the above bundle on $V_{i}$, and these sections agree
on the overlaps $V_{i}\cap V_{j}$. Since the union of $V_{i}$ has
complement of codimension $2$, we obtain a section on all of $V_{x}$,
and modifying the bundle $\mathcal{H}_{n+1}$ by this section yields
$\mathcal{F}_{\lambda^{-1},n+1}$. 

Now suppose $m>1$. Proceed by induction on $m$. The sheaf $\overline{\mathcal{M}}_{\widehat{\lambda^{-1}},W_{m}(k)}$
deforms the bundle $\overline{\mathcal{M}}_{\widehat{\lambda^{-1}},W_{m-1}(k)}$
in codimension $2$, and therefore, looking at local cohomology, yields
the section of a sheaf supported on points of codimension $2$. But
this section vanishes after inverting $\lambda^{-1}$, because $\overline{\mathcal{M}}_{\widehat{\lambda^{-1}},W_{m}(k)}[\lambda]$
is a vector bundle. So the class vanishes already, and $\overline{\mathcal{M}}_{\widehat{\lambda^{-1}},W_{m}(k)}$
is a bundle as required. 
\end{proof}

\subsection{\label{subsec:Arithmetic-Support}Arithmetic Support}

Now we devote ourselves to understanding the arithmetic support of
the objects $(\mathcal{L},\nabla)\star\overline{\mathcal{M}}_{\lambda,K}$
which we have constructed in the previous theorem. As this object
is an algebraic $\lambda$-connection over a field of characteristic
$0$, we can spread out and define $(\mathcal{L},\nabla)\star\overline{\mathcal{M}}_{\lambda,F}$
over a field $F$ which is finitely generated over $\mathbb{Q}$,
and then over the ring $R$ (after possibly extending $R$). Denote
the corresponding object $(\mathcal{L},\nabla)\star\overline{\mathcal{M}}_{\lambda}$.
Our goal in this section is to show this object has constant $p$-support
for $p>>0$. To that end, let $k'$ be an algebraically closed field
of positive characteristic $l$, with $R\to k'$. 

Since $R$ is assumed smooth over $\mathbb{Z}$, we can choose for
each $m\geq1$ a morphism $R\to W_{m}(k')$ which lifts the chosen
morphism $R\to W_{m-1}(k')$. If $R/l\to k'$ is injective, then the
map $R\to W(k')$ is flat. 
\begin{thm}
\label{thm:Uniqueness-of-lambda-conns}1) With notation as above,
the base change $(\mathcal{L},\nabla)\star\overline{\mathcal{M}}_{\lambda,k'}$
is $\theta$-regular at infinity, and satisfies $(\mathcal{L},\nabla)\star\overline{\mathcal{M}}_{\lambda,k'}/\lambda\tilde{=}\mathcal{L}\star\overline{\mathcal{M}}_{k'}$.
In particular, the $\lambda$-connection $(\mathcal{L},\nabla)\star\overline{\mathcal{M}}_{\lambda,k'}$
is the unique $\theta$-regular $\lambda$-connection over $k'$ whose
reduction mod $\lambda$ is $\mathcal{L}\star\overline{\mathcal{M}}_{k'}$. 

2) Let $F$ be a field of characteristic zero, suppose that $(\mathcal{L},\nabla)\star\overline{\mathcal{M}}_{\lambda,F}$
and $(\mathcal{L},\nabla)\star\overline{\mathcal{M}}'_{\lambda,F}$
are two vector bundles with $\lambda$-connection, both of which arise
from the construction of \prettyref{thm:Algebrization}, for possibly
different primes. Then, after possibly extending $F$, we have 
\[
(\mathcal{L},\nabla)\star\overline{\mathcal{M}}_{1,F}|_{U_{F}}\tilde{=}(\mathcal{L},\nabla)\star\overline{\mathcal{M}'}_{1,F}|_{U_{F}}
\]
as algebraic connections on $U_{F}$. In particular, all of the connections
$(\mathcal{L},\nabla)\star\overline{\mathcal{M}}_{1,F}|_{U_{F}}$
have constant arithmetic support, equal to $L_{U_{F}}$, of multiplicity
$1$. 
\end{thm}

As in the previous section, for notational convenience we just prove
the result for $\overline{\mathcal{M}}_{\lambda,F}$, the case of
$(\mathcal{L},\nabla)\star\overline{\mathcal{M}}_{\lambda,F}$ being
identical. We start with 
\begin{lem}
\label{lem:completion-is-regular}For each $m\geq1$, the sheaf $\overline{\mathcal{M}}_{\lambda}/\lambda^{m}$
is $\theta$-regular at infinity over $R$; when $m=1$ we have $\overline{\mathcal{M}}_{\lambda}/\lambda\tilde{=}\overline{\mathcal{M}}$.
Therefore, for any $k'$ as above, the completion $\overline{\mathcal{M}}_{\widehat{\lambda},k'}$
(along $(\lambda)$) agrees with the completion along $(\lambda)$
of the $\lambda$-connection constructed in \prettyref{prop:Extension-in-codim-2}
\end{lem}

\begin{proof}
By construction we have $\overline{\mathcal{M}}_{\lambda}/\lambda\tilde{=}\overline{\mathcal{M}}$,
and for each $m\geq1$ we have that $\overline{\mathcal{M}}_{\lambda,E}/\lambda^{m}$
is $\theta$-regular at infinity, for some field extension $E$ of
$F=\text{Frac}(R)$ (we can take $E$ to be any field which contains
$\text{Frac}(W(k))$ for the $k$ which we used to construct $\overline{\mathcal{M}}_{\lambda,W(k)}$).
We shall show that this implies $\overline{\mathcal{M}}_{\lambda}/\lambda^{m}$
is $\theta$-regular at infinity. 

First, we note that is is enough to prove this over an open subset
$V$ whose complement has codimension $2$. For, if we suppose (by
induction) that $\overline{\mathcal{M}}_{\lambda}/\lambda^{m-1}$
is $\theta$-regular at infinity, then the failure of $\overline{\mathcal{M}}_{\lambda}/\lambda^{m}$
to extended to a sheaf on all of $\overline{\tilde{X}}$ is given
by classes in the cohomology group $H^{0}(R^{1}j_{*}(\mathcal{O}_{\overline{L}_{V_{X,i}}}))$
(where the notation is as in \prettyref{lem:Adjust-the-extension}).
These classes vanish after base changing to a field extension, which
implies that they already vanish (for instance, by the explicit description
of the cohomology groups given in \prettyref{eq:explicit-groups}). 

Now we must show that $\overline{\mathcal{M}}_{\lambda}/\lambda^{m}|_{V}$
is $\theta$-regular for any $m\geq2$. To see this, we first note
that, since $\mathcal{M}_{\lambda}/\lambda^{m}$ is a $\mathcal{D}_{\lambda}$-module
on $U$, by micro-localization we can consider it as sheaf on $T^{*}U$.
As $\mathcal{M}$ is supported along $L\subset T^{*}U$, so is $\mathcal{M}_{\lambda}/\lambda^{m}$.
Let $x\in\tilde{D}$ be a point which is contained in a single component
of $\tilde{D}$. Let $\varphi:V_{x}\to\overline{\tilde{X}}$ be an
etale neighborhood as in \prettyref{cor:Local-Struc-of-Closure};
and let $z$ be a local coordinate for $D$ in $V_{x}$. Consider
the sheaf $\varphi^{*}(\overline{\mathcal{M}}_{\lambda}/\lambda^{m})$.
By the condition on the support, we have 
\[
\varphi^{*}(\overline{\mathcal{M}}_{\lambda}/\lambda^{m})[z^{-1}]=\bigoplus_{i}(\overline{\mathcal{M}}_{\lambda}/\lambda^{m})_{i}[z^{-1}]
\]
where $(\overline{\mathcal{M}}_{\lambda}/\lambda^{m})_{i}[z^{-1}]$
is the subsheaf of elements supported on $L_{V_{x},i}\backslash\tilde{D}$.
So if $m\in\varphi^{*}(\overline{\mathcal{M}}_{\lambda}/\lambda^{m})$
is any section, it admits a unique representation $m=\sum_{i}m_{i}$
where $m_{i}\in(\overline{\mathcal{M}}_{\lambda}/\lambda^{m})_{i}[z^{-1}]$.
On the other hand, we have 
\[
\varphi^{*}(\overline{\mathcal{M}}_{\lambda,E}/\lambda^{m})=\bigoplus_{i}(\overline{\mathcal{M}}_{\lambda,E}/\lambda^{m})_{i}
\]
since $\overline{\mathcal{M}}_{\lambda,E}/\lambda^{m}$ is $\theta$-regular,
and $(\overline{\mathcal{M}}_{\lambda,E}/\lambda^{m})_{i}[z^{-1}]$
is supported along $L_{V_{x},i,E}\backslash\tilde{D}_{E}$. So each
$m_{i}\in\varphi^{*}(\overline{\mathcal{M}}_{\lambda,E}/\lambda^{m})\cap\varphi^{*}(\overline{\mathcal{M}}_{\lambda}/\lambda^{m})[z^{-1}]=\varphi^{*}(\overline{\mathcal{M}}_{\lambda}/\lambda^{m})$
(the last equality since $\overline{\mathcal{M}}_{\lambda}/\lambda^{m}$
is a bundle near $x$). Thus we see that there is a decomposition
\[
\varphi^{*}(\overline{\mathcal{M}}_{\lambda}/\lambda^{m})=\bigoplus_{i}(\overline{\mathcal{M}}_{\lambda}/\lambda^{m})_{i}
\]
according to the support. 

Now we must check that each summand satisfies the conditions for $\theta$-regularity.
As always there are two cases to consider. In type I, we require that
$(\overline{\mathcal{M}}_{\lambda}/\lambda^{m})_{i}$, after possibly
adding a term of the form $A_{\alpha\frac{dz}{z}}$, for some $\alpha\in\mathbb{Q}$,
is a flat connection with no singularities. But this follows immediately
from the corresponding condition for $(\overline{\mathcal{M}}_{\lambda,E}/\lambda^{m})_{i}$.
In type II, we take the pullback of $\pi^{*}(\overline{\mathcal{M}}_{\lambda}/\lambda^{m})_{i}$
under the cyclic cover $\pi:V_{x}^{(l)}\to V_{x}$. In this case we
need to show that $\pi^{*}\overline{\mathcal{M}}_{\lambda,V_{x},i}$
is spanned by a $G_{l}$-invariant set of eigenvectors for the connection.
Any section $m\in\pi^{*}(\overline{\mathcal{M}}_{\lambda}/\lambda^{m})_{i}[z^{-1}]$
has a unique representation $m=\sum_{j}m_{j}$ as a sum of eigenvectors
for the connection; this follows by looking at the support. Arguing
as above, using the fact that $(\overline{\mathcal{M}}_{\lambda,E}/\lambda^{m})_{i}$
is $\theta$-regular at infinity, we see that each $m_{j}\in\pi^{*}(\overline{\mathcal{M}}_{\lambda}/\lambda^{m})_{i}$
as required. 
\end{proof}
This lemma gives us control over the behavior of $\overline{\mathcal{M}}_{\widehat{\lambda},k'}$.
We're going to use the geometry of the Hilbert scheme to obtain similar
control over $\overline{\mathcal{M}}_{\lambda,k'}$. We recall that,
by projectivizing the fibers of $T^{*}\overline{\tilde{X}}_{k'}\to\overline{\tilde{X}}_{k'}$,
the variety $T^{*}\overline{\tilde{X}}_{k'}$ admits a smooth, projective
compactification which we will denote $\overline{T^{*}(\tilde{X}_{k'})}$. 

We consider the reduced closure $(\overline{L_{k'}^{cl}})^{(1)}$
of $(L_{k'}^{cl})^{(1)}$ inside $\overline{T^{*}(\tilde{X}_{k'})}^{(1)}$.
Then we have the following 
\begin{lem}
\label{lem:Flatness-in-closure}Let $\mathcal{I}\subset O(T^{*}(\overline{\tilde{X}}_{k'})^{(1)}\times\mathbb{A}_{k}^{1})$
be the ideal sheaf $\mathcal{I}=\text{Ann}(j_{*}(\mathcal{M}_{\lambda,k'}|_{U_{k}}))$.
Let $\overline{\mathcal{I}}\subset O(\overline{T^{*}(\tilde{X}_{k'})}^{(1)}\times\mathbb{A}^{1})$
be the ideal sheaf of the scheme theoretic closure of $\mathcal{I}$
(i.e., $f\in\overline{\mathcal{I}}$ if the restriction of $f$ to
$T^{*}(\overline{\tilde{X}}_{k'})^{(1)}\times\mathbb{A}_{k}^{1}$
is in $\mathcal{I}$). Then $O(\overline{T^{*}(\tilde{X}_{k'})}^{(1)}\times\mathbb{A}^{1})/\overline{\mathcal{I}}$
is a flat deformation (over $\mathbb{A}_{k'}^{1}$) of $O(\overline{L_{k'}^{cl}})^{(1)}$. 
\end{lem}

\begin{proof}
This follows easily from the fact that $\mathcal{M}_{\lambda,k'}|_{U_{k}}$
is flat over $k'[\lambda]$. 
\end{proof}
We need one more piece of information to proceed. We recall some notions
from \cite{key-36}, chapter 1.2: let $Y_{k'}\subset Z_{k'}$ be a
closed immersion of schemes over the field $k'$. Suppose that we
have a subscheme $\tilde{Y}\subset Z_{k'}\times k'[\epsilon]/\epsilon^{n}$
which specializes to $Y_{k}$ at $\epsilon=0$. Then an infinitesimal
deformation of $\tilde{Y}$ is a $k'[\epsilon]/\epsilon^{n+1}$-flat
subscheme $\tilde{Y}'\subset Z_{k'}\times k'[\epsilon]/\epsilon^{n+1}$
which specializes to $\tilde{Y}$. The set of such deformations is
isomorphic to $H^{0}(\mathcal{N}_{Y_{k'}}):=H^{0}(\mathcal{H}om_{Y_{k'}}(\mathcal{I}/\mathcal{I}^{2},O_{Y_{k'}}))$
where $Y_{k'}=\tilde{Y}\times_{k'[\epsilon]/\epsilon^{n}}k'$ (this
is \cite{key-36}, Theorem 2.4).

This implies that infinitesimal deformations are parametrized by a
coherent sheaf which is torsion free over $Y_{k'}$. Therefore, if
$Y_{k'}$ is an integral scheme and a deformation is trivial at the
generic point of $Y_{k'}$, then it is trivial on all of $Y_{k'}$.
This we conclude: 
\begin{cor}
Consider the $\lambda$-adic completion of the sheaf $O(\overline{T^{*}(\tilde{X}_{k'})}^{(1)}\times\mathbb{A}^{1})/\overline{\mathcal{I}}$
(of \prettyref{lem:Flatness-in-closure}). Then this completion is
trivial (over $k'[[\lambda]]$) as a deformation of $(\overline{L_{k'}^{cl}})^{(1)}$. 
\end{cor}

\begin{proof}
Since $\overline{\mathcal{M}}_{\widehat{\lambda},k'}$ is $\theta$-regular
at infinity, it is the $\lambda$-adic completion of a unique $\overline{\mathcal{M}}'_{\lambda,k'}$
which is strongly $\theta$-regular at infinity. Therefore, the restriction
of this sheaf to $T^{*}U_{k'}^{(1)}[[\lambda]]$ is scheme-theoretically
supported along $L_{U_{k}}^{(1)}[[\lambda]]$. So the $\lambda$-completion
of $O(\overline{T^{*}(\tilde{X}_{k'})}^{(1)}\times\mathbb{A}^{1})/\overline{\mathcal{I}}$
becomes trivial (as a deformation of $L_{U_{k}}^{(1)}$) when restricted
to $T^{*}U_{k}^{(1)}$. So the remarks directly above imply the result. 
\end{proof}
Now, recall the Hilbert scheme 
\[
\mathcal{H}ilb_{\overline{L_{k'}^{cl}}}
\]
which is the scheme representing the functor which assigns to any
$k'$-scheme $T_{k'}$ the set of subschemes of $\overline{T^{*}(\tilde{X}_{k'})}\times T_{k'}$,
flat over $T_{k'}$, whose Hilbert polynomial is equal to that of
$\overline{L_{k'}^{cl}}\times T_{k'}$. (c.f., e.g., \cite{key-34}
for a very complete introduction). We have, by taking support, that
the bundle $\overline{\mathcal{M}}_{\lambda,k'}$ defines a morphism
$\mathbb{A}_{k'}^{1}\to\mathcal{H}ilb_{(\overline{L_{k'}^{cl}})^{(1)}}$
(i.e., the $p$-support is flat over $k'[\lambda]$). The completion
of this map at $\{0\}$ is necessarily the trivial map by the corollary.
Since $\mathcal{H}ilb_{(\overline{L_{k'}^{cl}}){}^{(1)}}$ is quasi-projective,
we see that the morphism $\mathbb{A}_{k'}^{1}\to\mathcal{H}ilb_{(\overline{L_{k'}^{cl}})^{(1)}}$
is the trivial morphism; i.e., it maps every element in $\mathbb{A}_{k'}^{1}$
to $(\overline{L_{k'}^{cl}}){}^{(1)}$. Therefore the (closure in
$\overline{T^{*}(\tilde{X}_{k'})}\times\mathbb{A}_{k'}^{1}$ of the)
$p$-support of this bundle is necessarily $(\overline{L_{k'}^{cl}})\times\mathbb{A}_{k}^{1}$. 

We now have the
\begin{prop}
\label{prop:Meromorphics-with-trivial-support}1) We have $\mathcal{M}_{\lambda,k'}\tilde{=}\pi_{*}(\mathcal{L}_{\lambda,k'})$
where $\mathcal{L}_{\lambda,k'}$ is a line bundle with $\lambda$-connection
on $L_{U_{k'}}$ whose $p$-curvature is equal to $\Gamma(df)^{(1)}\times\mathbb{A}_{k'}^{1}$. 

2) The bundle $\overline{\mathcal{M}}_{\lambda,k'}$ is strongly $\theta$-regular
at infinity.
\end{prop}

\begin{proof}
1) This follows immediately from \prettyref{lem:All-bundles-are-pi-push}
and the remarks directly above. 

2) This is similar to \prettyref{lem:completion-is-regular}. It suffices
to check the condition in codimension $2$ since $\overline{\mathcal{M}}_{\lambda,k'}$
is reflexive. Let $\{x\}$ be the generic point of a component of
$\tilde{D}_{k'}$, and let $\varphi:V_{x}\to\overline{\tilde{X}}_{k}$
be an etale neighborhood of $\{x\}$, so that 
\[
\varphi^{*}\overline{\mathcal{M}}_{k'}=\bigoplus_{i}\overline{\mathcal{M}}_{V_{x},i}
\]
as in \prettyref{cor:Local-Struc-of-Closure}. By $1)$, the action
of the $p$-curvature gives a decomposition 
\[
\varphi^{*}\overline{\mathcal{M}}_{\lambda,k'}[z^{-1}]=\bigoplus_{i}\overline{\mathcal{M}}_{\lambda,V_{x},i}[z^{-1}]
\]
(where $z$ is a local coordinate for $\varphi^{-1}(\tilde{D})$).
Let $m\in\varphi^{*}\overline{\mathcal{M}}_{\lambda,k'}$. Then we
have a unique representation $m=\sum_{i}m_{i}$ where $m_{i}\in\overline{\mathcal{M}}_{\lambda,V_{x},i}[z^{-1}]$.
On the other hand, the bundle $\varphi^{*}\overline{\mathcal{M}}_{\widehat{\lambda},k'}$
is $\theta$-regular at infinity; so that there is a decomposition
\[
\varphi^{*}\overline{\mathcal{M}}_{\widehat{\lambda},k'}=\bigoplus_{i}\overline{\mathcal{M}}_{\widehat{\lambda},V_{x},i}
\]
according to the action of $p$-curvature. So, regarding $m$ as a
section of $\varphi^{*}\overline{\mathcal{M}}_{\widehat{\lambda},k'}$,
we see that each $m_{i}\in\varphi^{*}\overline{\mathcal{M}}_{\widehat{\lambda},k'}$.
But $\varphi^{*}\overline{\mathcal{M}}_{\widehat{\lambda},k'}\cap\varphi^{*}\overline{\mathcal{M}}_{\lambda,k'}[z^{-1}]=\varphi^{*}\overline{\mathcal{M}}_{\lambda,k'}$,
so we see that each $m_{i}\in\varphi^{*}\overline{\mathcal{M}}_{\lambda,k'}$.
So we have a direct sum decomposition 
\[
\varphi^{*}\overline{\mathcal{M}}_{\lambda,k'}=\bigoplus_{i}\overline{\mathcal{M}}_{\lambda,V_{x},i}
\]
which lifts the analogous one for $\varphi^{*}\overline{\mathcal{M}}_{k'}$. 

Now we must check that each summand satisfies the conditions for $\theta$-regularity.
By looking at the $l$-curvature (or, the reduction mod $(\lambda)$)
we see that each summand is attached to a unique component $\overline{L}_{V_{x},i}$
of $\overline{L}_{V_{x}}$. As always there are two cases to consider:
first, if we are in type I. In this case, we require that, up to a
action by $z^{\alpha}$, $\overline{\mathcal{M}}_{\lambda,V_{x},i}$,
is a flat connection with no singularities. Now since $\overline{\mathcal{M}}_{\widehat{\lambda},V_{x},i}$
has no singularities, we see that $\overline{\mathcal{M}}_{(\lambda),V_{x},i}$
(the localization of $\overline{\mathcal{M}}_{\lambda,V_{x},i}$ at
$(\lambda)$ ) also has no singularities.

Now, we move on to type II, and complete the proof below. In this
case, we take the pullback of $\pi^{*}\overline{\mathcal{M}}_{\lambda,V_{x},i}$
under the cyclic cover $\pi:V_{x}^{(r_{i})}\to V_{x}$. In this case
we need to show that $\pi^{*}\overline{\mathcal{M}}_{\lambda,V_{x},i}$
is spanned by a $G_{r'}$-invariant set of eigenvectors for the connection.
Any section $m\in\pi^{*}\overline{\mathcal{M}}_{\lambda,V_{x},i}[z^{-1}]$
has a unique representation $m=\sum_{j}m_{j}$ as a sum of eigenvectors
for the connection (it is induced by the action of $l$-curvature).
Arguing as above, using the fact that $\overline{\mathcal{M}}_{\widehat{\lambda},V_{x},i}$
is $\theta$-regular at infinity, we see that each $m_{j}\in\pi^{*}\overline{\mathcal{M}}_{\lambda,V_{x},i}$
as required. 

So, applying this over all components of $\tilde{D}_{k'}$, we see
that there is an open subset $W\subset\overline{\tilde{X}}_{k'}\times\mathbb{A}_{k'}^{1}$
so that $\overline{\mathcal{M}}_{\lambda,k'}|_{W}$ is $\theta$-regular
at infinity (as explained in remark \prettyref{rem:Theta-reg-over-W}).
Clearly $U_{k'}\times\mathbb{A}_{k'}^{1}\subset W$, and, by what
we have already seen, $W\cap\overline{\tilde{X}}_{k'}$ contains an
open subset whose complement has codimension $2$. Therefore $W$
has codimension $2$ in $\overline{\tilde{X}}_{k'}\times\mathbb{A}_{k'}^{1}$. 

Now, applying the proof of \prettyref{thm:Pic-over-k} to strongly
$\theta$-regular $\lambda$-connections over $W$, with a given residue
mod $\lambda^{2}$, we see that the set of such connections is parametrized
by $\text{Pic}((\pi^{(1)})^{-1}(W))$ where $\pi^{(1)}:\overline{L}_{k'}^{(1)}\times\mathbb{A}_{k'}^{1}\to\overline{\tilde{X}}_{k'}\times\mathbb{A}_{k'}^{1}$
is the natural map. As this group is isomorphic to $\text{Pic}(\overline{L}_{k'}^{(1)}\times\mathbb{A}_{k'}^{1})=\text{Pic}(\overline{L}_{k'}^{(1)})$,
it follows that each $\theta$-regular $\lambda$-connection on $W$
extends uniquely to one on all of $\overline{\tilde{X}}_{k'}\times\mathbb{A}_{k'}^{1}$,
and the result follows. 
\end{proof}
Now we give the 
\begin{proof}
(of \prettyref{thm:Uniqueness-of-lambda-conns}) We have just proved
1). To prove $2)$, we consider the meromorphic $\lambda$-connection
$(\overline{\mathcal{M}}_{\lambda,F})\otimes(\overline{\mathcal{M}}'_{\lambda,F})^{*}$.
We can spread it out over $R$. By what we have just shown, there
is an isomorphism $\overline{\mathcal{M}}_{\lambda,k}\tilde{\to}\overline{\mathcal{M}}'_{\lambda,k}$
for all $k$ of large enough characteristic. Therefore the module
$\mathbb{H}_{dR}^{0}((\overline{\mathcal{M}}_{\lambda,F})\otimes(\overline{\mathcal{M}}'_{\lambda,F})^{*})$
must be nonzero; i.e., there is a nonzero map of meromorphic $\lambda$-connections
$\overline{\mathcal{M}}_{\lambda,F}\to\overline{\mathcal{M}}'_{\lambda,F}$. 

Now set $\lambda=1$. In this case, we have that $\mathcal{M}_{1,F}$
and $\mathcal{M}'_{1,F}$ are both irreducible connections on $U_{F}$-
indeed, we have seen both have an irreducible arithmetic support of
multiplicity $1$, which is equal to $L_{U_{F}}$ (which is finite
flat over $U_{F}$). Therefore the irreducibility follows from Bitoun's
theorem (\cite{key-7}, theorem 2.21), and the fact that the module
is a bundle over $U_{F}$ follows from the fact that it is so after
reduction mod $p$ for $p>>0$. Thus any nonzero morphism between
these connections is an isomorphism. So there is a map $\overline{\mathcal{M}}_{1,F}\to\overline{\mathcal{M}}'_{1,F}$
which becomes an isomorphism over $U_{F}$ as claimed. 
\end{proof}

\section{\label{sec:Some-P-adic-Microlocal}Some P-adic Micro-local Theory}

In this chapter, we will establish some general results on ``micro-local
analysis'' in the $p$-adic case. There are many things one might
mean by this, but our intention (serving the interests of the paper)
is to study the deformations (over $W_{m}(k)$) of some $D_{X_{k}}$-module
$\mathcal{M}$, which is a splitting bundle for a smooth Lagrangian
$L_{k}^{(1)}\subset T^{*}X_{k}^{(1)}$ (to avoid tricky issues with
symplectic forms, we'll assume in this section that $\text{char}(k)>2$).
Assuming that $X_{k}$ admits a flat lift to a smooth formal scheme
$\mathfrak{X}_{W(k)}$, what can be said about lifts of $\mathcal{M}_{k}$?
Looking at the ``canonical case'' of $L_{k}^{(1)}=X_{k}^{(1)}$,
$\mathcal{M}_{k}=\mathcal{O}_{X_{k}}$ and lifts to $W_{2}(k)$, we
have a given flat lift, namely $\mathcal{O}_{W_{2}(k)}$. For any
other flat lift, say $\mathcal{N}_{W_{2}(k)}$, we shall see below\footnote{Under the assumption that $\mathcal{N}_{W_{2}(k)}$ it itself the
reduction mod $p^{2}$ of a further lift} in \prettyref{def:p^m-curv}that there is a natural invariant (the
$p^{2}$-curvature) which measures the failure of $\mathcal{N}_{W_{2}(k)}$
to be locally isomorphic to $\mathcal{O}_{W_{2}(k)}$. As with the
usual $p$-curvature, it is given by a suitable differential form. 

We shall then go on to see that there is a similar structure in the
general case. Namely, if we fix a lift $\mathcal{M}_{W(k)}$ of $\mathcal{M}_{k}$,
then any other lift over $W_{m}(k)$ can be compared to $\mathcal{M}_{W_{m}(k)}$
via a version of the $p^{m}$-curvature; which vanishes iff the two
lifts are locally isomorphic; c.f. \prettyref{prop:p^m-curv-over-L}.
The set of lifts which are locally isomorphic to $\mathcal{M}_{W_{m}(k)}$
form a torsor over the group of lifts which are locally isomorphic
to $\mathcal{O}_{W_{m}(k)}$. The key to transferring from the case
of the trivial connection to the general case is the use of automorphisms
of differential operators (defined in the local case) to move any
Lagrangian to one which (locally) projects isomorphically to $X_{k}\subset T^{*}X_{k}$.
We develop the requisite theory in \prettyref{subsec:Local-Structure-of}
directly below. 

\subsection{\label{subsec:Local-Structure-of}Local Structure of Differential
operators}

We begin by working formally locally around a point in $L_{k}^{(1)}$.
Our aim is to show that, after suitably completing and applying an
automorphism, all modules supported along $L_{k}^{(1)}$ (and their
lifts) look trivial. In fact, we go a little bit further in classifying
such $\mathcal{D}$-modules in directly below; making heavy use of
the results of this section. 

We are going to work with the formal scheme $\mathfrak{X}_{W(K)}$,
as well as $T^{*}\mathfrak{X}{}_{W(k)}$, the formal cotangent bundle\footnote{As opposed to the algebraic cotangent bundle, given by taking the
relative spec of the sheaf of continuous derivations}- in other works, the direct limit of the schemes $T^{*}X_{W_{m}(k)}$.
We shall be working with differential operators on $\mathfrak{X}_{W(k)}$.
Following Berthelot's conventions, the usual (PD differential operators)
will be denoted $\mathcal{D}_{\mathfrak{X}_{W_{m}(k)}}^{(0)}$, instead
of $\mathcal{D}_{\mathfrak{X}_{W_{m}(k)}}$. This is because, in this
chapter, we will also be making use of the higher level differential
operators $\mathcal{D}_{\mathfrak{X}_{W_{m}(k)}}^{(i)}$ for $i\geq0$,
and we will need to distinguish them. 
\begin{defn}
\label{def:D-complete}We let ${\displaystyle \widehat{\mathcal{D}}_{\mathfrak{X}{}_{W(k)}}^{(0)}:=\lim_{m}\mathcal{D}_{X_{W_{m}(k)}}^{(0)}}$
(this is following \cite{key-61}). 
\end{defn}

To set things up, let $x$ be any $k$-point of $T^{*}X_{k}$. By
the infinitesimal lifting property, the point $x$ lifts to a $W(k)$-point
$\mathbf{x}$ of $T^{*}\mathfrak{X}{}_{W(k)}$; we shall regard $\mathbf{x}$
as a compatible collection of $W_{m}(k)$ points of $T^{*}\mathfrak{X}{}_{W_{m}(k)}$
for each $m$. 

More concretely, we let $\{x_{1},\dots,x_{n}\}$ be local coordinates
at $\pi(x)\in X_{k}$, and $\{x_{1},\dots,x_{n},\xi_{1},\dots\xi_{n}\}$
be local coordinates at $x\in T^{*}X_{k}$. As the construction is
local, we may shrink $X_{k}$ and assume that the inclusion
\[
k[x_{1},\dots,x_{n},\xi_{1},\dots,\xi_{n}]\subset\mathcal{O}_{T^{*}X_{k}}
\]
is etale. We may further suppose that these coordinates are chosen
so that, under the natural Poisson bracket, we have $\{x_{i},x_{j}\}=0=\{\xi_{i},\xi_{j}\}$
and $\{x_{i},\xi_{j}\}=\delta_{ij}$ for all $i$ and $j$. The $k$-point
$x$ yields an ideal of the form $\mathfrak{m}_{x}=\{x_{i}-a_{i},\xi_{j}-b_{j}\}_{1\leq i,j\leq n}$,
where $\{a_{i}\}$ and $\{b_{i}\}$ are in $k$. Then we may lift
these local coordinates to local coordinates on the formal scheme
$T^{*}\mathfrak{X}_{W(k)}$, which we will also denote by $\{x_{1},\dots,x_{n},\xi_{1},\dots\xi_{n}\}$.
The choice of point $\mathbf{x}$ corresponds to a choice of $2n$-tuple
$(A_{1},\dots,A_{n},B_{1},\dots,B_{n})$ in $W(k)$ which lifts $(a_{1},\dots,a_{n},b_{1},\dots,b_{n})$. 
\begin{defn}
\label{def:p-adic-microlocal-completion}We let $\mathfrak{n}_{x}\subset\mathcal{D}_{X_{W_{m}(k)}}^{(0)}$
be the two-sided ideal generated by the central elements $\{(x_{1}-\bar{A}_{1})^{p^{m}},\dots,(x_{n}-\bar{A}_{n})^{p^{m}}\}$
and $\{(\partial_{1}-\bar{B}_{1})^{p^{m}},\dots,(\partial_{n}-\bar{B}_{n})^{p^{m}}\}$
(here $\bar{A}_{i}$ and $\bar{B}_{j}$ denote the image of $A_{i}$
and $B_{j}$ in $W_{m}(k)$)\footnote{This ideal depends only on $x\in T^{*}X$, and not on the choice of
lift $\mathbf{x}$. This is because, if $a,b$ are any elements of
a commutative $W_{m}(k)$-algebra $R$, then $a\equiv b\phantom{i}\text{mod}\phantom{i}p$
implies $a^{p^{m}}=b^{p^{m}}$ }. Let $\mathcal{D}_{X_{W_{m}(k),}\widehat{x}}^{(0)}$ denote the completion
of $\mathcal{D}_{X_{W_{m}(k)}}^{(0)}$ along $\mathfrak{n}_{x}$.
Define
\[
\widehat{\mathcal{D}}_{\mathfrak{X}_{W(k)},\widehat{x}}^{(0)}:=\lim_{m}\widehat{\mathcal{D}}_{X_{W_{m}(k)},x}^{(0)}
\]
For each $m\geq1$ we have an algebra morphism $\mathcal{D}_{X_{W_{m}(k)}}^{(0)}\to\mathcal{D}_{X_{W_{m}(k),}\widehat{x}}^{(0)}$,
which in the limit yields $\widehat{\mathcal{D}}_{\mathfrak{X}_{W(k)}}^{(0)}\to\widehat{\mathcal{D}}_{\mathfrak{X}_{W(k)},\widehat{x}}^{(0)}$. 
\end{defn}

Here we have used the (easy) fact that the algebra\linebreak{}
 $\mathfrak{Z}:=W_{m}(k)[(x_{i}-\bar{A}_{i})^{p^{m}},(\partial_{j}-\bar{B}_{j})^{p^{m}}]_{1\leq i,j\leq n}$
is central inside $\mathcal{D}_{X_{W_{m}(k)}}^{(0)}$. As $\mathcal{D}_{X_{W_{m}(k)}}^{(0)}$
is finite flat over $\mathfrak{Z}$, we see that
\[
\mathcal{D}_{X_{W_{m}(k),}\widehat{x}}^{(0)}\tilde{=}\mathcal{D}_{X_{W_{m}(k)}}^{(0)}\otimes_{\mathfrak{Z}}\widehat{\mathfrak{Z}}
\]
where $\widehat{\mathfrak{Z}}$ is the completion of $\mathfrak{Z}$
along $((x_{i}-\bar{A}_{i})^{p^{m}},(\partial_{j}-\bar{B}_{j})^{p^{m}})$.
It follows that each $\mathcal{D}_{X_{W_{m}(k),}\widehat{x}}^{(0)}$
is flat over $W_{m}(k)$, and that natural map $\mathcal{D}_{X_{W_{m}(k),}\widehat{x}}^{(0)}\to\mathcal{D}_{X_{W_{m-1}(k),}\widehat{x}}^{(0)}$
is onto for each $m$ (as $\widehat{\mathfrak{Z}}$ can also be described
as the completion of $\mathfrak{Z}$ along $((x_{i}-\bar{A}_{i})^{p^{m'}},(\partial_{j}-\bar{B}_{j})^{p^{m'}})$
where $m'\geq m$ is any integer). Thus $\widehat{\mathcal{D}}_{\mathfrak{X}_{W(k)},\widehat{x}}^{(0)}$
is a $W(k)$-flat $p$-adically complete algebra such that 
\[
\widehat{\mathcal{D}}_{\mathfrak{X}_{W(k)},\widehat{x}}^{(0)}/p\tilde{=}\mathcal{D}_{X_{k},\widehat{x}}^{(0)}
\]

In fact, it is not difficult to describe this algebra as a more traditional
completion of $\widehat{\mathcal{D}}_{\mathfrak{X}_{W(k)}}^{(0)}$:
\begin{lem}
\label{lem:Desciption-of-D-hat}The elements
\[
\{(x_{1}-A_{1})^{p},\dots,(x_{n}-A_{n})^{p},(\partial_{1}-B_{1})^{p},\dots,(\partial_{n}-B_{n})^{p},p\}
\]
generate a proper two-sided ideal $\mathcal{I}\subset\widehat{\mathcal{D}}_{\mathfrak{X}_{W(k)}}^{(0)}$,
and there is an isomorphism of the completion of $\widehat{\mathcal{D}}_{\mathfrak{X}_{W(k)}}^{(0)}$
along $\mathcal{I}$ with $\widehat{\mathcal{D}}_{\mathfrak{X}_{W(k)},\widehat{x}}^{(0)}$.
In particular, the algebra $\widehat{\mathcal{D}}_{\mathfrak{X}_{W(k)},\widehat{x}}^{(0)}$
is flat over $\widehat{\mathcal{D}}_{\mathfrak{X}_{W(k)}}^{(0)}$,
and so the functor $\mathcal{N}\to\widehat{\mathcal{N}}_{x}:=\mathcal{N}\otimes_{\widehat{\mathcal{D}}_{\mathfrak{X}_{W(k)}}^{(0)}}\widehat{\mathcal{D}}_{\mathfrak{X}_{W(k)},\widehat{x}}^{(0)}$
is exact. 
\end{lem}

\begin{proof}
Without loss of generality we can take all $A_{i}$ and $B_{i}$ to
be $0$. To obtain the first statement, we must show that any element
in the left ideal generated by $\{x_{1}^{p},\dots,x_{n}^{p},\partial_{1}^{p},\dots,\partial_{n}^{p},p\}$
is also in the right ideal generated by these elements. This, in turn,
follows directly from the fact that for any $\Phi\in\widehat{\mathcal{D}}_{\mathfrak{X}_{W(k)}}^{(0)}$,
we have $[\Phi,x_{i}^{p}]=p\Phi'$ for some $\Phi'\in\widehat{\mathcal{D}}_{\mathfrak{X}_{W(k)}}^{(0)}$
(and similarly for $\partial_{i}^{p}$ instead of $x_{i}^{p}$). 

Now, let $(\widehat{\mathcal{D}}_{\mathfrak{X}_{W(k)}}^{(0)})_{\mathcal{I}}^{\widehat{}}$
denote the completion of $\widehat{\mathcal{D}}_{\mathfrak{X}_{W(k)}}^{(0)}$
along $\mathcal{I}$. For each $m\geq1$ one obtains a map 
\[
(\widehat{\mathcal{D}}_{\mathfrak{X}_{W(k)}}^{(0)})_{\mathcal{I}}^{\widehat{}}/p^{m}\to(\widehat{\mathcal{D}_{X_{W_{m}(k)}}^{(0)})}_{(\mathcal{I}/p^{m})}
\]
where the latter algebra denotes the completion of $\mathcal{D}_{X_{W_{m}(k)}}^{(0)}$
along $\mathcal{I}/p^{m}$; from the fact that $p$ is nilpotent in
$\mathcal{D}_{X_{W_{m}(k)}}^{(0)}$ and the fact that $\mathfrak{n}_{x}$
is generated by powers of elements in $\mathcal{I}/p^{m}$, one sees
directly that this algebra is isomorphic to $\mathcal{D}_{X_{W_{m}(k),}\widehat{x}}^{(0)}$.
Since $(p)\subset\mathcal{I}$, we have that $(\widehat{\mathcal{D}}_{\mathfrak{X}_{W(k)}}^{(0)})_{\mathcal{I}}^{\widehat{}}$
is $p$-adically complete, and so we can take the inverse limit to
obtain a map 
\[
(\widehat{\mathcal{D}}_{\mathfrak{X}_{W(k)}}^{(0)})_{\mathcal{I}}^{\widehat{}}\to\widehat{\mathcal{D}}_{\mathfrak{X}_{W(k)},\widehat{x}}^{(0)}
\]
of $p$-adically complete algebras. The reduction mod $(p)$ of this
map an isomorphism (both sides are identified with $\mathcal{D}_{X_{k},\widehat{x}}^{(0)}$),
so this map is surjective by the complete Nakayama lemma. Thus we
obtain an exact sequence 
\[
0\to\mathcal{K}\to(\widehat{\mathcal{D}}_{\mathfrak{X}_{W(k)}}^{(0)})_{\mathcal{I}}^{\widehat{}}\to\widehat{\mathcal{D}}_{\mathfrak{X}_{W(k)},\widehat{x}}^{(0)}\to0
\]
and since $\widehat{\mathcal{D}}_{\mathfrak{X}_{W(k)},\widehat{x}}^{(0)}$
is clearly $p$-torsion free we obtain that $\mathcal{K}/p=0$; since
$\mathcal{K}$ is an ideal in the noetherian, $p$-adically complete
ring $(\widehat{\mathcal{D}}_{\mathfrak{X}_{W(k)}}^{(0)})_{\mathcal{I}}^{\widehat{}}$,
$\mathcal{K}$ is also $p$-adically complete and so $\mathcal{K}=0$
as desired. 
\end{proof}
This allows us to prove
\begin{lem}
\label{lem:Complete-flatness}Let $\mathcal{N}$ be a finite $\widehat{\mathcal{D}}_{\mathfrak{X}_{W(k)}}^{(0)}$-module.
Then there is an isomorphism 
\[
\widehat{\mathcal{N}}\tilde{=}\widehat{\mathcal{D}}_{\mathfrak{X}_{W(k)},\widehat{x}}^{(0)}\otimes_{\widehat{\mathcal{D}}_{\mathfrak{X}_{W(k)}}^{(0)}}\mathcal{N}\tilde{=}\lim_{m}(\mathcal{D}_{X_{W_{m}(k)},\widehat{x}}^{(0)}\otimes_{\mathcal{D}_{X_{W_{m}(k)}}^{(0)}}\mathcal{N}/p^{m}\mathcal{N})
\]
where $\widehat{\mathcal{N}}$ denotes the completion of $\mathcal{N}$
along $\mathcal{I}$.
\end{lem}

\begin{proof}
The first isomorphism is standard (c.f., e.g., \cite{key-79} proposition
10.13). For the second; note that for each $m\geq1$ there is an isomorphism
\[
(\widehat{\mathcal{D}}_{\mathfrak{X}_{W(k)},\widehat{x}}^{(0)}\otimes_{\widehat{\mathcal{D}}_{\mathfrak{X}_{W(k)}}^{(0)}}\mathcal{N})/p^{m}\tilde{\to}\mathcal{D}_{X_{W_{m}(k)},\widehat{x}}^{(0)}\otimes_{\mathcal{D}_{X_{W_{m}(k)}}}\mathcal{N}/p^{m}\mathcal{N}
\]
\[
\tilde{\to}(\widehat{\mathcal{D}_{X_{W_{m}(k)}}^{(0)})}_{(\mathcal{I}/p^{m})}\otimes_{\mathcal{D}_{X_{W_{m}(k)}}^{(0)}}\mathcal{N}/p^{m}\mathcal{N}\tilde{\to}(\widehat{\mathcal{N}/p^{m}\mathcal{N}})_{(\mathcal{I}/p^{m})}
\]
where the latter denotes the completion of $\mathcal{N}/p^{m}\mathcal{N}$
along $\mathcal{I}/p^{m}$. Further, since 
\[
(\mathcal{N}/\mathcal{I}^{j})/p^{m}(\mathcal{N}/\mathcal{I}^{j})\tilde{=}(\mathcal{N}/p^{m}\mathcal{N})/(\mathcal{I}/p^{m})^{j}
\]
for all $j\geq m$, we also have
\[
(\widehat{\mathcal{N}/p^{m}\mathcal{N}})_{(\mathcal{I}/p^{m})}\tilde{=}\widehat{\mathcal{N}}/p^{m}\widehat{\mathcal{N}}
\]
Since $\widehat{\mathcal{N}}$ is $p$-adically complete, the result
follows. 
\end{proof}
Essentially by definition, the algebra $\widehat{\mathcal{D}}_{\mathfrak{X}_{W(k)},\widehat{x}}^{(0)}$
depends only on the formal local neighborhood of $x$:
\begin{lem}
\label{lem:little-iso}There is an isomorphism $\widehat{\mathcal{D}}_{\mathfrak{X}_{W(k)},\widehat{x}}^{(0)}\tilde{=}\widehat{\mathcal{D}}_{\mathfrak{A}_{W(k),}\widehat{0}}^{(0)}$
where $\mathfrak{A}_{W(k)}$ denotes the $p$-adic completion of $\mathbb{A}_{W(k)}^{n}$. 
\end{lem}

\begin{proof}
For each $m\geq1$ we have an automorphism of $\widehat{\mathcal{D}}_{\mathfrak{X}_{W(k)},\widehat{x}}^{(0)}$
which fixes $\mathcal{O}_{X_{W_{m}(k)}}$ and takes $\partial_{i}\to\partial_{i}-\overline{B}_{i}$.
This automorphism yields an isomorphism 
\[
\mathcal{D}_{X_{W_{m}(k),}\widehat{x}}^{(0)}\tilde{=}\mathcal{D}_{X_{W_{m}(k),}\widehat{y}}^{(0)}
\]
where $y$ is the image of $x$ under $T^{*}X_{k}\to X_{k}\subset T^{*}X_{k}$.
Taking the limit of these automorphisms yields 
\[
\widehat{\mathcal{D}}_{\mathfrak{X}_{W(k)},\widehat{x}}^{(0)}\tilde{=}\widehat{\mathcal{D}}_{\mathfrak{X}_{W(k)},\widehat{y}}^{(0)}
\]
Now let $\mathbf{y}=\pi(\mathbf{x})\in\mathfrak{X}_{k}$. There is
an etale morphism $\mathfrak{X}_{W(k)}\to\mathfrak{A}_{W(k)}$ which
sends the $W(k)$-point $\mathbf{y}$ to $\mathbf{0}$. Such a morphism
induces a map 
\[
\mathcal{\widehat{D}}_{\mathfrak{A}_{W(k)}}^{(0)}\to\mathcal{\widehat{D}}_{\mathfrak{X}_{W(k)}}^{(0)}
\]
which is easily seen to induce the required isomorphism. 
\end{proof}
Now we can analyze what happens to a module under a suitable choice
of automorphism. Let $\mathcal{N}_{k}$ denote a $\mathcal{D}_{X_{k}}^{(0)}$-module,
which is a splitting bundle for $\mathcal{D}_{X_{k}}^{(0)}|_{L_{k}^{(1)}}$
near some given point $x\in L_{k}^{(1)}$. Suppose $\mathcal{N}_{W(k)}$
is a $p$-adically complete, $p$-torsion-free lift of $\mathcal{N}_{k}$
to a $\mathcal{D}_{\mathfrak{X}_{W(k)}}^{(0)}$-module. Then:
\begin{prop}
\label{prop:Iso-All-Levels!}There is an isomorphism $\Psi:\mathcal{D}_{X_{k},\widehat{x}}^{(0)}\tilde{\to}\mathcal{D}_{\mathbb{A}_{k}^{n},\widehat{0}}^{(0)}$
under which $\widehat{\mathcal{N}}_{k}=\mathcal{D}_{X_{k},\widehat{x}}^{(0)}\otimes_{\mathcal{D}_{X_{k}}}\mathcal{N}_{k}$
corresponds to the standard module $\widehat{\mathcal{O}}_{\mathbb{A}_{k}^{n},0}$
over $\mathcal{D}_{\mathbb{A}_{k}^{n},\widehat{0}}^{(0)}$. There
is a lift of this isomorphism to an isomorphism $\Psi':\widehat{\mathcal{D}}_{\mathfrak{X}_{W(k)},\widehat{x}}^{(0)}\tilde{\to}\widehat{\mathcal{D}}_{\mathfrak{A}_{W(k),}\widehat{0}}^{(0)}$
under which $\mathcal{\widehat{N}}_{W(k)}$ corresponds to $\widehat{\mathcal{O}}_{\mathfrak{A}_{W(k),}0}$.
The analogous statement holds over $W_{m}(k)$ if we replace $\mathcal{N}_{W(k)}$
with a flat lift $\mathcal{N}_{W_{m}(k)}$ over $W_{m}(k)$.
\end{prop}

\begin{proof}
By the previous lemma, there is an isomorphism $\mathcal{D}_{X_{k},\widehat{x}}^{(0)}\tilde{\to}\mathcal{D}_{\mathbb{A}_{k}^{n},\widehat{0}}^{(0)}$
which can be lifted to $\widehat{\mathcal{D}}_{\mathfrak{X}_{W(k)},\widehat{x}}^{(0)}\tilde{\to}\widehat{\mathcal{D}}_{\mathfrak{A}_{W(k),}\widehat{0}}^{(0)}$.
By \cite{key-3}, section 5.2, one sees that the induced map on centers,
$Z\Phi:\widehat{\mathcal{O}}_{T^{*}X_{k}^{(1)},x}\tilde{\to}\widehat{\mathcal{O}}_{T^{*}\mathbb{A}_{k}^{n,(1)},0}$
preserves the natural symplectic structures on both sides . Thus the
formal subscheme $\widehat{L}_{k,x}^{(1)}\subset(\widehat{T^{*}X_{k}})_{x}^{(1)}$
corresponds, under $Z\Phi$, to a smooth formal Lagrangian subscheme
of $\widehat{\mathcal{O}}_{T^{*}\mathbb{A}_{k}^{n,}{}^{(1)},0}$,
which we shall also call $\widehat{L}_{k,x}^{(1)}$. 

We now identify $\mathbb{A}_{k}^{2n}$ with the tangent space to $0$
in $T^{*}\mathbb{A}_{k}^{n}$; it carries a natural symplectic form.
Let $\mathfrak{l}_{k}\subset\mathbb{A}_{k}^{2n}$ be the tangent space
to $\widehat{L}_{k,x}^{(1)}$. We may apply an element $\tilde{\sigma}$
of the linear symplectic group $\text{Sp}(\mathbb{A}_{k}^{2n})$ to
interchange $\mathfrak{l}_{k}$ with $T_{0}(\mathbb{A}_{k}^{n})$,
the tangent space to $\mathbb{A}_{k}^{n}\subset T^{*}\mathbb{A}_{k}^{n}$.
Further, by completing its standard action on $\mathcal{D}_{\mathbb{A}_{k}^{n}}^{(0)}$
, $\text{Sp}(\mathbb{A}_{k}^{2n})$ also acts on $\mathcal{D}_{\mathbb{A}_{k}^{n},\widehat{0}}^{(0)}$.
The induced action on the center is given via the $p$-th power isomorphism
$\text{Sp}(\mathbb{A}_{k}^{2n})\tilde{=}\text{Sp}(\mathbb{A}_{k}^{2n})^{(1)}$.
Therefore we obtain an isomorphism $\Phi:\mathcal{D}_{X_{k,}\widehat{x}}^{(0)}\to\mathcal{D}_{\mathbb{A}_{k}^{n},\widehat{0}}^{(0)}$,
so that $\Phi^{*}(\widehat{\mathcal{M}}_{k})$ is module, over $\mathcal{D}_{\mathbb{A}_{k}^{n},\widehat{0}}^{(0)}$
whose $p$-support, as a subscheme of $(\widehat{T^{*}\mathbb{A}_{k}^{n}})_{0}^{(1)}$,
has tangent space equal to $T_{0}^{*}(\mathbb{A}^{n})^{(1)}$. 

Let $\widehat{S}_{k}^{(1)}$ denote the $p$-support of $\Phi^{*}(\widehat{\mathcal{N}}_{k})$;
so that the ideal of $\widehat{S}_{k}^{(1)}$ corresponds, under $Z\Phi$,
to the ideal of $\widehat{L}_{k}^{(1)}$. Then the induced projection
map
\[
\widehat{S}_{k}^{(1)}\to(\widehat{\mathbb{A}_{k}^{n}})_{0}^{(1)}
\]
is an isomorphism. Therefore the module $\widehat{\mathcal{N}}_{k}$,
when regarded as a module over the center $\widehat{\mathcal{O}}_{(T^{*}\mathbb{A}_{k}^{n})^{(1)},0}$,
is a free module of rank $p^{n}$ over $\widehat{\mathcal{O}}_{(\mathbb{A}_{k}^{n}){}^{(1)},0}\subset\widehat{\mathcal{O}}_{(T^{*}\mathbb{A}^{n})^{(1)},0}$.
Since $\widehat{\mathcal{O}}_{\mathbb{A}_{k}^{n},0}$ is finite free
of rank $p^{n}$ over $\widehat{\mathcal{O}}_{(\mathbb{A}_{k}^{n}){}^{(1)},0}$.
we see\footnote{It is clearly reflexive by \cite{key-49}, proposition 1.6; since
the generic rank is $1$ it is free} that $\widehat{\mathcal{N}}_{k}$ is a free of rank $1$ over $\widehat{\mathcal{O}}_{\mathbb{A}_{k}^{n},0}$.
Thus we may write $\widehat{\mathcal{N}}_{k}=\widehat{\mathcal{O}}_{\mathbb{A}_{k}^{n},0}\cdot\mathbf{e}$
for some element $\mathbf{e}$. We have 
\begin{equation}
\partial_{i}\mathbf{e}=\bar{g}_{i}\mathbf{e}\label{eq:rank-1-connection}
\end{equation}
for $1\leq i\leq n$. Computing the $p$-th power yields 
\[
\partial_{i}^{p}\mathbf{e}=(\bar{g}_{i}^{p}+\partial_{i}^{p-1}(\bar{g}_{i}))\mathbf{e}
\]
so that $\widehat{S}_{k}^{(1)}$ is defined by the equations $\{\partial_{i}^{p}-\bar{g}_{i}^{p}-\partial_{i}^{p-1}(\bar{g}_{i})\}{}_{1\leq i\leq n}$.
This implies that each function $\bar{g}_{i}^{p}+\partial_{i}^{p-1}(\bar{g}_{i})$
is contained in the maximal ideal of $\widehat{\mathcal{O}}_{\mathbb{A}_{k}^{n},0}$. 

Choosing a lift of $\tilde{\sigma}$ to a linear symplectomorphism
of $T^{*}\mathbb{A}_{W(k)}^{n}$, we see that we can choose an isomorphism
$\widehat{\mathcal{D}}_{\mathfrak{X}_{W(k)},\widehat{x}}^{(0)}\tilde{\to}\widehat{\mathcal{D}}_{\mathfrak{A}_{W(k),}\widehat{0}}^{(0)}$
under which $\widehat{\mathcal{N}}_{W(k)}$ corresponds to a free
rank $1$ module $\widehat{\mathcal{O}}_{\mathbb{A}_{W(k)}^{n},0}\cdot\mathbf{e}$;
write 
\[
\partial_{i}\mathbf{e}=g_{i}\mathbf{e}
\]
for $1\leq i\leq n$. We now claim that there is an automorphism $\sigma$
of $\widehat{\mathcal{D}}_{\mathbb{A}_{W(k)}^{n},0}^{(0)}$ which
preserves $\widehat{\mathcal{O}}_{\mathbb{A}_{W(k)}^{n},0}$ and satisfies
\[
\partial_{i}\to\partial_{i}-g_{i}
\]

Indeed, we have 
\[
(\partial_{i}-g_{i})^{p}=\partial_{i}^{p}-g_{i}^{p}-\partial_{i}^{p-1}(g_{i})+p\Phi
\]
for some $\Phi\in\widehat{\mathcal{D}}_{\mathfrak{A}_{W(k),}\widehat{0}}^{(0)}$;
and the above discussion implies $g_{i}^{p}+\partial_{i}^{p-1}(g_{i})=\tilde{g}_{i}+p\alpha$
where $\tilde{g}_{i}$ is contained in the ideal of $\widehat{\mathcal{O}}_{\mathbb{A}_{W(k)}^{n},0}$
generated by $\{x_{1}^{p},\dots,x_{n}^{p}\}$, and $\alpha\in\widehat{\mathcal{O}}_{\mathbb{A}_{W(k)}^{n},0}$
is some function. In particular, we see that $(\partial_{i}-g_{i})^{p}$
is topologically nilpotent in $\widehat{\mathcal{D}}_{\mathfrak{X}_{W(k)},\widehat{x}}^{(0)}\tilde{\to}\widehat{\mathcal{D}}_{\mathfrak{A}_{W(k),}\widehat{0}}^{(0)}$.
Therefore the map which sends $x_{i}\to x_{i}$ and $\partial_{i}\to\partial_{i}-g_{i}$
does extend to an automorphism of $\widehat{\mathcal{D}}_{\mathfrak{X}_{W(k)},\widehat{x}}^{(0)}\tilde{\to}\widehat{\mathcal{D}}_{\mathfrak{A}_{W(k),}\widehat{0}}^{(0)}$
as claimed; but then composing with this automorphism proves the result.
The case where $W(k)$ is replaced by $W_{m}(k)$ is essentially identical. 
\end{proof}
Here is a typical corollary. To state it, recall from \cite{key-80}
that the center of $\mathcal{D}_{X_{W_{m}(k)}}^{(0)}$ is isomorphic
to the ring of Witt vectors on the $m$th Frobenius twist of the cotangent
bundle of $X_{k}$; $W_{m}(T^{*}X_{k}^{(m)})$. 
\begin{cor}
\label{cor:Support-of-N-W_m(k)}Let $m\ge1$, and let $\mathcal{N}_{W_{m}(k)}$
be a $W_{m}(k)$-flat $\mathcal{D}_{X_{W_{m}(k)}}^{(0)}$-module such
that $\mathcal{N}_{W_{m}(k)}/p$ is is a splitting bundle for $\mathcal{D}_{X_{k}}^{(0)}|_{L_{k}^{(1)}}$.
Then, under the action of $\mathcal{Z}(\mathcal{D}_{X_{W_{m}(k)}})\tilde{=}W_{n}(\mathcal{O}_{T^{*}X_{k}^{(m)}})$,
$\mathcal{N}_{W_{m}(k)}$ is scheme-theoretically supported along
$W_{m}(L_{k}^{(m)})$. 
\end{cor}

\begin{proof}
As this can be checked formally locally, we can apply the previous
proposition and assume $L_{k}=X_{k}\subset T^{*}X_{k}$. Choose local
coordinates $\{x_{1},\dots,x_{n},\partial_{1},\dots\partial_{n}\}$
on $X_{k}$. Then the proof comes down to checking that the annihilator
of any flat deformation (over $W_{m}(k)$) of $\mathcal{O}_{X_{k}}$
is exactly the ideal generated by 
\[
\{\partial_{1}^{p^{m+1}},\dots,\partial_{n}^{p^{m+1}},p\partial_{1}^{p^{m}},\dots,p\partial_{n}^{p^{m}},\dots,p^{m}\partial_{1}^{p},\dots,p^{m}\partial_{n}^{p}\}
\]
But this this a straightforward computation.
\end{proof}

\subsection{\label{subsec:Locally-trivial-connections}Locally trivial connections
in mixed characteristic}

In this section, we discuss the theory of the $p^{m}$-curvature alluded
to above. Let $k$ be a perfect field of positive characteristic,
and let $X_{k}$ be a smooth $k$-scheme. Let $\mathfrak{X}$ be a
flat lift to $W(k)$. Let $\mathcal{\widehat{D}}_{\mathfrak{X}}^{(i)}$
denote Berthelot's ring of differential operators of level $i$ on
$\mathfrak{X}$ (c.f. \cite{key-61}); recall that this is a $W(k)$-flat
$p$-adically complete sheaf of algebras; we denote $\mathcal{\widehat{D}}_{\mathfrak{X}}^{(i)}/p^{n}:=\mathcal{D}_{X_{W_{n}(k)}}^{(i)}$.
If $X$ is affine, then we have also the algebra $D_{\mathfrak{X}}^{(i)}$,
the algebra of finite order differential operators generated by operators
of order $\leq p^{i}$. It is not $p$-adically complete and does
not glue to a sheaf on $\mathfrak{X}$, however, if $\mathfrak{X}$
possesses local coordinates, then we do have $D_{\mathfrak{X}}^{(i)}/p^{n}\tilde{=}\Gamma(\mathcal{D}_{X_{W_{n}(k)}}^{(i)})$. 

We recall (\cite{key-57}, proposition 3.6) that the center $\mathcal{Z}(\mathcal{D}_{X_{k}}^{(i)})$
is isomorphic to $\mathcal{O}_{T^{*}X_{k}^{(i+1)}}$. Furthermore,
by \cite{key-57}, theorem 3.7, $\mathcal{D}_{X_{k}}^{(i)}$ is Azumaya
over its center. Let $\mathcal{I}\subset\mathcal{Z}(\mathcal{D}_{X_{k}}^{(i)})$
denote the ideal sheaf of $X_{k}^{(i+1)}\subset T^{*}X_{k}^{(i+1)}$.
If $\overline{\mathcal{D}_{X_{k}}^{(i)}}$ denotes the reduction by
the central ideal generated by $\mathcal{I}$, then this is a split
Azumaya algebra (the splitting bundle is $\mathcal{O}_{X}$); and
so we see that a $\mathcal{D}_{X_{k}}^{(i)}$-module is locally trivial
(i.e., locally isomorphic to a finite direct sum of copies of the
$\mathcal{D}_{X_{k}}^{(i)}$-module $\mathcal{O}_{X}$) iff it is
a vector bundle, equipped with an action of $\mathcal{D}_{X_{k}}^{(i)}$,
for which the ideal $\mathcal{I}$ acts trivially. 
\begin{thm}
\label{thm:Higher-diff-structure-on-deformations}Let $(\mathcal{E},\nabla)$
be a vector bundle with flat connection on $\mathfrak{X}$; set $\mathcal{E}_{n}:=\mathcal{E}/p^{n}$.
Suppose that $(\mathcal{E}_{n-1},\nabla):=(\mathcal{E}_{n}/p^{n},\nabla)$
is a locally trivial connection. Then for each $j\in\{1,\dots,n\}$
the bundle $\mathcal{E}_{j}:=\mathcal{E}/p^{j}$ inherits an action
of $\mathcal{D}_{X_{W_{j}(k)}}^{(n-j)}$. The connection $\mathcal{E}_{n}$
is locally trivial iff $\mathcal{E}_{j}$ is locally trivial over
$\mathcal{D}_{X_{W_{j}(k)}}^{(n-j)}$ for each $j$; in turn, this
holds iff $\mathcal{E}_{j}$ is locally trivial over $\mathcal{D}_{X_{W_{j}(k)}}^{(n-j)}$
for some $j\in\{1,\dots,n\}$. 

Setting $j=1$, we see that $\mathcal{E}_{n}$ is locally trivial
iff $\mathcal{E}_{1}$ is locally trivial over $\mathcal{D}_{X_{k}}^{(n-1)}$;
by the above remarks, this holds iff the ideal $\mathcal{I}\subset\mathcal{O}_{T^{*}X_{k}^{(n+1)}}$
acts trivially on $\mathcal{E}_{1}$. 
\end{thm}

\begin{proof}
When $n=1$ this is trivial, so we may assume $n>1$. As all statements
are local, we may suppose that we have coordinate derivations $\{\partial_{1},\dots,\partial_{d}\}$
on $X_{n}$, and that $\mathcal{E}_{n}$ is trivial as a vector bundle.

As the $\mathcal{D}_{X_{k}}^{(0)}$-module $\mathcal{E}_{1}$ has
$p$-curvature $0$, we see that $\partial_{i}^{p}(\mathcal{E})\subset p\cdot\mathcal{E}$.
Thus the operators $\partial_{i}^{[p]}={\displaystyle \frac{\partial_{i}^{p}}{p!}}$
act also on $\mathcal{E}$, and so we get an action of $D_{\mathfrak{X}}^{(1)}$
on $\mathcal{E}$. Thus $\mathcal{E}_{n-1}$ inherits an action of
$\mathcal{D}_{X_{W_{n-1}(k)}}^{(1)}$. This action is independent
of the choice of coordinates, and so exists when $X_{k}$ is not affine. 

By induction, since $\mathcal{E}_{2}$ is a locally trivial connection,
we may suppose that the $\mathcal{D}_{X_{k}}^{(1)}$-module $\mathcal{E}_{n-1}/p^{n-1}\tilde{=}\mathcal{E}_{1}$
is locally trivial. Therefore we have $(\partial_{i}^{[p]})^{p}=0$
on $\mathcal{E}_{1}$. As ${\displaystyle (\frac{\partial_{i}^{p}}{(p!)})^{p}=\frac{\partial_{i}^{p^{2}}}{(p!)^{p}}}$
we see ${\displaystyle \frac{\partial_{i}^{p^{2}}}{(p!)^{p}}}$ preserves
$\mathcal{E}$; therefore, as above, we can obtain an action of ${\displaystyle \frac{1}{p}\frac{\partial_{i}^{p^{2}}}{(p!)^{p}}}$
(and therefore an action of $\partial_{i}^{[p^{2}]}$) on $\mathcal{E}_{n-2}$.
Thus $\mathcal{E}_{n-2}$ inherits an action of $\mathcal{D}_{X_{W_{n-2}(k)}}^{(2)}$.
Iterating this construction proves the first claim. 

As for the second, it is clear that if $\mathcal{E}_{n}$ is locally
trivial then so if each $\mathcal{E}_{j}$. For the converse, proceed
by downward induction on $j$. Since $\mathcal{E}_{n-1}$ is locally
trivial as a connection, we can choose a basis $\{e_{i}\}_{i=1}^{m}$
of $\mathcal{E}_{n}$ for which
\[
\partial_{j}(e_{i})=\sum_{s=1}^{m}p^{n-1}\alpha_{is,j}e_{s}
\]
for some $\alpha_{is,j}\in\mathcal{O}_{X_{k}}$ (identifying $p^{n-1}\mathcal{O}_{X_{W_{n}(k)}}$
with $\mathcal{O}_{X_{k}}$). Applying $\partial_{j}^{p-1}$to this
equation yields 
\[
\partial_{j}^{p}(e_{i})=\sum_{s=1}^{m}p^{n-1}(\partial_{j}^{p-1}\alpha_{is,j})e_{s}
\]
So the $\mathcal{D}_{X_{W_{n-1}(k)}}^{(1)}$-module action on $\mathcal{E}_{n-1}$
is given, in the basis $\{e_{i}\phantom{i}\text{mod}\phantom{i}p^{n-1}\}_{i=1}^{m}$
(which we will again denote by $\{e_{i}\}$), by $\partial_{j}(e_{i})=0$
and 
\[
\partial_{j}^{[p]}(e_{i})=-\sum_{s=1}^{m}p^{n-2}(\partial_{j}^{p-1}\alpha_{is,j})e_{s}
\]
If we suppose that this $\mathcal{D}_{X_{W_{n-1}(k)}}^{(1)}$-module
is trivial, then we can choose a basis in which $\partial_{j}^{p-1}\alpha_{is,j}=0$
for all $i,j,s$. However, the one-form 
\[
{\displaystyle \sum_{s=1}^{m}\sum_{j=1}^{d}}\partial_{j}^{p-1}\alpha_{is,j}dz_{j}\cdot e_{i}^{*}\otimes e_{s}\in\text{End}(\mathcal{E}_{1})\otimes\Omega_{X_{k}}^{1}
\]
is exactly the Cartier operator applied to the one-form $\psi={\displaystyle \sum_{s=1}^{m}\sum_{j=1}^{d}\alpha_{is,j}dz_{j}e_{i}^{*}\otimes e_{s}}\in\text{End}(\mathcal{E}_{1})\otimes\Omega_{X_{k}}^{1}$;
i.e., it is the image of $\psi$ in $H_{dR}^{1}(\text{End}(\mathcal{E}_{1}))$
under the Cartier operator for the trivial connection $\mathcal{E}_{1}$.
If this vanishes, then there is a matrix $A\in\text{End}(\mathcal{E}_{1})$
with $dA=\psi$, and changing basis by $(1+p^{n-1}A)$ trivializes
the connection $(\mathcal{E}_{n},\nabla)$. 

Now, suppose $\mathcal{E}_{n-1}$ is trivial over $\mathcal{D}_{X_{W_{n-1}(k)}}^{(1)}$,
and consider the action of $\mathcal{D}_{X_{W_{n-2}(k)}}^{(2)}$ on
$\mathcal{E}_{n-2}$. To compute this, we start with the previous
formula for $\partial_{j}^{[p]}e_{i}$, and write $-\partial_{j}^{p-1}\alpha_{is,j}=\beta_{is,j}^{p}$
for some $\beta_{is,j}\in\mathcal{O}_{X_{k}}$. Using the fact that
$[\partial_{i}^{[p]},g^{p}]=(\partial_{i}(g))^{p}$ for all $g\in\mathcal{O}_{X_{k}}$,
we repeat the above computation to arrive at 
\[
\partial_{j}^{[p^{2}]}(e_{i})=-\sum_{s=1}^{m}p^{n-2}((\partial_{j}^{[p]})^{p-1}\beta_{is,j})e_{s}
\]
and as above the triviality of this $\mathcal{D}_{X_{W_{n-2}(k)}}^{(2)}$-module
is equivalent to the triviality of $\mathcal{E}_{n-1}$ as a $\mathcal{D}_{X_{W_{n-1}(k)}}^{(1)}$-module.
Continuing in this way proves this theorem. 
\end{proof}
It is useful to specialize the construction to the case of line bundles.
Suppose, in fact that we are working with the trivial line bundle
on $X_{W_{n}(k)}$, with connection 
\[
\nabla(1)=p^{n-1}\psi
\]
where $\psi$ is an arbitrary closed one-form in $\Omega_{X_{k}}^{1}$;
i.e., we make no assumption on the liftability of the connection.
We can look at the sequence $\psi^{(i)}\in\Omega_{X_{k}^{(i)}}^{1}$
defined as follows: $\psi^{(1)}=\text{image}(\psi)\in\mathcal{H}_{dR}^{1}(\Omega_{X_{k}}^{1})\tilde{=}\Omega_{X_{k}^{(1)}}^{1}$.
If $\psi^{(1)}$ is closed\footnote{which is guaranteed if the connection is liftable},
then we define $\psi^{(2)}=\text{image}(\psi^{(1)})\in\mathcal{H}_{dR}^{1}(\Omega_{X_{k}^{(1)}}^{1})\tilde{=}\Omega_{X_{k}^{(2)}}^{1}$.
If this is closed, we define $\psi^{(3)}$ analogously, and we continue
on to $\psi^{(n-1)}$. Then we make the 
\begin{defn}
\label{def:p^m-curv}The $p^{m}$-curvature of $\psi$, denoted $\Psi$,
is the $p$-curvature of $\psi^{(m-1)}$, i.e. 
\[
\Psi=(\psi^{(m-1)})^{p}-C(\psi^{(m-1)})\in\Omega_{X_{k}^{(m)}}^{1}
\]
Where $C:\Omega_{X_{k}^{(m-1)}}^{1,\text{cl}}\to\Omega_{X_{k}^{(m)}}^{1}$
denotes the Cartier operator.
\end{defn}

The condition that the $\{\psi^{(i)}\}$ are defined is guaranteed
by the existence $\mathcal{D}_{X_{W_{j}(k)}}^{(j)}$-module structure
of \prettyref{thm:Higher-diff-structure-on-deformations} is defined
for all $j\in\{0,\dots,m-1\}$; i.e., the coefficients of the one
form $\psi^{(i)}$ are determined by the action of the operators $\partial_{j}^{[p^{i}]}$.
Thus the above construction yields 
\begin{cor}
\label{cor:p^m-curvature}Let $\psi\in\Omega_{X_{k}}^{1}$ be such
that $\{\psi^{(i)}\}_{i=1}^{m-1}$ are defined (for instance, if the
one-form is liftable to $W(k)$). Then the $p^{n}$-curvature $\Psi=0$
iff there is some $u\in\Omega_{X_{W_{m}(k)}}^{1}$ such that 
\[
\frac{du}{u}=p^{m-1}\psi
\]
\end{cor}

It should be remarked right away that the consideration of the operators
$\{\psi^{(i)}\}_{i=1}^{m-1}$ is not new- indeed, they have been considered
in Illusie's fundamental work on the de Rham-Witt complex (\cite{key-82},
section 2.2), and the previous corollary can be deduced from the results
of that work. However, the connection with higher differential operators
seems to be new. Finally, we record the 
\begin{rem}
\label{rem:The-p^m-curvature-works}The theory works in an identical
fashion in the relative situation; e.g. for vector bundles with connection
over $\mathcal{D}_{X_{k}}[\lambda,\lambda^{-1}]$ or $\mathcal{D}_{X_{k}}[[\lambda]][\lambda^{-1}]$. 
\end{rem}

\subsection{\label{subsec:Local-Lifts-of-Connections}Local Lifts of Connections}

In this subsection we consider the structure of lifts of a given $\mathcal{D}_{X_{k}}^{(0)}$
module $\tilde{\mathcal{M}}_{k}$, which we assume to be a splitting
bundle for $L_{k}^{(1)}$. To that end, fix a $W_{m}(k)$-flat $\mathcal{D}_{X_{W_{m}(k)}}^{(0)}$-module
$\tilde{\mathcal{M}}_{W_{m-1}(k)}$ which lifts $\tilde{\mathcal{M}}_{k}$;
suppose this module admits a lift to a $W_{m}(k)$-flat $\mathcal{D}_{X_{W_{m}(k)}}^{(0)}$-module,
$\tilde{\mathcal{M}}_{W_{m}(k)}$. We start with 
\begin{lem}
\label{lem:End=00003DWitt}We have $\mathcal{E}nd_{\nabla}(\tilde{\mathcal{M}}_{W_{m}(k)})\tilde{=}W_{m}(\mathcal{O}_{L_{k}^{(m)}})$
\end{lem}

\begin{proof}
The action of the center yields a map $W_{m}(\mathcal{O}_{L_{k}^{(m)}})\to\mathcal{E}nd_{\nabla}(\tilde{\mathcal{M}}_{W_{m}(k)})$.
To show its an isomorphism, we can work formally locally. Then we
can apply \prettyref{prop:Iso-All-Levels!} and the analogous fact
for the trivial connection $\mathcal{O}_{X_{W_{m}(k)}}$. 
\end{proof}
Now we can analyze the structure of the set of lifts: 
\begin{prop}
\label{prop:Pic-for-Local-Lifts}For a fixed lift $\tilde{\mathcal{M}}_{W_{m}(k)}$,
then the subgroup of (isomorphism classes of) line bundles with locally
trivial connection on $L_{W_{m}(k)},$ acts simply transitively on
the set of lifts which are locally isomorphic to $\mathcal{\tilde{M}}_{W_{m}(k)}$. 
\end{prop}

\begin{proof}
Recall from \prettyref{cor:Marked-Descent} we have an equivalence
between line bundles with locally trivial connection on $L_{W_{m}(k)}$
and line bundles on $W_{m}(L_{k}^{(m)})$, given by pulling back under
the natural inclusion $\mathcal{O}_{W_{m}(L_{k}^{(m)})}\to\mathcal{O}_{L_{W_{m}(k)}}$.
To each such line bundle we can associate an open affine cover $\{U_{i}\}$
and an element $\alpha_{ij}\in\mathcal{O}_{W_{m}(U_{ij}^{(m)})}^{*}$.
We also have isomorphisms of $\mathcal{D}_{X_{W_{n}(k)}}^{(0)}$-modules
$\psi_{ij}:\mathcal{M}_{W_{n}(k)}(U_{i})|_{U_{ij}}\tilde{\to}\mathcal{M}_{W_{n}(k)}(U_{j})|_{U_{ij}}$.
By \prettyref{lem:End=00003DWitt}, $\alpha_{ij}$ acts on $\mathcal{M}_{W_{n}(k)}(U_{j})|_{U_{ij}}$,
and so, by the cocycle condition for $\{\alpha_{ij}\}$, we can obtain
a new $\mathcal{D}_{X_{W_{n}(k)}}^{(0)}$-module by replacing $\psi_{ij}$
with $\alpha_{ij}\cdot\psi_{ij}$; whence the result. 
\end{proof}
We'll also need
\begin{lem}
\label{lem:Ext for a splitting bundle}We have $\mathcal{E}xt_{\mathcal{D}_{X_{k}}^{(0)}}^{1}(\tilde{\mathcal{M}_{k}},\tilde{\mathcal{M}_{k}})\tilde{=}\Omega_{L_{k}^{(1)}}^{1}$
as sheaves on $T^{*}X_{k}^{(1)}$. 
\end{lem}

\begin{proof}
By assumption we have that $\mathcal{D}_{X_{k}}^{(0)}|_{L_{k}^{(1)}}$
is a split Azumaya algebra, and that $\tilde{\mathcal{M}_{k}}$ is
a splitting bundle. Let $L_{k,m}^{(1)}$ denote the $m$th infinitesimal
neighborhood of $L_{k}^{(1)}$ in $T^{*}X_{k}^{(1)}$. We claim that
$\mathcal{D}_{X_{k}}^{(0)}|_{L_{k,m}^{(1)}}$ is also split. To see
this, recall from \cite{key-23}, chapter 4, theorem 2.5, that there
is for any scheme $Y$ an injection $\text{Br}(Y)\to H_{et}^{2}(Y,\mathbb{G}_{m})$
(here $\text{Br}(Y)$ denotes the Brauer group). Since $L_{k}^{(1)}$
is affine, the nilpotent immersion $L_{k}^{(1)}\to L_{k,m}^{(1)}$
induces an isomorphism
\[
H_{et}^{2}(L_{k,m}^{(1)},\mathbb{G}_{m})\tilde{\to}H_{et}^{2}(L_{k}^{(1)},\mathbb{G}_{m})
\]
Therefore the class of $\mathcal{D}_{X_{k}}^{(0)}|_{L_{k,m}^{(1)}}$
in $\text{Br}(L_{k,m}^{(1)})$ must be trivial as claimed. 

Furthermore, as isomorphism classes of splitting bundles on $L_{k,m}^{(1)}$
form a torsor over $\text{Pic}(L_{k,m}^{(1)})\tilde{=}\text{Pic}(L_{k}^{(1)})$,
we see that these splitting bundles can be chosen compatibly with
the restriction from $L_{k,m}^{(1)}$ to $L_{k,m-1}^{(1)}$. Taking
the inverse limit yields a splitting bundle. $\widehat{\tilde{\mathcal{M}}}_{k}$,
for $\mathcal{D}_{X_{k}}$ on $\widehat{L_{k}^{(1)}}$, the full formal
neighborhood of $L_{k}^{(1)}$ in $T^{*}X_{k}^{(1)}$. Therefore,
there is an equivalence of categories 
\[
\text{QCoh}_{L_{k}^{(1)}}(T^{*}X_{k}^{(1)})\to\text{QCoh}_{L_{k}^{(1)}}(\mathcal{D}_{X_{k}}^{(0)})
\]
where the left hand side denotes quasi-coherent sheaves on $T^{*}X_{k}^{(1)}$
which are set-theoretically supported on $L_{k}^{(1)}$ and the right
hand side denotes quasi-coherent $\mathcal{D}_{X_{k}}^{(0)}$-modules
which are set-theoretically supported on $L_{k}^{(1)}$. The functor
is $\mathcal{F}\to\widehat{\tilde{\mathcal{M}}}_{k}\otimes_{\mathcal{O}_{T^{*}X_{k}^{(1)}}}\mathcal{F}$.
Thus we obtain an isomorphism 
\[
\mathcal{E}xt_{\mathcal{D}_{X_{k}}}^{1}(\tilde{\mathcal{M}_{k}},\tilde{\mathcal{M}_{k}})\tilde{=}\mathcal{E}xt_{T^{*}X_{k}^{(1)}}^{i}(\mathcal{O}_{L_{k}^{(1)}},\mathcal{O}_{L_{k}^{(1)}})
\]
and setting $i=1$ implies the result.
\end{proof}
Now we'll apply the theory of the $p^{m}$-curvature, developed above
in \prettyref{subsec:Locally-trivial-connections}, to this situation:
\begin{prop}
\label{prop:p^m-curv-over-L}Let $\tilde{\mathcal{N}}_{W_{m}(k)}$
be another lift of $\tilde{\mathcal{M}}_{k}$, and suppose that $\tilde{\mathcal{N}}_{W_{m-1}(k)}$
and $\tilde{\mathcal{M}}_{W_{m-1}(k)}$ are locally isomorphic. Then
there is a one-form $\Psi\in\Omega_{L_{k}^{(m)}}^{1}$ which vanishes
iff $\tilde{\mathcal{N}}_{W_{m}(k)}$ and $\tilde{\mathcal{M}}_{W_{m}(k)}$
are locally isomorphic. 
\end{prop}

\begin{proof}
By assumption, there exists an open affine cover $\{U_{i}\}$ of $L_{k}$,
so that there are isomorphisms $\psi_{i}:\mathcal{\tilde{M}}_{W_{m-1}(k)}(U_{i})\tilde{=}\mathcal{\tilde{N}}_{W_{m-1}(k)}(U_{i})$.
Via these isomorphisms, we can regard both $\mathcal{M}_{W_{n}(k)}(U_{i})$
and $\mathcal{N}_{W_{n}(k)}(U_{i})$ as being lifts of $\mathcal{N}_{W_{n-1}(k)}(U_{i})$;
then the class $[\mathcal{M}_{W_{n}(k)}(U_{i})]-[\mathcal{N}_{W_{n}(k)}(U_{i})]$
is an element $\varphi_{i}\in\text{Ext}{}_{U_{i}}^{1}(\mathcal{\tilde{M}}_{k},\mathcal{\tilde{M}}_{k})\tilde{=}\Omega_{L_{k}^{(1)}}^{1}(U_{i})$.
Pick a closed form $\varphi_{i}'\in\Omega_{L_{k}}^{1}(U_{i})$ whose
image in $\Omega_{L_{k}^{(1)}}^{1}(U_{i})$ is $\varphi_{i}$. We
claim that, in the notation of \prettyref{def:p^m-curv}, the one-forms
$\{(\varphi_{i}')^{(j)}\}_{j=0}^{m-1}$ are all defined. To see this,
note that this condition is checkable formally locally. So, applying
\prettyref{prop:Iso-All-Levels!} (which reduces the question to the
case of lifts of line bundles with connection) the result follows
from the fact that $\mathcal{M}_{W_{n}(k)}$ and $\mathcal{N}_{W_{n}(k)}$
are both liftable to $W(k)$. Therefore, we can take the $p^{m}$-curvature
$\Psi_{i}$; which a priori depends on the choice of $\psi_{i}$. 

Now we look at what happens when we change the map $\psi_{i}$. Suppose
we have $\alpha\in\mathcal{O}_{W_{m-1}(U_{ij}^{(m-1)})}^{*}$. Writing
${\displaystyle \alpha=g_{1}^{p^{m-1}}+pg_{2}^{p^{m-2}}+\dots+p^{m-2}g_{m-2}^{p}}$
for $g_{i}\in\mathcal{O}_{U_{ij,W_{m-1}(k)}}$, we see that, if $\tilde{\alpha}\in\mathcal{O}_{U_{ij,W_{m}(k)}}$
is a lift of $\alpha$, then ${\displaystyle \frac{d\tilde{\alpha}}{\tilde{\alpha}}\in p^{m-1}\Omega_{U_{ij,W_{m}(k)}}^{1}\tilde{=}}\Omega_{U_{ij,k}}^{1}$.
Denote the image of this closed one-form in $\Omega_{L_{k}^{(1)}}^{1}(U_{i})$
by ${\displaystyle [\frac{d\alpha}{\alpha}]}$. 

Then, if we alter the map $\psi_{i}$ by replacing it by $\alpha\cdot\psi_{i}$
for some $\alpha\in\mathcal{O}_{W_{m-1}(U_{ij}^{(m-1)})}^{*}$, then
this has the effect of replacing $\varphi_{i}$ with ${\displaystyle \varphi_{i}+[\frac{d\alpha}{\alpha}]}$.
To prove this, we may work (formally) locally and apply \prettyref{prop:Iso-All-Levels!}
to reduce to proving the analogous statement for line bundles with
connection on $L_{W_{m}(k)}$, where it is a direct computation. In
particular, replacing $\psi_{i}$ by $\alpha\cdot\psi_{i}$ does not
alter the one-form $\Psi_{i}$, because adding ${\displaystyle [\frac{d\alpha}{\alpha}]}$
corresponds to replacing the line bundle whose $p^{m}$-curvature
we are calculating with an isomorphic one, and the $p^{m}$-curvature
depends only the isomorphism class of the line bundle. It follows
that this construction glues, and we obtain a one-form $\Psi$ in
$\Omega_{L_{k}^{(m)}}^{1}$.

Clearly if $\tilde{\mathcal{N}}_{W_{m}(k)}$ and $\tilde{\mathcal{M}}_{W_{m}(k)}$
are locally isomorphic then $\Psi=0$. Conversely, suppose $\Psi=0$.
Choose $\{U_{i}\}$ as above and pick $\varphi_{i}'$ whose image
in $H_{dR}^{1}(U_{i})$ is $\varphi_{i}$. Then, by \prettyref{cor:p^m-curvature}
there is some unit $u\in\mathcal{O}_{L_{W_{m}(k)}}$ with ${\displaystyle p^{m-1}\varphi'_{i}=\frac{du}{u}}$.
But this equation implies that $\bar{u}$, the image of $u$ in $\mathcal{O}_{L_{W_{m-1}(k)}}$
satisfies $d\bar{u}=0$; in other words $\bar{u}\in\mathcal{O}_{W_{m-1}(U_{ij}^{(m-1)})}^{*}$.
Altering the map $\psi_{i}$ to $\bar{u}\cdot\psi_{i}$ then alters
$\varphi_{i}$ to $0$, which shows that $\tilde{\mathcal{N}}_{W_{m}(k)}(U_{i})\tilde{=}\tilde{\mathcal{M}}_{W_{m}(k)}(U_{i})$
as required. 
\end{proof}

\section{\label{sec:Quantizations-of-L}Quantizations of $L$}

The main goal of this chapter is to finish the proof of \prettyref{thm:1}.
Let us recall that, in \prettyref{thm:Uniqueness-of-lambda-conns},
we have constructed a family of meromorphic connections on $\overline{\tilde{X}}_{\mathbb{C}}$,
whose restriction to $U_{\mathbb{C}}$ has constant arithmetic support
(equal to $L_{U_{\mathbb{C}}}$). This family was indexed by line
bundles with log connection on $\overline{L}_{\mathbb{C}}$ with finite
monodromy group. Amongst these we can consider those which have trivial
monodromy along each divisor in $L_{\mathbb{C}}$. Since $H_{dR}^{1}(L_{\mathbb{C}})=0$,
this gives us a finite collection of connections, which we denoted
$\{(\mathcal{L},\nabla)\star\overline{\mathcal{M}}_{1,\mathbb{C}}\}$
(the choice of a particular $\overline{\mathcal{M}}_{1,\mathbb{C}}$
is for notational convenience only). 

Now, each such $(\mathcal{L},\nabla)\star\overline{\mathcal{M}}_{1,\mathbb{C}}|_{U_{\mathbb{C}}}$
is an irreducible flat connection. Thus we may consider the intermediate
extension, which is an irreducible holonomic $\mathcal{D}$-module
on $X_{\mathbb{C}}$; let us denote these holonomic $\mathcal{D}$-modules
by $(\mathcal{L},\nabla)\star\mathcal{M}_{\mathbb{C}}$, and (again
for notational convenience), let us fix our attention on one of these
modules, $\mathcal{M}_{\mathbb{C}}$. We want to calculate its arithmetic
support; i.e., the $p$-support of $\mathcal{M}_{k}$ for $\text{char}(k)>>0$. 

To accomplish this, we'll use the method of \prettyref{subsec:The-monodromy-divisor},
and apply a linear symplectomorphism $\sigma$ to analyze the situation
at a given point. It is not obvious that the resulting object, $\mathcal{M}^{\sigma}$,
is the specialization at $\lambda=1$ of a $\theta$-regular $\lambda$-connection.
However, this is true, and we'll use the theory developed in \prettyref{sec:Some-P-adic-Microlocal}
to show it. Namely, we'll first construct a suitable $\theta$-regular
$\lambda$-connection, and the, we'll use the theory of the $p^{m}$-curvature
to see that for $p>>0$, the $p$-adic completion of $\mathcal{M}^{\sigma}$
(over some open subset) is isomorphic to the $p$-adic completion
of a suitable $\theta$-regular $\lambda$-connection, specialized
to $\lambda=1$. Somewhat surprisingly, this turns out to be enough
to show that the two $\mathcal{D}$-modules are actually isomorphic
(c.f. the proof of \prettyref{thm:M-sigma-is-theta-reg} below). Once
this is done, the calculation of the $p$-support follows; at least
in codimension $2$; and then in general by the results of \prettyref{subsec:Calculation-of-the-p-support}
below. The same technique turns out to be the key point in showing
that the action of $\pi^{*}$ is transitive on the set of all $\mathcal{D}$-modules
with constant arithmetic support equal to $L_{\mathbb{C}}$, with
multiplicity $1$; as a by-product, we'll see that they all come from
the construction of the lemma above. 

We shall, as usual, suppose that $R$ is smooth over $\mathbb{Z}$,
so that any morphism $R\to k$ lifts to a flat morphism $R\to W(k)$. 

\subsection{\label{subsec:Analysis-of-M-sigma}Analysis of $\mathcal{M}^{\sigma}$. }

Now we can begin the proof that $\mathcal{M}_{\mathbb{C}}$ has constant
arithmetic support. We start with some preliminaries over $R$. Fix
a component $E'$ of the divisor $E=L\backslash L_{U}$. Choose an
embedding $X\subset\mathbb{A}^{m}$, and a linear symplectomorphism
$\sigma$ of $T^{*}\mathbb{A}^{m}$ such that the projection $T^{*}\mathbb{A}^{m}\to\mathbb{A}^{m}$
is an isomorphism in a formal neighborhood of $\sigma(E')$ in $\sigma(L)$.
For later use, we shall specify that $\sigma$ is of a particular
form. Let $\{x_{1},\dots,x_{m},\xi_{1},\dots,\xi_{m}\}$ be standard
coordinates on $T^{*}\mathbb{A}^{m}$. Then we demand that 
\[
\sigma=\sigma_{1}\circ\sigma_{2}\circ\cdots\circ\sigma_{r}
\]
where $\sigma_{i}$ is either the linear map which swaps $\xi_{r}$
and $\xi_{s}$ (for some $r$ and $s$), or the map which swaps $x_{t}$
and $\xi_{t}$ (for some $t$). It is elementary to see that there
is a $\sigma$, inducing an isomorphism on a formal neighborhood of
some point in $E'^{,\text{sm}}$, which factors this way. 

In particular, the projection $\sigma(L)\to\mathbb{A}^{m}$ is dominant;
and we have a function $f'$ on $\sigma(L)$ so that $df'=\theta|_{\sigma(L)}$.
In particular we can apply the notation and results of the previous
chapter to this situation; let $\tilde{\mathbb{P}}^{m}$ denote a
suitable modification of $\mathbb{P}^{m}$ (e.g., satisfying the assumptions
laid out in \prettyref{subsec:Higgs-Sheaves}). 

Replace $\mathcal{M}$ with ${\displaystyle \int_{i}\mathcal{M}}$.
Let $U'\subset\mathbb{A}^{m}$ be an open subset over which $\sigma(L)|_{U'}\to U'$
is finite etale. Note that $\mathcal{M}_{\lambda}(U)$ can be extended
a coherent $\mathcal{D}_{\lambda}$-module to all of $X$, and then
we can apply $\sigma^{*}$, and restrict to $U'$; shrinking $U'$
if necessary we have that the result, call it $\mathcal{M}_{\lambda}^{\sigma}$
, is a bundle on $U'$. We now replace $\sigma(L)$ with $L$ and
$\sigma(E')$ with $E'$, to ease notation. By our assumptions on
$\sigma$, we have that the image of $E'$ is a divisor $D'$ in $\mathbb{A}^{m}$. 

The main aim of this section is to prove the following preliminary
result: 
\begin{thm}
\label{thm:M-sigma-is-theta-reg}There is a $\theta$-regular $\lambda$-connection
$\mathcal{N}_{\lambda}$, whose residue along $E'$ is $0$, so that
$\mathcal{N}_{1}|_{U'}\tilde{=}\mathcal{M}_{1}^{\sigma}|_{U'}$. 
\end{thm}

In the next subsection, we'll show how this implies the main result
on the $p$-support of $\mathcal{M}$. Meanwhile, the proof of this
theorem will use all the tools developed in the previous chapters. 
\begin{prop}
\label{prop:Theta-regular-extension} Let the notations be as above.
For each $n\geq1$, $\mathcal{M}_{\lambda}^{\sigma}/\lambda^{n}$
extends to a $\theta$-regular meromorphic connection on $\tilde{\mathbb{P}}^{m}$,
denoted $\overline{\mathcal{N}_{\lambda,n}}$. The residue of this
connection is trivial along $E'$. 
\end{prop}

\begin{proof}
When $n=1$ this is clear as $\mathcal{M}_{\lambda}^{\sigma}/\lambda$
is a line bundle on $\sigma(L)_{U'}$; i.e., $\mathcal{M}_{\lambda}^{\sigma}/\lambda=\pi_{*}(\mathcal{L},df')$
for some line bundle $\mathcal{L}$ on $L_{U'}$. Shrinking $U'$
if necessary (for the last time), we'll assume $\mathcal{L}=\mathcal{O}_{L_{U'}}$. 

When $n=2$, we have (by \prettyref{lem:Deformations-over-U}) $\mathcal{M}_{\lambda}^{\sigma}/\lambda^{2}=\pi_{*}(\mathcal{L},df'+\lambda\psi)$
where $\psi$ is a one-form on $L_{U'}$. As $\sigma$ is a linear
symplectomorphism, we have (as in the proof of \prettyref{thm:Microlocal-form-of-E})
that $\mathcal{M}_{\lambda,k}^{\sigma}$ has $p$-curvature equal
to $\sigma(L_{k}^{(1)})\times\mathbb{A}_{k}^{1}$ for all $k$ of
characteristic $p>>0$. Therefore, applying the argument of \prettyref{prop:Difference-of-two-connections},
we see that $\psi$ is a closed one-form of $p$-curvature $0$ for
$p>>0$. It follows that the one-form $\psi$ has regular singularities
with rational monodromy around each divisor of $\sigma(L)$. Along
a given component of $\sigma(L)\backslash\sigma(L)_{U'}$, this monodromy
measures the failure of $\mathcal{M}_{\lambda}^{\sigma}/\lambda^{2}$
to extend to a flat deformation of a line bundle on $\sigma(L)$ (as
in \prettyref{prop:alternative-characterization}). However, $\mathcal{M}_{\lambda}^{\sigma}/\lambda^{2}$
certainly does extend to a flat deformation of a line bundle on $\sigma(L)$,
since $\mathcal{M}_{\lambda}(U)/\lambda^{2}$ extends to a deformation
of a line bundle on $L$. Thus this monodromy vanishes on each component
of $\sigma(L)\backslash\sigma(L)_{U'}$. Now, applying the argument
of \prettyref{lem:Adjust-the-extension} one deduces the existence
of $\overline{\mathcal{M}_{\lambda,2}^{\sigma}}$, with residue trivial
along $E'$ (as the projection of $E'\subset T^{*}X$ to $X$ is generically
an isomorphism, the blowup $\tilde{X}\to X$ is an isomorphism at
the generic point of $\pi(E)$. Thus we see that $E'$ must correspond
to a component of $\overline{L}_{V_{x},i}$ which is embedded, and
the residue of $\overline{\mathcal{M}_{\lambda,2}^{\sigma}}$ at $E'$
is equal to the residue of $\mathcal{M}_{\lambda,2}^{\sigma}$ at
$E'$, which is trivial). 

Now, let $n>2$. By induction, we shall assume that $\overline{\mathcal{N}_{\lambda,n-1}}$
exists, and that we have found an isomorphism $\overline{\mathcal{N}_{\lambda,n-1}}|_{U'}\tilde{=}\mathcal{N}_{\lambda}^{\sigma}/\lambda^{n-1}$.
Applying \prettyref{thm:Infinitesimal-Def}, we see that there is
a unique $\theta$-regular deformation $\overline{\mathcal{N}{}_{\lambda,n}}$
of $\overline{\mathcal{N}_{\lambda,n-1}}$. By the general deformation
theory, we have that $[\overline{\mathcal{N}_{\lambda,n}}|_{U'}]-[\mathcal{M}_{\lambda}^{\sigma}/\lambda^{n}]=\psi_{n}$
for some $\psi_{n}\in\Omega_{\sigma(L)|_{U'}}^{1}$; explicitly, if
we write $\overline{\mathcal{N}_{\lambda,n}}|_{U'}=\pi_{*}(\mathcal{O}[\lambda]/\lambda^{n},\varphi)$,
then $\mathcal{M}_{\lambda}^{\sigma}/\lambda^{n}=\pi_{*}(\mathcal{O}[\lambda]/\lambda^{n},\varphi+\lambda^{n-1}\psi_{n})$.
These two connections are isomorphic iff there exists some $g\in\mathcal{O}_{L_{U'}}$
with $dg=\psi_{n}$; the isomorphism is given by multiplying a generator
by $1+\lambda^{n-2}g$. 

On the other hand, the reduction to $k$ of both $\overline{\mathcal{N}_{\lambda,n}}|_{U'}$
and $\mathcal{M}_{\lambda}^{\sigma}/\lambda^{n}$ extend to $\lambda$-connections
over $k[\lambda]$ with $p$-curvature equal to $\sigma(L_{k}^{(1)})\times\mathbb{A}_{k}^{1}$
(the former by \prettyref{cor:At-most-one-N-lambda}, the latter by
definition, since it is $\theta$-regular at infinity). So $\overline{\mathcal{N}{}_{\lambda,n,k}}|_{U'}$
and $\mathcal{M}_{\lambda,k}^{\sigma}/\lambda^{n}$ are isomorphic
as $\lambda$-connections. But this implies that $\psi_{n,k}=dg$
for some function $g$. Therefore $\psi_{n}$ is a closed one-form,
whose image in the de Rham cohomology group $H_{dR}^{1}(L_{U',k})$
is $0$ for all $k$. Then, using the canonical injection 
\[
H_{dR}^{1}(L_{U'})/p\to H_{dR}^{1}(L_{U',k})
\]
we see that $\bar{\psi}_{n}\in p\cdot H_{dR}^{1}(L_{U'})$ for all
primes $p$ in $R$. (we use $\bar{\psi}_{n}$ to denote the image
of $\psi$ in de Rham cohomology) As we show below in \prettyref{lem:p-adically-seperated},
the group $H_{dR}^{1}(L_{U'})$ is a direct sum of finite type $R$-modules.
This shows that $\bar{\psi}_{n}$ is a torsion class in $H_{dR}^{1}(L_{U'})$;
which means it is a sum of $p^{r}$-torsion classes for finitely many
primes. Using the canonical surjection
\[
H_{dR}^{0}(L_{U',W_{r}(k)})\to p^{r}-\text{tors}(H_{dR}^{1}(L_{U'}))
\]
we see that, along any irreducible divisor $E''\subset\tilde{E}$,
if $\psi_{n}$ has a pole, then (up to adding an exact one-form) it
must be a sum of terms of the form of the form $p^{a}z^{-(p^{b}-1)}dz$
(where $z$is a local coordinate for $E''$, $b\geq1$, and $a\geq0$
is a suitable power with $a\leq b$). 

We claim that, for large $p$, no such description of the poles is
possible. If this is so, then $\bar{\psi}_{n}$ is the image of a
function in $H_{dR}^{0}(L_{U',W_{r}(k)})$ which has no poles along
any divisor in $\overline{L}_{W_{r}(k)}$, which implies $\bar{\psi}_{n}=0$
as required to finish the induction step of the proof. 

To prove this claim, consider the completion $\mathcal{M}_{\widehat{\lambda}}^{\sigma}$,
as a sheaf on $T^{*}U'$. As it is supported along $L_{U'}$, it follows
that, after possibly pulling back along a root cover, for each irreducible
divisor $E\subset\tilde{\mathbb{P}}^{m}\backslash U'$, there is a
decomposition 
\[
\mathcal{M}_{\widehat{\lambda}}^{\sigma}[z^{-1}]=\bigoplus_{i}\widehat{\mathcal{O}[\lambda,z^{-1}]\cdot e_{i}}
\]
where $z$ is a local equation for $E$, and $\widehat{\mathcal{O}[\lambda,z^{-1}]\cdot e_{i}}$
is the $\lambda$-adic completion of $\mathcal{O}[\lambda,z^{-1}]$,
equipped with some $\lambda$-connection. A priori, this connection
has the form 
\[
\nabla(e_{i})=\sum_{m\in\mathbb{Z}}\alpha_{m}z^{m}\cdot e_{i}
\]
where $\alpha_{m}\to0$ in the $\lambda$-adic topology as $m\to-\infty$.
However, as proved directly below, we actually have $\alpha_{m}=0$
for $m<N$, for some fixed negative number $N$; for some choice of
$\{e_{i}\}$. We also have that $\overline{\mathcal{N}_{\lambda,n}}[z^{-1}]$
admits an analogous decomposition, as it is $\theta$-regular at infinity.
Changing the basis $\{e_{i}\}$ to another such basis allows us to
add a term of the form $\lambda^{n-1}du$, however, since $\bar{\psi}_{n}\neq0$
we see that we cannot obtain $\psi_{n}$ this way. So, for $p>|N|$,
we cannot add $\psi_{n}$ to the connection form of $(\mathcal{M}_{\lambda}^{\sigma}/\lambda^{n})[z^{-1}]$
to obtain the connection form of $\overline{\mathcal{N}_{\lambda,n}}[z^{-1}]$.
Thus $\psi_{n}$ cannot have any poles and the result follows. 
\end{proof}
The proof above required a technical result. To state it, we introduce
some terminology. Let $U$ be an open subset of $\mathbb{A}^{m}$
and let $L_{U}\subset T^{*}U$ be a smooth exact Lagrangian whose
projection to $U$ is finite etale, of rank $r$. Let $\mathcal{F}_{\lambda}$
be a coherent $\mathcal{D}_{\lambda}$-module on $\mathbb{A}^{m}$,
so that $\mathcal{F}_{\lambda}/\lambda|_{U}\tilde{=}\pi_{*}(\mathcal{O}_{L_{U}})$
as Higgs bundles. Let $\{x\}$ be the generic point of an irreducible
divisor in $\tilde{\mathbb{P}}^{m}$ and let $V_{x}$ be an etale
neighborhood of $\{x\}$, with $V_{x}^{(l)}$ a root cover along the
divisor; denote by $\varphi:V_{x}^{(l)}\to\tilde{\mathbb{P}}^{m}$.
Then, exactly as in \prettyref{subsec:Higgs-Sheaves}, the Higgs bundle
$\pi_{*}(\mathcal{O}_{L_{U}})$ satisfies the following: 
\[
\varphi^{*}\pi_{*}(\mathcal{O}_{L_{U}})[z^{-1}]=\bigoplus_{i=1}^{r}\mathcal{L}_{i}[z^{-1}]
\]
where $z$ is the local coordinate for the divisor and $\mathcal{L}_{i}[z^{-1}]$
is a meromorphic Higgs bundle of rank $1$. In other words, by Abhyankar's
theorem, the Lagrangian breaks up into $r$ distinct pieces near $\{x\}$,
possibly after a root cover. 

Then we say that $\mathcal{F}_{\lambda}$ is micro-locally decomposable
if there is a decomposition 
\[
\varphi^{*}\widehat{\mathcal{F}_{\lambda}[z^{-1}]}=\bigoplus_{i=1}^{r}\widehat{\mathcal{F}_{i}[z^{-1}]}
\]
where $\widehat{\mathcal{F}_{\lambda}[z^{-1}]}$ is the $\lambda$-adic
completion of $\mathcal{F}_{\lambda}[z^{-1}]$ and $\widehat{\mathcal{F}_{i}[z^{-1}]}$
are line bundles with meromorphic $\lambda$-connection; i.e., there
is some $e_{i}\in\widehat{\mathcal{F}_{i}[z^{-1}]}$ which generates
it and for which $\nabla(e_{i})\in z^{-N}e_{i}$ for some $N\in\mathbb{N}$. 

As indicated in the previous proof, the fact that such a decomposition
exists, without the meromorphicity condition, is obvious on general
grounds. Weather or not it always exists, is not clear to me a this
time. However, we have
\begin{lem}
\label{lem:(Technical)-We-must}Suppose $\mathcal{F}_{\lambda}$ is
micro-locally decomposable, and let $\sigma$ be a linear symplectomorphism
of $T^{*}\mathbb{A}^{m}$ of the type considered at the beginning
of the chapter. Then $\mathcal{F}_{\lambda}^{\sigma}$ is also micro-locally
decomposable. 
\end{lem}

\begin{proof}
By definition, $\mathcal{F}_{\lambda}^{\sigma}$ is defined as the
tensor product 
\[
\mathcal{F}_{\lambda}^{\sigma}:=\mathcal{D}_{\lambda}^{\sigma}\otimes_{\mathcal{D}_{\lambda}}\mathcal{F}_{\lambda}
\]
where $\mathcal{D}_{\lambda}^{\sigma}$ is the invertible $(\mathcal{D}_{\lambda},\mathcal{D}_{\lambda})$
bimodule defined by the automorphism $\sigma$. Under the isomorphism
$\mathcal{D}_{\lambda}\otimes_{R[\lambda]}\mathcal{D}_{\lambda}\tilde{=}\mathcal{D}_{\mathbb{A}^{2m},\lambda}$
we may regard $\mathcal{D}_{\lambda}^{\sigma}$ as a $\mathcal{D}_{\mathbb{A}^{2m},\lambda}$-module,
and the we may rewrite the above in terms of the $\mathcal{D}_{\lambda}$-module
operations: 
\[
\mathcal{F}_{\lambda}^{\sigma}:=\int_{p_{2}}\mathcal{D}_{\lambda}^{\sigma}\otimes_{\mathcal{O}_{\mathbb{A}^{2m}}[\lambda]}^{L}p_{1}^{*}\mathcal{F}_{\lambda}
\]
where $p_{1},p_{2}$ are the two projections from $\mathbb{A}^{2m}=\mathbb{A}^{m}\times\mathbb{A}^{m}$
to $\mathbb{A}^{m}$ (this is a general fact about application of
bimodules to $\mathcal{D}$-modules). 

Now, suppose $\sigma$ is the symplectomorphism which interchanges
$\xi_{i}$ and $x_{i}$. Then $\mathcal{D}_{\lambda}^{\sigma}$ is
the $\mathcal{D}_{\mathbb{A}^{2m},\lambda}$-module $\text{exp}(x_{i}+y_{i})$;
i.e., it is a line bundle with $\lambda$-connection where $\nabla$
acts on the generator $e$ as $edx_{i}+edy_{i}$. And if $\sigma$
is the symplectomorphism which interchanges $\xi_{i}$ and $\xi_{j}$,
then $\mathcal{D}_{\lambda}^{\sigma}$ is the $\mathcal{D}_{\mathbb{A}^{2m},\lambda}$-module
$\text{exp}(x_{i}y_{i})$. So, in either case it follows directly
that the $\mathcal{D}_{\mathbb{A}^{2m},\lambda}$-module 
\[
\mathcal{D}_{\lambda}^{\sigma}\otimes_{\mathcal{O}_{\mathbb{A}^{2m}}[\lambda]}^{L}p_{1}^{*}\mathcal{F}_{\lambda}=\mathcal{D}_{\lambda}^{\sigma}\otimes_{\mathcal{O}_{\mathbb{A}^{2m}}[\lambda]}p_{1}^{*}\mathcal{F}_{\lambda}
\]
is micro-locally decomposable, as a module over $\mathcal{D}_{\mathbb{A}^{2m},\lambda}$.
But then the result for $\mathcal{F}_{\lambda}^{\sigma}$ follows
as well, as the pushforward over $p_{2}$ of a bundle with meromorphic
connection is easily seen to be meromorphic. 
\end{proof}
Let $\overline{\mathcal{N}}_{\lambda}$ denote the unique algebraic
$\theta$-regular $\lambda$-connection whose reduction mod $\lambda^{n}$
is $\overline{\mathcal{N}_{\lambda,n}}$ (constructed via \prettyref{thm:Algebrization}).
Fix $k$ of characteristic $p>>0$. For each $m\geq1$ we are going
to compare $\mathcal{M}_{\lambda,W_{m}(k)}^{\sigma}$ with $\overline{\mathcal{N}}_{\lambda,W_{m}(k)}$.
In fact we have 
\begin{cor}
For each $m\geq1$, the $\lambda$-connections $\mathcal{M}_{\lambda,W_{m}(k)}^{\sigma}$
and $\overline{\mathcal{N}}_{\lambda,W_{m}(k)}$ are locally isomorphic
over $U'_{W_{m}(k)}\times\mathbb{A}_{W_{m}(k)}^{1}$. 
\end{cor}

\begin{proof}
We proceed by induction on $m$, for $m=1$ is is clear by looking
at the $p$-support (and using \prettyref{prop:Difference-of-two-connections}).
For $m>1$, we start by inverting $\lambda$ and applying the theory
of the $p^{m}$-curvature (via \prettyref{rem:The-p^m-curvature-works})
to the connections $\mathcal{M}_{\lambda,W_{m}(k)}^{\sigma}[\lambda^{-1}]$
and $\overline{\mathcal{N}}_{\lambda,W_{m}(k)}[\lambda^{-1}]$. Thus
we obtain an element $\Psi\in\Omega_{(L_{U'_{k}})^{(m)}}^{1}[\lambda,\lambda^{-1}]$.
Further, since $\mathcal{M}_{\widehat{\lambda},W_{m}(k)}^{\sigma}\tilde{=}\overline{\mathcal{N}}_{\widehat{\lambda},W_{m}(k)}$
(by the previous proposition), we see that the image of $\Psi$ in
$\Omega_{(L_{U'_{k}})^{(m)}}^{1}[[\lambda]][\lambda^{-1}]$ vanishes.
So $\Psi$ vanishes also, and we see that $\mathcal{M}_{\lambda,W_{m}(k)}^{\sigma}[\lambda^{-1}]$
and $\overline{\mathcal{N}}_{\lambda,W_{m}(k)}[\lambda^{-1}]$ are
locally isomorphic. 

To see that they are locally isomorphic over all of $U'_{W_{m}(k)}\times\mathbb{A}_{W_{m}(k)}^{1}$,
note that, as they are both flat over $W_{m}(k)$, any map whose reduction
mod $p$ is an isomorphism is an isomorphism. We have a natural map
\[
\mathcal{H}om(\mathcal{M}_{\lambda,W_{m}(k)}^{\sigma},\overline{\mathcal{N}}_{\lambda,W_{m}(k)})/p\to\mathcal{H}om(\mathcal{M}_{\lambda,k}^{\sigma},\overline{\mathcal{N}}_{\lambda,k})
\]
The right hand side is a line bundle over $\mathcal{O}_{L_{U'_{k}}^{(1)}}[\lambda]$
(since we already know that $\mathcal{M}_{\lambda,k}^{\sigma}$ and
$\overline{\mathcal{N}}_{\lambda,k}$ are locally isomorphic on $U'_{k}\times\mathbb{A}_{k}^{1}$).
The left hand side is a coherent sheaf over $\mathcal{O}_{W_{m}(L_{U'_{k}}^{(m)})}[\lambda]$,
and it is a line bundle there; to see this, note that by flat descent
it can be checked after passing to $U'_{W_{m}(k)}[\lambda,\lambda^{-1}]$
and $U'_{W_{m}(k)}[[\lambda]]$ where the result is clear as the two
sheaves are locally isomorphic on each of these pieces. 

Let $\mathcal{H}$ denote the image of the above map, it is a sub
$\mathcal{O}_{L_{U'_{k}}^{(m)}}$-module of $\mathcal{H}om(\mathcal{M}_{\lambda,k}^{\sigma},\overline{\mathcal{N}}_{\lambda,k})$,
and we have 
\[
\mathcal{O}_{L_{U'_{k}}^{(1)}}\otimes_{\mathcal{O}_{L_{U'_{k}}^{(m)}}}\mathcal{H}\to\mathcal{H}om(\mathcal{M}_{\lambda,k}^{\sigma},\overline{\mathcal{N}}_{\lambda,k})
\]
is an isomorphism; again by checking it on the same flat cover. Therefore
$\mathcal{H}$ locally contains a non-vanishing section of the line
bundle $\mathcal{H}om(\mathcal{M}_{\lambda,k}^{\sigma},\overline{\mathcal{N}}_{\lambda,k})$,
which is necessarily an isomorphism, and the result follows.
\end{proof}
Now, applying $\mathcal{E}nd(\mathcal{M}_{\lambda,W_{m}(k)}^{\sigma})\tilde{=}\mathcal{O}_{W_{m}(L_{U'_{k}}^{(m)})}[\lambda]$,
we see that the two sheaves $\mathcal{M}_{\lambda,W_{m}(k)}^{\sigma}$
and $\overline{\mathcal{N}}_{\lambda,W_{m}(k)}$ define a unique line
bundle on $W_{m}(L_{U'_{k}}^{(m)})\times\mathbb{A}_{W_{m}(k)}^{1}$.
Since $W_{m}(L_{U'_{k}}^{(m)})$ is affine and is an infinitesimal
deformation of a regular scheme, we see that 
\[
\text{Pic}(W_{m}(L_{U'_{k}}^{(m)})\times\mathbb{A}_{W_{m}(k)}^{1})\tilde{\to}\text{Pic}(L_{U'_{k}}^{(m)}\times\mathbb{A}_{k}^{1})\tilde{\to}\text{Pic}(L_{U'_{k}}^{(m)})\tilde{\to}\text{Pic}(W_{m}(L_{U'_{k}}^{(m)}))
\]
Ans so our line bundle is the pullback of a unique line bundle on
$W_{m}(L_{U'_{k}}^{(m)})$; we therefore obtain via \prettyref{prop:Cartier-Over-W_m},
a line bundle with connection $(\mathcal{K}_{W_{m}(k)},\nabla)$ on
$L_{U'_{W_{m}(k)}}$. Taking the inverse limit, we obtain a line bundle
with continuous connection $(\mathcal{K}_{W(k)},\nabla)$ on the formal
scheme $L_{\mathfrak{U}'_{W(k)}}$. 
\begin{cor}
\label{cor:K-is-trivial}The connection $(\mathcal{K}_{W(k)},\nabla)$
is trivial on $L_{\mathfrak{U}'_{W(k)}}$; i.e., it is isomorphic
to $\mathcal{O}_{L_{\mathfrak{U}'_{W(k)}}}$ with the standard connection.
Therefore $\mathcal{\widehat{M}}_{W(k)}^{\sigma}\tilde{\to}\widehat{\mathcal{N}}_{W(k)}$,
where $\widehat{?}$ denotes the $p$-adic completion. 
\end{cor}

\begin{proof}
By construction we have, for each $m\geq1$, that $\mathcal{M}_{\lambda,W_{m}(k)}^{\sigma}\tilde{=}\pi_{*}(\mathcal{O}_{L_{U'_{W_{m}(k)}}}[\lambda],\varphi_{m})$
where $\varphi$ is a closed one-form in $\Omega_{L_{U'_{W_{m}(k)}}}^{1}[\lambda]$,
and similarly $\mathcal{N}_{\lambda,W_{m}(k)}\tilde{=}\pi_{*}(\mathcal{O}_{L_{U'_{W_{m}(k)}}}[\lambda],\varphi_{m}')$.
As the map $L_{U'}\to U'$ is etale, a (local) isomorphism between
these connections corresponds to a (local) isomorphism between the
connections $(\mathcal{O}_{L_{U'_{W_{m}(k)}}}[\lambda],\varphi_{m})$
and $(\mathcal{O}_{L_{U'_{W_{m}(k)}}}[\lambda],\varphi_{m}')$; as
is seen by pulling back along $L_{U'_{W_{m}(k)}}\to U'_{W_{m}(k)}$. 

If we set $\psi_{m}=\varphi_{m}'-\varphi_{m}$, then, because $\mathcal{M}_{\lambda,W_{m}(k)}^{\sigma}$
and $\mathcal{N}_{\lambda,W_{m}(k)}$ are locally isomorphic, we have
that $\psi_{m}$ is a closed one form which is locally trivial, i.e.,
${\displaystyle \psi_{m}=\lambda\frac{du}{u}}$ locally on $L_{U'_{W_{m}(k)}}$.
In particular we have $\psi_{m}=\lambda\psi_{0,m}$, where $\psi_{0,m}\in\Omega_{L_{U'_{W_{m}(k)}}}^{1}$.
The form $\psi_{0,m}$ defines a flat connection, which is precisely
$\mathcal{K}_{W_{m}(k)}$. Taking the inverse limit over $m$, we
obtain a closed one form $\psi_{0}\in\Omega_{L_{\mathfrak{U}'_{W(k)}}}^{1}$
which defines the connection $(\mathcal{K}_{W(k)},\nabla)$. 

Now, by construction there is an isomorphism 
\[
\mathcal{M}_{\lambda}^{\sigma}/\lambda^{2}\tilde{=}\mathcal{N}_{\lambda}/\lambda^{2}
\]
of $\lambda$-connections over $U'_{R}$. Base changing to $\mathfrak{U}'_{W(k)}$,
we obtain an isomorphism 
\[
\widehat{\mathcal{M}}_{\lambda,W(k)}^{\sigma}/\lambda^{2}\tilde{=}\mathcal{\widehat{N}}_{\lambda,W(k)}/\lambda^{2}
\]
or, equivalently, 
\[
\pi_{*}(\mathcal{O}_{L_{\mathfrak{U}_{W(k)}}}[\lambda]/\lambda^{2},\varphi)\tilde{\to}\pi_{*}(\mathcal{O}_{L_{\mathfrak{U}_{W(k)}}}[\lambda/\lambda^{2}],\varphi+\lambda\psi_{0})
\]
where $\varphi$ is the inverse limit of the $\varphi_{m}$, and we've
used $\varphi_{m}'=\varphi_{m}+\lambda\psi_{0,m}$. But this means
that there is a unit $u$ in $\mathcal{O}_{L_{\mathfrak{U}_{W(k)}}}$
so that $\lambda{\displaystyle \frac{du}{u}=\lambda\psi_{0}}$which
shows that the connection defined by $\psi_{0}$ is trivial as required. 
\end{proof}
Now we proceed to the
\begin{proof}
(of \prettyref{thm:M-sigma-is-theta-reg}) We are going to show that,
for any $k$ of large enough characteristic, we have a nonzero morphism
$(\overline{\mathcal{N}}'_{1,W(k)}|_{U_{W(k)}})\to\mathcal{M}_{1,W(k)}^{\sigma}|_{U_{W(k)}}$.
Since both of these are irreducible flat connections over $\text{Frac}(W(k))$,
after enlarging $R$ we obtain the desired isomorphism. Let $\mathcal{V}_{W(k)}:=\overline{\mathcal{N}}'_{1,W(k)}\otimes(\mathcal{M}_{1,W(k)}^{\sigma})^{*}$.
This is an algebraic vector bundle with flat connection on $U'_{W(k)}$.
Let $\widehat{\mathcal{V}}_{W(k)}$ denote its $p$-adic completion.
Then, we have the short exact sequence 
\[
0\to\mathcal{V}_{W(k)}\to\widehat{\mathcal{V}}_{W(k)}\to\mathcal{Q}_{W(k)}\to0
\]
where $\mathcal{Q}_{W(k)}$ is a module with flat connection over
$\mathcal{O}_{U_{W(k)}}$ on which $p$ acts invertible. Thus we obtain
the long exact sequence 
\[
\mathbb{H}_{dR}^{0}(\mathcal{Q}_{W(k)})\to\mathbb{H}_{dR}^{1}(\mathcal{V}_{W(k)})\to\mathbb{H}_{dR}^{1}(\widehat{\mathcal{V}}_{W(k)})\to\mathbb{H}_{dR}^{1}(\mathcal{Q}_{W(k)})
\]
and \prettyref{lem:p-adically-seperated} implies that $\mathbb{H}_{dR}^{1}(\mathcal{V}_{W(k)})$
is $p$-adically seperated. Thus the first map in this sequence is
zero, and we see that the induced map
\[
p-\text{tors}(\mathbb{H}_{dR}^{1}(\mathcal{V}_{W(k)}))\to p-\text{tors}(\mathbb{H}_{dR}^{1}(\widehat{\mathcal{V}}_{W(k)}))
\]
is an isomorphism. On the other hand, since $\mathcal{V}_{W(k)}$
is flat over $W(k)$ we have $\mathbb{H}_{dR}^{\cdot}(\mathcal{V}_{k})\tilde{=}\mathbb{H}_{dR}^{\cdot}(\mathcal{V}_{W(k)})\otimes_{W(k)}^{L}k$.
Therefore we have the short exact sequence 
\[
\mathbb{H}_{dR}^{0}(\mathcal{V}_{W(k)})/p\to\mathbb{H}_{dR}^{0}(\mathcal{V}_{k})\to p-\text{tors}(\mathbb{H}_{dR}^{1}(\mathcal{V}_{W(k)}))
\]
and the analogous one for $\widehat{\mathcal{V}}_{W(k)}$. Therefore
the map 
\[
\mathbb{H}_{dR}^{0}(\mathcal{V}_{W(k)})/p\to\mathbb{H}_{dR}^{0}(\widehat{\mathcal{V}}_{W(k)})/p
\]
is an isomorphism. Further, since $\overline{\mathcal{N}}'_{1,W(k)}$
and $\mathcal{M}_{1,W(k)}^{\sigma}$ become isomorphic after $p$-adic
completion, we see that $\mathbb{H}_{dR}^{0}(\widehat{\mathcal{V}}_{W(k)})\neq0$
(in fact $\mathbb{H}_{dR}^{0}(\widehat{\mathcal{V}}_{W(k)})=W(k)$
as can be easily seen from the fact that both $\overline{\mathcal{N}}'_{1,W(k)}$
and $\mathcal{M}_{1,W(k)}^{\sigma}$ are flat deformations of a splitting
bundle on $L_{U'_{k}}$). Therefore $\mathbb{H}_{dR}^{0}(\mathcal{V}_{W(k)})\neq0$
and the result follows. 
\end{proof}

\subsection{\label{subsec:Calculation-of-the-p-support}Calculation of the $p$-support}

Now we prove the first main assertion of \prettyref{thm:1}; namely,
that $\mathcal{M}$ has constant arithmetic support. Let $\mathfrak{X}_{W(k)}$
denote the formal completion of $X_{W(k)}$, considered as an affine
formal scheme. As in \prettyref{def:D-complete}, we let $\mathcal{\widehat{D}}_{\mathfrak{X}_{W(k)}}^{(0)}$
denote the $p$-adic completion of $\mathcal{D}_{X_{W(k)}}$; it is
a sheaf of noetherian, coherent algebras on $\mathfrak{X}_{W(k)}$.
We let $\mathfrak{U}_{W(k)}$ denote the formal scheme which is the
$p$-adic completion of $U_{W(k)}$; let $j:\mathfrak{U}_{W(k)}\to\mathfrak{X}_{W(k)}$
denote the inclusion. Finally, let $\widehat{\mathcal{M}}_{W(k)}$
denote the $p$-adic completion of $\mathcal{M}_{W(k)}$, and $\widehat{\mathcal{M}}_{\mathfrak{U},W(k)}$
the $p$-adic completion of $\mathcal{M}_{U,W(k)}$. Then, in \prettyref{subsec:V-filtrations}
below, we prove
\begin{thm}
\label{thm:Injectivity-for-M}The natural map $\widehat{\mathcal{M}}_{W(k)}\to j_{*}(\widehat{\mathcal{M}}_{\mathfrak{U},W(k)})$
is injective. 
\end{thm}

Now, $\widehat{\mathcal{M}}_{\mathfrak{U},W(k)}/p$ is a splitting
bundle for $\mathcal{D}_{U_{k}}|_{L_{U_{k}}^{(1)}}$. As we know from
\prettyref{lem:Extend-mod-p}, this splitting bundle extends to a
splitting bundle on all of $L_{k}^{(1)}$. Our next theorem shows
that (at least one such) extension admits a lift to $W(k)$. To state
this theorem, we recall from \cite{key-80} that the center of $\mathcal{D}_{X_{W_{n}(k)}}$
is isomorphic to the ring of Witt vectors on the $n$th Frobenius
twist of the cotangent bundle of $X_{k}$; $W_{n}(T^{*}X_{k}^{(n)})$.
In particular, the underlying topological space of $\text{Spec}(\mathcal{Z}(\mathcal{D}_{X_{W_{n}(k)}}))$
can be identified with that of $T^{*}X_{k}^{(n)}$.
\begin{thm}
\label{thm:p-adic-IC-v1}1) There is an open subset $V$ of $T^{*}X_{k}$,
whose intersection with $L_{k}$ has complement of codimension $2$
inside $L_{k}$, and such that $\mathcal{M}_{k}|_{V^{(1)}}$ is a
splitting bundle for $L_{k}^{(1)}\cap V^{(1)}$. For each $n\geq1$,
$\mathcal{M}_{W_{n}(k)}|_{W_{n}(V^{(n)})}$ is scheme theoretically
supported along $W_{n}(L_{k}^{(n)}\cap V^{(n)})$, and is a $W_{n}(k)$-flat
deformation of $\mathcal{M}_{k}|_{V^{(1)}}$.

2) Let $j:V\to T^{*}X_{k}$ denote the inclusion. Then, for each $n\geq1$,
$j_{*}(\mathcal{M}_{W_{n}(k)}|_{W_{n}(V^{(n)})})$ is a flat deformation
of $j_{*}(\mathcal{M}_{k}|_{V^{(1)}})$, which is a splitting bundle
on $L_{k}^{(1)}$. The module 
\[
\lim_{n\to\infty}j_{*}(\mathcal{M}_{W_{n}(k)}|_{W_{n}(V^{(n)})}):=\tilde{\mathcal{M}}_{W(k)}
\]
 is therefore a $p$-adically complete, $p$-torsion-free, finite
$\mathcal{\widehat{D}}_{\mathfrak{X}_{W(k)}}^{(0)}$-module $\tilde{\mathcal{M}}_{W(k)}$
such that $\tilde{\mathcal{M}}_{W(k)}/p$ is a splitting bundle on
$L_{k}^{(1)}$. Further, there is an isomorphism 
\[
\mathcal{\widehat{D}}_{\mathfrak{U}_{W(k)}}^{(0)}\widehat{\otimes}_{\mathcal{\widehat{D}}_{\mathfrak{X}_{W(k)}}^{(0)}}\tilde{\mathcal{M}}_{W(k)}\tilde{\to}j_{*}(\widehat{\mathcal{M}}_{\mathfrak{U},W(k)})
\]
where the term on the left denotes the $p$-adic completion of $\mathcal{\widehat{D}}_{\mathfrak{U}_{W(k)}}^{(0)}\otimes_{\mathcal{\widehat{D}}_{\mathfrak{X}_{W(k)}}^{(0)}}\tilde{\mathcal{M}}_{W(k)}$. 

In particular, there is a map $\widehat{\mathcal{M}}_{W(k)}\to\tilde{\mathcal{M}}_{W(k)}$
given by taking the inverse limit of the natural maps 
\[
\mathcal{M}_{W_{n}(k)}\to j_{*}(\mathcal{M}_{W_{n}(k)}|_{W_{n}(V^{(n)})})
\]
This map is surjective.
\end{thm}

Combining these two theorems yields immediately
\begin{thm}
\label{thm:generic-fibre-comparison}There is an isomorphism $\tilde{\mathcal{M}}_{W(k)}\tilde{=}\widehat{\mathcal{M}}_{W(k)}$.
In particular, the $p$-support of $\mathcal{M}_{k}$ is $L_{k}^{(1)}$,
with multiplicity $1$; in fact $\mathcal{M}_{k}$ is a splitting
bundle for $\mathcal{D}_{X_{k}}$ along $L_{k}^{(1)}$
\end{thm}

Now we proceed to the 
\begin{proof}
(of \prettyref{thm:p-adic-IC-v1}) 1) As in the previous section we
fix a component $E'$ of the divisor $E=L\backslash L_{U}$, and a
$\sigma$ so that the projection $T^{*}\mathbb{A}^{m}\to\mathbb{A}^{m}$
is an isomorphism in a formal neighborhood of $\sigma(E')$ in $\sigma(L)$.
Then, there is a component $D'$ of $\mathbb{A}^{m}\backslash U'$
so that $E'\to D'$ is generically an isomorphism.

Then, by \prettyref{thm:M-sigma-is-theta-reg}, we have that $\mathcal{M}^{\sigma}$
is the restriction to $U'$ of the reduction mod $(\lambda-1)$ of
$\overline{\mathcal{M}_{\lambda}^{\sigma}}$, which is $\theta$-regular
at infinity. Let $\tilde{\mathcal{M}}_{1}^{\sigma}$ denote the $\mathcal{D}_{\mathbb{A}^{m}}$-submodule
of $j_{*}(\mathcal{M}_{1}^{\sigma})$ generated by $\overline{\mathcal{M}_{1}^{\sigma}}|_{\mathbb{A}^{m}}$.
Then, localizing $R$ as needed, we have $\sigma^{*}\mathcal{M}\subset\tilde{\mathcal{M}}_{1}^{\sigma}$,
with cokernel flat over $R$. Thus we can suppose $\sigma^{*}\mathcal{M}_{W_{n}(k)}\subset\tilde{\mathcal{M}}_{1,W_{n}(k)}^{\sigma}$
for all $k$ with $R\to k$, and all $n\geq1$. 

By the $\theta$-regularity, we have for each $n\geq1$ a decomposition
\[
(\overline{\mathcal{M}_{\lambda}^{\sigma}})_{W_{n}(k)}=\bigoplus_{i}(\overline{\mathcal{M}_{\lambda}^{\sigma}})_{W_{n}(k),i}
\]
in an etale neighborhood $V_{x,W_{n}(k)}$ of the generic point of
$D'_{W_{n}(k)}$, and the sum is over components of the normalization
$\overline{\sigma(L)}$ of $\tilde{\mathbb{P}}^{m}$ in $K(\sigma(L))$
which live over $D'$. Since $\sigma(L)$ is smooth, we see that this
collection of components includes each component of $\sigma(L)$ which
lives over $D'$, including $E'$. This implies that the component
containing $E'$ is embedded in $T^{*}(\tilde{\mathbb{P}}^{m})$,
with trivial residue. Therefore, by the $\theta$-regularity condition,
the summand $(\widehat{\overline{\mathcal{M}_{\lambda}^{\sigma}}})_{W_{n}(k),i}$
which corresponds to $E'$ is itself a line bundle with $\lambda$-connection
over $\widehat{\mathcal{O}}_{D_{W_{n}(k)}'}[\lambda]$. Similarly,
by the choice of residue in \prettyref{thm:M-sigma-is-theta-reg},
the summands corresponding to other components of $\sigma(L)$ living
above $D'$ are vector bundles with $\lambda$-connection, while the
remaining summands correspond to divisors in $\overline{\sigma(L)}\backslash\sigma(L)$.
Taking the inverse limit, the same holds over $W(k)$.

Now we show that, after pulling back to $V_{x}$, we have the decomposition
\begin{equation}
\tilde{\mathcal{M}}_{1,k}=\bigoplus_{i,\text{finite}}(\overline{\mathcal{M}_{1}^{\sigma}})_{k,i}\oplus\bigoplus_{i,\text{infinite}}(\overline{\mathcal{M}_{1}^{\sigma}})_{k,i}[z^{-1}]\label{eq:Decomp-for-M-1-k}
\end{equation}
where the first sum runs over divisors in $\sigma(L)$, while the
second runs over divisors in $\overline{\sigma(L)}\backslash\sigma(L)$,
and $\{z=0\}$ is a local equation for $D'$. 

To see it, use the fact that, after taking a suitable root cover,
each $(\overline{\mathcal{M}_{1}^{\sigma}})_{W_{n}(k),i}$ (for $i$
infinite) becomes a direct sum of line bundles with connection, where
the connection is of the form ${\displaystyle \frac{d}{dz}e=f_{i}\cdot e}$
where $f_{i}$ has a pole in $z$. The $\mathcal{\widehat{D}}_{\mathfrak{V}_{x,W(k)}}^{(0)}$-module\footnote{here the $\mathfrak{V}_{x}$ is the $p$-adic completion of $V_{x}$}
generated by $(\overline{\mathcal{M}_{1}^{\sigma}})_{W(k),i}:=\lim_{n}(\overline{\mathcal{M}_{\lambda}^{\sigma}})_{W_{n}(k),i}$
is therefore equal to $\widehat{(\overline{\mathcal{M}_{1}^{\sigma}})_{W(k),i}[z^{-1}]}$.
For $i$ finite each $(\overline{\mathcal{M}_{1}^{\sigma}})_{W(k),i}$
is already a vector bundle with connection. So we see that the $\mathcal{D}_{\mathfrak{V}_{x,W(k)}}$-module
generated by $\{(\overline{\mathcal{M}_{1}^{\sigma}})_{W(k),i}\}$,
which is the $p$-adic completion of $\tilde{\mathcal{M}}_{1,Wk)}|_{V_{x,W(k)}}$,
is given by 
\[
\bigoplus_{i,\text{finite}}(\overline{\mathcal{M}_{1}^{\sigma}})_{W(k),i}\oplus\bigoplus_{i,\text{infinite}}\widehat{(\overline{\mathcal{M}_{1}^{\sigma}})_{Wk),i}[z^{-1}]}
\]
Thus $\tilde{\mathcal{M}}_{1,k}|_{V_{x}}$, which is the reduction
mod $p$ of the $p$-adic completion of $\tilde{\mathcal{M}}_{1,W(k)}|_{V_{x,W(k)}}$,
is exactly given by \prettyref{eq:Decomp-for-M-1-k} above. 

This implies that the $\mathcal{D}$-module $\tilde{\mathcal{M}}_{1,k}^{\sigma}$,
after restriction to an etale neighborhood of the generic point of
$D'$, has $p$-support equal to $\sigma(L_{k}^{(1)})$; and in fact
is a line bundle when restricted to $\widehat{\sigma(E')}_{k}$. Thus
$\sigma^{*}\mathcal{M}_{k}$ is torsion free of rank $1$ over $\sigma(L_{k}^{(1)})$
near $\sigma(E')_{k}$, and is therefore a line bundle in a neighborhood
of the generic point of $\sigma(E')_{k}$. Since $E'$ and $\sigma$
were arbitrary, we see that $\mathcal{M}_{k}$ is a vector bundle
on $L_{k}^{(1)}$ in codimension $2$, which is the first statement
of $1)$; the second follows as $\mathcal{M}_{W_{n}(k)}$ is flat
over $W_{n}(k)$. 

2) Since $\mathcal{M}_{k}|_{V^{(1)}}$ is a splitting bundle on $V^{(1)}$,
which has codimension $2$ in $L_{k}^{(1)}$, it follows (since $\mathcal{D}$
is split on $L_{k}^{(1)}$ by \prettyref{lem:Extend-mod-p}) that
$j_{*}(\mathcal{M}_{k}|_{v^{(1)}})$ is a splitting bundle on all
of $L_{k}^{(1)}$. To obtain the mixed-characteristic result, it suffices
to show that $j_{*}(\mathcal{M}_{W_{n}(k)}|_{W_{n}(V^{(n)})})$ is
a flat deformation of $j_{*}(\mathcal{M}_{k}|_{V^{(1)}})$. To do
so, choose a closed point $x\in L_{k}$. Let $e$ be a local section
of $j_{*}(\mathcal{M}_{k}|_{V^{(1)}})$ near $x$ which generates
it as a $\mathcal{D}$-module. We have the canonical morphism $\mathcal{M}_{k}\to j_{*}(\mathcal{M}_{k}|_{V^{(1)}})$.
We claim that this morphism is a surjection.

To prove this, we apply the $\mathcal{D}$-module duality functor
\[
\mathbb{D}\mathcal{E}=\mathcal{R}Hom(\mathcal{E},\mathcal{D})\otimes\omega_{X_{k}}[n]
\]
Since $\mathcal{M}_{k}$ is the reduction of a holonomic $\mathcal{D}_{X_{\mathbb{C}}}$-module,
we have that $\mathbb{D}\mathcal{M}_{k}$ is concentrated in degree
$0$. The same is also true for $j_{*}(\mathcal{M}_{k}|_{V^{(1)}})$;
this can be seen directly from the fact that $j_{*}(\mathcal{M}_{k}|_{V^{(1)}})$
is a splitting bundle for $L_{k}^{(1)}$. Furthermore, we can apply
all of the above constructions to the irreducible $\mathcal{D}_{X_{\mathbb{C}}}$
module $\mathbb{D}\mathcal{M}_{\mathbb{C}}$; which is the IC extension
of the dual vector bundle $\mathcal{M}^{*}|_{U_{\mathbb{C}}}$. This
yields a splitting bundle $j_{*}(\mathcal{M}_{k}^{*}|_{V^{(1)}})$,
which is the dual to $j_{*}(\mathcal{M}_{k}|_{V^{(1)}})$, and a map
$\mathbb{D}\mathcal{M}_{k}\to j_{*}(\mathcal{M}_{k}^{*}|_{V^{(1)}})$.
Taking the dual of this map, we obtain a map $j_{*}(\mathcal{M}_{k}|_{V^{(1)}})\to\mathbb{D}\mathbb{D}\mathcal{M}_{k}=\mathcal{M}_{k}$.
Composing with $\mathcal{M}_{k}\to j_{*}(\mathcal{M}_{k}|_{V^{(1)}})$
we get an isomorphism on $j_{*}(\mathcal{M}_{k}|_{V^{(1)}})$ (this
follows by restricting to $V^{(1)}$, where this map is clearly an
isomorphism). Therefore $\mathcal{M}_{k}\to j_{*}(\mathcal{M}_{k}|_{V^{(1)}})$
is a surjection as claimed; indeed, it is a split surjection. 

Now we proceed to prove point $2)$ by induction on $n$. Assume that
$j_{*}(\mathcal{M}_{W_{n-1}(k)}|_{V^{(n-1)}})$ is a flat deformation
of $j_{*}(\mathcal{M}_{k}|_{V^{(1)}})$ and that $\mathcal{M}_{W_{n-1}(k)}\to j_{*}(\mathcal{M}_{W_{n-1}(k)}|_{W^{(n-1)}})$
is onto. Choose an element $\bar{e}\in\mathcal{M}_{W_{n-1}(k)}$ whose
image in $j_{*}(\mathcal{M}_{W_{n-1}(k)}|_{V^{(n-1)}})$ is a generator
(over $\mathcal{D}_{X_{W_{n-1}(k)}}$). Let $e\in\mathcal{M}_{W_{n}(k)}$
be a lift of $\bar{e}$; we denote also by $e$ its image in $j_{*}(\mathcal{M}_{W_{n}(k)}|_{V^{(n)}})$
We have the exact sequence 
\[
0\to j_{*}(p\cdot\mathcal{M}_{W_{n}(k)}|_{V^{(n)}})\to j_{*}(\mathcal{M}_{W_{n}(k)}|_{V^{(n)}})\to j_{*}(\mathcal{M}_{k}|_{V^{(1)}})
\]
The right hand map is surjective, as the image of $e$ under $\mathcal{M}_{k}\to j_{*}(\mathcal{M}_{k}|_{V^{(1)}})$
generates it. By the induction hypothesis, $j_{*}(p\cdot\mathcal{M}_{W_{n}(k)}|_{V^{(n)}})$
is generated by the image of $pe$, which implies $j_{*}(p\cdot\mathcal{M}_{W_{n}(k)}|_{V^{(n)}})=p\cdot j_{*}(\mathcal{M}_{W_{n}(k)}|_{V^{(n)}})$.
Therefore
\[
j_{*}(\mathcal{M}_{W_{n}(k)}|_{V^{(n)}})/p\cdot j_{*}(\mathcal{M}_{W_{n}(k)}|_{V^{(n)}})\tilde{\to}j_{*}(\mathcal{M}_{k}|_{V^{(1)}})
\]
So $\mathcal{M}_{W_{n}(k)}\to j_{*}(\mathcal{M}_{W_{n}(k)}|_{V^{(n)}})$
is onto by Nakayama's lemma, and the image of $e$ generates $j_{*}(\mathcal{\mathcal{M}}_{W_{n}(k)}|_{V^{(n)}})$.
Now, let $m\in j_{*}(\mathcal{\mathcal{M}}_{W_{n}(k)}|_{V^{(n)}})$.
If $m\in\text{ker}(p^{n-1}\cdot)$, then $m\in\text{ker}(j_{*}(\mathcal{\mathcal{M}}_{W_{n}(k)}|_{V^{(n)}})\to j_{*}(\mathcal{M}_{k}|_{V^{(1)}}))$
since $\mathcal{M}_{W_{n}(k)}|_{V^{(n)}}$ is $W_{n}(k)$-flat. Therefore
$m\in\text{image}(p\cdot)$; so $j_{*}(\mathcal{\mathcal{M}}_{W_{n}(k)}|_{W^{(n)}})$
is flat over $W_{n}(k)$ by induction, and the result follows. The
surjectivity of the map $\widehat{\mathcal{M}}_{W(k)}\to\text{\ensuremath{\tilde{\mathcal{M}}}}_{W(k)}:={\displaystyle \lim_{n}j_{*}(\mathcal{\mathcal{M}}_{W_{n}(k)}|_{W^{(n)}})}$
now follows from the complete Nakayama lemma. 
\end{proof}

\subsection{\label{subsec:Torsor-Structure}Torsor Structure}

The results of the previous section ensure that we have constructed
a finite collection of holonomic $\mathcal{D}_{X_{\mathbb{C}}}$-modules,
$\{\mathcal{K}\star\mathcal{M}_{\mathbb{C}}\}$, where $\mathcal{K}$
is indexed by the character group $\pi^{*}$ of $\pi_{1}(L_{\mathbb{C}})$.
We now turn to checking that, first, each distinct $\mathcal{K}$
defines a distinct holonomic $\mathcal{D}$-module, and, second, that
any holonomic $\mathcal{D}$-module of constant arithmetic support
equal to $L_{\mathbb{C}}$ is isomorphic to such a $\mathcal{K}\star\mathcal{M}_{\mathbb{C}}$.
In fact, both statements will drop out of the same argument, and so
we turn to checking the second statement first. 

Let $\mathcal{N}_{\mathbb{C}}$ be a $\mathcal{D}$-module with constant
arithmetic support equal to $L_{\mathbb{C}}$ (with multiplicity $1$).
Our first task is to use this information about the arithmetic support
to constrain the ``behavior at infinity'' of $\mathcal{N}_{\mathbb{C}}$.
We let $\mathcal{N}$ denote an $R$-model of $\mathcal{N}_{\mathbb{C}}$,
and we let $\overline{\mathcal{N}}_{\mathbb{C}}$ denote an extension
of $\mathcal{N}_{\mathbb{C}}$ to a meromorphic connection on $\overline{\tilde{X}}_{\mathbb{C}}$. 
\begin{lem}
\label{lem:Local-Decomp-for-N}Let $\{x\}$ be the generic point of
some component of $\tilde{D}$. Let $\varphi:V_{x}\to\overline{\tilde{X}}$
be an etale neighborhood, and $V_{x}^{(l)}\to V_{x}$ a root cover.
Let $\varphi^{(l)}:V_{x}^{(l)}\to\overline{\tilde{X}}$ denote the
composed morphism. There is a decomposition
\[
(\varphi^{(l)})^{*}\widehat{\overline{\mathcal{N}}}[z_{1}^{-1}]=\bigoplus_{i=1}^{r'}\mathcal{O}_{\widehat{V_{x}^{(l)}}}[z_{1}^{-1}]\cdot e_{i}\oplus\widehat{\overline{\mathcal{N}}}[z_{1}^{-1}]_{reg}
\]
(as in \prettyref{cor:B-V-over-R}; here $\widehat{V_{x}^{(l)}}$
denotes the completion of $V_{x}^{(l)}$ along the divisor $\tilde{D}$).
We write $\partial_{j}e_{i}=\alpha_{ij}e_{i}$. 

1) After possibly renumbering, the (nonzero) image of $\alpha_{i}$
in $\mathcal{O}_{\widehat{V_{x}^{(l)}}}[z_{1}^{-1}]/\mathcal{O}_{\widehat{V_{x}^{(l)}}}\cdot z_{1}^{-1}$
agrees with the image of $\theta_{ij}$ in $\mathcal{O}_{\widehat{V_{x}^{(l)}}}[z_{1}^{-1}]/\mathcal{O}_{\widehat{V_{x}^{(l)}}}\cdot z_{1}^{-1}$. 

2) Let $R\to W(k)$. For any $m\geq1$, there is a decomposition 
\begin{equation}
\widehat{\overline{\mathcal{N}}}_{W_{m}(k)}[z_{1}^{-1}]=\bigoplus_{i=1}^{r}\mathcal{O}_{\widehat{V_{x,W_{m}(k)}^{(l)}}}[z_{1}^{-1}]\cdot e_{i}\label{eq:decomp}
\end{equation}
of meromorphic connections (i.e., we have $\nabla(e_{i})=\psi_{i}e_{i}$
for one-forms $\psi_{i}\in\Omega_{\widehat{V_{x,W_{m}(k)}^{(l)}}}^{1}[z_{1}^{-1}]$)
such that, for $1\leq i\leq r'$, we have that $e_{i}$ agrees with
the $e_{i}$ of the decomposition written above, and for $i>r'$ the
one-form $\psi_{i}$ has log poles.
\end{lem}

\begin{proof}
After reduction to $k$, the assumption on the $p$-curvature implies
that we have a decomposition 
\[
\widehat{\overline{\mathcal{N}}}_{k}[z_{1}^{-1}]=\bigoplus_{i=1}^{r}\mathcal{O}_{\widehat{V_{x,k}^{(l)}}}[z_{1}^{-1}]\cdot e_{i}
\]
of meromorphic connections; if we write $\nabla(e_{i})=\psi_{i}e_{i}$,
then we have $\nabla^{(p)}(e_{i})=(\psi_{i}^{p}-C(\psi_{i}))e_{i}$;
hence this is a complete set of eigenvectors for the operators $\{\partial_{i}^{p}\}_{i=1}^{n}$;
further, we have $(\psi_{i}^{p}-C(\psi_{i}))=\theta_{i}^{p}$. On
the other hand, the elements $\alpha_{i}^{p}-(\partial_{z_{1}})^{p-1}\alpha_{i}$
are amongst the eigenvalues for the operator $\partial_{z}^{p}$.
Writing $\alpha_{i}=a_{i,m}z^{-m}+\cdots+a_{i,2}z^{-2}+a_{i,1}z^{-1}+a_{0}$
(where $a_{0}$ has no poles in $z$) shows that 
\[
\alpha_{i}^{p}-(\partial_{j})^{p-1}\alpha_{i}=a_{i,m}^{p}z^{-mp}+\cdots a_{i,2}^{p}z^{-2p}+(a_{i,1}^{p}-a_{i,1})z^{-p}+a_{0}'
\]
(where $a_{0}'$ has no poles in $z$). Comparing these two sets of
eigenvectors thus gives $1)$. For $2)$, the decomposition \prettyref{eq:decomp}
follows from the fact that there are no nontrivial extensions between
the distinct connections $\{\mathcal{O}_{\widehat{V_{x,k}^{(l)}}}[z_{1}^{-1}]\cdot e_{i}\}_{i=1}^{r}$.
Comparing with the reduction to $W_{m}(k)$ of the decomposition 
\[
\widehat{\overline{\mathcal{N}}}[z_{1}^{-1}]=\bigoplus_{i=1}^{r'}\mathcal{O}_{\widehat{V_{x}^{(l)}}}[z_{1}^{-1}]\cdot e_{i}\oplus\widehat{\overline{\mathcal{N}}}[z_{1}^{-1}]_{reg}
\]
we see that the eigenvectors must agree, for each operator $\{\partial_{i}\}_{i=1}^{n}$
and the result follows.
\end{proof}
Now we are going to attach an element of $\pi^{*}$ to $\mathcal{N}_{\mathbb{C}}$.
Recall from (the proof of) \cite{key-7}, theorem 3.1.1, that $\mathcal{N}_{k}$
is a pure sheaf for all $k$ of characteristic $p>>0$. In particular,
since the $p$-support of $\mathcal{N}_{k}$ is the smooth scheme
$L_{k}^{(1)}$, we see that $\mathcal{N}_{k}$ is torsion-free as
a sheaf on $L_{k}^{(1)}$, and hence a vector bundle in codimension
$2$. As $\mathcal{D}_{X_{k}}$ is split on $L_{k}^{(1)}$, and $\mathcal{N}_{k}$
has multiplicity $1$, we see that in fact $\mathcal{N}_{k}$ is a
splitting bundle for $L_{k}^{(1)}$ in codimension $2$. As splitting
bundles form a torsor over $L_{k}^{(1)}$, there is a unique line
bundle $\mathcal{K}'_{k}$ on $L_{k}^{(1)}$ such that $\mathcal{K}'_{k}\star\mathcal{N}_{k}=\mathcal{M}_{k}$
(in codimension $2$). Let $(\mathcal{K}_{k},\nabla):=F^{*}\mathcal{K}_{k}'$
be the locally trivial connection on $L_{k}$ attached to $\mathcal{K}_{k}'$. 

Let us consider how to lift this to mixed characteristic. We note
that, since $\mathcal{N}_{W_{n}(k)}$ is a deformation of $\mathcal{N}_{k}$
for each $n\geq1$, it follows from \prettyref{cor:Support-of-N-W_m(k)}
that $\mathcal{N}_{W_{n}(k)}$ is scheme theoretically supported on
$W_{n}(L_{k}^{(n)})$. 
\begin{prop}
\label{prop:K_n-is-locally-trivial}For each $n\geq1$, the sheaves
$\mathcal{N}_{W_{n}(k)}$ and $\mathcal{M}_{W_{n}(k)}$ are locally
isomorphic in codimension $2$ on $W_{n}(L_{k}^{(n)})$; thus there
is a line bundle $(\mathcal{K}_{W_{n}(k)},\nabla)$ with locally trivial
connection (in codimension $2$ on $L_{W_{n}(k)}$) such that $\mathcal{K}_{W_{n}(k)}\star\mathcal{M}_{W_{n}(k)}\tilde{=}\mathcal{N}_{W_{n}(k)}$;
further, $(\mathcal{K}_{W_{n}(k)},\nabla)$ is a lift of $(\mathcal{K}_{k},\nabla)$. 
\end{prop}

\begin{proof}
For $n=1$ this is discussed above; we proceed by induction and suppose
we have constructed $\mathcal{K}_{W_{n-1}(k)}$. Applying the construction
of \prettyref{prop:p^m-curv-over-L}, we obtain a one-form $\Psi\in\Omega_{V_{k}^{(n)}}^{1}$,
where $V_{k}$ is an open subset of codimension $2$. 

We claim that $\psi$ has log poles along each divisor in $\overline{L}_{W_{n}(k)}$.
Consider a component $E_{i}$ of $\overline{L}_{W_{n}(k)}\backslash L_{W_{n}(k)}$.
Let the image of $E_{i}$ in $\overline{\tilde{X}}_{W_{n}(k)}$ be
a divisor $D_{i}$. Then, we can apply \prettyref{eq:decomp} over
the generic point of $D_{i}$; after pulling back to $\widehat{V}_{x}$
we have 
\[
\widehat{\overline{\mathcal{N}}}_{W_{n}(k)}[z_{1}^{-1}]=\bigoplus_{i=1}^{r}\mathcal{O}_{\widehat{V_{x,W_{n}(k)}^{(l)}}}[z_{1}^{-1}]\cdot e_{i}
\]
with $\nabla(e_{i})=\gamma_{i}$ for some $\gamma_{i}\in\Omega_{\widehat{V_{x}^{(l)}}}^{1}[z_{1}^{-1}]$
with the constraints on the poles of $\gamma_{i}$ listed in \prettyref{lem:Local-Decomp-for-N}.
By the same token, we have 
\[
\widehat{\overline{\mathcal{M}}}_{W_{n}(k)}[z_{1}^{-1}]=\bigoplus_{i=1}^{r}\mathcal{O}_{\widehat{V_{x,W_{n}(k)}^{(l)}}}[z_{1}^{-1}]\cdot e_{i}
\]
with $\nabla(e_{i})=\delta_{i}\in\Omega_{\widehat{V_{x}^{(l)}}}^{1}[z_{1}^{-1}]$
with the same constraints. In particular, by \prettyref{lem:Local-Decomp-for-N},
this difference has log poles. Computing the $p^{m}$-curvature, we
see that $\Psi$ has log poles as well. Since $\Gamma(\Omega_{\overline{L}_{k}^{(n)}}^{1}(\tilde{E}^{(n)}))=0$
we see that $\Psi=0$. Applying \prettyref{prop:Pic-for-Local-Lifts}
we deduce that $(\mathcal{K}_{W_{n-1}(k)},\nabla)$ lifts to a locally
trivial $(\mathcal{K}_{W_{n}(k)},\nabla)$, as desired. 
\end{proof}
\begin{rem}
In characteristic $0$, the assumption that $\mathbb{H}_{dR}^{1}(L_{\mathbb{C}})=0$
implies that any line bundle with connection on $L_{\mathbb{C}}$,
which has regular singularities, has finite order as a connection.
The above proof gives the analogous result in mixed characteristic:
any line bundle with connection on $L_{W_{n}(k)}$, which has log
poles along each divisor in $\overline{L}_{W_{n}(k)}$, is locally
trivial as a connection. 
\end{rem}

From this result we can deduce the 
\begin{cor}
For each $n\geq1$, the line bundle $(\mathcal{K}_{W_{n}(k)},\nabla)$
admits the structure of an $F^{s}$ module for some $s>0$ (which
does not depend on $n$). Further, $\mathcal{K}_{W_{n}(k)}$ extends
uniquely to an $F^{s}$-module on all of $L_{W_{n}(k)}$. Under the
Riemann-Hilbert correspondence for unit $F$-crystals, this sheaf
corresponds to a representation of $\pi_{1}^{ab,p'}(L_{k})$ into
$GL_{1}(W_{n}(\mathbb{F}_{p^{s}}))$. 
\end{cor}

\begin{proof}
Let $V_{k}\subset L_{k}$ be the open subset of codimension $2$ occurring
in the previous lemma. By the previous lemma and \prettyref{thm:Higher-diff-structure-on-deformations},
we see that each $\mathcal{K}_{W_{n}(k)}$ admits the structure of
a module over $\mathcal{D}_{V_{W_{n}(k)}}^{(\infty)}$. By induction
on $n$, we deduce that $j_{*}(\mathcal{K}_{W_{n}(k)})$ is a line
bundle on $L_{W_{n}(k)}$: this is clear for $n=1$, and for $n>1$
we note that the cokernel of the natural map $j_{*}(\mathcal{K}_{W_{n}(k)})\to j_{*}(\mathcal{K}_{k})$
must be $0$ since $j_{*}(\mathcal{K}_{k})$ is a simple $\mathcal{D}_{L_{k}}^{(\infty)}$-module
(being a line bundle); so the result follows by induction. 

We shall analyze this $\mathcal{D}^{(\infty)}$-module structure.
According to \cite{key-77}, section 4, giving a $\mathcal{D}_{L_{k}}^{(\infty)}$-module
structure on $\mathcal{K}_{k}$ is equivalent to giving an infinite
sequence of line bundles $\{\mathcal{K}_{k}^{(i)}\}_{i\geq0}$ with
$\mathcal{K}_{k}^{(0)}=\mathcal{K}_{k}$ and isomorphisms $F^{*}\mathcal{K}_{k}^{(i)}\tilde{\to}\mathcal{K}_{k}^{(i-1)}$.
In particular, each $\mathcal{K}_{k}^{(i)}$ possesses infinitely
many $p$th roots. On the other hand, since $H^{1}(\mathcal{O}_{\overline{L}_{k}})=0$,
we have that $\text{Pic}(\overline{L}_{k})$ is a finitely generated
abelian group; therefore the same is true of $\text{Pic}(L_{k})$.
So each $\mathcal{K}_{k}^{(i)}$ lives in the finite set of finite
order elements of $\text{Pic}(L_{k})$. Thus we must have $\mathcal{K}_{k}^{(i)}\tilde{=}\mathcal{K}_{k}^{(i+s)}$
for some $i$ and some $s>0$; but then we have $\mathcal{K}_{k}^{(j)}\tilde{=}\mathcal{K}_{k}^{(j+s)}$
for all $j$ by applying Frobenius descent. Thus $\mathcal{K}_{k}$
has the structure of an $F^{s}$-module. 

Now let us lift this to $\mathcal{K}_{W_{m}(k)}$. Since $L_{k}$
is smooth affine we may choose a lift of Frobenius to $L_{W_{m}(k)}$;
and, by \cite{key-90}, corollary 2.3, the $\mathcal{D}_{L_{W_{m}(k)}}^{(\infty)}$-module
$\mathcal{K}_{W_{m}(k)}$ determines an infinite sequence $\{\mathcal{K}_{W_{m}(k)}^{(i)}\}_{i\geq0}$
with $\mathcal{K}_{W_{m}(k)}^{(0)}=\mathcal{K}_{W_{m}(k)}$ and isomorphisms
$F^{*}\mathcal{K}_{W_{m}(k)}^{(i)}\tilde{\to}\mathcal{K}_{W_{m}(k)}^{(i-1)}$.
Each $\mathcal{K}_{W_{m}(k)}^{(i)}$ is necessarily a flat deformation
of $\mathcal{K}_{k}^{(i)}$, which makes it unique as a line bundle
(since $L_{k}$ is affine). The we again have $\mathcal{K}_{W_{m}(k)}^{(j)}\tilde{=}\mathcal{K}_{W_{m}(k)}^{(j+s)}$
for all $j$. Thus by the Riemann-Hilbert correspondence for unit
$F$-crystals (c.f. \cite{key-69}, corollary 16.2.8) there is associated
a representation of $\pi_{1}^{ab}(L_{W_{m}(k)})\tilde{\to}\pi_{1}^{ab}(L_{k})$
into $GL_{1}(W_{m}(\mathbb{F}_{p^{s}}))$. Since $GL_{1}(W_{m}(\mathbb{F}_{p^{s}}))$
has order prime to $p$, we see that this representation factors through
$\pi_{1}^{ab,p'}(L_{k})$ as claimed. 
\end{proof}
\begin{cor}
Suppose $(\mathcal{K}_{\mathbb{C}},\nabla)\star\mathcal{M}_{\mathbb{C}}\tilde{=}\mathcal{M}_{\mathbb{C}}$
for some $(\mathcal{K}_{\mathbb{C}},\nabla)$. Then $(\mathcal{K}_{\mathbb{C}},\nabla)$
is the trivial connection. 
\end{cor}

\begin{proof}
Since $(\mathcal{K}_{\mathbb{C}},\nabla)\star\mathcal{M}_{\mathbb{C}}$
has constant $p$-curvature equal to $L_{\mathbb{C}}$, with multiplicity
$1$, the previous corollary gives, for any $k$ of characteristic
$p>>0$, a line bundle $\mathcal{K}_{k}$ with an $F^{s}$-structure.
Comparing the constructions, is is clear that this line bundle is
simply the reduction to $k$ of $(\mathcal{K}_{\mathbb{C}},\nabla)|_{L_{\mathbb{C}}}$
in this case. But since $(\mathcal{K}_{\mathbb{C}},\nabla)\star\mathcal{M}_{\mathbb{C}}\tilde{=}\mathcal{M}_{\mathbb{C}}$,
we see that this is the trivial bundle for $p>>0$. This shows that
the associated representation of $\pi_{1}^{ab,p'}(L_{k})$ is trivial
for $p>>0$, which implies that the monodromy representation associated
to $(\mathcal{K}_{\mathbb{C}},\nabla)$ is also trivial, as desired.
\end{proof}
Now we prove 
\begin{thm}
\label{thm:Torsor!} There is a unique $(\mathcal{K}_{\mathbb{C}},\nabla)$
so that $\mathcal{N}_{\mathbb{C}}\tilde{=}(\mathcal{K}_{\mathbb{C}},\nabla)\star\mathcal{M}_{\mathbb{C}}$.
In particular, the set of such $\mathcal{N}_{\mathbb{C}}$ is a torsor
over $\pi^{*}$. 
\end{thm}

\begin{proof}
The uniqueness follows from the previous corollary. For the existence,
note that, due the set $\pi^{*}$ being finite, we must have some
$(\mathcal{K}_{\mathbb{C}},\nabla)$ and an infinite set of primes
for which $\mathcal{K}_{W_{n}(k)}\star\mathcal{M}_{W_{n}(k)}\tilde{=}\mathcal{N}_{W_{n}(k)}$
(for all $n$). Therefore, for such $k$, the $p$-adic completions
of these two connections are isomorphic, and we can apply (the proof
of) \prettyref{thm:M-sigma-is-theta-reg} to conclude.
\end{proof}

\subsection{Ext Vanishing}

Finally, we turn to the last part of \prettyref{thm:1}: 
\begin{thm}
\label{thm:Ext-vanishing} We have $\text{Ext}^{1}(\mathcal{M}_{\mathbb{C}},\mathcal{M}_{\mathbb{C}})=0$.
The same is true upon replacing $\mathcal{M}_{\mathbb{C}}$ by any
$(\mathcal{K}_{\mathbb{C}},\nabla)\star\mathcal{M}_{\mathbb{C}}$. 
\end{thm}

We'll simply prove the result for $\mathcal{M}_{\mathbb{C}}$, the
proof for $(\mathcal{K}_{\mathbb{C}},\nabla)\star\mathcal{M}_{\mathbb{C}}$
being identical. 

Consider an extension $\mathcal{P}_{\mathbb{C}}$ of $\mathcal{M}_{\mathbb{C}}$
by itself. Let $\overline{\mathcal{P}}_{\mathbb{C}}$ be an extension
of $\mathcal{P}_{\mathbb{C}}$ to a reflexive meromorphic connection
on $\overline{\tilde{X}}_{\mathbb{C}}$, which is a self extension
of $\overline{\mathcal{M}}_{\mathbb{C}}$ in codimension $2$. Let
$\mathcal{P}$ be an $R$-model for $\mathcal{P}_{\mathbb{C}}$ and
$\overline{\mathcal{P}}$ be an $R$-model for $\overline{\mathcal{P}}_{\mathbb{C}}$.
We have 
\begin{lem}
Let $\{x\}$ be the generic point of some component of $\tilde{D}$.
Let $\varphi:V_{x}\to\overline{\tilde{X}}_{\mathbb{C}}$ be an etale
neighborhood, and $V_{x}^{(l)}\to V_{x}$ a root cover, for which
the decomposition
\[
(\varphi^{(l)})^{*}\widehat{\overline{\mathcal{P}}}[z_{1}^{-1}]=\bigoplus_{i=1}^{2r'}\mathcal{O}_{\widehat{V_{x}^{(l)}}}[z_{1}^{-1}]\cdot e_{i}\oplus\widehat{\overline{\mathcal{P}}}[z_{1}^{-1}]_{reg}
\]
(as in \prettyref{cor:B-V-over-R}) holds. Write $\partial_{j}e_{i}=\alpha_{ij}e_{i}$. 

1) After possibly renumbering, for $i\in\{1,\dots,r'\}$ the image
of $\alpha_{ij}$ and $\alpha_{i+r',j}$ in $\mathcal{O}_{\widehat{V_{x}^{(l)}}}[z_{1}^{-1}]/\mathcal{O}_{\widehat{V_{x}^{(l)}}}\cdot z_{1}^{-1}$
agrees with the image of $\theta_{ij}$ in $\mathcal{O}_{\widehat{V_{x}^{(l)}}}[z_{1}^{-1}]/\mathcal{O}_{\widehat{V_{x}^{(l)}}}\cdot z_{1}^{-1}$.

2) Let $R\to k$. There is a decomposition 
\begin{equation}
\widehat{\overline{\mathcal{P}}}_{k}[z_{1}^{-1}]=\bigoplus_{i=1}^{2r'}\mathcal{O}_{\widehat{V_{x,k}^{(l)}}}[z_{1}^{-1}]\cdot e_{i}\oplus\bigoplus_{i=1}^{r-r'}\widehat{\overline{\mathcal{P}}}_{i}[z_{1}^{-1}]_{reg}\label{eq:decomp-1}
\end{equation}
of meromorphic connections (i.e., we have $\nabla(e_{i})=\psi_{i}e_{i}$
for one-forms $\psi_{i}\in\Omega_{\widehat{V_{x,k}^{(l)}}}^{1}[z_{1}^{-1}]$)
such that, for $1\leq i\leq r'$, we have that the image of $\psi_{i}$
in $\Omega_{\widehat{V_{x}^{(l)}}}^{1}[z_{1}^{-1}]/\Omega_{\widehat{V_{x}^{(l)}}}^{1}\cdot z_{1}^{-1}$
agrees with the image of $\theta_{i}$ in $\Omega_{\widehat{V_{x}^{(l)}}}^{1}[z_{1}^{-1}]/\Omega_{\widehat{V_{x}^{(l)}}}^{1}\cdot z_{1}^{-1}$;
and each $\widehat{\overline{\mathcal{P}}}_{i}[z_{1}^{-1}]_{reg}$
is a rank $2$ log-connection which is a self-extension of a line
bundle with log connection. 
\end{lem}

This follows directly from the proceeding characterization of $\overline{\mathcal{M}}_{R}$,
as well as the eigenvalue decomposition of $\widehat{\overline{\mathcal{M}}}_{k}[z_{1}^{-1}]$
(just as in the proof of \prettyref{lem:Local-Decomp-for-N}) Now
let us give the 
\begin{proof}
(of \prettyref{thm:Ext-vanishing}) First we shall show that the extension
$\mathcal{P}_{k}$ is trivial for each $k$ of characteristic $p>>0$.
To do so, note that $\mathcal{P}_{k}$ defines a class $\phi$ in
$\text{Ext}^{1}(\mathcal{M}_{k},\mathcal{M}_{k})\tilde{=}\Omega_{L_{k}}^{1}$.
We claim that $\phi$ has log poles along each component of the compactification
$\overline{L}$. If not, then then $\phi$ must have a pole of order
$\geq2$ on some component $E_{i}$ of $\overline{L}_{W_{n}(k)}\backslash L_{W_{n}(k)}$.
Let the image of $E_{i}$ in $\overline{\tilde{X}}_{k}$ be a divisor
$D_{i}$. Then, we can apply \prettyref{eq:decomp-1} over the generic
point of $D_{i}$; after pulling back to $\widehat{V}_{x,k}^{(l)}$
we have 
\[
\widehat{\overline{\mathcal{P}}}_{k}[z_{1}^{-1}]=\bigoplus_{i=1}^{2r'}\mathcal{O}_{\widehat{V_{x,k}^{(l)}}}[z_{1}^{-1}]\cdot e_{i}\oplus\bigoplus_{i=1}^{r-r'}\widehat{\overline{\mathcal{P}}}_{i}[z_{1}^{-1}]_{reg}
\]
On the other hand, we have the decomposition 
\[
\widehat{\overline{\mathcal{M}}}_{k}[z_{1}^{-1}]=\bigoplus_{i=1}^{r'}\mathcal{O}_{\widehat{V_{x,k}^{(l)}}}[z_{1}^{-1}]\cdot e_{i}\oplus\bigoplus_{i=1}^{r-r'}\widehat{\overline{\mathcal{M}}}_{i}[z_{1}^{-1}]_{reg}
\]
where $\partial_{z_{1}}e_{i}=\alpha_{i}$ and each $\widehat{\overline{\mathcal{M}}}_{i}[z_{1}^{-1}]_{reg}$
is a line bundle with log-connection; note that each summand corresponds
to a component of the formal completion of $\overline{L}^{(l)}$ over
$V_{x}^{(l)}$. So the extension $\widehat{\overline{\mathcal{P}}}_{k}[z_{1}^{-1}]$
is given by taking $\widehat{\overline{\mathcal{M}}}_{k}[z_{1}^{-1}]\oplus\widehat{\overline{\mathcal{M}}}_{k}[z_{1}^{-1}]$
and adding the restriction of $\phi$ (to the corresponding component)
to the connection form. Since the poles of order $\geq2$ in $\widehat{\overline{\mathcal{P}}}_{k}[z_{1}^{-1}]$
match those in $\widehat{\overline{\mathcal{M}}}_{k}[z_{1}^{-1}]\oplus\widehat{\overline{\mathcal{M}}}_{k}[z_{1}^{-1}]$,
we see that $\phi$ has log poles. 

On the other hand, we have $H^{0}(\Omega_{L}^{1}(\tilde{E}))=0$.
So $\phi=0$, and the extension $\mathcal{P}_{k}$ is trivial as claimed. 

Now consider the class $[\mathcal{P}_{U}]\in\text{Ext}^{1}(\mathcal{M}_{U},\mathcal{M}_{U})$.
As $\mathcal{M}_{U}$ is a bundle over $U$, the $R$-module $\text{Ext}^{1}(\mathcal{M}_{U},\mathcal{M}_{U})$
is the first cohomology group of the de Rham complex for $\mathcal{M}_{U}\otimes_{\mathcal{O}_{U}}\mathcal{M}_{U}^{*}$,
which is a complex of free $R$-modules. Thus for any ring $A$ such
that $R\to A$ we have that $\text{Ext}^{1}(\mathcal{M}_{U_{A}},\mathcal{M}_{U_{A}})$
is the the first cohomology of the de Rham complex for $\mathcal{M}_{U_{A}}\otimes_{\mathcal{O}_{U_{A}}}\mathcal{M}_{U_{A}}^{*}$.
Applying this with $A=R/p$, from the spectral sequence for the base
change from $R$ to $R/p$ we obtain the injection 
\[
\text{Ext}^{1}(\mathcal{M}_{U},\mathcal{M}_{U})/p\to\text{Ext}^{1}(\mathcal{M}_{U_{R/p}},\mathcal{M}_{U_{R/p}})
\]
and from the discussion in the previous paragraph we deduce that the
image of $[\mathcal{P}_{U}]$ in $\text{Ext}^{1}(\mathcal{M}_{U_{R/p}},\mathcal{M}_{U_{R/p}})$
is $0$. Thus the image of $[\mathcal{P}_{U}]$ in $\text{Ext}^{1}(\mathcal{M}_{U},\mathcal{M}_{U})/p$
is $0$ as well. As this is true for all $p>>0$, then, using the
fact that $\text{Ext}^{1}(\mathcal{M}_{U},\mathcal{M}_{U})$ is a
direct sum of finite type $R$-modules (by \prettyref{lem:p-adically-seperated}),
we see that $[\mathcal{P}_{U}]$ defines a torsion class. Thus the
image of $[\mathcal{P}_{U}]$ in $\text{Ext}^{1}(\mathcal{M}_{U_{\mathbb{C}}},\mathcal{M}_{U_{\mathbb{C}}})$
must be $0$ as well. Thus $\mathcal{P}_{U_{\mathbb{C}}}$ is a split
extension. Furthermore, $\mathcal{P}_{\mathbb{C}}=j_{!*}(\mathcal{P}_{U_{\mathbb{C}}})$
(this follows from $\mathcal{M}_{\mathbb{C}}=j_{!*}(\mathcal{M}_{U_{\mathbb{C}}})$).
So $\mathcal{P}_{\mathbb{C}}$ is a split extension as well and we
are done.
\end{proof}

\section{\label{sec:Applications}Applications}

In this section we present some basic applications of \prettyref{thm:1}.
We'll start by explaining how the famous Abhyankar-Moh theorem can
now be derived from $\mathcal{D}$-module theory, in particular Arinkin's
construction of rigid local systems. Then we'll give the promised
application to the Weyl algebra. 

\subsection{Abhyankar-Moh Theorem}

In this subsection we make some brief remarks about the relation of
the $m=1$ case of \prettyref{thm:1} and the Abhyankar-Moh theorem
about embeddings of $\mathbb{A}_{\mathbb{C}}^{1}$ into $\mathbb{A}_{\mathbb{C}}^{2}$.
As it turns out, the two theorems are essentially equivalent. To see
why, we recall (one version of) the statement of the theorem: let
$\{x,y\}$ be standard coordinates on $\mathbb{A}_{\mathbb{C}}^{2}$.
Then:
\begin{thm}
(Abhyankar-Moh) Consider any embedding $i:\mathbb{A}_{\mathbb{C}}^{1}\to\mathbb{A}_{\mathbb{C}}^{2}$.
Then there exists an automorphism $a$ of $\mathbb{A}_{\mathbb{C}}^{2}$
such that $a\circ i$ is the standard embedding of the $x$-axis.
We can choose $a$ so that the Jacobian $J(a)=1$. 
\end{thm}

This is essentially ({[}AM{]} theorem 1.6); we remark that the claim
about choosing $J(a)=1$ is not mentioned there; however, if one has
found an $a$ as in the statement of the first sentence of the theorem,
then $J(a)=c$ is necessarily a constant function on $\mathbb{A}_{\mathbb{C}}^{2}$.
Multiplying $a$ on the right by the transformation given by $y\to y$
and $x\to c^{-1}x$ yields an automorphism with $J=1$ whose action
on the $x$-axis is the same as that of $a$. 

Further, we assume throughout the rest of this section that the image
of $i$, called $L_{\mathbb{C}}$, has a dominant projection to the
$x$-axis $\mathbb{A}_{\mathbb{C}}^{1}$. If not, the image of $i$
is simply a line in $\mathbb{A}_{\mathbb{C}}^{2}$ for which the theorem
is easy. 

To see why this applies to our situation, we recall that there is
a natural map 
\[
\mbox{Aut}(D_{\mathbb{A}_{\mathbb{C}}^{m}})\to\mbox{Aut}_{Symp}(\mathbb{A}_{\mathbb{C}}^{2m})
\]
 constructed in {[}ML{]} (for $m=1$) and then \cite{key-25} and
\cite{key-3} in general (see the subsection directly below for more
on this). Here, the group on the right is the group of algebraic symplectomorphisms
of $\mathbb{A}_{\mathbb{C}}^{2m}$ when equipped with the standard
symplectic form. 

In the case $m=1$, this map is known to be an automorphism from \cite{key-31};
in addition, in this case, the group on the right is simply the group
of automorphisms of $\mathbb{A}_{\mathbb{C}}^{2}$ whose Jacobian
is equal to $1$. Furthermore, in this case the construction of the
inverse map 
\[
\mbox{Aut}_{Symp}(\mathbb{A}_{\mathbb{C}}^{2})\to\mbox{Aut}(D_{\mathbb{A}_{\mathbb{C}}^{1}})
\]
can be done quite explicitly; namely, it is known\footnote{Due, I believe, to Dixmier and reproved, e.g., in \cite{key-31}}
that the group $\mbox{Aut}_{Symp}(\mathbb{A}_{\mathbb{C}}^{2})$ is
generated by $\mbox{SL}_{2}(\mathbb{C})$ and transformations of the
form 
\[
{x\to x \brace y\to y+f(x)}
\]
where $f(x)$ is an arbitrary polynomial. One sees quite directly
that both $\mbox{SL}_{2}(\mathbb{C})$ and all of these elements act
on $D_{\mathbb{A}_{\mathbb{C}}^{1}}$ by replacing $y$ with ${\displaystyle \frac{d}{dx}}$
everywhere in the formulas; after checking some relations this yields
the required map. Furthermore, one may see quite directly that, given
an automorphism $a$ of $\mathbb{A}_{\mathbb{C}}^{2}$ which goes
to $a'\in\mbox{Aut}(D_{\mathbb{A}_{\mathbb{C}}^{1}})$, after reduction
mod $p$, $a'$ acts on ${\displaystyle Z(D_{\mathbb{A}_{\mathbb{C}}^{1}})\tilde{=}k[x^{p},(\frac{d}{dx})^{p}]}$
by $a^{(1)}$. Indeed, this is the key to the proof of the automorphism
given in \cite{key-31}. 

So, given a curve $L_{\mathbb{C}}$ isomorphic to $\mathbb{A}_{\mathbb{C}}^{1}$
inside $\mathbb{A}_{\mathbb{C}}^{2}$, the element $a$ from the previous
theorem may be used to construct a $\mathcal{D}$-module: namely,
we take $a^{*}O_{X}$, which, by construction, will have constant
arithmetic support equal to $L_{\mathbb{C}}$. 

On the other hand, the one dimensional case of \prettyref{thm:1},
combined with the main theorem of Arinkin's paper \cite{key-28},
may also be used to reprove the Abhyankar-Moh theorem. Namely, Arinkin
extends Katz's constructive algorithm for rigid irreducible connections
to the irregular case, thus yielding a finite procedure which constructs
$M_{\mathbb{C}}$ from $O_{X}$. Let us recall corollary 2.5 of that
paper: 
\begin{thm}
(Arinkin) Let $M_{\mathbb{C}}$ be a rigid irreducible connection
on some open subset $U_{\mathbb{C}}\subset\mathbb{P}_{\mathbb{C}}^{1}$.
Then $M_{\mathbb{C}}$ can be constructed out of $O_{\mathbb{A}_{\mathbb{C}}^{1}}$
by a sequence of operations of the following types: 

a) Tensor by a rank one connection. 

b) Pullback by an automorphism of $\mathbb{P}_{\mathbb{C}}^{1}$. 

c) Fourier transform. 
\end{thm}

Because $M_{\mathbb{C}}$ has no singularities except at $\{\infty\}$,
the only allowable rank one local systems have no singularities except
at $\{\infty\}$, and hence are of the form $e^{f}$ for some $f\in O(\mathbb{A}_{\mathbb{C}}^{1})$.
Similarly, the only allowable automorphisms of $\mathbb{P}_{\mathbb{C}}^{1}$
are those which fix $\{\infty\}$; i.e., the multiplication by a constant. 

Now let us consider the effect on the arithmetic support of each of
the three operations in the corollary. So, suppose we have a $\mathcal{D}$-module
$N_{\mathbb{C}}$ with constant arithmetic support $\mathcal{L}_{\mathbb{C}}$.
Tensoring by $e^{f}$ corresponds to the operation 
\[
\mathcal{L}_{\mathbb{C}}\to\mathcal{L}_{\mathbb{C}}+df
\]
Now consider the automorphism $a_{c}$ defined by multiplication by
$c$on $\mathbb{A}_{\mathbb{C}}^{1}$. Let $\tilde{a}_{c}$ be the
automorphism of $\mathbb{A}_{\mathbb{C}}^{2}$ defined by $x\to cx$
and $y\to c^{-1}y$. Then the pullback by $a_{c}$ corresponds to
the operation 
\[
\mathcal{L}_{\mathbb{C}}\to\tilde{a}_{c}^{*}(\mathcal{L}_{\mathbb{C}})
\]
Finally, the Fourier transform corresponds to 
\[
\mathcal{L}_{\mathbb{C}}\to r^{*}(\mathcal{L}_{\mathbb{C}})
\]
where $r$ is the automorphism $x\to y$, $y\to-x$ of $\mathbb{A}_{\mathbb{C}}^{2}$.
Each of these three operations, therefore, moves the arithmetic support
by the action of an automorphism of $\mathbb{A}_{\mathbb{C}}^{2}$
(and it follows from the previous discussion that these elements generate
the automorphism group). Therefore, our proof of the one-dimensional
case of \prettyref{thm:1} combined with Arinkin's theorem, reproves
the Abhyankar-Moh theorem. 

\subsection{Autoequivalences of the Weyl Algebra}

In this section we will prove \prettyref{thm:3}. First, following
\cite{key-3} and \cite{key-25}, we recall the following 
\begin{thm}
There is a natural map from the group $\mbox{MAut}(D_{m,\mathbb{C}})$
of Morita autoequivalences of the Weyl algebra $D_{m,\mathbb{C}}$
to the group $\mbox{Aut}_{\mathrm{Symp}}(T^{*}\mathbb{A}_{\mathbb{C}}^{m})$
of symplectomorphisms of $T^{*}\mathbb{A}_{\mathbb{C}}^{m}$. 
\end{thm}

In fact, the proof is a very quick consequence of the existence and
uniqueness theorems and the results of the cited papers. Let us briefly
recall the relevant set-up, following the notation of \cite{key-25}
(the way of \cite{key-3} is extremely similar): first, he introduces
the field $\mathbb{Q}_{U}^{(\infty)}$ which is a subfield of the
ring 
\[
\prod\bar{\mathbb{F}}_{p}
\]
where the product ranges over all prime numbers, and the $U$ denotes
a principle ultrafilter on the set of primes. The field $\mathbb{Q}_{U}^{(\infty)}$
is isomorphic, non canonically, to $\mathbb{C}$. Then the main construction
of \cite{key-25} (proposition 7.1) goes as follows: given any endomorphism
$\phi$ of the Weyl algebra $D_{m,\mathbb{C}}$, one may regard it
as an endomorphism of $D_{m,\mathbb{Q}_{U}^{(\infty)}}$. From this
it follows that there exists a family of endomorphisms $\phi_{p}$
of $D_{m,\bar{\mathbb{F}}_{p}}$ (for some infinite set of primes)
whose limit is equal to $\phi$. But then each $\phi_{p}$ gives an
endomorphism of $Z(D_{m,\bar{\mathbb{F}}_{p}})=T^{*}(\mathbb{A}^{n}(\bar{\mathbb{F}}_{p}))^{(1)}$
which respects the symplectic form. In this way we have a natural
map 
\[
\mbox{End}(D_{m,\mathbb{Q}_{U}^{(\infty)}})\to\mbox{End}(T^{*}\mathbb{A}_{\mathbb{Q}_{U}^{(\infty)}}^{m})
\]
 which is therefore equivalent to a map 
\[
\mbox{End}(D_{m,\mathbb{C}})\to\mbox{End}(T^{*}\mathbb{A}^{m}(\mathbb{C}))
\]

Now, instead of starting with an endomorphism of $D_{m,\mathbb{C}}$,
we could have started with an invertible bimodule over $D_{m,\mathbb{C}}$;
i.e., a $D_{m,\mathbb{C}}$ bimodule $A_{\mathbb{C}}$ such that there
exists a $D_{m,\mathbb{C}}$ bimodule $B$ with $A\otimes_{D_{m,\mathbb{C}}}B\tilde{=}D_{m,\mathbb{C}}$.
Applying the same technique, one obtains a map 
\[
\mbox{MAut}(D_{m,\mathbb{Q}_{U}^{(\infty)}})\to\mbox{Cor}{}_{symp}(T^{*}\mathbb{A}_{\mathbb{Q}_{U}^{(\infty)}}^{m})
\]
where on the right hand side we have the monoid of symplectic correspondences
of $T^{*}\mathbb{A}^{m}$, i.e., the monoid (under composition of
correspondences) of coherent sheaves on $T^{*}\mathbb{A}_{\mathbb{Q}_{U}^{(\infty)}}^{m}\times T^{*}\mathbb{A}_{\mathbb{Q}_{U}^{(\infty)}}^{m}$
whose support is a Lagrangian subvariety. The image under this map
will consist of invertible correspondences, and in fact it is not
hard to show the 
\begin{lem}
The group of invertible symplectic correspondences of $T^{*}\mathbb{A}_{k}^{m}$
(where $k$ is any algebraically closed field) is isomorphic to the
group of symplectomorphisms of $T^{*}\mathbb{A}_{k}^{m}$.
\end{lem}

Thus we obtain the promised map $\mbox{MAut}(D_{m,\mathbb{C}})\to\mbox{Aut}{}_{symp}(T^{*}\mathbb{A}_{\mathbb{C}}^{m})$.
Since $T^{*}\mathbb{A}_{\mathbb{C}}^{m}$ is an affine space, the
underlying variety of every invertible symplectic correspondence will
satisfy the cohomology vanishing assumptions of \prettyref{thm:1};
thus to each one we may associate a unique $D_{m,\mathbb{C}}$ bimodule
by the existence and uniqueness theorems; and so we deduce that the
map $\mbox{MAut}(D_{m,\mathbb{C}})\to\mbox{Aut}{}_{symp}(T^{*}\mathbb{A}_{\mathbb{C}}^{m})$
is an isomorphism; this proves \prettyref{thm:3}.

\section{\label{sec:Appendix}Appendix}

Here we gather together a few general results which we need in body
of the text, but which do not seem to exist in the literature in the
form we need.

\subsection{Spectral Splitting and decomposition of meromorphic connections over
$R$}

In this section we give the versions of two well-known facts about
meromorphic connections -the spectral splitting lemma and the BNR
correspondence that we use in this paper. Since these have not quite
appeared in the form that we need, we include proofs. Let us begin
with the spectral splitting lemma; whose statement (and proof) we
essentially take from \cite{key-22}, in a generalized form. 

The set up here is as follows: Let $Y=\text{Spec}(A)$ be an affine
scheme, smooth over $R$, which possesses coordinates $\{z_{1},\dots z_{n}\}$,
and coordinate derivations $\{\partial_{1},\dots,\partial_{n}\}$.
Let $\widehat{A}$ be the completion of $A$ along $(z_{1})$. Let
$\mathcal{V}$ be a free $\widehat{A}$-module such that $\mathcal{V}[z_{1}^{-1}]$
is equipped with a flat connection $\nabla$. 

For some $r\in\mathbb{Z}$ we have the morphism $\partial_{1}:\mathcal{V}\to z^{r}\mathcal{V}$.
If we suppose $r\leq-2$, then we have, for all $v\in\mathcal{V}$,
$\partial_{1}(z_{1}v)=v+z_{1}\partial_{1}(v)\in z^{r+1}\mathcal{V}$.
Therefore we obtain a morphism 
\[
\bar{\partial}_{1}:\mathcal{V}/z_{1}\to z_{1}^{r}\mathcal{V}/(z_{1}^{r+1}\mathcal{V})\tilde{=}\mathcal{V}/z_{1}
\]
which is an endomorphism in the category of vector bundles. Therefore,
after possibly shrinking $Y$ and passing to a finite etale cover,
we can suppose that there is a basis of $\mathcal{V}/z_{1}$ for which
$\bar{\partial}_{1}$ is upper-triangular. In particular, the action
of $\overline{\partial}_{1}$ satisfies a generalized eigenspace decomposition
\[
\mathcal{V}/z_{1}=\bigoplus_{i}(\mathcal{V}/z_{1})_{i}
\]
for which, if $\{\alpha_{i}\}$ are the eigenvalues, we have $\alpha_{i}-\alpha_{j}$
are units in $A$ for $i\neq j$. Then we have
\begin{lem}
\label{lem:(Spectral-Splitting)} 1) (Spectral Splitting) With notation
as above, we have a direct sum decomposition 
\[
\mathcal{V}=\bigoplus\mathcal{V}_{i}
\]
such that $\mathcal{V}_{i}/z_{1}=(\mathcal{V}/z_{1})_{i}$ and the
action of $\partial_{1}$ preserves $\mathcal{V}_{i}[z_{1}^{-1}]$
to $\mathcal{V}_{i}[z_{1}^{-1}]$. 

2) The subspaces $\mathcal{V}_{i}$ are also preserved under the action
of the $\{\partial_{2},\dots,\partial_{n}\}$. 
\end{lem}

\begin{proof}
1) We shall construct the required decomposition inductively inside
$\mathcal{V}/z_{1}^{m}\mathcal{V}$ and then take the limit; for $m=1$
there is nothing to prove.

Consider the induction step for $m\geq2$. Each $\mathcal{V}/z_{1}^{m}\mathcal{V}$
is a module over $A/z_{1}^{m}$, and we obtain a map
\[
\overline{\partial}_{1}:\mathcal{V}/z_{1}^{m}\mathcal{V}\to z_{1}^{r}\mathcal{V}/z_{1}^{r+m}\mathcal{V}\tilde{\to}\mathcal{V}/z_{1}^{m}\mathcal{V}
\]
We obtain that $\mathcal{V}/z_{1}^{m}\mathcal{V}$ is a module over
the ring $(A/z_{1}^{m})<\overline{\Theta}>$ which has the relations
\[
[\overline{\Theta},a]=z_{1}^{r}\partial_{1}(a)
\]
Now, by induction the $(A/z_{1}^{m-1})<\overline{\Theta}>$-module
$\mathcal{V}/z^{m-1}\mathcal{V}$ spits as a direct sum which lifts
the generalized eigenspace decomposition $\mathcal{V}/z_{1}$. Write
\[
\mathcal{V}/z_{1}^{m-1}\mathcal{V}=\bigoplus(\mathcal{V}/z_{1}^{m-1}\mathcal{V})_{i}
\]

Now, since $V/z^{m}V$ is free as an $A/z^{m}$-module, we may choose
a direct sum decomposition of $V/z^{m}V$ as an $A/z^{m}$-module
which lifts the given decomposition of $V/z^{m-1}V$; let $\{\pi_{i}\}$
denote the associated projection operators; they are $A/z^{m}$-linear
but may not commute with $\overline{\Theta}$. So, we write 
\[
\overline{\Theta}(v)=v+z^{m}T(\overline{v})
\]
where $T:V/zV\to V/zV$ is $A/z$-linear; in the above, $\overline{v}$
denotes the reduction of $v$ to $\mathcal{V}/z_{1}$. Using this,
one computes that the $A/z$-linear map ${\displaystyle \frac{1}{z_{1}^{m}}(\pi_{i}\overline{\Theta}-\overline{\Theta}\pi_{i})}:\mathcal{V}/z_{1}\to\mathcal{V}/z_{1}$
takes $(\mathcal{V}/z\mathcal{V})_{l}$ to $(\mathcal{V}/z\mathcal{V})_{i}$
for $l\neq i$ and $(\mathcal{V}/z\mathcal{V})_{i}$ to ${\displaystyle \bigoplus_{s\neq i}(\mathcal{V}/z\mathcal{V})_{s}}$.
 In other words, if we consider ${\displaystyle \frac{1}{z_{1}^{m}}(\pi_{i}\overline{\Theta}-\overline{\Theta}\pi_{i})}$
as an element of ${\displaystyle \text{End}(\mathcal{V}/z\mathcal{V})\tilde{=}\bigoplus_{s,t}(\mathcal{V}/z\mathcal{V})_{s}\otimes(\mathcal{V}/z\mathcal{V})_{t}^{*}}$,
then this element is contained in ${\displaystyle \bigoplus_{s\neq t}(\mathcal{V}/z\mathcal{V})_{s}\otimes(\mathcal{V}/z\mathcal{V})_{t}^{*}}$. 

On the other hand, if we consider $[\overline{\partial}_{1},\cdot]$
as an endomorphism of $\text{End}(\mathcal{V}/z\mathcal{V})$, then
the subspaces $(\mathcal{V}/z_{1}\mathcal{V})_{s}\otimes(\mathcal{V}/z_{1}\mathcal{V})_{t}^{*}$
are the just the generalized eigenspaces, with eigenvalue $\alpha_{s}-\alpha_{t}$.
Since each such eigenvalue is assumed to be a unit (for $s\neq t$),
we see that $\overline{\partial}_{1}$ is surjective on $(\mathcal{V}/z_{1}\mathcal{V})_{s}\otimes(\mathcal{V}/z_{1}\mathcal{V})_{t}^{*}$
when $s\neq t$. In particular, we see that ${\displaystyle \frac{1}{z_{1}^{m}}(\pi_{i}\overline{\Theta}-\overline{\Theta}\pi_{i})}$
is in the range of $[\overline{\partial}_{1},\cdot]$. This implies
that if we modify the maps $\pi_{i}$ to $\tilde{\pi}_{i}=\pi_{i}+z^{m}S_{i}$,
where $S_{i}:\mathcal{V}/z_{1}\to\mathcal{V}/z_{1}$ are chosen so
that 
\[
[S_{i},\overline{\partial}_{1}]=\frac{1}{z_{1}^{m}}(\pi_{i}\overline{\Theta}-\overline{\Theta}\pi_{i})
\]
then $[\tilde{\pi}_{i},\overline{\partial}_{1}]=0$. 

In sum, we have shown that the projectors $\overline{\pi}_{i}\in\text{End}_{(A/z_{1}^{m-1})<\overline{\Theta}>}(\mathcal{V}/z^{m-1}\mathcal{V})$
lift to elements $\tilde{\pi}_{i}\in\text{End}_{(A/z_{1}^{m})<\overline{\Theta}>}(\mathcal{V}/z^{m}\mathcal{V})$.
But the natural map 
\[
\text{End}_{(A/z_{1}^{m})<\overline{\Theta}>}(\mathcal{V}/z^{m}\mathcal{V})\to\text{End}_{(A/z_{1}^{m-1})<\overline{\Theta}>}(\mathcal{V}/z^{m-1}\mathcal{V})
\]
has a nilpotent kernel, so the classical lifting of idempotents theorem
(\cite{key-42}, Theorem 21.28) implies that in fact the $\overline{\pi}_{i}$
lift to a complete set of orthogonal idempotents in $\text{End}_{(A/z_{1}^{m})<\overline{\Theta}>}(\mathcal{V}/z^{m}\mathcal{V})$;
which gives the required decomposition. 

2) To show this, we recall from \cite{key-68}, lemme 4.2, that the
analogous result holds, for the canonical decomposition of $\mathcal{V}[z_{1}^{-1}]$,
with respect to $\partial_{1}$, after passing to small neighborhood
(in the analytic topology) of a generic closed point in $D_{\mathbb{C}}$.
But the canonical decomposition clearly refines the decomposition
that we have constructed. So the subspaces are preserved by $\{\partial_{2},\dots\partial_{n}\}$
after passage to some faithfully flat ring extension; therefore they
are preserved already over $A$. 
\end{proof}
Now we give the main application of this result: let $X$ be a smooth
scheme over $R$, and let $(\mathcal{V},\nabla)$ be a vector bundle
with connection of rank $r$ on $X$. Suppose that $\overline{X}$
is a smooth compactification of $X$, and let $D$ be an irreducible
divisor in $\overline{X}$. Take an extension $\overline{\mathcal{V}}$
to a meromorphic connection on $\overline{X}$, which we assume to
be a vector bundle over the generic point of $D$; let $\{z_{1}=0\}$
be a local equation for $D$. 

Writing $U=\text{Spec}(A)$, we have the finite flat cover $U_{m}:=A[y]/(y-z^{m})$.
We use the same notation when replacing $U$ by any etale neighborhood
$U'$. Finally, let $\widehat{U}_{m}$ denote the $z$-adic completion\footnote{equivalently, the $y$-adic completion}
of $U_{m}$, considered as a formal scheme. Then we have 
\begin{cor}
\label{cor:B-V-over-R}There is a finite etale neighborhood $U'\to U$,
and an integer $m\geq1$, so that, after pulling back to $\widehat{U}'_{m}$,
we have 
\[
\widehat{\overline{\mathcal{V}}}_{m}=\bigoplus_{i=1}^{n'}\widehat{\overline{\mathcal{V}}}_{m,i}\oplus\widehat{\overline{\mathcal{V}}}_{m,reg}
\]
where the $\widehat{\overline{\mathcal{V}}}_{m,i}$ are vector bundles
with meromorphic connection, where the connection takes the form $\nabla=\alpha_{i}I+\nabla'$
where $\alpha_{i}\in\Omega^{1}[z_{1}^{-1}]$ is a one-form over $\widehat{U}'_{m}$,
$\nabla'$ is a connection with log singularities, and $\widehat{\overline{\mathcal{V}}}_{m,reg}$
is a bundle over $\mathcal{O}_{\widehat{U}_{m}}$ with log singularities. 
\end{cor}

Before proving this, we recall the main technical result from \cite{key-22},
in a form suitable for this situation. Namely, let $F$ be an algebraically
closed field of characteristic $0$, and consider a free module $\mathcal{W}$
of rank $r$ over $F[[z]]$, so that that $\mathcal{W}[z^{-1}]$ is
equipped with the action of an operator $\partial/\partial z$, which
is continuous (for the $z$-adic topology) and satisfies the Leibniz
rule. For each $m>0$ we choose an $m$'th root of $z$, denoted $y$,
and we have the finite flat map $F[[z]]\to F[[y]]$. Set $\mathcal{W}_{m}:=F[[y]]\otimes_{F[[z]]}\mathcal{W}$.
Then $\mathcal{W}_{m}[z^{-1}]$ admits an action of $\partial/\partial y:=m\cdot y^{m-1}(\partial/\partial z)$,
which is again continuous and satisfies the Leibniz rule. 

Pick a basis for $\mathcal{W}$ and write 
\[
[\partial/\partial z]=\sum_{i=r}^{\infty}z^{i}A_{i}
\]
where $[\partial/\partial z]$ is the matrix of the action of $\partial/\partial z$,
and the $A_{i}$ are $n\times n$ matrices over $F$. We refer to
$A_{r}$ as the leading term of $[\partial/\partial z]$. Suppose
$r\leq-2$, and $n>1$. Then 
\begin{prop}
\label{prop:(Babbitt-Varadarajan)-Main-Lemma}(Babbitt-Varadarajan)
If $A_{r}\neq0$ is nilpotent, there is an $m>0$ so that $\mathcal{W}_{m}$
admits a basis in which the matrix of $\partial/\partial y$ has leading
term $A_{r'}$, for some $r'\leq r$, and $A_{r'}$ has at least two
distinct eigenvalues. 
\end{prop}

This is proposition 4.6 in \cite{key-22} (c.f. also the proof of
Theorem 6.3 in loc. cit). It is the key point in their approach to
the formal reduction theory; once the leading term has more than one
eigenvalue, one may apply the spectral splitting lemma to show that
the connection is a direct sum of connections of smaller rank. Using
this result, we modify their technique to give the 
\begin{proof}
(of \prettyref{cor:B-V-over-R}) Let $\mathcal{O}_{D}$ be the local
ring of the subscheme $D$. Let $\mathcal{O}_{D_{\mathbb{C}}}$ be
the local ring of $D_{\mathbb{C}}$ in $X_{\mathbb{C}}$. After completing
along $(z)$, the ring $\widehat{\mathcal{O}}_{D_{\mathbb{C}}}$ is
a complete DVR of equicharacteristic $0$, and therefore isomorphic
to $\mathbb{C}(D)[[z]]$, where $\mathbb{C}(D)$ denotes the function
field of $D_{\mathbb{C}}$, which is the residue field of $\mathcal{O}_{D_{\mathbb{C}}}$.
The previous proposition therefore applies in this setup, at least
after passing to the algebraic closure of $\mathbb{C}(D)$. 

Let $\overline{\mathcal{V}}$ be as in the statement of the corollary,
let $\overline{\mathcal{V}}_{D}$ be its localization at $D$; this
is a finite free module over $\mathcal{O}_{D}$. Choosing some basis
for $\overline{\mathcal{V}}_{D}$, we obtain a basis for $\overline{\mathcal{V}}_{D}\otimes_{\mathcal{O}_{D}}\mathcal{O}_{D_{\mathbb{C}}}$.
If the connection has log singularities, we are done; so we assume
this is not the case. Let $A_{r}$ denote the leading term for the
action of $\partial_{1}$ in this basis; if $r\geq-1$ then $\overline{\mathcal{V}}$
is regular along $D$, so we may suppose $r\leq-2$. 

Let $\mathbb{C}(D)'$ be a finite field extension of $\mathbb{C}(D)$
in which $A_{r}$ has a generalized eigenspace decomposition. So,
after replacing $R$ with a finite ring extension, there is a finite
etale neighborhood $U'$ of the generic point of $D$ over which $A_{r}$
is defined and has a generalized eigenspace decomposition. Localizing
as needed, we may assume that the differences between distinct eigenvalues
are units in $U'$. Therefore, applying \prettyref{lem:(Spectral-Splitting)};
we see that if there are two distinct eigenvectors then the completion
$\widehat{\overline{\mathcal{V}}}_{D}[z^{-1}]$ splits along the generalized
eigenspace decomposition. Therefore we may suppose there is a single
generalized eigenspace of dimension $>1$, with eigenvalue $\alpha$.
Replacing the action of $\partial/\partial z$ by $\partial/\partial z-z^{r}\alpha$,
we see that we may suppose $A_{r}$ is nilpotent.

If $A_{r}\neq0$, then applying \prettyref{prop:(Babbitt-Varadarajan)-Main-Lemma},
we see that there exists $m>0$ and a finite etale extension $\widehat{\mathcal{O}}_{D'_{\mathbb{C}}}$
of $\widehat{\mathcal{O}}_{D_{\mathbb{C}}}$ so that, after pulling
back to $\widehat{\mathcal{O}}_{D'_{\mathbb{C}}}[y]/(y^{m}-z)$, the
module $\overline{\mathcal{V}}_{D_{\mathbb{C}}}\otimes_{\mathcal{O}_{D_{\mathbb{C}}}}\widehat{\mathcal{O}}_{D'_{\mathbb{C}}}[y]/(y^{m}-z)$
admits a basis, for which the leading term of $[\partial/\partial y]$
has a generalized eigenspace decomposition with more than one eigenvalue.
Let $\{\overline{e}_{i}\}$ denote the image of this basis in 
\[
(\overline{\mathcal{V}}_{D_{\mathbb{C}}}\otimes_{\mathcal{O}_{D_{\mathbb{C}}}}\widehat{\mathcal{O}}_{D'_{\mathbb{C}}}[y]/(y^{m}-z))/y\tilde{=}(\overline{\mathcal{V}}_{D_{\mathbb{C}}}/z)\otimes_{\mathbb{C}(D)}\mathbb{C}(D)'
\]
Then, again extending $R$ if needed, we may find an etale neighborhood
$U'$ of $D$ so that, $\overline{\mathcal{V}}(U'_{m})$ has a basis
whose reduction mod $y$ is equal to $\{\overline{e}_{i}\}$. Since
the leading term of $[\partial/\partial y]$ only depends on the reduction
mod $y$, we see that the leading term of $[\partial/\partial y]$
has a generalized eigenspace decomposition with at least two eigenvalues.
Thus as above we may split the connection again.

If, on the other hand, $A_{r}=0$, then the connection $\nabla-z^{r}dz$
has a leading term of strictly smaller degree. If this leading term
is still $\leq-2$ we may repeat the above; if it is $\geq-1$ then
this connection is of the form $\widehat{\overline{\mathcal{V}}}_{m,i}$
(for the action of $\partial_{1}$ only). Thus we arrive at a place
where, as a module for $\partial_{1}$, there is a decomposition 
\[
\widehat{\overline{\mathcal{V}}}_{m}=\bigoplus_{i=1}^{n'}\widehat{\overline{\mathcal{V}}}_{m,i}\oplus\widehat{\overline{\mathcal{V}}}_{m,reg}
\]
where $\partial_{1}$ acts on $\widehat{\overline{\mathcal{V}}}_{m,i}$
as $\alpha I+M$, where $\alpha$ is a function with poles in $z_{1}$
of order $\geq2$, at $M$ has only log poles; and $\partial_{1}$
acts on $\widehat{\overline{\mathcal{V}}}_{m,reg}$ with log poles
in $z_{1}$. Further, each of these summands is preserved under $\{\partial_{2},\dots,\partial_{n}\}$.
This tells us that $\widehat{\overline{\mathcal{V}}}_{m,reg}$ is
already a log connection. Now consider $\widehat{\overline{\mathcal{V}}}_{m,i}$.
After inverting finitely many integers in $R$, we can find a function
$g$ such that ${\displaystyle \frac{\partial g}{\partial z_{1}}}=[\alpha]$,
where $[\alpha]$ is the sum of the terms in $\alpha$ whose order
in $z_{1}$ is $\leq-2$. Then the meromorphic connection $e^{-g}\otimes\widehat{\overline{\mathcal{V}}}_{m,i}$
has log singularities, which is exactly the result. 
\end{proof}

\subsection{Marked Descent of Line bundles in mixed characteristic}

In this section we provide a few useful general facts about locally
trivial vector bundles with connection in mixed characteristic. In
positive characteristic, such bundles are exactly the ones with $p$-curvature
zero, and, via the famous Cartier descent theorem (\cite{key-24},
theorem 7.2) these are the bundles with connection arising from Frobenius
pullback. We provide analogues of these statements in mixed characteristic.
Finally, to close out this appendix we give some basic results about
Frobenius descent in mixed characteristic, and we explain the ``marked''
Frobenius descent of line bundles. 

To set things up, recall that the kernel of the differential $d:\mathcal{O}_{X_{W_{m}(k)}}\to\Omega_{X_{W_{m}(k)}}^{1}$
is isomorphic to $\mathcal{O}_{W_{m}(X_{k}^{(m)})}$, the structure
sheaf of the $m$th Witt vectors of $X_{k}^{(m)}$ (a local section
$(g_{1},\dots,g_{n})$ of $\mathcal{O}_{W_{n}(X_{k}^{(n)})}$ is sent
to $g_{1}^{p^{m}}+pg_{2}^{p^{m-1}}+\dots+p^{m}g_{m}$). We let $\Phi:\mathcal{O}_{W_{m}(X_{k}^{(m)})}\to\mathcal{O}_{X_{W_{m}(k)}}$
denote the inclusion. Then the result reads
\begin{prop}
\label{prop:Cartier-Over-W_m}The functor $\Phi^{*}$ induces an equivalence
between vector bundles on $W_{m}(X_{k}^{(m)})$, and vector bundles
with locally trivial connection on $X_{W_{m}(k)}$. The inverse is
given by $\mathcal{E}\to\text{ker}(\nabla:\mathcal{E}\to\mathcal{E})$. 
\end{prop}

Now suppose $X_{W_{m}(k)}$ is equipped with a normal crossings divisor,
$D_{W_{m}(k)}$. Let $(\mathcal{L},\nabla)$ be a line bundle with
log connection (with respect to $D_{W_{m}(k)}$). Suppose that the
restriction of $(\mathcal{L},\nabla)$ to $U_{W_{m}(k)}:=X_{W_{m}(k)}\backslash D_{W_{m}(k)}$
is locally trivial. Then $(\mathcal{L},\nabla)$ must have residues
in $\mathbb{Z}/p^{m}$, and in fact, near a point which is contained
in $s$ components of the divisor $D_{W_{m}(k)}$, the one-form of
the connection can locally be written as 
\[
\sum_{i=1}^{s}\alpha_{i}\frac{dz_{i}}{z_{i}}
\]
for $\alpha_{i}\in\mathbb{Z}/p^{m}$ (this can be shown by induction
on $m$, using the method of proof of \prettyref{thm:Higher-diff-structure-on-deformations}).
So we have
\begin{cor}
\label{cor:Marked-Descent}There is a bijection between line bundles
with log connection on $(X_{W_{m}(k)},D_{W_{m}(k)})$ which are locally
trivial on $U_{W_{m}(k)}$, and line bundles on $W_{m}(X_{k}^{(m)})$,
along with an element of $\{0,1,2,\dots,p^{m}-1\}$ which is attached
to each component of $D_{W_{m}(k)}$. 
\end{cor}

\begin{proof}
To a line bundle with log connection $(\mathcal{L},\nabla)$ we attach
$\text{ker}(\nabla)$. By the local form of the connection above this
is a line bundle on $W_{m}(X_{k}^{(m)})$, and along each component
of $D_{W_{m}(k)}$ we attach the unique element $\alpha\in\{0,1,2,\dots,p^{m}-1\}$
so that $z^{\alpha}\cdot\mathcal{L}$ is preserved under the connection
(here $z$ is a local equation for $D_{W_{m}(k)}$). 
\end{proof}

\subsection{\label{subsec:V-filtrations}$V$-filtrations over $R$, applications}

In this subsection, we'll develop the general theory of $V$-filtrations
for $R$-models of holonomic $\mathcal{D}$-modules. As an application,
we'll prove the missing injectivity statement (\prettyref{thm:Injectivity-for-M})
in the proof of \prettyref{thm:generic-fibre-comparison}; as well
as the important technical lemma \prettyref{lem:p-adically-seperated}. 

Suppose $\mathcal{E}_{\mathbb{C}}$ is a holonomic $\mathcal{D}$-module
on a complex algebraic variety $Y_{\mathbb{C}}$, and let $t\in\mathbb{C}[Y]$
be a regular function on $Y_{\mathbb{C}}$, such that there is a derivation
$\partial$ on $Y$ satisfying $\partial(t)=1$. In particular, the
map $t:Y_{\mathbb{C}}\to\mathbb{A}_{\mathbb{C}}^{1}$ is smooth; and
we have the smooth variety $X_{\mathbb{C}}=t^{-1}(0)$. Then for each
$U\subset Y_{\mathbb{C}}$, we have the sequence of ideals $\{(t)^{k}\}_{k\in\mathbb{Z}}$
where by convention we have $(t)^{k}=\mathcal{O}(U)$ for $k<0$.
We then set 
\begin{equation}
V_{\leq k}(\mathcal{D}_{Y_{\mathbb{C}}}(U))=\{Q\in\mathcal{D}_{Y}(U)|Q\cdot(t)^{j}\subset(t)^{j-k}\phantom{i}\text{for all}\phantom{i}j\in\mathbb{Z}\}\label{eq:V-filt}
\end{equation}
 This is a separated, exhaustive, multiplicative filtration on the
sheaf $\mathcal{D}_{Y_{\mathbb{C}}}$, for which $t\in V_{\leq-1}(\mathcal{D}_{Y})$,
$\partial\in V_{\leq1}(\mathcal{D}_{Y})$, and 
\[
V_{\leq0}(\mathcal{D}_{Y_{\mathbb{C}}})/V_{\leq-1}(\mathcal{D}_{Y_{\mathbb{C}}})\tilde{=}i_{*}(\mathcal{D}_{X_{\mathbb{C}}}[t\partial])
\]
(where $i:X_{\mathbb{C}}\to Y_{\mathbb{C}}$ is the inclusion). As
usual, there is the notion of a filtered module over this filtered
ring; and such a filtration is called good if its associated graded
is a locally finitely generated module over $\text{gr}(\mathcal{D}_{Y_{\mathbb{C}}})$.
We shall denote the associated graded module by $\text{gr}(\mathcal{E}_{\mathbb{C}})$;
each graded summand $\text{gr}_{i}(\mathcal{E}_{\mathbb{C}})$ is
naturally a module over $V_{\leq0}(\mathcal{D}_{Y_{\mathbb{C}}})/V_{\leq-1}(\mathcal{D}_{Y_{\mathbb{C}}})$,
and hence, by the above, a $\mathcal{D}_{X_{\mathbb{C}}}$-module
with an action of the operator $t\partial$. In the holonomic case,
this operator acts via a generalized eigenspace decomposition; in
fact, summarizing the main theorems of \cite{key-67}, sections 4.3
and 4.4, we have 
\begin{thm}
\label{thm:V-filt-exists}The module $\mathcal{E}_{\mathbb{C}}$ possesses
a unique good filtration $V_{\leq i}(\mathcal{E}_{\mathbb{C}})$ (called
the Kashiwara-Malgrange filtration) such that each $\text{gr}_{i}(\mathcal{E}_{\mathbb{C}})=V_{\leq i}(\mathcal{E}_{\mathbb{C}})/V_{\leq i-1}(\mathcal{E}_{\mathbb{C}})$
is annihilated by an operator of the form 
\[
\prod_{l=1}^{m}(t\partial+1+\alpha_{i,l})^{s_{l}}
\]
for which $i-1<\text{Re}(\alpha_{l})\leq i$. Each $\text{gr}_{i}(\mathcal{E})$
is a coherent (in fact, holonomic, by \cite{key-67} 4.6.3) $\mathcal{D}_{Y}$-module.
Further, the collection of all $\{\alpha_{i,l}\}$ which appear in
the above operators consists of integer shifts of some finite subset
of $\{z\in\mathbb{C}|-1<\text{Re}(z)\leq0\}$. 
\end{thm}

From these conditions, one can deduce that $t:V_{\leq i}(\mathcal{E}_{\mathbb{C}})\to V_{\leq i-1}(\mathcal{E}_{\mathbb{C}})$
is an isomorphism for all $i\leq0$, and that $\partial:\text{gr}_{i}(\mathcal{E}_{\mathbb{C}})\to\text{gr}_{i+1}(\mathcal{E}_{\mathbb{C}})$
is bijective for $i\geq1$ (c.f \cite{key-67}, lemmas 4.5.1 and 4.5.4).
In particular, $V_{\leq1}(\mathcal{E}_{\mathbb{C}})$ generates $\mathcal{E}_{\mathbb{C}}$
as a $\mathcal{D}$-module.

Now we explain how to build analogous filtrations for modules over
$R$. Suppose that we are given $R$-models for everything in sight.
Then certainly the definition \prettyref{eq:V-filt} makes sense for
$\mathcal{D}_{Y}$. We have 
\begin{lem}
\label{lem:V-filt-over-R}After possibly localizing and extending
$R$, the filtration $V_{\leq i}(\mathcal{E}_{\mathbb{C}})\cap\mathcal{E}$
is a good filtration (with respect to $V_{\leq i}(\mathcal{D}_{Y})$).
We have that $t:V_{\leq0}(\mathcal{E})\to V_{\leq-1}(\mathcal{E})$
is an isomorphism, and $\partial:\text{gr}_{i}(\mathcal{E})\to\text{gr}_{i+1}(\mathcal{E})$
is a bijection for $i\geq1$. In particular, $V_{\leq1}(\mathcal{E})$
generates $\mathcal{E}$ as a $\mathcal{D}_{Y}$-module. 
\end{lem}

\begin{proof}
By definition each $V_{\leq i}(\mathcal{E}_{\mathbb{C}})$ is coherent
over $V_{\leq0}(\mathcal{D}_{Y_{\mathbb{C}}})$. So, after possibly
localizing and extending $R$, choose coherent $V_{\leq0}(\mathcal{D}_{Y})$-models
$V_{\leq i}(\mathcal{E})$ inside $\mathcal{E}$ for $V_{\leq i}(\mathcal{E}_{\mathbb{C}})$
when $i=\{-1,0,1\}$; we can suppose, after localizing $R$, that
$t\cdot V_{\leq0}(\mathcal{E})\subset V_{\leq-1}(\mathcal{E})$. We
have that $V_{\leq-1}/(t\cdot V_{\leq0}(\mathcal{E}))$ is a torsion
$R$-module; since it is coherent over $V_{\leq0}(\mathcal{D}_{Y})$
we can localize $R$ again and suppose it is $0$. By generic freeness
we can also suppose that $\text{gr}_{i}(\mathcal{E})$ is free over
$R$, for $i=\{0,1\}$. 

Now, define $V_{\leq i}(\mathcal{E}):=\partial^{i-1}\cdot V_{\leq1}(\mathcal{E})$
for $i\geq1$ and $V_{\leq-i-1}(\mathcal{E}):=t^{i}\cdot V_{\leq-1}(\mathcal{E})$
for $i\geq0$. We have $\mathcal{D}_{Y}\cdot V_{\leq1}(\mathcal{E})\subset\mathcal{E}$
is an inclusion of coherent $\mathcal{D}_{Y}$-modules which becomes
an equality after passing to the fraction field of $R$. Thus it is
an equality after localizing $R$ at some element, and so we can suppose
$\mathcal{D}_{Y}\cdot V_{\leq1}(\mathcal{E})=\mathcal{E}$. It follows
that we have defined an exhaustive, good filtration on $\mathcal{E}$
over $V_{\cdot}(\mathcal{D}_{Y})$. 

Next, we need to show $V_{\leq i}(\mathcal{E})=V_{\leq i}(\mathcal{E}_{\mathbb{C}})\cap\mathcal{E}$.
Clearly $V_{\leq i}(\mathcal{E})\subseteq V_{\leq i}(\mathcal{E}_{\mathbb{C}})\cap\mathcal{E}$
for all $i$. Let $m\in\mathcal{E}$ be any section, and suppose $m\in V_{\leq i}(\mathcal{E}_{\mathbb{C}})$
for some $i\geq1$. Since $\mathcal{E}=\mathcal{D}_{Y}\cdot V_{\leq1}(\mathcal{E})$
we can write 
\begin{equation}
m=\partial^{r}\cdot m_{1}+\Phi\cdot m_{2}\label{eq:V-filt-rep}
\end{equation}
where $r\geq0$, $m_{1},m_{2}\in V_{\leq1}(\mathcal{E})$, and $\Phi\in V_{\leq r-1}(\mathcal{D}_{Y})$.
If $r+1>i$, then the image of $m$ in $\text{gr}_{r+1}(\mathcal{E}_{\mathbb{C}})$
is $0$. Therefore \prettyref{eq:V-filt-rep} implies that the image
of $m_{1}$ in $\text{gr}_{1}(\mathcal{E}_{\mathbb{C}})$ is $0$,
since $\partial^{r}:\text{gr}_{1}(\mathcal{E}_{\mathbb{C}})\to\text{gr}_{r+1}(\mathcal{E}_{\mathbb{C}})$
is a bijection. But we have that that $\text{gr}_{1}(\mathcal{E}_{\mathbb{C}})=\text{gr}_{1}(\mathcal{E})\otimes_{R}\mathbb{C}$,
and $\text{gr}_{1}(\mathcal{E})$ is torsion-free over $R$, so that
the image of $m_{1}$ in $\text{gr}_{1}(\mathcal{E})$ is also $0$.
So in fact $m_{1}\in V_{\leq0}(\mathcal{E})$ and we can rewrite \prettyref{eq:V-filt-rep}
with $r$ replaced by $r-1$. Continuing in this way, we eventually
an expression for $m$ as in \prettyref{eq:V-filt-rep} in which $r=i-1$,
which forces $m\in V_{\leq i}(\mathcal{E})$. So we see that $V_{\leq i}(\mathcal{E})=V_{\leq i}(\mathcal{E}_{\mathbb{C}})\cap\mathcal{E}$
for all $i\geq1$. Further, if $i=0$ we see that $m\in V_{\leq1}(\mathcal{E})\cap V_{\leq0}(\mathcal{E}_{\mathbb{C}})$.
This implies $m\in V_{\leq0}(\mathcal{E})$ since $\text{gr}_{1}(\mathcal{E}_{\mathbb{C}})=\text{gr}_{1}(\mathcal{E})\otimes_{R}\mathbb{C}$
and $\text{gr}_{1}(\mathcal{E})$ is free over $R$. We can then repeat
the argument for $i=-1$ to deduce $V_{\leq-1}(\mathcal{E})=V_{\leq-1}(\mathcal{E}_{\mathbb{C}})\cap\mathcal{E}$. 

We now have to show $V_{\leq i}(\mathcal{E})=V_{\leq i}(\mathcal{E}_{\mathbb{C}})\cap\mathcal{E}$
for $i<-1$. By definition $V_{\leq i}(\mathcal{E})=t^{-i-1}\cdot V_{\leq-1}(\mathcal{E})$
for any such $i$, and since multiplication by $t$ is injective on
all of $V_{\leq-1}(\mathcal{E})$, we see that $t^{-i-1}\cdot$ is
an isomorphism $V_{\leq-1}(\mathcal{E})\tilde{\to}V_{\leq i}(\mathcal{E})$.
The same holds for $V_{\leq i}(\mathcal{E}_{\mathbb{C}})$. So we
see that each $V_{\leq i}(\mathcal{E})$ is free over $R$ and we
have $V_{\leq i}(\mathcal{E})\otimes_{R}\mathbb{C}\tilde{\to}V_{\leq i}(\mathcal{E}_{\mathbb{C}})$
and $\text{gr}_{i}(\mathcal{E})\otimes_{R}\mathbb{C}\tilde{\to}\text{gr}_{i}(\mathcal{E}_{\mathbb{C}})$.
Now we proceed as above, arguing by induction on $i$. 

To finish the proof of the lemma, we need to show that $\partial:\text{gr}_{i}(\mathcal{E})\to\text{gr}_{i+1}(\mathcal{E})$
is a bijection for $i\geq1$. It is onto by definition. Further, since
$V_{\leq i}(\mathcal{E})=V_{\leq i}(\mathcal{E}_{\mathbb{C}})\cap\mathcal{E}$,
we have that $\text{gr}_{i}(\mathcal{E})\to\text{gr}_{i}(\mathcal{E}_{\mathbb{C}})$
is injective, so the result follows from the analogous one for $\mathcal{E}_{\mathbb{C}}$. 
\end{proof}
Now we want to apply this to the inclusion $\mathcal{E}\to\mathcal{E}[t^{-1}]:=\mathcal{F}$.
Although the complexification $\mathcal{F}_{\mathbb{C}}$ is coherent,
the $\mathcal{D}_{Y}$-module $\mathcal{F}$ generally will not be.
However, it is coherent over $\mathcal{D}_{Y}[t^{-1}]$. We have 
\begin{lem}
\label{lem:Inclusion-Q-to-P} Let $V_{\leq i}(\mathcal{F})=V_{\leq i}(\mathcal{E})$
for all $i\leq1$, and $V_{\leq i}(\mathcal{F})=t^{-i+1}\cdot V_{\leq1}(\mathcal{F})=t^{-i+1}\cdot V_{\leq1}(\mathcal{E})$
for all $i>1$. Then, after possibly localizing and extending $R$,
this is an exhaustive filtration of $\mathcal{P}$ over $V_{\cdot}(\mathcal{D}_{Y})$.
Each $V_{\leq i}(\mathcal{F})$ is coherent over $V_{\leq0}(\mathcal{D}_{Y})$,
and $V_{\leq i}(\mathcal{F})=V_{\leq i}(\mathcal{F}_{\mathbb{C}})\cap\mathcal{F}=V_{\leq i}(\mathcal{E}_{\mathbb{C}})\cap\mathcal{F}$
for all $i$. In particular, the filtered inclusion $\mathcal{E}\to\mathcal{F}$
is strict. 
\end{lem}

\begin{proof}
Choose coherent $V_{\leq0}(\mathcal{D}_{Y})$ modules $V_{\leq i}(\mathcal{F})\subset\mathcal{F}$
(for $i=\{0,1,-1\}$) satisfying $V_{\leq i}(\mathcal{F})\otimes_{R}\mathbb{C}=V_{\leq i}(\mathcal{F}_{\mathbb{C}})$.
As in the previous proof, localizing $R$ as needed, we can suppose
$\text{gr}_{i}(\mathcal{F})$ is free over $R$, for $i\in\{0,1\}$,
that $\text{gr}_{i}(\mathcal{F})\otimes_{R}\mathbb{C}=\text{gr}_{i}(\mathcal{F}_{\mathbb{C}})$
(for $i\in\{0,1\}$) and that $t\cdot V_{\leq0}(\mathcal{F})=V_{\leq-1}(\mathcal{F})$.
Define $V_{\leq i}(\mathcal{F})=t^{-i+1}\cdot V_{\leq1}(\mathcal{F})$
for all $i>1$, and $V_{\leq i}(\mathcal{F})=t^{-i}\cdot V_{\leq0}(\mathcal{F})$
for all $i\leq0$. Clearly $V_{\leq i}(\mathcal{F})\subset V_{\leq i}(\mathcal{F}_{\mathbb{C}})\cap\mathcal{F}$
for all $i$. 

Now, suppose $x\in V_{\leq i}(\mathcal{F})$, and its image in $\mathcal{F}_{\mathbb{C}}$
is contained in $V_{\leq i-1}(\mathcal{F}_{\mathbb{C}})$. Multiplying
by an appropriate power of $t$, we may assume $i\in\{0,1\}$. But
then since $\text{gr}_{i}(\mathcal{F})$ is free over $R$, and $\text{gr}_{i}(\mathcal{F})\otimes_{R}\mathbb{C}=\text{gr}_{i}(\mathcal{F}_{\mathbb{C}})$
we see that $x\in V_{\leq i-1}(\mathcal{F})$. So $V_{\leq i}(\mathcal{F})=V_{\leq i}(\mathcal{F}_{\mathbb{C}})\cap\mathcal{F}$
for all $i$. Since $\mathcal{F}_{\mathbb{C}}=\mathcal{E}_{\mathbb{C}}$
and $V_{\leq i}(\mathcal{E})=V_{\leq i}(\mathcal{E}_{\mathbb{C}})\cap\mathcal{E}$,
we see that the filtered inclusion $\mathcal{E}\to\mathcal{F}$ is
strict. Further, after localizing $R$ again, we can suppose $V_{\leq1}(\mathcal{F})=V_{\leq1}(\mathcal{E})$
and $V_{\leq0}(\mathcal{F})=V_{\leq0}(\mathcal{E})$. This implies
the description of the filtration given in the lemma; further, since
$V_{\leq1}(\mathcal{F})$ is coherent over $V_{\leq0}(\mathcal{D}_{Y})$
we see that each $V_{\leq i}(\mathcal{F})$ is coherent over $V_{\leq0}(\mathcal{D}_{Y})$. 

Finally, note that 
\[
\bigcup_{i}V_{\leq i}(\mathcal{F})\subset\mathcal{F}
\]
is an inclusion of coherent $\mathcal{D}_{Y}[t^{-1}]$-modules, which
becomes an equality on passing to $\mathbb{C}$. So it is an equality
after localizing $R$, and the filtration is indeed exhaustive. 
\end{proof}
Now we turn to the proof of \ref{thm:Injectivity-for-M}. To set things
up, we recall that we are considering a coherent $\mathcal{D}_{X}$-module
$\mathcal{M}\subset j_{*}(\mathcal{M}_{U})$ which is a model for
$j_{!*}(\mathcal{M}_{U_{\mathbb{C}}})$. Now, the $D_{X_{R}}$-module
$j_{*}(\mathcal{M}_{U})$ is, in general, not coherent; however, it
becomes coherent after passing to $\text{Frac}(R)$. Therefore we
may consider a coherent $\mathcal{D}_{X}$-module 
\[
\int_{j}\mathcal{M}_{U}\subset j_{*}(\mathcal{M}_{U})
\]
which is a lattice for $j_{*}(\mathcal{M}_{U_{R}})\otimes_{R}F$ (where
$F=\text{Frac}(R)$). Then the main theorem of this section is 
\begin{thm}
\label{thm:Injectivity-for-j-push-M} For all $p>>0$, the $p$-adic
completion of the map ${\displaystyle \int_{j}\mathcal{M}_{U_{W(k)}}\to j_{*}(\mathcal{M}_{U_{W(k)}})}$
is injective. 
\end{thm}

We point out that this directly implies the 
\begin{proof}
(of \prettyref{thm:Injectivity-for-M}) Localizing $R$ if necessary,
we may suppose that there is an injection ${\displaystyle \mathcal{M}\to\int_{j}\mathcal{M}_{U}}$.
Since these are both coherent $\mathcal{D}_{X}$-modules, the Artin-Rees
lemma implies that the $p$-adic completion of this map is injective;
similarly, the $p$-adic completion of ${\displaystyle \mathcal{M}_{W(k)}\to\int_{j}\mathcal{M}_{U_{W(k)}}}$
is injective. By the above, the $p$-adic completion of ${\displaystyle \int_{j}\mathcal{M}_{U_{W(k)}}\to j_{*}(\mathcal{M}_{U_{W(k)}})}$
is injective also; since $\mathcal{M}_{U_{R}}$ is coherent over $\mathcal{D}_{U_{R}}$,
the $p$-adic completion of $\mathcal{M}_{U_{R}}\otimes_{R}W(k)$
is precisely $j_{*}(\mathcal{D}_{\mathfrak{U}_{W(k)}}\widehat{\otimes}_{\mathcal{D}_{U_{W(k)}}}\mathcal{M}_{U_{W(k)}})$,
and the result follows. 
\end{proof}
Recall that we denote by $g:X_{R}\to\mathbb{A}_{R}^{1}$ the function
whose zero set is the complement of $U_{R}$. Let $\Gamma\subset X_{R}\times_{R}\mathbb{A}_{R}^{1}$
denote the graph. We have $\Gamma\tilde{=}X_{R}$ and we let $i$
denote the inclusion of $\Gamma$ into $X_{R}\times_{R}\mathbb{A}_{R}^{1}$.
To analyze the inclusion ${\displaystyle \int_{j}\mathcal{M}_{U_{R}}}\subset j_{*}(\mathcal{M}_{U_{R}})$,
it will be necessary to consider the $\mathcal{D}$-module push-forward
under $i$. We have the 
\begin{lem}
\label{lem:First-Reduction}Suppose that the $p$-adic completion
of the map ${\displaystyle i_{*}\int_{j}\mathcal{M}_{U_{W(k)}}\to i_{*}j_{*}(\mathcal{M}_{U_{W(k)}})}$
is an injection. Then \prettyref{thm:Injectivity-for-j-push-M} is
true. 
\end{lem}

\begin{proof}
For each $r\geq0$, we have the submodule $\mathcal{N}_{r}:={\displaystyle p^{r}\cdot i_{*}j_{*}(\mathcal{M}_{U_{W(k)}})\cap i_{*}\int_{j}\mathcal{M}_{U_{W(k)}}}$,
and we denote by $\mathcal{N}_{r,m}$ (for $m\geq0$) the image of
$\mathcal{N}_{r}$ inside ${\displaystyle (i_{*}\int_{j}\mathcal{M}_{U_{W(k)}})/p^{m}}$.
Then the injectivity statement of the lemma is that, for each $m\geq0$,
we have 
\[
\bigcap_{r\geq0}\mathcal{N}_{r,m}=0
\]
Now, Let $y$ be a function on $X_{R}\times_{R}\mathbb{A}_{R}^{1}$
such that $\Gamma$ is the (scheme-theoretic) zero set of $y$; let
$\partial_{y}$ be a derivation such that $\partial_{y}(y)=1$. Then
we have isomorphisms (of $\mathcal{D}_{X_{R}}$-modules)
\[
i_{*}\int_{j}\mathcal{M}_{U_{W(k)}}\tilde{=}\bigoplus_{i=0}^{\infty}\int_{j}\mathcal{M}_{U_{W(k)}}\cdot\partial_{y}^{i}
\]
and 
\[
i_{*}j_{*}(\mathcal{M}_{U_{W(k)}})\tilde{=}\bigoplus_{i=0}^{\infty}j_{*}(\mathcal{M}_{U_{W(k)}})\cdot\partial_{y}^{i}
\]
Therefore, if the condition ${\displaystyle \bigcap_{r\geq0}\mathcal{N}_{r,m}=0}$
holds then the analogous statement must hold for the inclusion ${\displaystyle \int_{j}\mathcal{M}_{U_{W(k)}}\to j_{*}(\mathcal{M}_{U_{W(k)}})}$
as it is a summand (in the category of $W(k)$-modules) of the inclusion
${\displaystyle i_{*}\int_{j}\mathcal{M}_{U_{W(k)}}\to i_{*}j_{*}(\mathcal{M}_{U_{W(k)}})}$. 
\end{proof}
This lemma allows us to replace the function $g$ by the coordinate
function $t:X\times\mathbb{A}^{1}\to\mathbb{A}^{1}$. Let $Y=X\times\mathbb{A}^{1}$;
to ease notation we denote $\mathcal{P}:=i_{*}j_{*}(\mathcal{M}_{U})$
and ${\displaystyle \mathcal{Q}:=i_{*}\int_{j}(\mathcal{M}_{U})}$
; the natural injection $\mathcal{Q}\to\mathcal{P}$ will be called
$\iota$. We have that $\mathcal{Q}$ is finitely generated over $\mathcal{D}_{Y}$
and $\mathcal{P}=\mathcal{Q}[t^{-1}]$ is finitely generated over
$\mathcal{D}_{Y}[t^{-1}]$; in addition $\mathcal{P}\otimes_{R}F=\mathcal{Q}\otimes_{R}F$.
We consider the $V$-filtrations on these modules as developed above. 

To use them, we first prove the following general
\begin{lem}
Let $(\mathcal{E}_{W(k)},V)$ be a $p$-torsion-free $\mathcal{D}_{Y_{W(k)}}$-module,
filtered with respect to $V_{\cdot}(\mathcal{D}_{Y_{W(k)}})$. Let
${\displaystyle \text{Rees}(\mathcal{E}_{W(k)})=\bigoplus_{i=\infty}^{\infty}V_{\leq i}(\mathcal{E}_{W(k)})}$
be the associated Rees module, graded by putting $V_{\leq i}(\mathcal{E}_{W(k)})$
is degree $i$, and equipped with the action of the operator $\tau$
of degree $1$ which acts by including $V_{\leq i}(\mathcal{E}_{W(k)})$
in degree $i$ into $V_{\leq i+1}(\mathcal{E}_{W(k)})$ in degree
$i+1$. Then there is an embedding 
\[
\widehat{\text{Rees}(\mathcal{E}_{W(k)})}\to\prod_{i=-\infty}^{\infty}\widehat{(V_{\leq i}(\mathcal{E}_{W(k)}))}
\]
where $\widehat{?}$ denotes $p$-adic completion
\end{lem}

\begin{proof}
There is a canonical embedding 
\[
\iota:\bigoplus_{i=\infty}^{\infty}V_{\leq i}(\mathcal{E}_{W(k)})\to\prod_{i=-\infty}^{\infty}V_{\leq i}(\mathcal{E}_{W(k)})
\]
and we denote the cokernel by $\mathcal{C}$. An element $x=(x_{i})_{i\in\mathbb{Z}}$
is not in the image of $\iota$ iff $x_{i}\neq0$ for infinitely many
$i$. Since each component is torsion-free over $W(k)$, this implies
that $rx$ is not in the image of $\iota$ for all $r\in W(k)\backslash\{0\}$.
So $\mathcal{C}$ is torsion-free over $W(k)$ as well. Thus we obtain
short exact sequences 
\[
(\bigoplus_{i=\infty}^{\infty}V_{\leq i}(\mathcal{E}_{W(k)}))/p^{m}\to(\prod_{i=-\infty}^{\infty}V_{\leq i}(\mathcal{E}_{W(k)}))/p^{m}\to\mathcal{C}/p^{m}
\]
for all $m$. Taking the inverse limit, we see that 
\[
\widehat{(\bigoplus_{i=\infty}^{\infty}V_{\leq i}(\mathcal{E}_{W(k)}))}\to\widehat{(\prod_{i=-\infty}^{\infty}V_{\leq i}(\mathcal{E}_{W(k)}))}
\]
is an injection. In addition, one sees directly that the natural map
\[
(\prod_{i=-\infty}^{\infty}V_{\leq i}(\mathcal{E}_{W(k)}))/p^{m}\to(\prod_{i=-\infty}^{\infty}V_{\leq i}(\mathcal{E}_{W(k)})/p^{m})
\]
is an isomorphism. A moment's thought also shows that the map
\[
\prod_{i=-\infty}^{\infty}\widehat{(V_{\leq i}(\mathcal{E}_{W(k)}))}\to\lim_{m}(\prod_{i=-\infty}^{\infty}V_{\leq i}(\mathcal{E}_{W(k)})/p^{m})
\]
is an isomorphism. Thus we obtain 
\[
\widehat{(\prod_{i=-\infty}^{\infty}V_{\leq i}(\mathcal{E}_{W(k)}))}\tilde{=}\prod_{i=-\infty}^{\infty}\widehat{(V_{\leq i}(\mathcal{E}_{W(k)}))}
\]
whence the result. 
\end{proof}
Now we can give the 
\begin{proof}
(of \prettyref{thm:Injectivity-for-j-push-M}) Let $\mathcal{C}$
be the cokernel of $\mathcal{Q}_{W(k)}\to\mathcal{P}_{W(k)}$. We
have ${\displaystyle \text{ker}(\widehat{\mathcal{Q}_{W(k)}}\to\widehat{\mathcal{P}_{W(k)}})=\lim_{\leftarrow}\mathcal{C}[p^{n}]}$
where $\mathcal{C}[p^{n}]$ is the $p^{n}$-torsion submodule. Let
$V_{\leq i}(\mathcal{Q}_{W(k)}):=V_{\leq i}(\mathcal{Q})\otimes_{R}W(k)$
and similarly for $\mathcal{P}$. Then these are exhaustive, good
filtrations of $\mathcal{Q}_{W(k)}$ (and $\mathcal{P}_{W(k)}$) over
$V_{\cdot}(\mathcal{D}_{Y_{W(k)}})$. Since $R\to W(k)$ is flat,
we see that $\mathcal{Q}_{W(k)}\to\mathcal{P}_{W(k)}$ is a strict
inclusion with respect to these filtrations.

For each $i$, we have that $V_{\leq i}(\mathcal{Q}_{W(k)})\subset V_{\leq i}(\mathcal{P}_{W(k)})$
is an inclusion of finite $V_{\leq0}(\mathcal{D}_{Y_{R}})$-modules.
Thus the Artin-Rees lemma implies 
\[
\widehat{V_{\leq i}(\mathcal{Q}_{W(k)})}\to\widehat{V_{\leq i}(\mathcal{P}_{W(k)})}
\]
is also injective; hence the same is true of the natural map 
\[
\prod_{i=-\infty}^{\infty}\widehat{V_{\leq i}(\mathcal{Q}_{W(k)})}\to\prod_{i=-\infty}^{\infty}\widehat{V_{\leq i}(\mathcal{P}_{W(k)})}
\]
Therefore the previous lemma now implies ${\displaystyle \widehat{\text{Rees}(\mathcal{Q}_{W(k)})}\to\widehat{\text{Rees}(\mathcal{P}_{W(k)})}}$
is injective. 

Since $\mathcal{Q}_{W(k)}\to\mathcal{P}_{W(k)}$ is strict, $\text{Rees}(\mathcal{C})$
is the cokernel of $\text{Rees}(\mathcal{Q}_{W(k)})\to\text{Rees}(\mathcal{P}_{W(k)})$.
So the above implies ${\displaystyle \lim_{\leftarrow}\text{Rees}(\mathcal{C})[p^{n}]}=0$.
We shall show that this implies ${\displaystyle \lim_{\leftarrow}\mathcal{C}[p^{n}]=0}$.
For each $n$ we have a short exact sequence
\[
\text{Rees}(\mathcal{C})[p^{n}]\xrightarrow{\tau-1}\text{Rees}(\mathcal{C})[p^{n}]\to\mathcal{C}[p^{n}]
\]
Since ${\displaystyle \lim_{\leftarrow}\text{Rees}(\mathcal{C})[p^{n}]}=0$,
to prove ${\displaystyle \lim_{\leftarrow}\mathcal{C}[p^{n}]=0}$
we must show that $\tau-1$ acts injectively on ${\displaystyle \text{R}^{1}\lim_{\leftarrow}\text{Rees}(\mathcal{C})[p^{n}]}$.
Recall that this module is the cokernel of 
\[
\eta:\prod_{n=1}^{\infty}\text{Rees}(C)[p^{n}]\to\prod_{n=1}^{\infty}\text{Rees}(C)[p^{n}]
\]
where $\eta(c_{1},c_{2},c_{3},\dots)=(c_{1}-pc_{2},c_{2}-pc_{3},\dots)$.
Now, since each $\text{Rees}(C)[p^{n}]$ is graded, we may define
a homogenous element of degree $i$ in ${\displaystyle \prod_{n=1}^{\infty}\text{Rees}(C)[p^{n}]}$
to be an element $(c_{1},c_{2},\dots)$ such that each $c_{j}$ has
degree $i$. Any element of $d\in{\displaystyle \prod_{n=1}^{\infty}\text{Rees}(C)[p^{n}]}$
has a unique representation of the form ${\displaystyle \sum_{i=0}^{\infty}d_{i}}$
where $d_{i}$ is homogenous of degree $i$ (this follows by looking
at the decomposition by grading of each component). Since the map
$\eta$ preserves the set of homogenous elements of degree $i$, we
have ${\displaystyle \eta(\sum_{i=0}^{\infty}d_{i})=\sum_{i=0}^{\infty}\eta(d_{i})}$. 

Suppose that $(\tau-1)d=\eta(d')$. Write ${\displaystyle d=\sum_{i=j}^{\infty}d_{i}}$
where $d_{j}\neq0$. Then 
\[
(\tau-1){\displaystyle \sum_{i=j}^{\infty}d_{j}}=-d_{j}+\sum_{i=j+1}^{\infty}(\tau d_{i-1}-d_{i})=\sum_{i=0}^{\infty}\eta(d_{i}')
\]
So we obtain $d_{j}=-\eta(d_{j}')$, and $d_{i}=\tau d_{i-1}+\eta(d_{i}')$
for all $i>j$, which immediately gives $d_{i}\in\text{image}(\eta)$
for all $i$; so $d\in\text{image}(\eta)$ and $\tau-1$ acts injectively
on $\text{coker}(\eta)$ as required. 
\end{proof}
Now, we'll give the other key application of the $V$-filtration theory: 
\begin{lem}
\label{lem:p-adically-seperated}Let $\mathcal{E}$ be a $\mathcal{D}_{\mathbb{A}^{n}}$
module such that $\mathcal{E}_{\mathbb{C}}$ is holonomic. Then, after
possibly localizing $R$, we have that, $\mathbb{H}_{dR}^{i}(\mathcal{E})$
is a direct sum of finite type $R$-modukes. For all $R\to W(k)$,
the same is true of $\mathbb{H}_{dR}^{i}(\mathcal{E}_{W(k)})$. Therefore
this $W(k)$-module is is $p$-adically separated for all. 

The statement of the lemma holds also over any affine algebraic variety.
\end{lem}

\begin{proof}
The second sentence follows from the first by choosing an embedding
into affine space. So we assume we are working over $\mathbb{A}^{n}$
from now on.

After taking the Fourier transform of $\mathcal{E}$, it suffices
to show that $L^{j}\iota^{*}(\mathcal{E}_{W(k)})$ is $p$-adically
seperated for each $j$, where $\iota:\{0\}\to\mathbb{A}_{W(k)}^{n}$
is the inclusion. To see this, we'll use the V-filtration on $\mathcal{E}$
as presented above in . Let $\{t_{1},\dots,t_{n}\}$ be the coordinates
on $\mathbb{A}^{n}$; we thus have $n$ distinct $V$-filtrations
on $\mathcal{D}_{\mathbb{A}^{n}}$ (each of which is actually a grading),
and the $i$th one has $\text{deg}(t_{i})=-1$ and $\text{deg}(\partial_{i})=1$. 

Consider $\mathcal{E}$ with the $V$ filtration with respect to $t_{1}$.
We have that $t_{1}:V_{\leq0}(\mathcal{E})\to V_{\leq-1}(\mathcal{E})$
is an isomorphism. Therefore the complex $\mathcal{E}\xrightarrow{t_{1}}\mathcal{E}$
is quasi-isomorphic to $\mathcal{E}/V_{\leq0}(\mathcal{E})\xrightarrow{t_{1}}\mathcal{E}/V_{\leq-1}(\mathcal{E})$.
Now, we claim that the $V$-filtration is split on both $\mathcal{E}/V_{\leq0}(\mathcal{E})$
and $\mathcal{E}/V_{\leq-1}(\mathcal{E})$. To see this, we localize
$R$ so that all of the differences between the roots of the characteristic
polynomial of $t_{1}d_{1}$ on $\text{gr}_{0}(\mathcal{E})$ and $\text{gr}_{1}(\mathcal{E})$
are units in $R$. Then the module $V_{\leq1}(\mathcal{E})/V_{\leq-1}(\mathcal{E})$
splits as $\text{gr}_{0}(\mathcal{E})\oplus\text{gr}_{1}(\mathcal{E})$
by using the eigenspace decomposition. Now we can split the filtration
on any $V_{\leq i}(\mathcal{E})/V_{\leq-1}(\mathcal{E})$ by using
the subspaces $d_{1}^{j}\cdot\text{gr}_{1}(\mathcal{E})$ for $0\leq j\leq i$. 

Therefore, if we denote by $\text{gr}^{(1)}(\mathcal{E})$ the associated
graded of $\mathcal{E}$ with respect to this $V$ filtration, we
also obtain that $\mathcal{E}/V_{\leq0}(\mathcal{E})\xrightarrow{t_{1}}\mathcal{E}/V_{\leq-1}(\mathcal{E})$
is quasi-isomorphic to $\text{gr}^{(1)}(\mathcal{E})\xrightarrow{t_{1}}\text{gr}^{(1)}(\mathcal{E})$. 

If $\iota_{1}$ denotes the inclusion $\{t_{1}=0\}\to\mathbb{A}^{n}$,
we have shown $L\iota_{1}^{*}(\mathcal{E})\tilde{=}L\iota_{1}^{*}(\text{gr}^{(1)}(\mathcal{E}))$.
It follows that $Li^{*}(\mathcal{E})\tilde{=}Li^{*}(\text{gr}^{(1)}(\mathcal{E}))$. 

Now, each graded piece $\text{gr}_{i}^{(1)}(\mathcal{E})$ is a coherent
$R<t_{2},\dots,t_{n},\partial_{2},\dots,\partial_{n}>$ module (which
is holonomic upon passing to $\mathbb{C}$) and we have isomorphisms
$\text{gr}_{i}^{(1)}(\mathcal{E})\tilde{=}\text{gr}_{0}^{(1)}(\mathcal{E})$
for all $i<0$ and $\text{gr}_{i}^{(1)}(\mathcal{E})\tilde{=}\text{gr}_{1}^{(1)}(\mathcal{E})$
for all $i>1$. Thus we may choose simultaneous $R$-models for each
$\text{gr}_{i}^{(1)}(\mathcal{E})$ which possess $V$-filtrations
with respect to $t_{2}$. Setting 
\[
V_{\leq t}(\text{gr}^{(1)}(\mathcal{E}))=\bigoplus_{i}V_{\leq t}(\text{gr}_{i}^{(1)}(\mathcal{E}))
\]
then puts a $V$-filtration (with respect to $t_{2}$) on $\text{gr}^{(1)}(\mathcal{E})$.
Arguing as above we obtain $Li^{*}(\text{gr}^{(1)}(\mathcal{E}))\tilde{=}Li^{*}(\text{gr}^{(2)}(\text{gr}^{(1)}(\mathcal{E})))$;
and $(\text{gr}^{(2)}(\text{gr}^{(1)}(\mathcal{E}))$ is a $\mathbb{Z}^{2}$-graded
module, with each $(\text{gr}_{i}^{(2)}(\text{gr}_{j}^{(1)}(\mathcal{E}))$
coherent over $R<t_{3},\dots,t_{n},\partial_{3},\dots,\partial_{n}>$
(and holonomic upon passing to $\mathbb{C}$). 

Continuing in this way, we deduce 
\[
Li^{*}(\mathcal{E})\tilde{=}Li^{*}(\text{gr}^{(1)}(\mathcal{E}))\tilde{=}Li^{*}(\text{gr}^{(2)}(\text{gr}^{(1)}(\mathcal{E})))\tilde{=}\dots\tilde{=}Li^{*}(\text{gr}^{(n)}(\cdots\text{gr}^{(2)}(\text{gr}^{(1)}(\mathcal{E})))
\]
The last module here is a $\mathbb{Z}^{n}$-graded module, which has
the property that each multi-degree is finite and flat over $R$.
Base changing to $W(k)$, we obtain 
\[
Li^{*}(\mathcal{E}_{W(k)})\tilde{=}Li^{*}(\text{gr}^{(n)}(\cdots\text{gr}^{(2)}(\text{gr}^{(1)}(\mathcal{E}_{W(k)})))
\]
 The last complex can be computed via the Koszul complex with respect
to $\{t_{1},\dots,t_{n}\}$, which has homogenous differentials (after
shifting the grading appropriately). Thus we see that each $L^{j}\iota^{*}(\mathcal{E}_{W(k)})$
is a direct sum of modules whose $p$-torsion is bounded; and the
result follows immediately. 
\end{proof}

University of Illinois. csdodd2@illinois.edu
\end{document}